%% file: Main.tex
\documentclass[10 pt, singlespacing, headsepline]{book}
\usepackage[numbers]{natbib}
\usepackage[unicode]{hyperref}
\usepackage{bookmark}
\usepackage{lipsum}
\usepackage[utf8]{inputenc}
\usepackage{soul}
\usepackage[T1]{fontenc}
\usepackage{tikz, tikz-cd, enumerate, amsfonts, stmaryrd, amsmath, amsthm, amssymb, comment, todonotes, mathrsfs, mathtools}
\usepackage[shortlabels]{enumitem}
\usetikzlibrary{arrows}
\usetikzlibrary{positioning}
\usepackage[utf8]{inputenc}
\usepackage{bbm}
\usepackage{graphicx}
\usepackage[all]{xypic}
\usepackage[toc,page,title,titletoc,header]{appendix}
\usepackage{xcolor, color}
\usepackage[all,color]{xy}
\usepackage[left=1in, right=1in, top=1in, bottom=1in, includefoot, headheight=13.6pt]{geometry}
\usetikzlibrary{matrix,arrows,decorations.pathmorphing}
\usepackage{mdframed}
\usepackage{eucal}
\usetikzlibrary{positioning}
\usepackage{fancyhdr}
\usepackage{longtable}
\usepackage[french, english]{babel}
\usepackage{imakeidx}
\usepackage{hyperref}
\usetikzlibrary{decorations.markings}
\usepackage{authblk}

\input xy
\xyoption{all}

\newcommand{\bb}{\textbf}

\newcommand{\ft}{\footnote}
\newcommand{\la}{\langle}
\newcommand{\ra}{\rangle}
\newcommand{\ie}{\textit{i.e. }}
\newcommand{\cf}{\textit{cf }}

\newtheorem{theorem}[subsection]{Theorem}
\newtheorem{lemma}[subsection]{Lemma}
\newtheorem{proposition}[subsection]{Proposition}
\newtheorem{corollary}[subsection]{Corollary}

\newtheorem{example}[subsection]{Example}

\newtheorem*{theoremetoile}{Theorem}

\theoremstyle{definition}

\newtheorem{definition}[subsection]{Definition}

\newtheorem{notation}[subsection]{Notation}

\newmdtheoremenv{lem}{Lemma}
\newmdtheoremenv{pro}{Proposition}
\newmdtheoremenv{cor}{Corollary}
\newmdtheoremenv{rem}{Remark}
\newmdtheoremenv{ide}{Idea}

\newtheorem{lemm}[subsection]{Lemma}
\newtheorem{prop}[subsection]{Proposition}
\newtheorem{coro}[subsection]{Corollary}

\theoremstyle{definition}

\newtheorem{remark}[subsection]{Remark}

\DeclareMathOperator{\W}{\mathsf{W}}

\DeclareMathOperator{\ZZ}{\mathbf{Z}} 
\DeclareMathOperator{\QQ}{\mathbf{Q}} 
\DeclareMathOperator{\FF}{\mathbf{F}}
\DeclareMathOperator{\NN}{\mathbf{N}} 
\DeclareMathOperator{\RR}{\mathbf{R}} 
 
\DeclareMathOperator{\BB}{\mathbf{B}}
\renewcommand{\AA}{\mathbf{A}}
\renewcommand{\BB}{\mathbf{B}}
\DeclareMathOperator{\EE}{\mathbf{E}}
\DeclareMathOperator{\T}{\mathsf{T}}

\DeclareMathOperator{\R}{\mathsf{R}}

\DeclareMathOperator{\Mm}{\mathfrak{M}}
\DeclareMathOperator{\Oo}{\mathcal{O}}
\DeclareMathOperator{\Oe}{\mathcal{O}_{\mathcal{E}}}

\DeclareMathOperator{\Ee}{\mathcal{E}}
\DeclareMathOperator{\Rr}{\mathcal{R}}
\DeclareMathOperator{\DD}{\mathbb{D}}
\DeclareMathOperator{\mm}{\mathfrak{m}}

\DeclareMathOperator{\Ss}{\mathfrak{S}}
\DeclareMathOperator{\Ff}{\mathcal{F}}

\DeclareMathOperator{\Ext}{\mathsf{Ext}}

\DeclareMathOperator{\Id}{\mathsf{Id}}

\DeclareMathOperator{\Mod}{\mathbf{Mod}}

\DeclareMathOperator{\Ab}{\mathbf{Ab}}

\DeclareMathOperator{\id}{id}

\DeclareMathOperator{\Cc}{\mathcal{C}}

\DeclareMathOperator{\BT}{\mathbf{BT}}
\DeclareMathOperator{\rank}{rank}
\DeclareMathOperator{\res}{res}

\DeclareMathOperator{\inj}{\hookrightarrow}

\DeclareMathOperator{\Tr}{Tr}

\DeclareMathOperator{\Ind}{\mathsf{Ind}}

\DeclareMathOperator{\cris}{cris}

\DeclareMathOperator{\dR}{dR}

\DeclareMathOperator{\Frac}{\mathsf{Frac}} 
\DeclareMathOperator{\Gal}{\mathsf{Gal}}

\DeclareMathOperator{\Ker}{\mathsf{Ker}}
\DeclareMathOperator{\im}{\mathsf{Im}}
\DeclareMathOperator{\Coker}{\mathsf{Coker}}

\DeclareMathOperator{\GL}{\mathsf{GL}}

\DeclareMathOperator{\Hom}{\mathsf{Hom}}
 
\DeclareMathOperator{\Rep}{\mathbf{Rep}}

\DeclareMathOperator{\rad}{rad}
\DeclareMathOperator{\sep}{sep}
\DeclareMathOperator{\ur}{ur}
\DeclareMathOperator{\tors}{tors}

\DeclareMathOperator{\TR}{\mathsf{TR}}

\DeclareMathOperator{\naif}{\mathsf{naive}}
\DeclareMathOperator{\np}{\mathsf{np}}
\DeclareMathOperator{\pa}{\mathsf{pa}}

\renewcommand{\H}{\mathsf{H}} 
\newcommand{\rig}{\mathsf{rig}}

\newcommand{\Lim}[1]{\raisebox{0.5ex}{\scalebox{0.8}{$\displaystyle \lim_{#1}\;$}}}
\newcommand{\pLim}[1]{\raisebox{0.5ex}{\scalebox{0.8}{$\displaystyle \varprojlim_{#1}\;$}}}
\newcommand{\iLim}[1]{\raisebox{0.5ex}{\scalebox{0.8}{$\displaystyle \varinjlim_{#1}\;$}}}

\DeclareMathOperator{\HT}{\mathsf{HT}}

\renewcommand{\H}{\mathsf{H}} %vecteurs de Witt

\def \ooverline #1#2#3%
{\mkern #1mu \overline{\mkern -#1mu #3 \mkern -#2mu }\mkern #2mu }

\def \Kbar {\ooverline40K{\mkern 1mu}{}}

\makeindex

\author{Luming ZHAO
}

\newcommand{\Addresses}{{% additional braces for segregating \footnotesize
  \bigskip
  \footnotesize

  \textsc{IMB, CNRS UMR 5251, Université Bordeaux, 33405 Talence, France }\par\nopagebreak
  \textit{E-mail address}: \texttt{luming.zhao@math.u-bordeaux.fr}

}}
  
\begin{document}
	\frontmatter
	
	\pagestyle{plain}
	
\title{
\textbf{Galois cohomology of $p$-adic fields and $(\varphi, \tau)$-modules}
}
\author{Luming ZHAO}	
\date{}

\maketitle

	\chapter*{Abstract}
	
	Let $p$ be an odd prime number and $K$ a complete discrete valuation field of characteristic $0$, with perfect residue field of characteristic $p$. The goal of this work is to build complexes, defined in terms of invariants attached to a $p$-adic representation of the absolute Galois group of $K$, and whose homology is isomorphic to the Galois cohomology of the representation. In his thesis, Herr constructed such a three-term complex using the $(\varphi,\Gamma)$-module associated to the representation (defined from the cyclotomic extension of $K$). For many questions, however, it is useful to work with a Breuil-Kisin extension, obtained from $K$ by adding to it a compatible system of $p^n$-th roots of a uniformizer of $K$. An essential difference (and a notable difficulty) compared to the cyclotomic theory is that the extension obtained is not Galois. A natural solution, provided by Tavares Ribeiro in his thesis, is to work with the composite extension of the cyclotomic extension with a Breuil-Kisin extension and the corresponding Galois group, which provides a four-term complex. Since then, Caruso has developed the theory of $(\varphi,\tau)$-modules, which are to a Breuil-Kisin extension what $(\varphi,\Gamma)$-modules are to the cyclotomic extension: they provide a complete classification of $p$-adic representations (integral or not). Our first result is the construction of a three-term complex, defined in terms of the $(\varphi,\tau)$-module of a $p$-adic representation, and whose homology is isomorphic to the Galois cohomology of the representation. We prove that it refines that of Tavares Ribeiro in the finite residue field case, by building a quasi-isomorphism between the two. Then, we construct an operator $\psi$ (analogous to the one existing in the cyclotomic theory), and show that in our complex, we can substitute the Frobenius operator with it. Using the overconvergence of $(\varphi,\tau)$-modules (proved by Gao-Poyeton, and which we refine for integral representations), we define overconvergent versions of our complexes, and prove that they calculate the correct $\H^0$ and $\H^1$. Moreover, using Poyeton's results, we construct a complex over the Robba ring, simpler than the previous ones (the $\tau$ operator is replaced by a derivation), and whose $\H^0$ and $\H^1$ are isomorphic to the inductive limits of the Galois cohomology $\H^0$ and $\H^1$ along a Breuil-Kisin extension. Finally, we apply the above to the computation of Galois cohomology of the Tate module of a $p$-divisible group over the ring of integers of $K$, in terms of its associated Breuil-Kisin module.

	\tableofcontents	
	\mainmatter	  % Begin normal, numeric (1,2,3...) page numbering
	\pagestyle{fancy}
	\fancyhf{}
	\fancyhead[RE]{\slshape\nouppercase{\leftmark}}
	\fancyhead[LO]{\slshape\nouppercase{\rightmark}}
	\fancyfoot[C]{\thepage}
	\renewcommand{\headrulewidth}{0.53pt}
	\cleardoublepage
	\phantomsection
	\addcontentsline{toc}{chapter}{Introduction}
	%\include{Chapters/Introduction(French)}
	\include{Introduction}

	\cleardoublepage
	\phantomsection
	\addcontentsline{toc}{chapter}{Notations}
	\include{Notations}

\include{Chapter1}

\include{Chapter2}

\include{Chapter3}

\include{Chapter4}

\include{Chapter5}

\include{Chapter6}

	\cleardoublepage
	\phantomsection
	\addcontentsline{toc}{chapter}{Index}
	\printindex

	\cleardoublepage
	\phantomsection
	\addcontentsline{toc}{chapter}{Biblography}

	\Addresses
	
	%---------------------------------------------------------------------------------------
\end{document}

%% file: Introduction.tex
% Introduction

\chapter*{Introduction} % Main chapter title

\label{Introduction} % Change X to a consecutive number; for referencing this chapter elsewhere, use \ref{ChapterX}

Let $p$ be a prime number and $K$ a complete discrete valuation field, of characteristic $0$, whose residue field $k$ is perfect of characteristic $p$. Fix an algebraic closure $\Kbar$ of $K$ and put $\mathscr{G}_K=\Gal(\Kbar/K)$. For any strictly arithmetic profinite subextension $K_\infty/K$ of $\Kbar/K$, Fontaine and Wintenberger have associated perfect (the tilt of the completion of $K_\infty$) and imperfect (\cf \cite{Win83}) fields of norms. The latter is isomorphic to a field of formal Laurent series in one variable with coefficients in a finite extension of $k$, and its absolute Galois group is isomorphic to $\Gal(\Kbar/K_\infty)$.

\smallskip

The most studied case is the one where $K_{\infty}$ is the cyclotomic extension: fix $\varepsilon=(\zeta_{p^n})_{n\in \NN}$ a compatible system of primitive $p^n$-th roots of unity, and assume $K_\infty=K_\zeta:=\bigcup \limits_{n=0}^\infty K(\zeta_{p^n}).$ The extension $K_\infty/K$ is then Galois and its Galois group $\Gamma$ is identified, via the cyclotomic character, with an open subgroup of $\ZZ_p^\times.$ In \cite{Fon90}, Fontaine lifted the imperfect field of norms to characteristic $0$: he constructed a Cohen ring $\mathbf{A}_K$ endowed with a lifted Frobenius map and an action of the group $\Gamma$ which commutes with Frobenius. It is a sub-ring of the ring of formal Laurent series in one variable with coefficients in a finite extension of the ring of Witt vectors $\W(k)$, and which converges on the annulus of thickness $0$ and radius $1.$ This allowed him to build an equivalence of categories between the category of integral $p$-adic representations of $\mathscr{G}_K$ (constituted by $\ZZ_p$-modules of finite type with a continuous linear action of $\mathscr{G}_K$) and the category of \'etale $(\varphi, \Gamma)$-modules on $\AA_K:$  these are $\mathbf{A}_K$-modules of finite type endowed with a semi-linear Frobenius endomorphism, and a semi-linear action of $\Gamma$ which commutes with Frobenius ("\'etale" means that the linearization of Frobenius is an isomorphism).

\smallskip

This theory has been refined by Cherbonnier-Colmez in \cite{CC98}, in which they showed the overconvergence of $p$-adic representations. This means that in the theory of $(\varphi,\Gamma)$-modules which precedes, one can replace the ring $\mathbf{A}_K$ by the sub-ring $\mathbf{A}_K^\dagger$ constituted by overconvergent elements, that is to say, which converge on an annulus of non-zero thickness and outer radius $1$ (and $\mathbf{A}_K$ is its $p$-adic completion). One of the interests of this refinement is that it allows to relate the $(\varphi,\Gamma)$-module associated to a $p$-adic representation to its invariants coming from $p$-adic Hodge theory (\cf \cite{Ber02}, \cite{Fon04}, \cite{Ber08} and \cite{BC08}). The infinitesimal action of $\Gamma$ allows in particular to endow the $(\varphi,\Gamma)$-module over the Robba ring with a connection, thanks to which Berger proved that Crew's conjecture (proved independently by Andr\'e, Kedlaya and Mebkhout) implies Fontaine's $p$-adic monodromy conjecture (\cf \cite{Ber02}). 

\smallskip

The theory of $(\varphi,\Gamma)$-modules has many applications: the Langlands correspondence for $\GL_2(\QQ_p)$ (constructed by Colmez, \cf \cite{Col10}), and whose generalization is one of the main motivations of recent developments of the theory, reciprocity laws, theory of $p$-adic $L$ functions and Iwasawa theory (\cf \cite{Ben00}, \cite{Ber03}, \cite{Ben11}). Logically, one can compute the Galois cohomology of a $p$-adic representation directly from its $(\varphi, \Gamma)$-module using a simple three-terms complex: 

\begin{theoremetoile}
	\textup{(Herr, \cite{Her98})} Let $T$ be an integral $p$-adic representation of $\mathscr{G}_K$ and $D$ the associated $(\varphi,\Gamma)$-module. Assume $\Gamma$ is topologically generated by an element $\gamma$. Then the homology of the complex
	$$\xymatrix{
		0\ar[rrr] &&& D\ar[r] & D\oplus D\ar[r] & D\ar[r] & 0\\
		&&& x\ar@{|->}[r] & ((\varphi-1)(x),(\gamma-1)(x)) & & \\
		&&& & (y,z)\ar@{|->}[r] & (\gamma-1)(y)-(\varphi-1)(z) & }$$
	is canonically isomorphic to the Galois cohomology of $T$.
\end{theoremetoile}

 This result allowed him to reprove Tate's duality theorem (\cf \cite{Her01}). The analogues with $(\varphi, \Gamma)$-modules over overconvergent rings and over the Robba ring are also valid (see \cite{Ben14}, \cite{RLiu08}).

\smallskip

The theory of $(\varphi,\Gamma)$-modules has been developed in many directions: in the relative case by Andreatta (\cf \cite{And06}), for Lubin-Tate extensions (\cf \cite{Ber13}, \cite{Ber16}), by Schneider \emph{et al.} (the latter is strongly motivated by the Langlands $p$-adic correspondence).

\medskip

Since the work of Breuil (\cf \cite{Bre00}) and Kisin (\cf \cite{Kis06}), it has been clear that for many questions (applications to $p$-divisible groups, \og integral\fg{} $p$-adic Hodge theory, study of the deformations of $p$-adic representations), it is judicious to work with other deeply ramified extensions: the Breuil-Kisin extensions. To construct them, we fix a uniformizer $\pi$ of $K$, and a compatible system $\widetilde{\pi}=(\pi_n)_{n\in\NN}$ of $p^n$-roots of $\pi$ (\ie such that $\pi_0=\pi$ and $\pi_{n+1}^p=\pi_n$ for all $n\in\NN$): the associated extension is then $K_\infty=K_\pi:=\bigcup\limits_{n=0}^\infty K(\pi_n)$. Again, one can construct a Cohen ring (which we will note $\mathcal{O}_{\mathcal{E}}$) for the corresponding field of norms, but unlike the cyclotomic case, the extension $K_\pi/K$ is not Galois, so that a Galois action on the \'etale $\varphi$-modules is \og missing\fg, which is needed to associate $p$-adic representations in this framework. Nevertheless, this allowed Kisin to associate invariants (bundles over domains in the open unit disk) to certain semi-stable representations, which allowed him to classify crystalline representations, $p$-divisible groups and finite flat group schemes over $\mathcal{O}_K$, and to show that crystalline representations of Hodge-Tate weights $0$ and $1$ all come from $p$-divisible groups. Since then, this theory has given rise to an abundant literature, notably under the impulse of Caruso and T. Liu (\cf \cite{Liu07}, \cite{Liu08}, \cite{CL09}, \cite{Liu10}, \cite{CL11}, \cite{Car11}, \cite{Liu12}, \cite{Car13}, \cite{Liu13}, \cite{Liu15a}, \cite{Liu15b}, \cite{GLS15}, \cite{CL16}, \cite{CL19}, \cite{GL20}).

\smallskip

In this context, it is natural to try to compute the Galois cohomology of a $p$-adic representation from its invariants attached to a Breuil-Kisin extension, by means of a complex similar to the Herr complex mentioned above. This is the first objective of our work. As we have seen, the main obstacle is that the extension $K_\pi/K$ is not Galois. It is natural to consider the Galois closure of $K_\pi$ in $\Kbar$: it is the compositum $L=K_\pi K_\zeta$ of $K_\pi$ with the cyclotomic extension. The Galois group $\Gal(L/K)$ is then a semi-direct product of $\ZZ_p(1)$ by an open subgroup of $\ZZ_p^{\times}$ (the extension $L/K$ is said to be m\'etab\'elian). This point of view is used by Tavares Ribeiro in his thesis (\cf \cite{Tav11}), in which he constructed a theory analogous to that of $(\varphi,\Gamma)$-modules, and a four-term complex computing the Galois cohomology of the representation.

\begin{theoremetoile}
	\textup{(Tavares Ribeiro, \cite[\S 1.5]{Tav11})} Suppose that $\Gamma$ is topologically generated by an element $\gamma$ and let $\tau$ be a topological generator of $\Gal(L/K_\zeta)$. Let $T$ be an integral $p$-adic representation of $\mathscr{G}_K$ and let $M=D_L(T)=\big(\mathcal{O}_{\widehat{\Ee^{\ur}}}\otimes_{\ZZ_p}T  \big)^{\mathscr{G}_L}$. Then the homology of the complex
	\[0\to M \xrightarrow{\tilde{\alpha}} M\oplus M\oplus M \xrightarrow{\tilde{\beta}}M\oplus M\oplus M\xrightarrow{\tilde{\eta}}M \to 0 \] 
	where
	\[	\tilde{\alpha}=\begin{pmatrix}
	\varphi-1\\
	\gamma-1\\
	\tau-1
	\end{pmatrix}  
	, \ 
	\tilde{\beta}=\begin{pmatrix}
	\gamma-1 &1-\varphi& 0\\
	\tau-1  &0&  1-\varphi\\
	0 &\tau^{\chi(\gamma)}-1& \delta-\gamma
	\end{pmatrix}   \]
	\[ \tilde{\eta}=\begin{pmatrix}
	\tau^{\chi(\gamma)}-1, \delta-\gamma, \varphi-1
	\end{pmatrix}  \]
	and $\delta=(\tau^{\chi(\gamma)}-1)(\tau-1)^{-1}\in \ZZ_p[\![\tau-1]\!]$ is canonically isomorphic to the Galois cohomology of $T$. 
\end{theoremetoile}
This allowed him to prove the Br\"uckner-Vostokov reciprocity law for a formal group.

\smallskip

This said, Caruso's $(\varphi,\tau)$-module theory (\cf \cite{Car13}) provides an avatar of $(\varphi,\Gamma)$-module theory in the context of Breuil-Kisin extensions. Given a $p$-adic (let's say integral) representation $T$ of $\mathscr{G}_K$ (with $p$ an odd prime), the idea is to consider not only the associated $\varphi$-module $\mathscr{D}(T)$ over $\mathscr{O}_{\mathscr{E}}$, but also the action of a topological generator $\tau$ of $\Gal(L/K_\zeta)$ on $\mathcal{D}(T)_\tau: =\mathcal{O}_{\mathcal{E}_\tau}\otimes_{\mathcal{O}_{\mathcal{E}}}\mathcal{D}(T)$ (where $\mathcal{E}_\tau$ is a suitable extension of the fraction field $\mathcal{E}$ of $\mathcal{O}_{\mathcal{E}}$ and the action is usually denoted $\tau_D)$. Explicitly we have:

\begin{theoremetoile}
	\textup{(Caruso, \cite[\S 1.3]{Car13}, \cf section \ref{section 1.1})} Let $p$ be an odd prime number. The functor
	\begin{align*}
	\Rep_{\ZZ_p}(\mathscr{G}_K) &\to \Mod_{\Oo_{\Ee},\Oo_{\Ee_\tau}}(\varphi,\tau)\\		
	T &\mapsto \big(\mathcal{O}_{\widehat{\Ee^{\ur}}}\otimes_{\ZZ_p}T\big)^{\mathscr{G}_{K_\pi}}
	\end{align*}		
	is an equivalence between the category of integral $p$-adic representations and that of $(\varphi, \tau)$-modules over $\Oo_{\Ee}.$
\end{theoremetoile}
The first contribution of our work is the construction of a three term complex close to Herr's, built from the $(\varphi,\tau)$-module $\mathcal{D}(T)$, and which computes the Galois cohomology of $T$ 

\smallskip

More precisely, we have

\begin{theoremetoile}
		\textup{(\cf theorem \ref{thm main result})}. Let $p$ be an odd prime number. Let $T$ be an integral $p$-adic representation of $\mathscr{G}_K$ and $(D,D_\tau)$ the associated $(\varphi,\tau)$-module. Put
	$$D_{\tau,0}=\big\{x\in D_\tau,\;(\forall g\in\mathscr{G}_{K_\pi})\;\chi(g)\in\ZZ_{>0}\Rightarrow (g\otimes1)(x)=x+\tau_D(x)+\cdots+\tau_D^{\chi(g)-1}(x)\big\}.$$
	Then the homology of the complex
	$$\xymatrix{
		0\ar[rrr] &&& D\ar[r] & D\oplus D_{\tau,0}\ar[r] & D_{\tau,0}\ar[r] & 0\\
		&&& x\ar@{|->}[r] & ((\varphi-1)(x),(\tau_D-1)(x)) & & \\
		&&& & (y,z)\ar@{|->}[r] & (\tau_D-1)(y)-(\varphi-1)(z) & }$$
	is canonically isomorphic to the Galois cohomology of $T.$
\end{theoremetoile}

 The techniques used in the proof are standard (effaceability, d\'evissage and passing to limit), but a bit more difficult (compared to the cyclotomic case) by the fact that the complex involves a subgroup $\mathcal{D}(T)_{\tau, 0}$ of $\mathcal{D}(T)_{\tau}$ which is not easy to understand (which is necessary, as the \og naive\fg{} complex does not provide the correct cohomology groups, \cf section \ref{sec a counter example} ).

\smallskip

In the second part, we construct a morphism between the Tavares Ribeiro complex and ours, and prove that it is a quasi-isomorphism when the residue field is finite, which shows that the latter is a refinement of the former in this case (\cf theorem \ref{thm map from HC to TR}).

\smallskip

In the third part, we construct a $\psi$ operator (a left inverse of the Frobenius) and use it to construct other complexes that compute Galois cohomology. Since the residue field of $\mathcal{E}_\tau$ is perfect, such a $\psi$ operator cannot be constructed on it: it is necessary to use a refinement of $(\varphi,\tau)$-module theory, more precisely, it is necessary to work with partially unperfected coefficients $\mathcal{O}_{\mathcal{E}_{u,\tau}}$ (\cf \cite[\S 1.2.2]{Car13}, section \ref{section modules over unperfected coefficients}). It is shown that complexes analogous to the one in theorem \ref{thm main result}, but with coefficients in $\mathcal{O}_{\mathcal{E}_{u,\tau}}$ and using operator $\varphi$ or $\psi$ compute the Galois cohomology (\cf theorems \ref{coro complex over non-perfect ring works also Zp-case} and \ref{prop quais-iso psi}).

\smallskip

In the fourth part, we build overconvergent avatars of the previous complexes, using either the Frobenius operator $\varphi$, or the $\psi$ operator. Of course, we rely crucially on the overconvergence of $(\varphi,\tau)$-modules, proved by Gao and Poyeton (\cf \cite{GP21}), which we refine to the case of integral representations (\cf proposition \ref{prop Zp surconvergence}). This allows us to prove that the overconvergent complexes compute the correct $\H^0$ and $\H^1$. The case of $\H^2$ remains open. Let us point out that, as far as we know, the fact that the \og classical\fg{} overconvergent Herr complex (corresponding to the case of the cyclotomic extension) is shown to compute the Galois cohomology only in the case where $k$ is finite ($\ie$where $K$ is a finite extension of $\QQ_p$, \cf \cite{RLiu08}, \cite{Ben14}).

\smallskip

In the fifth part, we construct a three-term complex from the $(\varphi,N_\nabla)$-module over the Robba ring $\Rr$ associated to a representation (\ie associated to its $(\varphi,\tau)$-module).

More precisely, if $V$ is a $p$-adic representation of $\mathscr{G}_K$ and $D_{\rig}^\dagger$ the associated $(\varphi,N_\nabla)$-module, the complex in question is of the form

$$\xymatrix{
	0\ar[rrr] &&& D_{\rig}^\dagger\ar[r] & D_{\rig}^\dagger\oplus D_{\rig}^\dagger\ar[r] & D_{\rig}^\dagger\ar[r] & 0\\
	&&& x\ar@{|->}[r] & \big((\varphi-1)(x),N_\nabla(x)\big) & & \\
	&&& & (y,z)\ar@{|->}[r] & N_\nabla(y)-(c\varphi-1)(z) & }$$
where $c\in \Rr$ is an explicit element. It is a matter, after extension of the scalars to a suitable Robba ring, of replacing the operator $\tau_D$ by the (suitably normalized) infinitesimal action of $\Gal(L/K_\zeta)$ (\cf \cite[\S 2.2]{Poy21}). For this reason, the complex of course cannot compute Galois cohomology of the Galois representation $V$ from which we started, but it can almost be computed: we show (\cf proposition \ref{prop H0 H1 of complex C_R(V)}) that its $\H^i$ is isomorphic to $\iLim{n}\H^i(\mathscr{G}_{K_n},V)$ for $i\in \{0,1\}$. Again, the case of $\H^2$ is still problematic. In the case where the residue field $k$ is finite, we construct a pairing analogous to the one which gives rise to the Tate duality (\cf \cite{Her01}), but its non-degeneracy (which would allow to extend the above mentioned assertion to the case $i=2$) is not demonstrated. Let us observe that, although a little less fine than the complexes constructed previously (which compute the Galois cohomology), the complex considered in this part has the great advantage of not involving a group of the type $D_{\tau,0}$ (which as we said, is rather intangible), but only the $(\varphi,N_\nabla)$-module: it is thus \textit{a priori} easier to deal with.

\smallskip

In the sixth and last part, we apply the above to representations coming from $p$-divisible groups over $\mathcal{O}_K$. From the work of Breuil and Kisin (\cf \cite{Kis06}), we know that they are classified by some $\varphi$-modules over $\mathfrak{S}:=W[\![u]\!]$. The idea is then to relate the $(\varphi,N_\nabla)$-module associated to the Tate module of a $p$-divisible group with the associated $\mathfrak{S}$-module. Let us point out that according to Caruso (\cf \cite{Car13}), we are able to find the action of $\tau$ from that of $N_\nabla$ in this case, which allows us to compute the Galois cohomology of the dual of the Tate module of a $p$-divisible group, using its Breuil-Kisin module (\cf corollary \ref{coro chpater 6 module BK}).

\medskip

Based on our results in this work, we would also expect to compute the Galois cohomology of the Tate module of a $p$-divisible group over a general affine base in terms of its Zink's display (\cf \cite{Zin01}, \cite{Zin02}) in the future.

\medskip

\bb{Acknowledgements.} The main part of this work is contained in my thesis. I express my deepest gratitude to my supervisor Olivier Brinon for his numerous encouragements and comments: without him this work would not have been possible. I also benefit a lot from suggestions given by Xavier Caruso: I express my sincere thanks. My thanks equally go to Laurent Berger and Tong Liu for their careful reading and important comments on an earlier version of the manuscript. Finally, I also thank Abhinandan and L\'eo Poyeton for helpful discussions.

%% file: Notations.tex
\chapter*{Notations}

\subsection*{Extensions of fields}

Let $p$ be an odd prime, $K$\index{$K$} a complete discretly valued extension of $\QQ_p,$ with perfect residue field $k.$ Let $v$\index{$v$} be the normalized valuation on $K.$ Fix an algebraic closure $\overline{K}$\index{$\overline{K}$} of $K:$ the valuation $v$ extends to $\overline{K},$ denote by $C$\index{$C$} its completion. The group $\mathscr{G}_K=\Gal(\overline{K}/K)$\index{$\mathscr{G}_K$} acts by continuity on $C.$

\medskip

Let
$$C^\flat=\pLim{x\mapsto x^p}C=\big\{(x^{(n)})_{n\in\NN}\in C^{\NN}\,;\, (\forall n\in\NN)\, (x^{(n+1)})^p=x^{(n)}\big\}$$\index{$C^\flat$}
be the tilt of $C$ ($\cf$ \cite[\S 3]{Sch12}): with multiplication and addition defined by
\[(x^{(n)})_{n\in\NN}.(y^{(n)})_{n\in\NN}=(x^{(n)}y^{(n)})_{n\in\NN},\]
and
\[(x^{(n)})_{n\in\NN}+(y^{(n)})_{n\in\NN}=(z^{(n)})_{n\in\NN},\]
where 
\[z^{(n)}=\Lim{m\to\infty}\big(x^{(n+m)}+y^{(n+m)}\big)^{p^m}.\]
Recall that $C^\flat$ is then an algebraically closed field of characteristic $p$ and complete for the valuation given by
\begin{align*}
v^\flat\colon C^\flat &\to \RR\cup\, \{\infty\}\\
\big(x^{(n)}\big)_{n\in\NN} &\mapsto v_{p}\big(x^{(0)}\big).
\end{align*}
We denote $\mathcal{O}_{C^\flat}$ as the ring of integers of $C^\flat$ and recall that the natural map
\begin{align*}
\mathcal{O}_{C^\flat} &\to \pLim{x\mapsto x^p}\mathcal{O}_{\overline{K}}/p\mathcal{O}_{\overline{K}}\\
\big(x^{(n)}\big)_{n\in\NN} &\mapsto \big(x^{(n)} \, \text{mod} \, p\big)_{n\in\NN}
\end{align*}\index{$v^\flat$}
is an isomorphism.

\medskip

Fix $\tilde{\pi}=(\pi_n)_{n\in \NN}\in\Oo_{C^{\flat}}$\index{$\tilde{\pi}$} with $\pi_0=\pi$ and $\varepsilon=(\zeta_{p^n})_{n\in \NN}\in \Oo_{C^{\flat}}$\index{$\varepsilon$} such that $\zeta_p$ is a primitive $p$-th root of unity. Let 
\[K_{\pi}:=\bigcup\limits_n K(\pi_n), \quad K_{\zeta}:=\bigcup\limits_n K(\zeta_{p^n}),\] \index{$K_{\pi}$} \index{$K_{\zeta}$}
and
\[L:=K_{\pi}(\zeta_{p^{\infty}})=K_{\pi}K_{\zeta}.\]\index{$L$}
We denote the corresponding Galois groups as follows:
\[\mathscr{G}_{K_{\pi}}=\Gal(\overline{K}/K_{\pi}), \quad  \mathscr{G}_{K_{\zeta}}=\Gal(\overline{K}/K_{\zeta}),\quad  \mathscr{G}_{L}=\Gal(\overline{K}/L).    \] \index{$\mathscr{G}_{K_{\pi}}$} \index{$\mathscr{G}_{K_{\zeta}}$} \index{$\mathscr{G}_{L}$}
Let $\chi\colon \mathscr{G}_K\to\ZZ_p^\times$\index{$\chi$} be the cyclotomic character. It induces an injective morphism with open image:
$$\Gamma:=\Gal(K_{\zeta}/K)\inj \ZZ_p^\times.$$\index{$\Gamma$}

\medskip

Remark that $p$ is odd implies that $\Gamma$ has a topological generator. Moreover, we can choose a particular generator by the following lemma.

\begin{lemma}\label{lemm gamma with finite caracter}
	There exists a topological generator $\gamma\in\Gamma$ with $\chi(\gamma)\in\ZZ_{>0}.$\index{$\gamma$}
\end{lemma}

\begin{proof}
	Let $\gamma_0$ be a topological generator of $\Gamma.$ As $\chi(\Gamma)$ is an open subgroup of $\ZZ_p,$ there exists $N\in\ZZ_{>0}$ such that $1+p^N\ZZ_p\subset\chi(\Gamma),$ hence $\chi(\gamma_0)+p^N\ZZ_p\subset\chi(\Gamma).$ As $\ZZ_{\geq0}$ is dense in $\ZZ_p,$ we can choose $\gamma\in\Gamma$ such that
	
	$$\chi(\gamma)\in\ZZ_{\geq0}\,\cap\ (\chi(\gamma_0)+p^N\ZZ_p).$$
	
	We have $\overline{\langle\chi(\gamma)\rangle}\subset\chi(\Gamma)$ and both have the same image under $\ZZ^\times_p\to(\ZZ/p^N\ZZ)^\times:$ we have
	$$\chi(\gamma_0)\chi(\gamma)^{-1}\in\chi(\Gamma)\cap(1+p^N\ZZ_p),$$
	so we can write $\chi(\gamma_0)\chi(\gamma)^{-1}=\chi(\gamma_0)^z$ with $z\in p\ZZ_p$ (if $N\gg0$). We have $\chi(\gamma_0)=\chi(\gamma)\chi(\gamma_0)^z,$ so
	$$\chi(\gamma_0)=\chi(\gamma)^{1+z}\chi(\gamma_0)^{z^2},$$
	and by induction
	$$\chi(\gamma_0)=\chi(\gamma)^{(1+z+z^2+\cdots+z^r)}\chi(\gamma_0)^{z^{r+1}},$$
	and passing to the limit gives
	$$\chi(\gamma_0)=\chi(\gamma)^{(1+z+z^2\cdots)}$$
	so that $\gamma_0=\gamma^{\frac{1}{1-z}}\in\overline{\langle\gamma\rangle}.$
\end{proof}

\subsection*{The element $\tau$}

If $g\in \mathscr{G}_K,$ we denote by $c(g)$\index{$c(\cdot)$} the unique element in $\ZZ_p$ such that $g(\tilde{\pi})=\varepsilon^{c(g)}\tilde{\pi}.$ Indeed, for any $n\in \NN,$ there exists a unique element $c_n(g)\in\ZZ/p^n\ZZ$ such that $g(\pi_n)=\zeta_{p^n}^{c_n(g)}\pi_n.$ As $c_{n+1}(g)\equiv c_{n}(g)$ mod $p^n,$ the sequence $\big( c_n(g) \big)_n$ defines the element $c(g)\in \ZZ_p.$ Notice that the map $c\colon\mathscr{G}_K\to \ZZ_p(1)$ is a 1-cocycle (with $c^{-1}(0)=\mathscr{G}_{K_{\pi}}$), $\ie$for any $g, h\in \mathscr{G}_K:$
\[c(gh)=c(g)+\chi(g)c(h).\]

\medskip

Fix an element $\tau\in \Gal(\overline{K}/K_{\zeta})$\index{$\tau$} such that $ \tau(\tilde{\pi})=\varepsilon\tilde{\pi}$ ($\ie$$c(\tau)=1$). As $\mathscr{G}_L\subset c^{-1}(0),$ we still denote $\tau$ for its image in $\Gal(L/K_{\zeta})$ ($\cf$ diagram below).

\begin{remark}\label{rem def tau}
	The sequence 
	\[ 1\to \mathscr{G}_L \to \mathscr{G}_K \xrightarrow{\chi_{\infty}} \ZZ_p\rtimes \Gamma \to 1  \] is exact with $\chi_{\infty}=(c, \chi)$ and we have $\ZZ_p\rtimes \Gamma\simeq \mathscr{G}_K/\mathscr{G}_L=\Gal(L/K).$ Moreover, $\Gal(L/K)\simeq \overline{\la \tau \ra} \rtimes \overline{\la \gamma \ra}$ with $\gamma\tau\gamma^{-1}=\tau^{\chi(\gamma)}.$ 
\end{remark}

\begin{remark}
	If $g\in \mathscr{G}_{K_\pi},$ then
	$$\begin{cases}
	g\tau g^{-1}(\tilde{\pi})=\varepsilon^{\chi(g)}\tilde{\pi}=\tau^{\chi(g)}(\tilde{\pi})\\
	g\tau g^{-1}(\varepsilon)=\varepsilon=\tau^{\chi(g)}(\varepsilon)
	\end{cases}$$
	so that $g\tau g^{-1}$ and $\tau^{\chi(g)}$ have the same image in $\mathscr{G}_K/\mathscr{G}_L.$
\end{remark}

We summarize these notations in a diagram:

\[  \begin{tikzcd}[every arrow/.append style={dash}]
&\overline{K} \arrow[ldd, bend right, "\mathscr{G}_{K_{\pi}}"'] \arrow[d, "\mathscr{G}_L"]\arrow[rdd, bend left, "\mathscr{G}_{K_{\zeta}}"]&\\
&L\arrow[ld, "\overline{\la \gamma \ra}"']\arrow[dd]\arrow[rd,"\overline{\la \tau \ra}"]&\\		
K_{\pi}\arrow[rd]&   &K_{\zeta}\arrow[ld, "\Gamma=\overline{\la\gamma \ra}"]\\
& K &
\end{tikzcd}
\]

\medskip

\subsection*{Topologies on \texorpdfstring{$\W(C^\flat)$}{W(C\unichar{"266D})}}

Let $\W(C^\flat)$\index{$\W(C^\flat)$} be the ring of Witt vectors with coefficients in $C^\flat.$ The group $\mathscr{G}_K$ naturally acts on $\W(C^\flat)$ and we now describe some topologies on $\W(C^{\flat}).$

\begin{definition}
	We define the \emph{weak topology} on $\W(C^{\flat})$ as follows. Write
	\[\W(C^{\flat})=\pLim{n}\W_n(C^{\flat})  \] 
	and endow $\W(C^{\flat})$ with the inverse limit topology ($\ie$the topology induced by the product topology on $\prod_n\W_n(C^\flat)$) where we endow $\W_n(C^{\flat})=(C^{\flat})^n$ with the topology induced by the valuation topology on $C^\flat.$ 
\end{definition}

\begin{remark}
	
	\item (1) The $p$-adic topology on $\W(C^{\flat})$ induces the discrete topology on $C^{\flat}.$
	
	\item (2) The weak topology induces the valuation topology on $C^{\flat}.$
	
	\item (3) The $\mathscr{G}_K$-action is not continuous on $C^{\flat}$ with the discrete topology, while it is continuous on $C^{\flat}$ with the valuation topology. Indeed we have $\Lim{g\to \id}\varepsilon^{\chi(g)}=\varepsilon$ for the valuation topology, but not for the discrete topology.
\end{remark}

%% file: Chapter1.tex
\chapter{The complex \texorpdfstring{$\Cc_{\varphi, \tau}$}{C\textpinferior\texthinferior\textiinferior,\texttinferior\textainferior\textuinferior}}\label{chapter The complex C varphi tau}

In this chapter, we introduce the category $\Mod_{\Ee, \Ee_{\tau}}(\varphi, \tau)$ of $(\varphi, \tau)$-modules over $(\Ee, \Ee_{\tau})$ (introduced by Caruso $\cf$ \cite{Car13}), which is categorically equivalent to the category of $p$-adic representations. Then for any $p$-adic representation $V,$ we construct a complex $\Cc_{\varphi, \tau}(V)$ (which is functorial) using the $(\varphi, \tau)$-module associated to $V,$ and we show that this complex computes the continuous Galois cohomology of $V.$ Hence, this is indeed a variant of Herr complex for $p$-adic representations ($\cf$ \cite{Her98}).

\section{Construction of the complex}\label{section 1.1}

\begin{notation}
	Let $\Rep_{\ZZ_p}(\mathscr{G}_K)$\index{$\Rep_{\ZZ_p}(\mathscr{G}_K)$} (resp. $\Rep_{\QQ_p}(\mathscr{G}_K),$\index{$\Rep_{\QQ_p}(\mathscr{G}_K)$} resp. $\Rep_{\FF_p}(\mathscr{G}_K)$)\index{$\Rep_{\FF_p}(\mathscr{G}_K)$} be the category of $\ZZ_p$-representations (resp. $p$-adic representations, resp. $\FF_p$-representations), whose objects are $\ZZ_p$-modules (resp. $\QQ_p$-vector spaces, resp. $\FF_p$-vector spaces) of finite type endowed with a linear and continuous action of $\mathscr{G}_K.$ Let $\Rep_{\ZZ_p, \tors}(\mathscr{G}_K)$\index{$\Rep_{\ZZ_p, \tors}(\mathscr{G}_K)$} be the subcategory of $\ZZ_p$-representations that are killed by some power of $p.$ Moreover, for a fixed $r\in \NN,$ let $\Rep_{\ZZ_p, p^r\textrm{-tors}}(\mathscr{G}_K)$\index{$\Rep_{\ZZ_p, p^r\textrm{-tors}}(\mathscr{G}_K)$} be the subcategory of $\ZZ_p$-representations that are killed by $p^r.$ 
\end{notation}

\begin{notation}
	Let $F_{\tau}=(C^{\flat})^{\mathscr{G}_L},\ F_0=k(\!(\tilde{\pi})\!)$\index{$F_{\tau}$}\index{$F_0$} and $F_0^{\sep}$\index{$F_0^{\sep}$}  be the separable closure of $F_0$ in $C^{\flat}.$ We then have $F_0\subset (F_0^{\sep})^{\mathscr{G}_L}\subset F_\tau.$ Put
	\[\Oo_{\Ee}=\Big\{ \sum\limits_{i\in \ZZ} a_iu^i;\ a_i\in W(k),\ \Lim{i\to -\infty}a_i=0 \Big\},\]\index{$\Oo_{\Ee}$}and embed it into $\W(C^{\flat})$ by sending $u$ to $[\tilde{\pi}].$ Endow $\Oo_{\Ee}$ with the $p$-adic valuation: it is a discrete valuation ring with residue field $F_0.$ Put $\mathcal{E}=\Frac(\mathcal{O}_{\mathcal{E}}).$\index{$\mathcal{E}$}  Let $\Oo_{\Ee^{\ur}}$\index{$\Oo_{\Ee^{\ur}}$}  be the unique ind-\'etale sub-algebra of $\W(C^{\flat})$ whose residue field is $F_0^{\sep}\subset C^{\flat}$ and $\Oo_{\widehat{\Ee^{\ur}}}$\index{$\Oo_{\widehat{\Ee^{\ur}}}$}  its $p$-adic completion. We put $\Oo_{\Ee_{\tau}}=\W(F_{\tau})$ and $\Ee_{\tau}=\Frac \Oo_{\Ee_{\tau}}.$\index{$\Oo_{\Ee_{\tau}}$} 
\end{notation}

\begin{remark}
	We have $\Gal(\Frac \Oo_{\Ee^{\ur}}/ \Frac \Oo_{\Ee})\simeq \Gal(F_0^{\sep}/F_0)\simeq \mathscr{G}_{K_{\pi}}$ ($\cf$ \cite[Th\'eor\`eme 3.2.2]{Win83}).
\end{remark}

\begin{definition}\label{Phi-tau-module}(\cf \cite[D\'efinition 1.17]{Car13})
	A \emph{$(\varphi, \tau)$-module over $(\Oo_{\Ee}, \Oo_{\Ee_{\tau}})$} consists of
	\item (i) an \'etale $\varphi$-module $D$ over $\Oo_{\Ee}$ (this means that the linearization $1\otimes \varphi\colon\Oo_{\Ee}\otimes_{\varphi, \Oo_{\Ee}}D \to D $ is an isomorphism);
	\item (ii) a $\tau$-semi-linear endomorphism $\tau_D$ on $D_\tau:=\Oo_{\Ee_{\tau}}\otimes_{\Oo_{\Ee}}D$ which commutes with $\varphi_{\Oo_{\Ee_{\tau}}}\otimes \varphi_D$ (where $\varphi_{\mathcal{O}_{\mathcal{E}_{\tau}}}$ is the Frobenius map on $\mathcal{O}_{\mathcal{E}_{\tau}}$ and $\varphi_D$ the Frobenius map on $D$) and such that:
	\[ (\forall x\in D)\ (g\otimes 1)\circ \tau_D(x) = \tau_D^{\chi(g)}(x), \]
	for all $g\in \mathscr{G}_{K_{\pi}}/\mathscr{G}_L$ such that $\chi(g)\in \ZZ_{>0}.$
	
	\medskip
	
	Let $\Mod_{\Oo_{\Ee},\Oo_{\Ee_{\tau}}}(\varphi,\tau)$\index{$\Mod_{\Oo_{\Ee},\Oo_{\Ee_{\tau}}}(\varphi,\tau)$} be the corresponding category. One defines similarly the notion of $(\varphi,\tau)$-module over $(\mathcal{E},\mathcal{E}_\tau),$ and the corresponding category $\Mod_{\mathcal{E},\mathcal{E}_\tau}(\varphi,\tau).$\index{$\Mod_{\mathcal{E},\mathcal{E}_\tau}(\varphi,\tau)$}
\end{definition}

\begin{remark}\label{remark on defi varphi-tau module gamma tau relation} (1) In the following we will use light notations, take $\Mod_{\Oo_{\Ee},\Oo_{\Ee_{\tau}}}(\varphi,\tau)$ for example: we will simply denote $D$ or $(D, D_{\tau})$ for objects of this category. 
	\item (2) The condition for $(\varphi,\tau)$-modules can be rewritten as follows (\cf \cite[\S 1.2.3]{Car13}): if $g\in \mathscr{G}_K/\mathscr{G}_L$ is such that $\chi(g)\in\ZZ_{>0}$ then
	$$(g\otimes1)\circ\tau_D=\tau_D^{\chi(g)}\circ(\tau^{-\chi(g)}g\tau\otimes1)=\tau_D^{\chi(g)}\circ(g\otimes1).$$ Notice that the first equality is hold on $D_{\tau},$ while the second equality is hold on $D.$
\end{remark}

\begin{theorem}\label{thm equivalence of cats}
	The functors
	\begin{align*}
	\Rep_{\ZZ_p}(\mathscr{G}_K) &\simeq \Mod_{\Oo_{\Ee},\Oo_{\Ee_{\tau}}}(\varphi, \tau)\\
	T&\mapsto \mathcal{D}(T)=(\mathcal{O}_{\widehat{\mathcal{E}^{\ur}}}\otimes T)^{\mathscr{G}_{K_{\pi}}}\\
	\mathcal{T}(D)=(\mathcal{O}_{\widehat{\mathcal{E}^{\ur}}}\otimes_{\Oo_{\Ee}}D)^{\varphi=1} &\mapsfrom D 
	\end{align*}
	establish quasi-inverse equivalences of categories. 
	
	\medskip 
	
	Similarly, we have quasi-inverse equivalences of categories \[\Rep_{\QQ_p}(\mathscr{G}_K)\simeq\Mod_{\mathcal{E},\mathcal{E}_\tau}(\varphi,\tau).\]
\end{theorem}

\begin{proof} 
	$\cf$ \cite[\S 1.3]{Car13}.
\end{proof}

\begin{notation}\label{nota D-tau}
	Let $(D, D_{\tau})\in\Mod_{\Oo_{\Ee},\Oo_{\Ee_{\tau}}}(\varphi,\tau),$ then we put 
	\[D_{\tau, 0}:=\big\{ x\in D_{\tau};\ (\forall g\in \mathscr{G}_{K_{\pi}})\ \chi(g)\in \ZZ_{>0} \Rightarrow (g\otimes 1)(x)=x+\tau_D(x)+\cdots +\tau_D^{\chi(g)-1}(x) \big\}.\]\index{$D_{\tau, 0}$}
\end{notation}

Recall that we have the following lemma:
\begin{lemma}\label{lemm iso for D-tau}
	Let $T\in \Rep_{\ZZ_p}(\mathscr{G}_{K_{\pi}}),$	we then have an isomorphism of $\Gamma$-modules:
	\[ \mathcal{D}(T)_{\tau}=\Oo_{\Ee_{\tau}}\otimes_{\Oo_{\Ee}}\mathcal{D}(T)\simeq  (\W(C^{\flat})\otimes_{\ZZ_p}T)^{\mathscr{G}_L}.\] 
\end{lemma}

\begin{proof}
	$\cf$ \cite[Lemma 1.18]{Car13}.
\end{proof}

\begin{notation}
	If $z\in\ZZ_p$ and $n\in\ZZ_{\geq0},$ let
	$$\binom{z}{n}=\frac{z(z-1)\cdots(z-n+1)}{n!}\in\ZZ_p.$$\index{$\binom{z}{n}$}
	As $\Lim{m\to\infty}\tau_D^{p^m}=1$ ($\cf$ \cite[Proposition 1.3]{Car13}), the operator $\tau_D-1$ is topologically nilpotent on $D_{\tau},$ and we may define
	$$\tau_D^z=\sum\limits_{n=0}^\infty\binom{z}{n}(\tau_D-1)^n$$\index{$\tau_D^z$}
	for all $z\in\ZZ_p.$ In particular, we have
	$$\delta:=1+\tau_D+\cdots+\tau_D^{\chi(\gamma)-1}=\frac{\tau_D^{\chi(\gamma)}-1}{\tau_D-1}=\sum\limits_{n=1}^\infty\binom{\chi(\gamma)}{n}(\tau_D-1)^{n-1}.$$\index{$\delta$}
\end{notation} 

\begin{lemma}\label{Replace g witi gamma}
	We have $\mathcal{D}(T)_{\tau, 0}= \big\{ x\in \mathcal{D}(T)_{\tau};\ (\gamma\otimes 1)(x)=x+\tau_D(x)+\cdots +\tau_D^{\chi(\gamma)-1}(x) \big\}.$
\end{lemma}

\begin{proof} Let $x\in \mathcal{D}_{\tau}$ and suppose 
	\[(\gamma\otimes 1)(x)=x+\tau_D(x)+\cdots +\tau_D^{\chi(\gamma)-1}(x).\]
	Then for any $g\in \mathscr{G}_{K_{\pi}}/\mathscr{G}_L$ with $\chi(g)\in \ZZ_{>0},$ we show that 
	\[(g\otimes 1)(x)=x+\tau_D(x)+\cdots +\tau_D^{\chi(g)-1}(x).\]
	Let $g\in\mathscr{G}_{K_\pi}.$ Under the isomorphism of lemma \ref{lemm iso for D-tau}, the action of $g\otimes1$ on $\Oo_{\Ee_{\tau}}\otimes_{\Oo_{\Ee}}\mathcal{D}(T)$ corresponds to the diagonal action $g\otimes g$ on $\W(C^{\flat})\otimes_{\ZZ_p}T.$ Similarly, lemma \ref{lemm iso for D-tau} also tells that the action $\tau_D$ corresponds to $\tau\otimes\tau$ on $\mathcal{D}(T)_{\tau}\subset \W(C^{\flat})\otimes_{\ZZ_p}T.$ We prove by induction on $n\in\ZZ_{>0}$ that
	$$(g\otimes1)(x)=(1+\tau_D+\cdots+\tau_D^{\chi(g)-1})(x)$$
	when $g=\gamma^n$ (this is true for $n=1$ by hypothesis). Assume $n>1.$ Seeing $x$ as an element of $\Oo_{\Ee_{\tau}}\otimes_{\Oo_{\Ee}}\mathcal{D}(T),$ we have
	
	\begin{align*}
	(g\otimes g)(x) &= (\gamma\otimes\gamma)\big((\gamma^{n-1}\otimes\gamma^{n-1})(x)\big)\\
	&= (\gamma\otimes\gamma)\Big(\sum\limits_{i=0}^{\chi(\gamma)^{n-1}-1}\tau^i\otimes\tau^i\Big)(x)\\
	&= \Big(\sum\limits_{i=0}^{\chi(\gamma)^{n-1}-1}\tau^{i\chi(\gamma)}\otimes\tau^{i\chi(\gamma)}\Big)(\gamma\otimes\gamma)(x)\\
	&= \Big(\sum\limits_{i=0}^{\chi(\gamma)^{n-1}-1}\tau^{i\chi(\gamma)}\otimes \tau^{i\chi(\gamma)}\Big)\Big(\sum\limits_{j=0}^{\chi(\gamma)-1}\tau^j\otimes\tau^j\Big)(x)\\
	&= \Big(\sum\limits_{k=0}^{\chi(\gamma)^n-1}\tau^k\otimes\tau^k\Big)(x).
	\end{align*}
	
	Let $g\in\mathscr{G}_{K_\pi}$ be such that $\chi(g)\in\ZZ_{>0}.$ As $\gamma$ is a topological generator of $\Gamma,$ there exists a sequence $(b_m)_{m\in\ZZ_{\geq0}}$ in $\ZZ_p$ such that $\Lim{m\to\infty}\gamma^{b_m}=g.$ As $\ZZ_{\geq0}$ is dense in $\ZZ_p,$ there exists $a_m\in\ZZ_{\geq0}$ such that $b_m-a_m\in p^n\ZZ_p.$ Then $\Lim{m\to\infty}\gamma^{a_m}=g.$ We have
	$$(\gamma^{a_m}\otimes1)(x)=(1+\tau_D+\cdots+\tau_D^{\chi(\gamma)^{a_m}-1})(x)=\sum\limits_{n=1}^\infty\binom{\chi(\gamma)^{a_m}}{n}(\tau_D-1)^n(x)$$
	for all $m.$ Passing to the limit as $m\to\infty,$ we get
	$$(g\otimes1)(x)=\sum\limits_{n=1}^\infty\binom{\chi(g)}{n}(\tau_D-1)^n(x)=(1+\tau_D+\cdots+\tau_D^{\chi(g)-1})(x).$$
\end{proof}

\begin{lemma}\label{lemm 1.1.11 well-defined tau-1 varphi-1}
	The maps $\varphi-1\colon D_{\tau}\to D_{\tau}$ and $\tau_D-1\colon D_{\tau}\to D_{\tau}$ induce maps $\varphi-1\colon D_{\tau,0}\to D_{\tau, 0}$ and $\tau_D-1\colon D\to D_{\tau, 0}.$
\end{lemma}

\begin{proof}
	For $\varphi-1,$ this results from the fact that $\varphi\colon D_{\tau}\to D_{\tau}$ commutes with $\tau_D$ and the action of $\mathscr{G}_{K_{\pi}}.$
	\item Let $x\in D$ and $y:=(\tau_D-1)(1\otimes x)\in D_{\tau},$ we then claim that $y\in D_{\tau, 0}.$ Indeed, if $g\in \mathscr{G}_{K_{\pi}}$ satisfies $\chi(g)=n\in \ZZ_{>0},$ then
	\begin{align*}
	(1+\tau_D +\cdots + \tau_D^{n-1})(y)&= (1+\tau_D +\cdots + \tau_D^{n-1})((\tau_D-1)(1\otimes x))\\
	&= (\tau_D^n-1)(1\otimes x)\\
	&= (g\otimes 1)(\tau_D (1\otimes x))-1\otimes x\\
	&= (g\otimes 1)((\tau_D-1)(1\otimes x))\\
	&= g(y).
	\end{align*}
	Indeed, $\tau_D^n(1\otimes x)=\tau_D^n\circ (g\otimes 1)(1\otimes x)=(g\otimes 1)\tau_D(1\otimes x),$ where the last equality follows from remark \ref{remark on defi varphi-tau module gamma tau relation} (2). 	
\end{proof}

\begin{definition}\label{def C-phi-tau}
	Let $(D, D_{\tau})\in\Mod_{\Oo_{\Ee},\Oo_{\Ee_{\tau}}}(\varphi,\tau).$ We define a complex $\Cc_{\varphi, \tau}(D)$\index{$\Cc_{\varphi, \tau}(D)$} as follows:  
	\[ \xymatrix{
		0\ar[rr] && D \ar[r] & D\oplus D_{\tau, 0}  \ar[r]&  D_{\tau, 0}  \ar[r] & 0   \\
		&& x  \ar@{|->}[r] &((\varphi-1)(x), (\tau_D-1)(x))  &{}\\
		&&&{} (y,z) \ar@{|->}[r]& (\tau_D-1)(y)-(\varphi-1)(z)
	}   \]
	
	where the first term is of degree $-1.$
	
	\medskip
	
	If $T\in \Rep_{\ZZ_p}(\mathscr{G}_{K}),$ we have in particular the complex $\Cc_{\varphi, \tau}(\mathcal{D}(T)),$ which will also be simply denoted $\Cc_{\varphi, \tau}(T).$\index{$\Cc_{\varphi, \tau}(T)$}
	
	\medskip
	
	Similarly, we can attach a complex $\mathcal{C}_{\varphi,\tau}(D)$ to any $D\in\Mod_{\mathcal{E},\mathcal{E}_\tau}(\varphi,\tau),$ and in particular attach a complex $\mathcal{C}_{\varphi,\tau}(V)$\index{$\mathcal{C}_{\varphi,\tau}(V)$} to any $V\in \Rep_{\QQ_p}(\mathscr{G}_{K}).$
\end{definition}

\begin{theorem}\label{thm main result}
	For any $T\in \Rep_{\ZZ_p}(\mathscr{G}_K)$ and natural integer $i,$ there is a canonical and functorial isomorphism
	$$\H^i(\mathcal{C}_{\varphi,\tau}(T))\simeq \H^i(\mathscr{G}_K,T).$$
	
	\medskip
	
	Similarly, if $V\in\Rep_{\QQ_p}(\mathscr{G}_K),$ we have a canonical and functorial isomorphism $\H^i(\mathcal{C}_{\varphi,\tau}(V))\simeq \H^i(\mathscr{G}_K,V)$ for all $i.$
\end{theorem}

\begin{notation} If $\mathscr{C}$ is a category, we denote by $\Ind\mathscr{C}$\index{$\Ind\mathscr{C}$} the associated ind-category, whose objects are inductive systems of objects of $\mathscr{C}$ indexed by some filtered category. 
\end{notation}

\begin{remark}	(1) The category $\Ind \Rep_{\ZZ_p, p^r\textrm{-tors}}(\mathscr{G}_K)$\index{$\Ind \Rep_{\ZZ_p, p^r\textrm{-tors}}(\mathscr{G}_K)$} coincides with the category of discrete $\mathscr{G}_K$-modules killed by $p^r,$ $\ie$of discrete $(\ZZ/p^r\ZZ)[\![\mathscr{G}_K]\!]$-modules.
	
	\item (2) The equivalences of theorem \ref{thm equivalence of cats} extend into equivalences
	$$\Ind \Rep_{\ZZ_p}(\mathscr{G}_K)\simeq\Ind \Mod_{\Oo_{\Ee}, \Oo_{\Ee_{\tau}}}(\varphi,\tau).$$
\end{remark}

\begin{definition}
	Let $T\in \mathbf{\Rep}_{\ZZ_p, p^r\textrm{-tors}}(\mathscr{G}_K).$ The \emph{induced module} of $T$ is $\Ind_{\mathscr{G}_K}(T):=C^0(\mathscr{G}_K, T),$\index{$\Ind_{\mathscr{G}_K}(T)$} the set of all continuous maps from $\mathscr{G}_K$ to $T.$ Endow $\Ind_{\mathscr{G}_K}T$ with the discrete topology and the action of $\mathscr{G}_K$ given by
	\begin{align*}
	\mathscr{G}_K \times \Ind_{\mathscr{G}_K}(T) &\to \Ind_{\mathscr{G}_K}(T) \\
	(g, \eta) &\mapsto [x\mapsto \eta(xg)].
	\end{align*}
	Then $\Ind_{\mathscr{G}_K}(T)\in \Ind \mathbf{\Rep}_{\FF_p}(\mathscr{G}_K),$ and $T$ canonically injects into $\Ind_{\mathscr{G}_K}(T)$ by sending $v \in T$ to $\eta_v,$ where $\eta_v(g)=g(v)$ for any $g\in \mathscr{G}_K.$  
\end{definition}

\begin{lemma}\label{lemmcoho0nonGalois Zp}
	Let $T\in \Rep_{\ZZ_p, \tors}(\mathscr{G}_K)$ and $K^{\prime}$ be any subfield of $\overline{K}$ containing $K.$ We then have \[ \H^1(\mathscr{G}_{K^{\prime}}, \Ind_{\mathscr{G}_K}T)=0. \]
\end{lemma}

\begin{proof}
	Put $U=\Ind_{\mathscr{G}_K}T.$ Let $I_{K^{\prime}}$ be the set of finite subextensions of $K^{\prime}/K.$ Put $H=\bigcap\limits_{M\in I_{K^{\prime}}}\mathscr{G}_M:$ this is a closed subgroup of $\mathscr{G}_K$ (since all $\mathscr{G}_M$ are) and $\mathscr{G}_{K^{\prime}}\subset H.$ Hence we have $\overline{K}^H\subset \overline{K}^{\mathscr{G}_{K^{\prime}}}=K^{\prime}.$ On the other hand, for each $M\in I_{K^{\prime}},$ we have $H\leq \mathscr{G}_M,$ whence $M=\overline{K}^{\mathscr{G}_M}\subset \overline{K}^{H}.$ As this holds for each $M,$ we get $K^{\prime}=\bigcup\limits_{M\in I_{K^{\prime}}}M\subset\overline{K}^H,$ $\ie$$\overline{K}^{H}=K^{\prime}=\overline{K}^{\mathscr{G}_{K^{\prime}}}:$ by Galois correspondence we have $H=\mathscr{G}_{K^{\prime}}.$ 
	We have 	
	\begin{equation}\label{equ Tavares Ribeiro 1}
	\H^1(\mathscr{G}_{K^{\prime}}, U)=\iLim{M} \H^1(\mathscr{G}_M, U)
	\end{equation}
	by \cite[Chapitre I, Proposition 8]{Ser68}. Indeed, we saw that the groups $\{ \mathscr{G}_M \}_{M\in I_{K^{\prime}}}$ form, for inclusion, a projective system with limit $\bigcap \mathscr{G}_M=\mathscr{G}_{K^{\prime}}$ and this system is compatible with the inductive system formed by $U,$ seen as a $\mathscr{G}_M$-module by restriction. While the limit is $U,$ seen as a $\mathscr{G}_{K^{\prime}}$-module by restriction again. 
	Now we claim that $\H^1(\mathscr{G}_{K^{\prime}}, U)=0.$ Indeed, $\mathscr{G}_M$ being open in $\mathscr{G}_K,$ we have the finite decomposition $\mathscr{G}_K=\bigsqcup\limits_{t\in \mathscr{G}_K/\mathscr{G}_M}t\mathscr{G}_M$ from which we deduce that, as a $\mathscr{G}_M$-module, $U$ admits a decomposition 
	\[ U=\bigoplus_{t\in \mathscr{G}_K/\mathscr{G}_M}C^{0}(t\mathscr{G}_M, T)\simeq \bigoplus_{t\in \mathscr{G}_K/\mathscr{G}_M}C^{0}(\mathscr{G}_M, T)\simeq \bigoplus_{t\in \mathscr{G}_K/\mathscr{G}_M}\Ind_{\mathscr{G}_M}T. \]
	Indeed, the homeomorphism is given by 
	
	\begin{align*}
	C^{0}(t\mathscr{G}_M, T)&\simeq C^{0}(\mathscr{G}_M, T)\\ 
	f&\mapsto \tilde{f},
	\end{align*}  
	where $\tilde{f}(g)=f(tg)$ for all $g\in \mathscr{G}_M.$ It is an isomorphism of $\mathscr{G}_M$-modules: if $f\in C^0(t\mathscr{G}_M, T)$ and $g, h\in \mathscr{G}_M,$ we have 
	\[  (g\cdot \tilde{f})(h)=\tilde{f}(hg)=f(thg)=(g\cdot f)(th)=\widetilde{g\cdot f}(h),\]
	$\ie$$g\cdot \tilde{f}=\widetilde{g\cdot f}.$
	
	\smallskip
	
	Thus, we have
	\[\H^1(\mathscr{G}_M, U)\simeq \bigoplus_{\mathscr{G}_K/\mathscr{G}_M} \H^1(\mathscr{G}_M, \Ind_{\mathscr{G}_M} T), \]
	and the summands of the right-hand side are zero by classical results, refering to \cite[VII, Proposition 1]{Ser68}.  Now (\ref{equ Tavares Ribeiro 1}) implies that $ \H^1(\mathscr{G}_{K^{\prime}}, U)=0.$ 
\end{proof}

\begin{lemma}\textup{(\cf \cite[Lemma 1.8]{Tav11})} \label{lemm conti section}
	The following two maps 
	\[\mathcal{O}_{\widehat{\mathcal{E}^{\ur}}}\xrightarrow{\varphi-1} \mathcal{O}_{\widehat{\mathcal{E}^{\ur}}}\] 
	\[\W(C^{\flat}) \xrightarrow{\varphi-1} \W(C^{\flat})\]
	admit continuous sections. \ft{Remark that the sections are not ring morphisms.}
\end{lemma}

\begin{proof}
	For any element $y$ of $F_0^{\sep},$ there exists a finite field extension $F_1/F_0$ such that $y\in F_1.$ If $v^{\flat}(y)>0,$ then we have  $-\sum\limits_{k=0}^{\infty}\varphi^k(y)\in F_1$ and $y=(\varphi-1)(-\sum\limits_{k=0}^{\infty}\varphi^k(y)).$ This defines a continuous section $\overline{s}\colon \mm_{F_0^{\sep}}\to \mm_{F_0^{\sep}}.$ We have $F_0^{\sep}=\bigsqcup\limits_{i\in I}(i+ \mm_{F_0^{\sep}}):$ this is a union of disjoint opens, where $I$ is a fixed set of representatives of the quotient $F_0^{\sep}/\mm_{F_0^{\sep}}.$ For any $i\in I$ choose a $y_i\in F_0^{\sep}$ such that $(\varphi-1)(y_i)=i.$ Indeed, we can solve the equations $T^p-T=i$ in $F_0^{\sep}.$ Observe that $y_i\in\mathcal{O}_{F_0^{\sep}}$ when $i\in\mathcal{O}_{F_0^{\sep}}.$ We define a continuous section on $F_0^{\sep}$ by sending $i+x$ to $y_i+\overline{s}(x)$ (note that $\overline{s}$ maps $\mathcal{O}_{F_0^{\sep}}$ into itself). As $F_0^{\sep}$ is dense in $C^{\flat},$ this section extends by continuity into a continuous section $\overline{s}$ of $\varphi-1\colon C^{\flat}\to C^{\flat}$ that maps $\mathcal{O}_{C^\flat}$ into itself.\\
	The map $s_1\colon\W(C^\flat)\to\W(C^\flat)$ defined by $s_1(x)=[\overline{s}(\overline{x})]$ is a continuous (for the weak topology) section of $\varphi-1$ modulo $p,$ that maps $\W(\mathcal{O}_{C^\flat})$ into itself. We deform $s_1$ inductively into maps $s_n\colon\W(C^\flat)\to\W(C^\flat)$ which are sections of $\varphi-1$ modulo $p^n$ and map $\W(\mathcal{O}_{C^\flat})$ into itself: if $s_n$ is constructed, we have
	$$(\varphi-1)\circ s_n=\Id+p^nf_n,$$
	where $f_n\colon\W(C^\flat)\to\W(C^\flat)$ is a continuous map that maps $\W(\mathcal{O}_{C^\flat})$ into itself. Then
	$$s_{n+1}=s_n-p^ns_1\circ f_n$$
	has the required properties. The sequence $(s_n)_{n\in\ZZ_{>0}}$ converges to a section $s$ of $\varphi-1$ on $\W(C^\flat)$ that maps $\W(\mathcal{O}_{C^\flat})$ into itself.
	
	Let $y\in\Oo_{\widehat{\mathcal{E}^{\ur}}}:$ there exists $x\in\Oo_{\widehat{\mathcal{E}^{\ur}}}$ such that $(\varphi-1)(x)=y,$ so that $(\varphi-1)(s(y)-x)=0,$ $\ie$ $s(y)-x\in\ZZ_p,$ so that $s(y)\in\Oo_{\widehat{\mathcal{E}^{\ur}}}.$ This proves that $s(\Oo_{\widehat{\mathcal{E}^{\ur}}})\subset\Oo_{\widehat{\mathcal{E}^{\ur}}},$ so that $s$ induces a continuous section of $\varphi-1\colon \Oo_{\widehat{\mathcal{E}^{\ur}}}\to \Oo_{\widehat{\mathcal{E}^{\ur}}}.$
	
\end{proof}

\begin{corollary}\label{coro H1 null}
	Let $T\in\Rep_{\ZZ_p\text{-}\tors}(\mathscr{G}_K)$ and $U=\Ind_{\mathscr{G}_K}T.$ We have
	\begin{align*}
	\H^1(\mathscr{G}_{L}, U)&=0\\
	\H^1(\mathscr{G}_{K_{\pi}}, U)&=0
	\end{align*}	
	and hence the following exact sequences:
	\begin{align*}
	0\to U^{\mathscr{G}_L}\to \mathcal{D}(U)_{\tau}\xrightarrow{\varphi-1}  \mathcal{D}(U)_{\tau} &\to 0\\
	0\to U^{\mathscr{G}_{K_{\pi}}}\to \mathcal{D}(U)\xrightarrow{\varphi-1}  \mathcal{D}(U) &\to 0.
	\end{align*}
\end{corollary}

\begin{proof}
	Let $K^{\prime}=L,$ then lemma \ref{lemmcoho0nonGalois Zp} gives $ \H^1(\mathscr{G}_L, U)=0.$  Consider the following exact sequence
	\[ 0\to \ZZ_p \to \W(C^{\flat}) \xrightarrow{\varphi-1} \W(C^{\flat}) \to 0. \]
	Tensorize with the injective $U$ and then, by lemma \ref{lemm conti section}, we can take Galois invariants to get a long exact sequence
	\[  0\to U^{\mathscr{G}_L} \to  \mathcal{D}(U)_{\tau} \xrightarrow{\varphi-1}  \mathcal{D}(U)_{\tau}\to \H^1(\mathscr{G}_{L}, U)=0. \]
	The case $K^{\prime}=K_{\pi}$ can be proved similarly.  
\end{proof}

\subsection{Tensor product and internal Hom}

\begin{remark}
	Let $T\in \Rep_{\ZZ_p}(\mathscr{G}_K),$ we then have a $\mathscr{G}_{K_{\pi}}$-equivariant isomorphism:
	\[ \mathcal{D}(T)\otimes_{\Oo_{\Ee}}\Oo_{\widehat{\Ee^{\ur}}}\simeq T\otimes_{\ZZ_p}\Oo_{\widehat{\Ee^{\ur}}}.  \]
\end{remark}

\begin{proposition}\label{prop internal hom and tensor product}
	Let $T_1, T_2\in \Rep_{\ZZ_p}(\mathscr{G}_K),$ then we have 
	\begin{align*}
	&\mathcal{D}(T_1\otimes_{\ZZ_p} T_2)\simeq \mathcal{D}(T_1)\otimes_{\Oo_{\Ee}}\mathcal{D}(T_2)\\
	&\mathcal{D}(\Hom_{\ZZ_p}(T_1, T_2))\simeq \Hom_{\Oo_{\Ee}}(\mathcal{D}(T_1), \mathcal{D}(T_2)).
	\end{align*}
\end{proposition}

\begin{proof}
	We have the following $\mathscr{G}_{K_{\pi}}$-equivariant isomorphisms:	
	
	\begin{align*}
	\Oo_{\widehat{\Ee^{\ur}}}\otimes_{\Oo_{\Ee}}\mathcal{D}(T_1\otimes_{\ZZ_p} T_2) &\simeq T_1\otimes_{\ZZ_p} T_2	\otimes_{\ZZ_p} \Oo_{\widehat{\Ee^{\ur}}}\\
	&\simeq (T_1\otimes_{\ZZ_p}\Oo_{\widehat{\Ee^{\ur}}})\otimes_{\Oo_{\widehat{\Ee^{\ur}}}}(T_2\otimes_{\ZZ_p}\Oo_{\widehat{\Ee^{\ur}}})\\
	&\simeq(\mathcal{D}(T_1)\otimes_{\Oo_{\Ee}}\Oo_{\widehat{\Ee^{\ur}}})\otimes_{\Oo_{\widehat{\Ee^{\ur}}}}(\mathcal{D}(T_2)\otimes_{\Oo_{\Ee}}\Oo_{\widehat{\Ee^{\ur}}})\\
	&\simeq \mathcal{D}(T_1)\otimes_{\Oo_{\Ee}}\mathcal{D}(T_2)\otimes_{\Oo_{\Ee}} \Oo_{\widehat{\Ee^{\ur}}}.
	\end{align*}
	Recall that $\mathscr{G}_{K_{\pi}}$ acts trivially on $\mathcal{D}(T).$ Taking fixed points under  $\mathscr{G}_{K_{\pi}}$ on both sides of
	$	\Oo_{\widehat{\Ee^{\ur}}}\otimes_{\Oo_{\Ee}}\mathcal{D}(T_1\otimes_{\ZZ_p} T_2) \simeq \mathcal{D}(T_1)\otimes_{\Oo_{\Ee}}\mathcal{D}(T_2)\otimes_{\Oo_{\Ee}} \Oo_{\widehat{\Ee^{\ur}}}  $
	we have 
	\[  \mathcal{D}(T_1\otimes_{\ZZ_p} T_2)\simeq \mathcal{D}(T_1)\otimes_{\Oo_{\Ee}}\mathcal{D}(T_2). \]
	
	We also have the following $\mathscr{G}_{K_{\pi}}$-equivariant isomorphisms:	
	
	\begin{align*}
	\Oo_{\widehat{\Ee^{\ur}}}\otimes_{\Oo_{\Ee}}\mathcal{D}(\Hom_{\ZZ_p}(T_1, T_2))&\simeq\Oo_{\widehat{\Ee^{\ur}}}\otimes_{\ZZ_p} \Hom_{\ZZ_p}(T_1, T_2)\\
	&\simeq \Hom_{\ZZ_p}(T_1, T_2\otimes_{\ZZ_p}\Oo_{\widehat{\Ee^{\ur}}})\\
	&\simeq \Hom_{\Oo_{\widehat{\Ee^{\ur}}}}(T_1\otimes_{\ZZ_p}\Oo_{\widehat{\Ee^{\ur}}}, T_2\otimes_{\ZZ_p}\Oo_{\widehat{\Ee^{\ur}}})\\
	&\simeq \Hom_{\Oo_{\widehat{\Ee^{\ur}}}}(\mathcal{D}(T_1)\otimes_{\Oo_{\Ee}}\Oo_{\widehat{\Ee^{\ur}}}, \mathcal{D}(T_2)\otimes_{\Oo_{\Ee}}\Oo_{\widehat{\Ee^{\ur}}})\\
	&\simeq \Oo_{\widehat{\Ee^{\ur}}}\otimes_{\Oo_{\Ee}}\Hom_{\Oo_{\Ee}}(\mathcal{D}(T_1), \mathcal{D}(T_2)).
	\end{align*}
	Recall the Galois action on $\Hom_{\Oo_{\Ee}}(\mathcal{D}(T_1), \mathcal{D}(T_2))$ is defined as $g(f)=g\circ f\circ g^{-1}.$ Taking fixed points under  $\mathscr{G}_{K_{\pi}}$ on both sides of 
	\[\Oo_{\widehat{\Ee^{\ur}}}\otimes_{\Oo_{\Ee}}\mathcal{D}(\Hom_{\ZZ_p}(T_1, T_2))\simeq \Oo_{\widehat{\Ee^{\ur}}}\otimes_{\Oo_{\Ee}}\Hom_{\Oo_{\Ee}}(\mathcal{D}(T_1), \mathcal{D}(T_2)),\] and we have 
	\[  \mathcal{D}(\Hom_{\ZZ_p}(T_1, T_2))\simeq \Hom_{\Oo_{\Ee}}(\mathcal{D}(T_1), \mathcal{D}(T_2)). \]
\end{proof}

\section{Cohomological functoriality}\label{section functorial}

\begin{definition}\label{def Fi 1}
	For $i\in \ZZ_{\geq 0},$ we put 
	\begin{align*}
	\Ff^i\index{$\Ff^i$}\colon \Ind \mathbf{\Rep}_{\ZZ_p}(\mathscr{G}_K)&\to \Ab\\
	T&\mapsto \H^i(\Cc_{\varphi, \tau}(T)).
	\end{align*}
\end{definition}

To prove theorem \ref{thm main result}, we use the strategy of \cite[1.5.2]{Tav11}: we show that $\{\Ff^i\}_{i\in\ZZ_{\geq0}}$ forms a $\delta$-functor that coincides with invariants under $\mathscr{G}_K$ in degree $0$ ($\cf$ proposition \ref{prop delta functor Zp}), and we prove its effaceability ($\cf$ section \ref{section effaceability}). As in \textit{loc. cit.}, we firstly show the result for torsion representations, then pass to the limit. For torsion representations, it is necessary to work in a category with sufficiently many injectives: we have to embed $\Rep_{\ZZ_p,\tors}(\mathscr{G}_K)$ in its ind-category.

\begin{notation}
	For $T\in\mathbf{\Rep}_{\ZZ_p, \tors}(\mathscr{G}_K),$ we put $U:=\Ind_{\mathscr{G}_K}(T)\in \Ind \mathbf{\Rep}_{\ZZ_p, \tors}(\mathscr{G}_K).$\index{$U$}
\end{notation}

\begin{lemma}\label{lemm H0} If  $T\in\mathbf{\Rep}_{\ZZ_p}(\mathscr{G}_K),$ we have $\Ff^0(T)=\H^0(\mathscr{G}_K, T).$
\end{lemma}

\begin{proof}
	By definition we have:
	\[ \Ff^0(T)=(\mathcal{O}_{\widehat{\mathcal{E}^{\ur}}}\otimes_{\ZZ_p}T)^{\mathscr{G}_{K_{\pi}}, \varphi=1, \tau_D=1}=((\mathcal{O}_{\widehat{\mathcal{E}^{\ur}}})^{\varphi=1}\otimes_{\ZZ_p}T)^{\mathscr{G}_{K_{\pi}}, \tau_D=1}=T^{\la \mathscr{G}_{K_{\pi}}, \tau\ra}=T^{\mathscr{G}_K}=\H^0(\mathscr{G}_K, T).\]
\end{proof}

\begin{lemm}\label{lemmcoho0 Zp tor}
	If $T\in\Rep_{\ZZ_p, \tors}(\mathscr{G}_K),$ then
	\begin{align*}
	\H^i(\mathscr{G}_L,\W(C^{\flat})\otimes_{\ZZ_p}T)&=0\\
	\H^i(\mathscr{G}_{K_\pi},\W(C^{\flat})\otimes_{\ZZ_p}T)&=0\\
	\H^i(\mathscr{G}_{K_\zeta},\W(C^{\flat})\otimes_{\ZZ_p}T)&=0
	\end{align*}
	for all $i>0.$
\end{lemm}

\begin{proof}
	Recall that if $M$ is a subextension of $\overline{K}/K$ whose completion is perfectoid (this is the case for $K_\pi,$ $K_\zeta$ and $L$) and $i\in\ZZ_{>0},$ then $\H^i(\mathscr{G}_M,\Oo_{C^{\flat}}\otimes_{\FF_p}T)$ is almost zero (this follows from the almost vanishing of $\H^i(\mathscr{G}_M,\mathcal{O}_{\overline{K}}/p\mathcal{O}_{\overline{K}})$), so that $\H^i(\mathscr{G}_M,C^{\flat}\otimes_{\FF_p}T)=0.$ Then we proceed by induction on $r\in\ZZ_{>0}$ such that $p^rT=0:$ the exact sequence
	$$0\to pT\to T\to T/pT\to0$$
	induces the exact sequence
	$$\H^i\big(\mathscr{G}_M,\W(C^{\flat})\otimes_{\ZZ_p}pT\big)\to \H^i\big(\mathscr{G}_M,\W(C^{\flat})\otimes_{\ZZ_p}T\big)\to \H^i\big(\mathscr{G}_M,\W(C^{\flat})\otimes_{\ZZ_p}(T/pT)\big),$$
	so that $\H^i\big(\mathscr{G}_M,\W(C^{\flat})\otimes_{\ZZ_p}T\big)=0$ since $\H^i\big(\mathscr{G}_M,\W(C^{\flat})\otimes_{\ZZ_p}pT\big)=0$ (because $pT$ is killed by $p^{r-1}$) and $\H^i\big(\mathscr{G}_M,\W(C^{\flat})\otimes_{\ZZ_p}(T/pT)\big)=0$ (because $T/pT\in\Rep_{\FF_p}(\mathscr{G}_K)$).
\end{proof}

\begin{corollary}\label{lemmcoho0 Zp}
	If $T\in\Rep_{\ZZ_p}(\mathscr{G}_K),$ then
	\begin{align*}
	\H^i(\mathscr{G}_L,\W(C^{\flat})\otimes_{\ZZ_p}T)&=0\\
	\H^i(\mathscr{G}_{K_\pi},\W(C^{\flat})\otimes_{\ZZ_p}T)&=0\\
	\H^i(\mathscr{G}_{K_\zeta},\W(C^{\flat})\otimes_{\ZZ_p}T)&=0
	\end{align*}
	for all $i>0.$
\end{corollary}

\begin{proof}
	By \cite[Theorem 2.3.4]{NSW13}, we have the exact sequence
	$$0\to \R^1\pLim{n}\H^{i-1}(\mathscr{G}_L,\W_n(C^{\flat})\otimes_{\ZZ_p}T)\to
	\H^i(\mathscr{G}_L,\W(C^{\flat})\otimes_{\ZZ_p}T)\to\pLim{n} \H^i(\mathscr{G}_L,\W_n(C^{\flat})\otimes_{\ZZ_p}T)\to0.$$
	By lemma \ref{lemmcoho0 Zp tor} we have 
	\[\H^{i-1}(\mathscr{G}_L,\W_n(C^{\flat})\otimes_{\ZZ_p}T)=0, \text{ if } i>1,\]
	while when $i=1$ we know that $\{\H^0(\mathscr{G}_L,\W_n(C^{\flat})\otimes_{\ZZ_p}T)\}_n$ has the Mittag-Leffler property. This implies \[\H^i(\mathscr{G}_L,\W_n(C^{\flat})\otimes_{\ZZ_p}T)=0, \text{ if } i>0 \] and 
	\[\R^1\pLim{n}\H^{i-1}(\mathscr{G}_L,\W_n(C^{\flat})\otimes_{\ZZ_p}T)=0, \text{ if } i>0,\]
	hence 
	\[\H^i(\mathscr{G}_L,\W(C^{\flat})\otimes_{\ZZ_p}T)=0, \text{ if } i>0.\] The proofs of the other statements are similar.
\end{proof}

\begin{notation}
	Let $T\in\Rep_{\ZZ_p}(\mathscr{G}_K),$ we will put $D_\tau=\mathcal{D}(T)_\tau$ in the rest of this section to keep light notations.
\end{notation}

\begin{lemma}\label{lemm coho procyclic}
	Let $H$ be a profinite group homeomorphic to $\ZZ_p$ with $\sigma$ a topological generator, and $M$ be a $p$-torsion $H$-module. Then the continuous cohomology $\H^i(H,M)$ is computed by the complex
	$$[M\xrightarrow{\sigma-1}M]$$
	in which the first term is in degree $0.$
\end{lemma}

\begin{proof}
	$\cf$ \cite[Proposition 1.6.13]{NSW13}.
\end{proof}

\begin{lemm}\label{lemm surjective gamma tau}
	If $n\in\ZZ_{>0},$ the operators $\gamma-1$ and $\tau_D^n-1$ are surjective on $D_\tau.$
\end{lemm}

\begin{proof}
	As $\H^i(\mathscr{G}_L,\W(C^{\flat})\otimes_{\ZZ_p}T)=0$ for all $i>0,$ together with lemma \ref{lemm coho procyclic} we have
	$$0=\H^1(\mathscr{G}_{K_\pi},\W(C^{\flat})\otimes_{\ZZ_p}T)=\H^1(\Gal(L/K_\pi),D_\tau)=\Coker\big(D_\tau\xrightarrow{\gamma-1}D_\tau\big)$$
	since $\Gal(L/K_\pi)=\overline{\langle\gamma\rangle}.$ Similarly, put $K_{\zeta,n}=L^{\overline{\la \tau^n \ra}}:$ this is a finite extension of $K_\zeta$ and we have
	$$0=\H^1(\mathscr{G}_{K_{\zeta,n}},\W(C^{\flat})\otimes_{\ZZ_p}T)=\H^1(\Gal(L/K_{\zeta,n}),D_\tau)=\Coker\big(D_\tau\xrightarrow{\tau_D^n-1}D_\tau\big)$$
	since $\Gal(L/K_{\zeta,n})=\overline{\langle\tau^n\rangle}.$
\end{proof}

\begin{lemm}\label{lemm amazing but easy equation}
	We have $(\delta-\gamma\otimes1)\circ(\tau_D-1)=\big(1-\tau_D^{\chi(\gamma)}\big)\circ(\gamma\otimes1-1)$ on $D_\tau.$
\end{lemm}

\begin{proof}
	As $(\gamma\otimes1)\circ\tau_D=\tau_D^{\chi(\gamma)}\circ(\gamma\otimes1)$, we have
	\begin{align*}
	(\delta-\gamma\otimes1)\circ(\tau_D-1) &= \delta\circ(\tau_D-1)-\gamma\circ\tau_D+\gamma\otimes1\\
	&= \tau_D^{\chi(\gamma)}-1-\tau_D^{\chi(\gamma)}\circ(\gamma\otimes1)+\gamma\otimes1\\
	&= \big(1-\tau_D^{\chi(\gamma)}\big)\circ(\gamma\otimes1-1).
	\end{align*}
\end{proof}

\begin{prop}\label{propdeltagamma Zp}
	The map $\delta-\gamma\otimes1$ is surjective on $D_\tau.$
\end{prop}

\begin{proof}
	As $\gamma-1$ and $\tau_D^{\chi(\gamma)}-1$ are surjective by lemma \ref{lemm surjective gamma tau}, so is $(\delta-\gamma\otimes1)\circ(\tau_D-1).$ Hence $\delta-\gamma$ is surjective.
\end{proof}

\begin{coro}\label{Exactness D-tau-0 Zp}
	If $0\to T^\prime\to T\to T^{\prime\prime}\to0$ is an exact sequence in $\Rep_{\ZZ_p}(\mathscr{G}_K),$ then the sequences
	$$0\to\mathcal{D}(T^\prime)\to\mathcal{D}(T)\to\mathcal{D}(T^{\prime\prime})\to0$$
	$$0\to\mathcal{D}(T^\prime)_\tau\to\mathcal{D}(T)_\tau\to\mathcal{D}(T^{\prime\prime})_\tau\to0$$
	$$0\to\mathcal{D}(T^\prime)_{\tau,0}\to\mathcal{D}(T)_{\tau,0}\to\mathcal{D}(T^{\prime\prime})_{\tau,0}\to0$$
	are exact. In particular, the functor $T\mapsto\mathcal{C}_{\varphi,\tau}(T)$ is exact.
\end{coro}

\begin{proof}
	As $\W(C^{\flat})$ is torsion-free, we have the exact sequence
	$$0\to\W(C^{\flat})\otimes_{\ZZ_p}T^\prime\to\W(C^{\flat})\otimes_{\ZZ_p}T\to\W(C^{\flat})\otimes_{\ZZ_p}T^{\prime\prime}\to0$$
	which induces the exact sequence
	$$0\to\mathcal{D}(T^\prime)_\tau\to\mathcal{D}(T)_\tau\to\mathcal{D}(T^{\prime\prime})_\tau\to\H^1(\mathscr{G}_L,\W(C^{\flat})\otimes_{\ZZ_p}T^\prime).$$
	By corollary \ref{lemmcoho0 Zp}, we get the second exact sequence. Similarly, by the observation $\H^1(\mathscr{G}_{K_{\pi}}, \Oo_{\widehat{\Ee^{\ur}}}\otimes_{\ZZ_p}T^\prime)=0$ (where $\Oo_{\widehat{\Ee^{\ur}}}$ is endowed with the $p$-adic topology, \cf lemma \ref{lemm H1 G infty vanishes}) we have the first exact sequence. Moreover, we have the commutative diagram
	$$\xymatrix{0\ar[r] & \mathcal{D}(T^\prime)_\tau\ar[r]\ar[d]_{\delta-\gamma\otimes1} & \mathcal{D}(T)_\tau\ar[r]\ar[d]_{\delta-\gamma\otimes1} & \mathcal{D}(T^{\prime\prime})_\tau\ar[r]\ar[d]_{\delta-\gamma\otimes1} & 0\\
		0\ar[r] & \mathcal{D}(T^\prime)_\tau\ar[r] & \mathcal{D}(T)_\tau\ar[r] & \mathcal{D}(T^{\prime\prime})_\tau
		\ar[r] & 0}.$$
	The snake lemma and proposition \ref{propdeltagamma Zp} provide the last exact sequence.
\end{proof}

\begin{proposition}\label{prop delta functor Zp}
	The functors $\{\Ff^i\}_{i\in \NN}$ form a $\delta$-functor.
\end{proposition}

\begin{proof}
	Let $0\to T^\prime\to T\to T^{\prime\prime}\to0$ be a short exact sequence in $\Ind \mathbf{\Rep}_{\ZZ_p}(\mathscr{G}_K).$ We have a short exact sequence of complexes $0\to\Cc_{\varphi, \tau}(T^{\prime})\to \Cc_{\varphi, \tau}(T)\to \Cc_{\varphi, \tau}(T^{\prime\prime})\to 0$ by corollary \ref{Exactness D-tau-0 Zp}. Classical result ($\cf$ \cite[Theorem 1.3.1]{Wei95}) gives the desired long exact sequence of cohomological groups.
\end{proof}

\section{Computation of \texorpdfstring{$\H^1(\Cc_{\varphi, \tau})$}{H\textonesuperior(C\textpinferior\texthinferior\textiinferior,\texttinferior\textainferior\textuinferior)}}

\begin{notation}
	For any $(\varphi, \tau)$-module $D,$ we denote $\Ext_{\Mod_{\Oo_{\Ee},\Oo_{\Ee_{\tau}}}(\varphi,\tau)}(\Oo_{\Ee}, D)$\index{$\Ext_{\Mod_{\Oo_{\Ee},\Oo_{\Ee_{\tau}}}(\varphi,\tau)}(\Oo_{\Ee}, D)$} the group of extensions of $\Oo_{\Ee}$ (the unit object in $\Mod_{\Oo_{\Ee}, \Oo_{\Ee_{\tau}}}(\varphi,\tau)$) by $D.$ More precisely, the group of exact sequences
	\[ 0\to D\to E\to \Oo_{\Ee}\to 0 \]
	in $\Mod_{\Oo_{\Ee}, \Oo_{\Ee_{\tau}}}(\varphi,\tau)$  modulo equivalence. Two extensions $E_1,E_2$ are equivalent when there exists a map $f$ making the diagram commute:
	\[ \begin{tikzcd}
	0\arrow[r] &D \arrow[r]\arrow[d, equal] &E_1 \arrow[r]\arrow[d, "f"] &\Oo_{\Ee} \arrow[r]\arrow[d, equal]& 0\\
	0\arrow[r] &D \arrow[r] &E_2 \arrow[r] &\Oo_{\Ee} \arrow[r] &0.
	\end{tikzcd}   \]
\end{notation}

\begin{lemma}\label{Extension}
	Define two submodules of $D \oplus D_{\tau}$ as follows:
	\[ M:=\Big\{(\lambda, \mu) \in D\oplus D_{\tau};\ \begin{cases}
	(\varphi-1)\mu=(\tau_D-1)\lambda\\
	(\forall g\in \mathscr{G}_{K_{\pi}})\  \chi(g)\in \ZZ_{>0} \Rightarrow g(\mu)=\mu+ \tau_D(\mu)+ \tau_D^2(\mu)+ \cdots + \tau_D^{\chi(g)-1}(\mu) 
	\end{cases} \Big\}, \]
	
	\[N:= \Big\{ \big((\varphi-1)d, (\tau_D-1)d\big);\ d\in D \Big\}. \]
	Then there is a group isomorphism  \[ \Ext(\Oo_{\Ee}, D)\simeq M/N. \]
\end{lemma}

\begin{proof}
	Consider an extension of $\Oo_{\Ee}$ by $D:$
	\[ 0\to D \to E \xrightarrow{\epsilon} \Oo_{\Ee} \to 0. \]
	It is equivalent to giving a $(\varphi, \tau)$-module structure to the $\Oo_{\Ee}$-module  $E=D\oplus \Oo_{\Ee}\cdot x,$ where $x\in E$ is a preimage of 1 under $\epsilon.$ Since $D$ is already a $(\varphi, \tau)$-module, it suffices to specify the images of $x$ by $\varphi$ and $\tau_D.$ Since $\epsilon(\varphi(x)-x)=0,$ we must have $\varphi(x)-x\in D$ and hence we put
	\[\varphi(x)=x+\lambda \text{ with } \lambda\in D.\]
	For $\tau_D,$ we tensorize the original exact sequence with $\Oo_{\Ee_{\tau}}$ as follows
	\[  0\to D_{\tau} \to \Oo_{\Ee_{\tau}}\otimes E \xrightarrow{\epsilon} \Oo_{\Ee_{\tau}} \to 0  \] 
	then similarly we have $\epsilon((\tau_D-1)(1\otimes x))=0$ and hence $(\tau_D-1)(1\otimes x)\in D_{\tau}.$ We then put
	\[ \tau_D(1\otimes x)=1\otimes x+\mu \text{ with } \mu\in D_{\tau}.\] 
	Such a pair $(\lambda, \mu)$ has to satisfy two conditions to give an extension $E.$ The first condition is that $\varphi$ commutes with $\tau_D,$ $\ie$$\varphi(\tau_D(x))=\tau_D(\varphi(x)).$ Notice 
	\[ \varphi(\tau_D(1\otimes x))=\tau_D(\varphi(1\otimes x)) \iff\varphi(1\otimes x+\mu)=\tau_D(1\otimes x+\lambda)\iff 1\otimes x+\lambda +\varphi(\mu)=1\otimes x+\mu+ \tau_D(\lambda). \]
	This is equivalent to: 
	\begin{equation}\label{equ Caruso 1}
	\varphi(\mu)-\mu=\tau_D(\lambda)-\lambda.
	\end{equation}
	The second condition is: $(g\otimes 1)(\tau_D(1\otimes x))=\tau_D^{\chi(g)}(1\otimes x)$ whenever $g\in \mathscr{G}_{K_{\pi}}/\mathscr{G}_L$ is such that $\chi(g)\in \ZZ_{>0}.$ By induction we have $\tau_D^{\chi(g)}(1\otimes x)=1\otimes x+\mu +\tau_D(\mu)+\cdots +\tau_D^{\chi(g)-1}(\mu).$ Thus the second condition rewrites as
	\begin{equation}\label{equ Caruso 1-2}
	g(\mu)=\mu+ \tau_D(\mu)+ \tau_D^2(\mu)+ \cdots + \tau_D^{\chi(g)-1}(\mu).
	\end{equation}
	Indeed
	\[(g\otimes 1)(\mu)=(g\otimes 1)(\tau_D(1\otimes x)-1\otimes x)=\tau_D^{\chi(g)}(1\otimes x)-1\otimes x=\mu+ \tau_D(\mu)+ \tau_D^2(\mu)+ \cdots + \tau_D^{\chi(g)-1}(\mu).\]

	Hence we are left to show that an extension $E$ arising from the pair $(\lambda, \mu)$ is trivial if and only if there exists $d\in D$ such that $\lambda=(\varphi-1)(d)$ and $\mu=(\tau_D-1)(d).$ Indeed, $E$ is trivial if and only if there exists $d\in D$ such that $\Oo_{\Ee}\cdot (x-d)$ is a sub-$(\varphi,\tau)$-module of $E.$ 
	%\ft{Verify that we have $\varphi(x-d)=x+(\varphi-1)d-\varphi(d)=x-d$} 
	In other words, $E=D\oplus \Oo_{\Ee}\cdot(x-d)$ as $(\varphi, \tau)$-modules. This is equivalent to the existence of $\alpha\in \Oo_{\Ee}$ and $\beta\in \Oo_{\Ee_{\tau}}$ such that $\varphi(x-d)=\alpha(x-d)$ and $\tau_D(x-d)=\beta(x-d).$ More precisely, it is $x+\lambda-\varphi(d)=\alpha(x-d)$ and after applying the map $\epsilon$ we have $\alpha=1$ and $\beta=1.$ 
	We now have $\varphi(x-d)=x-d$ and it gives directly $(\varphi-1)(d)=\lambda.$ Similarly we have $\mu=(\tau_D-1)(d).$ Hence we find the desired $d.$
\end{proof}

\begin{proposition}\label{prop Caruso H1 Zp}
	Let $T\in\Rep_{\ZZ_p}(\mathscr{G}_K),$ then we have
	\[ \H^1(\Cc_{\varphi, \tau}(T))=\H^1(\mathscr{G}_K, T).\]
\end{proposition}

\begin{proof}
	As $\mathcal{D}$ establishes the equivalence of categories, we have 
	\[\H^1(\mathscr{G}_K, T)\simeq \Ext_{\mathbf{\Rep}_{\ZZ_p}(\mathscr{G}_K)}(\ZZ_p, T)\simeq \Ext_{\Mod_{\Oo_{\Ee},\Oo_{\Ee_{\tau}}}(\varphi,\tau)}(\Oo_{\Ee}, \mathcal{D}(T))\] (since $\mathcal{D}(\ZZ_p)=\Oo_{\Ee}$ is the trivial $(\varphi, \tau)$-module). Consider the complex
	\[ 0\to \mathcal{D} \xrightarrow{\alpha} \mathcal{D}\oplus \mathcal{D}_{\tau, 0} \xrightarrow{\beta} \mathcal{D}_{\tau, 0} \to 0.  \]
	We firstly describe $\H^1(\Cc_{\varphi, \tau}(T)).$ A pair $(\lambda, \mu)\in \mathcal{D}\oplus \mathcal{D}_{\tau}$ is in $\Ker \beta$ if and only if it satisfies the following two conditions:
	\begin{equation}\label{equ Caruso 2-1}
	\varphi(\mu)-\mu=\tau_D(\lambda)-\lambda
	\end{equation}
	
	\begin{equation}\label{equ Caruso 2-2}
	\mu \in D_{\tau, 0}
	\end{equation} 
	
	We see that (\ref{equ Caruso 2-1}) and (\ref{equ Caruso 2-2}) correspond to (\ref{equ Caruso 1}) and (\ref{equ Caruso 1-2}) respectively. It suffices to show that two pairs $(\lambda_1, \mu_1)$ and $(\lambda_2, \mu_2)$ give equivalent extensions if and only if $(\lambda_2, \mu_2)-(\lambda_1, \mu_1)\in \im \alpha,$ but this is clear. 
\end{proof}

\section{Effaceability}\label{section effaceability}

\begin{lemma}\label{lemmcoker0UGL Zp}
	The map $\delta-\gamma\otimes 1\colon U^{\mathscr{G}_L}\to U^{\mathscr{G}_L}$ is surjective.
\end{lemma}

\begin{proof}
	By lemma \ref{lemm amazing but easy equation}, it is sufficient to prove that $\gamma-1$ and $\tau_D^{\chi(\gamma)}-1$ are surjective over $U^{\mathscr{G}_L},$ which is equivalent to the vanishing of $\H^1(\overline{\la \gamma\ra},  U^{\mathscr{G}_L})$ and $\H^1(\overline{\la \tau^{\chi(\gamma)}\ra},  U^{\mathscr{G}_L})$ according to lemma \ref{lemm coho procyclic}. Indeed, by lemma \ref{lemmcoho0nonGalois Zp} we have $\H^1(\mathscr{G}_L, U)=0,$ which implies that 
	\[\H^1(\overline{\la \gamma \ra},  U^{\mathscr{G}_L})=\H^1(\overline{\la \gamma\ra},  \H^0(\mathscr{G}_L, U))=\H^1(\mathscr{G}_{K_{\pi}}, U).\] The group $\H^1(\mathscr{G}_{K_{\pi}}, U)$ vanishes by lemma \ref{lemmcoho0nonGalois Zp}. Assume $\chi(\gamma)=n$ and put $K_{\zeta, n}=L^{\overline{\la \tau^n \ra}}.$ Then similarly we have 
	\[\H^1(\overline{\la \tau^{\chi(\gamma)} \ra},  U^{\mathscr{G}_L})=\H^1(\overline{\la \tau^{\chi(\gamma)} \ra},  \H^0(\mathscr{G}_L, U))=\H^1(\mathscr{G}_{K_{\zeta, n}},  U)=0,\]
	where the last equality follows from lemma \ref{lemmcoho0nonGalois Zp}.
\end{proof}

\begin{lemma}\label{Surjectivity of varphi-1 on D-tau-0 Zp}
	The map $\varphi-1\colon \mathcal{D}(U)_{\tau, 0}\to \mathcal{D}(U)_{\tau, 0}$ is surjective.
\end{lemma}

\begin{proof}
	By corollary \ref{coro H1 null} we have the following exact sequence:
	
	\[  0\to U^{\mathscr{G}_L} \to\mathcal{D}(U)_\tau\xrightarrow{\varphi-1}\mathcal{D}(U)_\tau\to0.\]
	Consider the following map of complexes:
	$$\xymatrix{0\ar[r] & U^{\mathscr{G}_L}\ar[r]\ar[d]_{\delta-\gamma\otimes 1} & \mathcal{D}(U)_\tau\ar[r]^{\varphi-1}\ar[d]_{\delta-\gamma\otimes 1} & \mathcal{D}(U)_\tau\ar[r]\ar[d]_{\delta-\gamma\otimes 1} & 0\\
		0\ar[r] & U^{\mathscr{G}_L}\ar[r] & \mathcal{D}(U)_\tau\ar[r]^{\varphi-1} & \mathcal{D}(U)_\tau
		\ar[r] & 0.}$$
	By snake lemma we have the exact sequence: 
	\[ \mathcal{D}(U)_{\tau, 0}\xrightarrow{\varphi-1} \mathcal{D}(U)_{\tau, 0} \to \Coker (\delta-\gamma\otimes 1\colon U^{\mathscr{G}_L}\to U^{\mathscr{G}_L})=0 \]
	where the last equality is lemma \ref{lemmcoker0UGL Zp}.
\end{proof}

\begin{proof}[Proof of theorem \ref{thm main result}]\label{pf thm main}
	Write $\Cc_{\varphi, \tau}(\mathcal{D}(U))$ as follows:
	
	\[\xymatrix{
		0\ar[rr]&& \mathcal{D}(U) \ar[r]^{\alpha\quad \quad \quad } & \mathcal{D}(U)\oplus \mathcal{D}(U)_{\tau, 0}  \ar[r]^{\beta} & \mathcal{D}(U)_{\tau, 0}  \ar[r]& 0   \\
		&&	x  \ar@{->}[r] &((\varphi-1)(x), (\tau_D-1)(x))  &{}\\
		&&	&{} (y,z)  \ar@{->}[r] &(\tau_D-1)(y)-(\varphi-1)(z).
	}\]

	As remarked in section \ref{section functorial}, it suffices to show the effaceability of $\Ff^i,$ more precisely that $\Ff^i(U)=0$ for $i\in \{ 1, 2\}.$ For $i=1,$ by proposition \ref{prop Caruso H1 Zp} we have $\H^1(\Cc_{\varphi, \tau}(U))=\H^1(\mathscr{G}_K, U),$ which is zero by lemma \ref{lemmcoho0nonGalois Zp}. For $i=2,$ it follows from lemma \ref{Surjectivity of varphi-1 on D-tau-0 Zp}. 
\end{proof}

\begin{remark}
We recover the fact that $\H^i(\mathscr{G}_K,T)=0$ for $T\in \Rep_{\ZZ_p}(\mathscr{G}_K)$ and $i\geq3.$
\end{remark}

\section{A counter example}\label{sec a counter example}

A more natural complex can be defined as follows.

\begin{definition}\label{def naive}
	%Let $T\in \Rep_{\ZZ_p}(\mathscr{G}_K)$ with $(D, D_{\tau})$ its $(\varphi, \tau)$-module over $(\Oo_{\Ee}, \Oo_{\Ee_{\tau}}).$ 
	Let $(D, D_{\tau})\in \Mod_{\Oo_{\Ee}, \Oo_{\Ee_{\tau}}}(\varphi, \tau).$ We define a complex $\Cc_{\varphi, \tau}^{\naif}(D)$\index{$\Cc_{\varphi, \tau}^{\naif}(D)$} as follows:
	\[ \xymatrix{
		0\ar[rr] && D \ar[r] & D\oplus D_{\tau} \ar[r] &  D_{\tau}  \ar[r]  & 0   \\
		&& x  \ar@{|->}[r] & ((\varphi-1)(x), (\tau_D-1)(x))  &{}\\
		&&&{} (y,z) \ar[r]  & (\tau_D-1)(y)-(\varphi-1)(z).
	}  \]
	If $T\in \Rep_{\ZZ_p}(\mathscr{G}_K),$ we have in particular the complex $\Cc_{\varphi, \tau}^{\naif}(\mathcal{D}(T)),$ which will also be simply denoted $\Cc_{\varphi, \tau}^{\naif}(T).$\index{$\Cc_{\varphi, \tau}^{\naif}(T)$}
\end{definition}

However, this complex \emph{does not} compute the continuous Galois cohomology in general, as showed in the following example.

\medskip

\begin{example}
	Let $T=\FF_p$ be the trivial representation: then $D=F_0$ and $D_{\tau}=F_{\tau}.$ Now we look at the following diagram, which embeds the complex $\Cc_{\varphi, \tau}(\FF_p)$ into $\Cc_{\varphi, \tau}^{\naif}(\FF_p)$ and we denote the cokernel complex by $\Cc(\FF_p).$
	
	\[\xymatrix{
		\Cc_{\varphi, \tau}(\FF_p)& 	0 \ar[r] & F_0 \ar@{=}[d]\ar[rr]^{(\varphi-1, \tau-1)}&& F_0\oplus F_{\tau, 0} \ar@{^(->}[d]\ar[rr]^{(\tau-1)\ominus(\varphi-1)}&& F_{\tau,0} \ar@{^(->}[d]\ar[r] & 0\\
		\Cc_{\varphi, \tau}^{\naif}(\FF_p)&		0 \ar[r] & F_0 \ar[d]\ar[rr]^{(\varphi-1, \tau-1)}&& F_0\oplus F_{\tau} \ar@{->>}[d]\ar[rr]^{(\tau-1)\ominus(\varphi-1)}&& F_{\tau} \ar@{->>}[d]\ar[r] & 0\\
		\Cc(\FF_p) &	0 \ar[r] & 0 \ar[rr]^{}&& F_{\tau}/F_{\tau, 0} \ar[rr]^{\varphi-1}&& F_{\tau}/F_{\tau,0} \ar[r] & 0\\
	}\]
	
	The associated long exact sequence is
	
	\begin{align*}
	0&\to \H^0(\mathscr{G}_K, \FF_p) \to \H^0_{\naif}(\mathscr{G}_K, \FF_p) \to 0\to \H^1(\mathscr{G}_K, \FF_p)\to \H^1_{\naif}(\mathscr{G}_K, \FF_p)\\
	&\to \Ker(\varphi-1)\to \H^2(\mathscr{G}_K, \FF_p) \to \H^2_{\naif}(\mathscr{G}_K, \FF_p) \to \Coker(\varphi-1)\to 0
	\end{align*}
	
	where the subscript "$\naif$" refers to the complex $\Cc_{\varphi, \tau}^{\naif}(\FF_p)$ in the middle of the diagram, while $\Ker(\varphi-1)$ and $\Coker(\varphi-1)$ refer to the last complex $F_{\tau}/F_{\tau, 0}\xrightarrow{\varphi-1}F_{\tau}/F_{\tau, 0}.$
	
	\medskip
	
	Notice that $1\not\in F_{\tau, 0}$ and $1\in \Ker(F_{\tau}\xrightarrow{\varphi-1}F_{\tau}),$ this implies that $\Ker(\varphi-1)$ is not trivial. In particular, this implies that 
	\[\H^1(\mathscr{G}_K, \FF_p)\subsetneqq \H^1_{\naif}(\mathscr{G}_K, \FF_p).\]
	More precisely, take any $\alpha\in k,$ then $(\alpha, 1)\in \Ker((\tau-1)\ominus (\varphi-1)) \setminus (F_0\oplus F_{\tau, 0})$ induces an element of $\H^1_{\naif}(\mathscr{G}_K, \FF_p)\setminus \H^1(\mathscr{G}_K, \FF_p).$
\end{example}

\begin{remark} 
	Let $T\in \Rep_{\ZZ_p}(\mathscr{G}_K)$ with $(D, D_{\tau})\in \Mod_{\Oo_{\Ee}, \Oo_{\Ee_{\tau}}}(\varphi, \tau)$ its $(\varphi, \tau)$-module. Then we have  \[(\tau_D-1)D\subset D_{\tau, 0} \subset D_{\tau} \]
	($\cf$ lemma \ref{lemm 1.1.11 well-defined tau-1 varphi-1}), and it is natural to consider the following complex:
	\[ \xymatrix{
		0\ar[rr] && D \ar[r] & D\oplus (\tau_D-1)D \ar[r] &  (\tau_D-1)D  \ar[r]  & 0   \\
		&& x  \ar@{|->}[r] & ((\varphi-1)(x), (\tau_D-1)(x))  &{}\\
		&&&{} (y,z) \ar[r]  & (\tau_D-1)(y)-(\varphi-1)(z).
	}  \]
	However, this complex \emph{does not} compute the continuous Galois cohomology of $T$ in general as it has trivial $H^2$.
\end{remark}

%% file: Chapter2.tex
\chapter{Relation with Tavares Ribeiro's complex}

In this chapter, we show that the complex $\Cc_{\varphi, \tau}$ defined in chapter \ref{chapter The complex C varphi tau} is a refinement of Tavares Ribeiro's complex introduced in \cite[\S 1.5]{Tav11}, at least in the finite residue field case.

\begin{theorem}
	Let $T$ be an integral $p$-adic representation of $\mathscr{G}_K$ and let $M=D_L(T)=\big(\mathcal{O}_{\widehat{\Ee^{\ur}}}\otimes_{\ZZ_p}T  \big)^{\mathscr{G}_L}$. Then the homology of the complex $\Cc_{\varphi, \gamma, \tau}(M)$\index{$\Cc_{\varphi, \gamma, \tau}(M)$} defined as follows:
	\[0\to M \xrightarrow{\tilde{\alpha}} M\oplus M\oplus M \xrightarrow{\tilde{\beta}}M\oplus M\oplus M\xrightarrow{\tilde{\eta}}M \to 0 \] 
	where
	\[	\tilde{\alpha}=\begin{pmatrix}
	\varphi-1\\
	\gamma-1\\
	\tau-1
	\end{pmatrix}  
	, \ 
	\tilde{\beta}=\begin{pmatrix}
	\gamma-1 &1-\varphi& 0\\
	\tau-1  &0&  1-\varphi\\
	0 &\tau^{\chi(\gamma)}-1& \delta-\gamma
	\end{pmatrix}   \]
	\[ \tilde{\eta}=\begin{pmatrix}
	\tau^{\chi(\gamma)}-1, \delta-\gamma, \varphi-1
	\end{pmatrix}  \]
	with $\delta=(\tau^{\chi(\gamma)}-1)(\tau-1)^{-1}\in \ZZ_p[\![\tau-1]\!],$ identifies canonically and functorially with the continuous Galois cohomology of $T.$ 
\end{theorem}

\begin{proof}
	$\cf$ \cite[Theorem 1.5]{Tav11}.
\end{proof}

\section{Tavares Ribeiro's complex with \texorpdfstring{$D_{\tau}$}{D\texttinferior\textainferior\textuinferior}}\label{section 2.1}

By replacing $D_L(T)$ in \cite[Lemma 1.9]{Tav11} with $\mathcal{D}(T)_{\tau},$ we have:

\begin{lemma}\label{lemm replace DL with D-tau} 
	For any $T\in \Ind \Rep_{\ZZ_p, p^r-tor}(\mathscr{G}_K)$ and $\alpha\in \ZZ_p^{\times},$ we have: 
	\begin{align*}
	0&\to \Ind_{\widehat{G}}(T)\to \mathcal{D}(\Ind_{\mathscr{G}_K}(T))_{\tau}\xrightarrow{\varphi-1} \mathcal{D}(\Ind_{\mathscr{G}_K}(T))_{\tau}\to 0\\
	0&\to \Ind_{\Gamma}(T) \to  \Ind_{\widehat{G}}(T) \xrightarrow{\tau^{\alpha}-1}\Ind_{\widehat{G}}(T) \to 0\\
	0&\to \Ind_{\mathscr{G}_K}(T)^{\mathscr{G}_K} \to \Ind_{\Gamma}(T) \xrightarrow{\gamma-1} \Ind_{\Gamma}(T)\to 0.
	\end{align*}
\end{lemma}

\begin{proof}
	The first exact sequence follows from corollary \ref{coro H1 null}, while the proof for the rest ones are the same as in \cite[Lemma 1.9]{Tav11}.
\end{proof}

\begin{definition}\label{def Fi 2}
	For $i\in \NN,$ we put 
	\begin{align*}
	\Ff^i\colon \Rep_{\ZZ_p}(\mathscr{G}_K)&\to \Ab \\
	T &\to \H^i(\Cc_{\varphi, \gamma, \tau}(\mathcal{D}(T)_{\tau})).
	\end{align*}\index{$\Ff^i$}
\end{definition}

\begin{remark}
The functor $\{ \Ff^i \}_{i}$ just defined is different from that of the previous chapter ($\cf$ definition \ref{def Fi 1}).
\end{remark}

\begin{lemma}\label{lemm right H1 C-TR}
	We have
	\[ \Ff^0(T)\simeq \H^0(\mathscr{G}_K,T).  \]
\end{lemma}
\begin{proof}
We have $\Ff^0(T)=(\W(C^{\flat})\otimes_{\ZZ_p} T)^{\mathscr{G}_L, \varphi=1, \gamma=1, \tau=1}=(\W(C^{\flat})^{\varphi=1}\otimes_{\ZZ_p} T)^{ \la \mathscr{G}_L,\tau, \gamma \ra }=T^{ \la \mathscr{G}_L,\tau, \gamma \ra }=T^{\mathscr{G}_K}.$
\end{proof}

\begin{lemma}\label{lemm H1 G infty vanishes}
	Let $\mathcal{M}$ be an $\mathcal{O}_{\widehat{\mathcal{E}^{\ur}}}$-module of finite type with a continuous (for the $p$-adic topology) semi-linear action of $\mathscr{G}_{K_{\pi}},$ $\ie$$(\forall g\in \mathscr{G}_{K_{\pi}}) (\forall \lambda\in \mathcal{O}_{\widehat{\mathcal{E}^{\ur}}}) (\forall m\in \mathcal{M})\  g(\lambda m)=g(\lambda)g(m),$ then
	\[ \H^1(\mathscr{G}_{K_{\pi}}, \mathcal{M})=0. \]
\end{lemma}

\begin{proof}
	\begin{enumerate}
		\item [(1)] We first assume that $\mathcal{M}$ is a $F_0^{\sep}$-vector space of finite dimension, and hence $p\mathcal{M}=0.$ Denote $d=\dim_{F_0^{\sep}}\mathcal{M}$ and denote $\underline{e}=(e_1, \dots, e_d)$ a basis of $\mathcal{M}$ over $F_0^{\sep}.$ For any $g\in \mathscr{G}_{K_{\pi}},$ we denote by $U_g\in \GL_d(F_0^{\sep})$ the matrix of the action of $g$ under the fixed basis $\underline{e}.$ There is a bijection between the classes of $F_0^{\sep}$-representations of $\mathscr{G}_{K_{\pi}}$ and the set $\H^1(\mathscr{G}_{K_{\pi}}, \GL_d(F_0^{\sep})),$ which has one point by Hilbert 90 theorem (here $F_0^{\sep}$ is endowed with the discrete topology). This implies that $\mathcal{M}\simeq (F_0^{\sep})^d$ as a $\mathscr{G}_{K_\pi}$-module and hence $\H^1(\mathscr{G}_{K_{\pi}}, \mathcal{M})\simeq \H^1(\mathscr{G}_{K_{\pi}}, F_0^{\sep})^d=0.$
		\item [(2)]  Now we assume that $\mathcal{M}$ is killed by a power of $p.$ We use induction as follows: if $\mathcal{M}$ is killed by $p^n,$ then we have an exact sequence
		\[ 0\to p^{n-1}\mathcal{M} \to \mathcal{M} \to \mathcal{M}/p^{n-1}\mathcal{M} \to 0, \]
		hence an exact sequence 
		\[ \cdots \to \H^1(\mathscr{G}_{K_{\pi}}, p^{n-1}\mathcal{M}) \to \H^1(\mathscr{G}_{K_{\pi}}, \mathcal{M}) \to \H^1(\mathscr{G}_{K_{\pi}}, \mathcal{M}/p^{n-1})\to \cdots. \]
		As $p^{n-1}\mathcal{M}$ is killed by $p$ and $\mathcal{M}/p^{n-1}$ by $p^{n-1},$ the induction hypothesis and (1) implies the vanishing of two sides, hence the vanishing in the middle.
		\item [(3)]  If $\mathcal{M}$ is an $\mathcal{O}_{\widehat{\mathcal{E}^{\ur}}}$-module of finite type, we have $\mathcal{M}=\pLim{n} \mathcal{M}/p^n\mathcal{M}$ and hence we have the following exact sequence
		\[ \R^1\pLim{n} \H^0(\mathscr{G}_{K_{\pi}}, \mathcal{M}/p^n)\to \H^1(\mathscr{G}_{K_{\pi}}, \mathcal{M})\to \pLim{n} \H^1(\mathscr{G}_{K_{\pi}}, \mathcal{M}/p^{n}). \] 
		The first term vanishes since the short exact sequence
		\[ 0\to p^n\mathcal{M}/p^{n+1}\to \mathcal{M}/p^{n+1} \to \mathcal{M}/p^n \to 0 \]
		gives the exact sequence \[\H^0(\mathscr{G}_{K_{\pi}}, \mathcal{M}/p^{n+1}) \to \H^0(\mathscr{G}_{K_{\pi}}, \mathcal{M}/p^n) \to \H^1(\mathscr{G}_{K_{\pi}}, p^n\mathcal{M}/p^{n+1}),\]
		and we know the right term vanishes by situation (2) we discussed. Hence we have Mittag-Leffler condition for the system of $\H^0$ and then $\R^1\pLim{n} \H^0(\mathscr{G}_{K_{\pi}}, \mathcal{M}/p^n)=0.$ Notice that $\pLim{n} \H^1(\mathscr{G}_{K_{\pi}}, \mathcal{M}/p^{n})=0$ by case (2), and hence $\H^1(\mathscr{G}_{K_{\pi}}, \mathcal{M})=0.$
	\end{enumerate}	
\end{proof}

\begin{lemma}\label{lemm continuous action p-adic topology ur}
	The group $\mathscr{G}_{K_{\pi}}$ acts continuously on the ring $\Oo_{\widehat{\mathcal{E}^{\ur}}},$ where the latter is endowed with the $p$-adic topology. 
\end{lemma}

\begin{proof}
	The extension $\mathcal{E}^{ur}/\mathcal{E}$ is Galois with group $\mathscr{G}_{K_\pi}.$The $p$-adic valuation is the unique valuation that extends the $p$-adic valuation on $\mathcal{E}:$ this implies that $\mathscr{G}_{K_\pi}$ acts by isometries hence continuously on $\mathcal{E}^{\ur}.$ In particular, the action over $\mathcal{O}_{\widehat{\mathcal{E}^{ur}}}$ is continuous for the $p$-adic topology.
\end{proof}

\begin{corollary}\label{coro D and D tau are exact}
	The functors $D(-)$ and $D(-)_{\tau}$ are exact.
\end{corollary}

\begin{proof}
	The functor $D(-)$ is left exact by construction. If $T\in\Rep_{\ZZ_p}(\mathscr{G}_{K_\pi}),$ the semi-linear action on $\mathcal{O}_{\widehat{\mathcal{E}^{\ur}}}\otimes_{\ZZ_p}T$ is continuous for the $p$-adic topology (because it is continuous on $T$ and $\Oo_{\widehat{\Ee^{\ur}}}$) : by lemma \ref{lemm H1 G infty vanishes}, the functor $D(-)$ is also right exact. Recall $D(T)_{\tau}=\Oo_{\Ee_{\tau}}\otimes_{\Oo_{\Ee}}\mathcal{D}(T),$ and hence it is exact.
\end{proof}

\begin{lemma}\label{lemm delta fct C-TR}
	The functors $\{ \Ff^i \}_{i\in \NN}\colon \Rep_{\ZZ_p}(\mathscr{G}_K)\to \Ab$ form a $\delta$-functor.
\end{lemma}

\begin{proof}
	This follows directly from corollary \ref{coro D and D tau are exact}.
\end{proof}

\begin{theorem}
	Let $T\in \Rep_{\ZZ_p}(\mathscr{G}_K)$ and $i\in\NN,$ we then have
	\[ \H^i(\mathscr{G}_K, T)\simeq \Ff^i(T). \]
\end{theorem}

\begin{proof}
	By lemmas \ref{lemm right H1 C-TR} and \ref{lemm delta fct C-TR}, it suffices to prove the effaceability of $\{\Ff^i \}_{i\in \NN}.$ Similar proof as \cite[Proposition 1.7]{Tav11} works, as it is based on \cite[Lemma 1.9]{Tav11}. 
\end{proof}

\begin{notation}
	Let $T\in \Rep_{\ZZ_p}(\mathscr{G}_K).$ To make light notations, in the rest of this chapter we denote $\Cc_{\varphi, \tau}$\index{$\Cc_{\varphi, \tau}$} for the complex $\Cc_{\varphi, \tau}(T),$ and $\Cc_{\TR}$\index{$\Cc_{\TR}$} for Tavares Ribeiro's complex attached to $\mathcal{D}(T)_{\tau},$ $\ie$$\Cc_{\varphi, \gamma, \tau}(\mathcal{D}(T)_{\tau}).$
\end{notation}

Using the morphisms introduced by Tavares Ribeiro, we have another description of $\mathcal{D}(T)_{\tau, 0}$ as follows:

\begin{lemma}\label{D tau 0 appears}
	Let $x\in \mathcal{D}(T)_{\tau},$ then $x\in \mathcal{D}(T)_{\tau,0} \iff (\delta-\gamma)x=0.$ 
\end{lemma}

\begin{proof} This is a translation of lemma \ref{Replace g witi gamma} with new notations.
\end{proof}

\begin{notation}
	Let $F_0^{\rad}$\index{$F_0^{\rad}$} be the radical closure of $F_0$ in $C^{\flat},$ and $\widehat{F_0^{\rad}}$\index{$\widehat{F_0^{\rad}}$} the closure (for the valuation topology) of $F_0^{\rad}$ in $C^{\flat}:$ this is a complete perfect subfield of $C^{\flat}.$ Let $\W(\widehat{F_0^{\rad}})$ be the ring of Witt vectors of $\widehat{F_0^{\rad}}.$\index{$\W(\widehat{F_0^{\rad}})$}
\end{notation}

\begin{lemma}\label{lemm gamma=1}
	In $\mathbf{\Rep}_{\ZZ_p}(\mathscr{G}_K),$ we have $(\mathcal{D}(T)_{\tau})^{\gamma=1}=\W(\widehat{F_0^{\rad}})\otimes \mathcal{D}(T).$
\end{lemma}
\begin{proof}
	$(\mathcal{D}(T)_{\tau})^{\gamma=1}=(\mathcal{D}(T)_{\tau})^{\mathscr{G}_{K_{\pi}}}=(\Oo_{\Ee_{\tau}}\otimes_{\Oo_{\Ee}} \mathcal{D}(T))^{\mathscr{G}_{K_{\pi}}}=(\Oo_{\Ee_{\tau}})^{\mathscr{G}_{K_{\pi}}}\otimes_{\Oo_{\Ee}} \mathcal{D}(T)=\W(\widehat{F_0^{\rad}})\otimes \mathcal{D}(T).$ Indeed, $F_\tau^{\gamma=1}=F_\tau^{\mathscr{G}_{K_{\pi}}}=\widehat{F_0^{\rad}}.$ 
\end{proof}

Let $T\in\Rep_{\ZZ_p}(\mathscr{G}_K)$ and put $D=\mathcal{D}(T).$ We have the following morphism of complexes from $\Cc_{\varphi, \tau}$ to $\Cc_{\TR}$
$$\xymatrix@R=15pt@C=30pt{
	0\ar[r] & D\ar[r]^-\alpha\ar[d]_{\textrm{incl}} & D\oplus D_{\tau,0}\ar[r]^-\beta\ar[d]_u & D_{\tau,0}\ar[r]\ar[d]^v & 0\ar[r]\ar[d] & 0\\
	0\ar[r] & D_\tau\ar[r]^-{\widetilde{\alpha}} & D_\tau\oplus D_\tau\oplus D_\tau\ar[r]^-{\widetilde{\beta}} & D_\tau\oplus D_\tau\oplus D_\tau\ar[r]^-{\widetilde{\eta}} & D_\tau\ar[r] & 0.}$$
where
\begin{align*}
\alpha(d) &= ((\varphi-1)d,(\tau_D-1)d),\\
\beta(x,z) &= (\tau_D-1)x-(\varphi-1)z,\\
\widetilde{\alpha} &= \left(\begin{smallmatrix}\varphi-1\\ \gamma-1\\ \tau_D-1\end{smallmatrix}\right),\\
\widetilde{\beta} &= \left(\begin{smallmatrix} \gamma-1 & 1-\varphi & 0\\ \tau_D-1 & 0 & 1-\varphi\\ 0 & \tau_D^{\chi(\gamma)}-1 & \delta-\gamma\end{smallmatrix}\right),\\
\widetilde{\eta} &= (\tau_D^{\chi(\gamma)}-1,\delta-\gamma,\varphi-1)
\end{align*}
and
$$\xymatrix@R=15pt@C=30pt{
	d\ar@{|->}[r]^-\alpha\ar@{|->}[d] & (x=(\varphi-1)d,z=(\tau_D-1)d)\ar@{|->}[r]^-\beta\ar@{|->}[d]^u & t=(\tau_D-1)x-(\varphi-1)z\ar@{|->}[r]\ar@{|->}[d]^-v & 0\ar@{|->}[d]\\
	d\ar@{|->}[r]^-{\widetilde{\alpha}} & (x,0,z)\ar@{|->}[r]^-{\widetilde{\beta}} & (0,t,0)\ar@{|->}[r]^-{\widetilde{\eta}} & 0}$$

\begin{theorem}\label{thm map from HC to TR}
	This morphism of complexes is a quasi-isomorphism when the residue field of $K$ is finite.
\end{theorem}

\section{The intermediate complex \texorpdfstring{$\Cc_{\varphi, \tau}^{\rad}$}{C\textpinferior\texthinferior\textiinferior,\texttinferior\textainferior\textuinferior \unichar{"0072}\unichar{"0061}\unichar{"0064}rad}}

\begin{definition}
	%Let $T\in \Rep_{\ZZ_p}(\mathscr{G}_K)$ and denote $D=\mathcal{D}(T)$ for light notation. We then define a complex $\mathcal{C}_{\varphi,\tau}^{\rad}(D)$ (or simply $\mathcal{C}_{\varphi,\tau}^{\rad}$) as follows: 
	Let $(D, D_{\tau})\in \Mod_{\Oo_{\Ee}, \Oo_{\Ee_{\tau}}}(\varphi, \tau).$ We define a complex $\mathcal{C}_{\varphi,\tau}^{\rad}(D)$\index{$\mathcal{C}_{\varphi,\tau}^{\rad}(D)$} as follows:
	%	\begin{eqnarray*}
	%		0\to \W(\widehat{F_0^{\rad}})\otimes D \xrightarrow{\alpha} & \W(\widehat{F_0^{\rad}})\otimes D\bigoplus D_{\tau, 0}  &\xrightarrow{\beta}  D_{\tau, 0}  \to 0   \\
	%		x  \mapsto &((\varphi-1)x, (\tau_D-1)x)  &{}\\
	%		&{} (y,z) &\mapsto (\tau_D-1)y-(\gamma-1)z
	%	\end{eqnarray*}
	\[ \xymatrix{
		0\ar[rr] && \W(\widehat{F_0^{\rad}})\otimes D \ar[r]^{{\alpha \quad \quad }} & \W(\widehat{F_0^{\rad}})\otimes D\bigoplus D_{\tau, 0} \ar[r]^{\quad \quad \quad \beta} &  D_{\tau, 0}   \ar[r]  & 0   \\
		&& x  \ar@{|->}[r] & ((\varphi-1)(x), (\tau_D-1)(x))  &{}\\
		&&&{} (y,z) \ar[r]  & (\tau_D-1)(y)-(\varphi-1)(z).
	}  \]
	If $T\in \Rep_{\ZZ_p}(\mathscr{G}_K),$ we have in particular the complex $\mathcal{C}_{\varphi,\tau}^{\rad}(\mathcal{D}(T)),$ which will also be simply denoted $\mathcal{C}_{\varphi,\tau}^{\rad}(T).$\index{$\mathcal{C}_{\varphi,\tau}^{\rad}(T)$}
\end{definition}

\begin{remark}
	The morphism $\tau_D-1\colon  \W(\widehat{F_0^{\rad}})\otimes D \to D_{\tau, 0}$ is well-defined. Indeed, for any $x\in\W(\widehat{F_0^{\rad}})\otimes D$ ($\ie$$x\in D_{\tau}$ with $(\gamma-1)x=0$ $\cf$ lemma \ref{lemm gamma=1}), we have by lemma \ref{lemm amazing but easy equation}
	\[ (\delta-\gamma)(\tau_D-1)x= (1-\tau_D^{\chi(\gamma)})(\gamma-1)x=0.\]
	This exactly tells that $(\tau_D-1)x\in D_{\tau, 0}.$ Hence the complex $\Cc_{\varphi, \tau}^{\rad}$ is well defined.
\end{remark}

\begin{remark} We have a morphism from $\Cc_{\varphi, \tau}^{\rad}$ to $\Cc_{\TR}$ as follows:
	$$\xymatrix@R=15pt@C=30pt{
		0\ar[r] & \W(\widehat{F_0^{\rad}})\otimes D\ar[r]^-\alpha\ar[d]_{\textrm{incl}} & \W(\widehat{F_0^{\rad}})\otimes D\bigoplus D_{\tau,0}\ar[r]^-\beta\ar[d]_u & D_{\tau,0}\ar[r]\ar[d]^v & 0\ar[r]\ar[d] & 0\\
		0\ar[r] & D_\tau\ar[r]^-{\widetilde{\alpha}} & D_\tau\oplus D_\tau\oplus D_\tau\ar[r]^-{\widetilde{\beta}} & D_\tau\oplus D_\tau\oplus D_\tau\ar[r]^-{\widetilde{\eta}} & D_\tau\ar[r] & 0}$$
	where the maps are defined similarly as in section \ref{section 2.1}.
\end{remark}

\begin{theorem}\label{thm main result for C-phi-tau-rad}
	Let $T\in \Rep_{\ZZ_p}(\mathscr{G}_K),$ the homology of the complex $\Cc_{\varphi, \tau}^{\rad}(T)$ computes $\H^i(\mathscr{G}_K, T).$ 
\end{theorem}

To prove the theorem, we need the following results.

\begin{lemma}
	The constructed map from $\Cc_{\varphi, \tau}^{\rad}$ to $\Cc_{\TR}$ induces an isomorphism on $\H^1.$
\end{lemma}
\begin{proof}
	For injectivity: suppose a pair $(x, z)\in \Ker \beta$ is mapped to $0$ in $\H^1(\Cc_{\TR})$ by $u,$ $\ie$there exists $d\in D_{\tau}$ such that 
	\[\begin{cases}
	(\varphi-1)d=x\\
	(\gamma-1)d=0\\
	(\tau_D-1)d=z.
	\end{cases} \] 
	Then $d\in (D_{\tau})^{\gamma=1}=\W(\widehat{F_0^{\rad}})\otimes D$ and $\alpha(d)=(x, z)\in \im \alpha$ and hence $[(x,z)]=0\in \H^1(\Cc_{\varphi, \tau}^{\rad}).$
	
	\medskip
	For surjectivity: given $(a, b, c)\in \Ker \tilde{\beta},$ $\ie$
	\[
	\begin{cases}
	(\gamma-1)a+(1-\varphi)b=0\\
	(\tau_D-1)a+(1-\varphi)c=0\\
	(\tau_D^{\chi(\gamma)}-1)b+(\delta-\gamma)c=0.
	\end{cases}
	\]
	By lemma \ref{lemm surjective gamma tau}, we can fix an element $s\in D_{\tau}$ such that $(\gamma-1)s=b.$ Denote $x^{\prime}:= (\varphi-1)s-a$ and $z^{\prime}:=c-(\tau_D-1)s\in D_{\tau}.$ We have $(x^{\prime}, z^{\prime})\in \Ker \beta$ so that $[u(x^{\prime},z^{\prime})]=[(a, b, c)]\in \H^1(\Cc_{\TR}).$  Indeed, we have 
	\[(\gamma-1)x^{\prime}=(\gamma-1)((\varphi-1)s-a)=(\varphi-1)b-(\gamma-1)a=0\]
	and
	\begin{align*}
	(\delta-\gamma)z^{\prime}&=(\delta-\gamma)(c-(\tau_D-1)s)\\
	&=-(\delta-\gamma)c+(\delta-\gamma)((\tau_D-1)s)\\
	&=-(\delta-\gamma)c+ (\tau_D^{\chi(\gamma)}-1)s-\gamma(\tau_D-1)s\\
	&=-(\delta-\gamma)c+ (\tau_D^{\chi(\gamma)}-1)s-(\tau_D^{\chi(\gamma)}\gamma-\gamma)s\\
	&=-(\delta-\gamma)c+ (\tau_D^{\chi(\gamma)}-1)(1-\gamma)s\\
	&=-(\delta-\gamma)c+ (\tau_D^{\chi(\gamma)}-1)b\\
	&=0
	\end{align*}
	hence $x^{\prime}\in (D_{\tau})^{\gamma=1}=\W(\widehat{F_0^{\rad}})\otimes D$ and $z^{\prime}\in D_{\tau,0}.$ Finally we have
	\begin{align*}
	&(\tau_D-1)x^{\prime}-(\varphi-1)z^{\prime}\\
	=&(\tau_D-1)((\varphi-1)s-a)-(\varphi-1)(c-(\tau_D-1)s)\\
	=&(\tau_D-1)a-(\varphi-1)c\\
	=&0
	\end{align*}
	hence $(x^{\prime}, z^{\prime})\in \Ker \beta$ is as required.
\end{proof}

\begin{lemma}\label{Coro H2 injec C-Rad}
	The constructed map from $\Cc_{\varphi, \tau}^{\rad}$ to $\Cc_{\TR}$ induces an injective map on $\H^2.$
\end{lemma}

\begin{proof}
	Take $t\in D_{\tau, 0}$ and suppose $v(t)$ vanishes in $\H^2(\Cc_{\varphi, \tau}).$ Hence there exists $(a, b, c)\in D_\tau^{\oplus 3}$ such that $\tilde{\beta}(a, b, c)=(0, t, 0),$ $\ie$the following relations holds:

	\begin{equation}\label{equ inj H2 (1)}
	(\gamma-1)a+(1-\varphi)b=0
	\end{equation}
	
	\begin{equation}\label{equ inj H2 (2)}
	(\tau_D-1)a+(1-\varphi)c=t
	\end{equation}
	
	\begin{equation}\label{equ inj H2 (3)}
	(\tau_D^{\chi(\gamma)}-1)b+(\delta-\gamma)c=0.
	\end{equation}
	
	It suffices to show there exist $a^{\prime}\in \W(\widehat{F_0^{\rad}})\otimes_{\Oo_{\Ee}}D$ and $c^{\prime}\in D_{\tau, 0}$ such that $(\tau_D-1)a^{\prime}+(\varphi-1)c^{\prime}=t.$ \\
	By lemma \ref{lemm surjective gamma tau}, we can fix an element $b^{\prime}\in D_{\tau}$ such that $(\gamma-1)b^{\prime}=b.$ From (\ref{equ inj H2 (1)}) we have
	\begin{equation}
	(\gamma-1)(a+(1-\varphi)b^{\prime})=0
	\end{equation} \label{equ H2 inj (1)''}
	and from (\ref{equ inj H2 (3)}) we have $(\tau_D^{\chi(\gamma)}-1)(\gamma-1)b^{\prime}+(\delta-\gamma)c=0.$  
	Lemma \ref{lemm amazing but easy equation} then tells 
	\begin{equation}\label{equ H2 inj (3)''}
	(\delta-\gamma)(c-(\tau_D-1)b^{\prime})=0.
	\end{equation} 
	Now (\ref{equ inj H2 (2)}) implies $(\tau_D-1)(a+(1-\varphi)b^{\prime})+(1-\varphi)(c-(\tau_D-1)b^{\prime})=t.$ Hence $a^{\prime}:=a+(1-\varphi)b^{\prime}$ and $c^{\prime}:=-c+(\tau_D-1)b^{\prime}$ are good candidates, as they satisfy $(\tau_D-1)a^{\prime}+(\varphi-1)c^{\prime}=t.$ Indeed, we have now
	\begin{align*}
	(\gamma-1)a^{\prime}&=0\\
	(\delta-\gamma)c^{\prime}&=0		
	\end{align*} 
	which tells exactly $a^{\prime}\in D_{\tau}^{\mathscr{G}_{K_{\pi}}}=\W(\widehat{F_0^{\rad}})\otimes_{F_0}D$ and  $c^{\prime}\in D_{\tau, 0}.$
\end{proof}

\begin{lemma}\label{Coro H2 surj C-Rad}
	The constructed map from $\Cc_{\varphi, \tau}^{\rad}$ to $\Cc_{\TR}$ induces a surjective map on $\H^2.$
\end{lemma}

\begin{proof}
	Let $(\tilde{a}, \tilde{b}, \tilde{c})\in \Ker \tilde{\eta},$ $\ie$
	\begin{equation}\label{equ (4) surj H2}
	(\tau_D^{\chi(\gamma)}-1)\tilde{a} + (\delta-\gamma)\tilde{b} +(\varphi-1)\tilde{c}=0. 
	\end{equation} 
	
	It suffices to find $t\in D_{\tau, 0}$ such that $[v(t)]=[(\tilde{a}, \tilde{b}, \tilde{c})]\in \H^2(\Cc_{\TR}).$ More precisely, to find $(a, b, c)\in D_{\tau}^{\oplus 3}$ and $t\in D_{\tau, 0}$ that satisfy the following system of equations:
	
	\[
	\begin{cases}
	(\gamma-1)a+(1-\varphi)b=\tilde{a}\\
	(\tau_D-1)a+(1-\varphi)c+t=\tilde{b}\\
	(\tau_D^{\chi(\gamma)}-1)b+(\delta-\gamma)c=\tilde{c}.
	\end{cases}
	\]
	
	Take $b=0:$ it suffices to solve
	\[
	\begin{cases}
	(\gamma-1)a=\tilde{a}\\
	(\tau_D-1)a+(1-\varphi)c+t=\tilde{b}\\
	(\delta-\gamma)c=\tilde{c}.
	\end{cases}
	\]
	
	By lemma \ref{lemm surjective gamma tau}, we can find $a,\ c\in D_{\tau}$ such that $(\gamma-1)a=\tilde{a},\ (\delta-\gamma)c=\tilde{c}.$ Now we have
	\begin{align*}
	(\delta-\gamma)((\tau_D-1)a+(1-\varphi)c)=& (1-\tau_D^{\chi(\gamma)})(\gamma-1)a+(1-\varphi)\tilde{c}\\
	=& (1-\tau_D^{\chi(\gamma)})\tilde{a}+(1-\varphi)\tilde{c}\\
	=& (\delta-\gamma)\tilde{b}
	\end{align*}
	where the first equality follows from lemma \ref{lemm amazing but easy equation} and the last one is from (\ref{equ (4) surj H2}). This shows that \[t:=\tilde{b}-(\tau_D-1)a-(1-\varphi)c\] belongs to $D_{\tau,0},$ as required.
\end{proof}

\begin{remark}
	We hence proved theorem \ref{thm main result for C-phi-tau-rad} without refering to proposition \ref{prop Caruso H1 Zp}.
\end{remark}

\section{Proof of Theorem \ref{thm map from HC to TR}}

\begin{lemma}\label{lemm0}
	If $\delta\in \W(\widehat{F_0^{\rad}})\otimes_{\Oo_{\Ee}}D$ is such that $(\varphi-1)(\delta)\in D,$ then $\delta\in D.$ 
\end{lemma}

\begin{proof}\begin{enumerate}
		\item [(1)] First consider the case when $D$ is killed by $p:$ we have to prove that if $\delta\in\widehat{F_0^{\rad}}\otimes_{F_0}D$ satisfies $(\varphi-1)\delta\in D,$ then $\delta\in D.$\\
		We have
		$$F_0^{\rad}=\bigoplus\limits_{\lambda\in\Lambda}F_0\pi^\lambda$$
		where $\Lambda=\ZZ\big[p^{-1}\big]\cap[0,1).$ This shows that
		$$F_\tau^{\gamma=1}=\Big\{\sum\limits_{\lambda\in\Lambda}\alpha_\lambda\pi^\lambda\,;\,(\alpha_\lambda)_{\lambda\in\Lambda}\in F_0^\Lambda,\,\Lim{\lambda}\alpha_\lambda=0\Big\}$$
		where the limit is taken for the filter complement of finite subsets in $\Lambda$ (this means that for every $C\in\RR,$ the set $\{\lambda\in\Lambda\,;\,v_\pi(a_\lambda)<C\}$ is finite). This implies that
		$$D_\tau^{\gamma=1}=F_\tau^{\gamma=1}\otimes_{F_0}D=\Big\{\sum\limits_{\lambda\in\Lambda}\pi^\lambda\otimes d_\lambda\,;\,(d_\lambda)_{\lambda\in\Lambda}\in D^\Lambda,\,\Lim{\lambda} d_\lambda=0\Big\}.$$
		Here we use a basis $\mathfrak{B}=(e_1,\ldots,e_r)$ of $D$ over $F_0,$ and endow $D$ with the "valuation" $$v_{\mathfrak{B},\pi}(x_1e_1+\cdots+x_re_r)=\min\limits_{1\leq i\leq r}v_\pi(x_i).$$
		Put $\mathcal{D}=\mathcal{D}_{\mathfrak{B}}:=\bigoplus\limits_{i=1}^rk[\![\pi]\!]e_i,$ this is a lattice in $D.$
		
		\noindent
		Note that the writting $\sum\limits_{\lambda\in\Lambda}\pi^\lambda\otimes d_\lambda$ of an element in $F_\tau^{\gamma=1}\otimes_{F_0}D$ as above is unique.
		
		\smallskip
		
		\noindent
		Write $\delta=\sum\limits_{\lambda\in\Lambda}\pi^\lambda\otimes d_\lambda$ as above: we have
		\begin{align*}
		(\varphi-1)\delta &= \sum\limits_{\lambda\in\Lambda}(\pi^{p\lambda}\otimes\varphi(d_\lambda)-\pi^\lambda\otimes d_\lambda)\\
		&= \sum\limits_{\lambda\in\Lambda}\pi^\lambda\otimes\Big(\sum\limits_{m=0}^{p-1}\pi^m\varphi\Big(d_{\frac{m+\lambda}{p}}\Big)-d_\lambda\Big)
		\end{align*}
		so that $d_\lambda=\sum\limits_{m=0}^{p-1}\pi^m\varphi\Big(d_{\frac{m+\lambda}{p}}\Big)$ for all $\lambda\in\Lambda\setminus\{0\}.$
		
		\medskip

		\noindent
		For $k\in\ZZ_{>0},$ put $c_k=\min\limits_{\substack{\lambda\in\Lambda\\ p^{k-1}\lambda\notin\ZZ}}v_{\mathfrak{B},\pi}(d_\lambda).$ We have $\Lim{k\to\infty}c_k=+\infty.$
		
		\noindent
		Let $k\in\ZZ_{>0}$ and $\lambda\in\Lambda$ such that $p^k\lambda\in\ZZ$ and $p^{k-1}\lambda\notin\ZZ.$ If $m\in\{0,\ldots,p-1\}$ and $\mu=\frac{m+\lambda}{p},$ we have $p^k\mu\notin\ZZ,$ so $v_{\mathfrak{B},\pi}(d_\mu)\geq c_{k+1}.$ Assume $c_{k+1}>0:$ we have $v_{\mathfrak{B},\pi}(\varphi(d_\mu))\geq pc_{k+1}.$ This implies that $v_{\mathfrak{B},\pi}(d_\lambda)\geq pc_{k+1}\geq c_{k+1}.$ Thus we have
		$$c_{k+1}>0\Rightarrow c_k\geq c_{k+1}>0.$$
		Let $c\in\RR_{>0}:$ we have $c_k\geq c$ for $k\gg0,$ and the above shows that $c_k\geq c$ for all $k\in\ZZ_{>0}.$ As this holds for all $c>0,$ this means that $c_1=+\infty,$ $\ie$$d_\lambda=0$ whenever $\lambda\neq0,$ so that
		$$\delta=d_0\in D.$$
		
		\item [(2)] In the general case, let $\delta\in\textrm{W}(\widehat{F_0^{\textrm{rad}}})\otimes_{\mathcal{O}_{\mathcal{E}}}D$ be such that $(\varphi-1)\delta\in D.$ We show that
		$$\delta\in D+p^n\textrm{W}(\widehat{F_0^{\textrm{rad}}})\otimes_{\mathcal{O}_{\mathcal{E}}}D$$
		by induction on $n\in\ZZ_{>0}.$ The case $n=1$ follows from the special case above. Assume $n\geq1$ and that we have $\delta\in D+p^n\textrm{W}(\widehat{F_0^{\textrm{rad}}})\otimes_{\mathcal{O}_{\mathcal{E}}}D:$ write $\delta=\delta_n+\delta_n'$ with $\delta_n\in D$ and $\delta_n'\in p^n\textrm{W}(\widehat{F_0^{\textrm{rad}}})\otimes_{\mathcal{O}_{\mathcal{E}}}D.$ Then
		$$(\varphi-1)\delta_n'=(\varphi-1)\delta_n-(\varphi-1)\delta\in D\cap p^n\textrm{W}(\widehat{F_0^{\textrm{rad}}})\otimes_{\mathcal{O}_{\mathcal{E}}}D=p^nD.$$
		If we apply the special case to the image of $\delta_n'$ in $p^nD/p^{n+1}D$ (which is a $\varphi$-module over $F_0$)\ft{Observe that the proof of the particular case ($D$ killed by $p$) uses nothing but the $\varphi$-module structure of $D:$ it holds with $p^nD/p^{n+1}D.$}, we deduce that $\delta_n^\prime\in p^nD+p^{n+1}\textrm{W}(\widehat{F_0^{\textrm{rad}}})\otimes_{\mathcal{O}_{\mathcal{E}}}D,$ which shows that
		$$\delta\in D+p^{n+1}\textrm{W}(\widehat{F_0^{\textrm{rad}}})\otimes_{\mathcal{O}_{\mathcal{E}}}D.$$
		As this holds for all $n\in\ZZ_{>0},$ this shows that $\delta$ belongs to the closure of $D$ in $\textrm{W}(\widehat{F_0^{\textrm{rad}}})\otimes_{\mathcal{O}_{\mathcal{E}}}D$ for the $p$-adic topology: as $D$ is closed, we deduce that $\delta\in D.$ 
		
	\end{enumerate}
\end{proof}

\subsection{Morphism from \texorpdfstring{$\Cc_{\varphi, \tau}$ to $\Cc_{\varphi, \tau}^{\rad}$}{C\textpinferior\texthinferior\textiinferior,\texttinferior\textainferior\textuinferior to C\textpinferior\texthinferior\textiinferior,\texttinferior\textainferior\textuinferior rad}}

We have a morphism from $\Cc_{\varphi, \tau}$ to $\mathcal{C}^{\rad}_{\varphi,\tau}$ as follows
$$\xymatrix{
	0\ar[r] & D\ar[d]_{\mathrm{incl}}\ar[r]^\alpha & D\oplus D_{\tau,0}\ar[d]_{\mathrm{incl}}\ar[r]^\beta & D_{\tau,0}\ar@{=}[d]\ar[r] & 0\\
	0\ar[r] & \W(\widehat{F_0^{\rad}})\otimes_{\Oo_{\Ee}}D\ar[r]^-{\alpha^{\rad}} &  \W(\widehat{F_0^{\rad}})\otimes_{\Oo_{\Ee}}D\bigoplus D_{\tau,0}\ar[r]^-{\beta^{\rad}} & D_{\tau,0}\ar[r] & 0.}$$

\medskip

\begin{proposition}
	The above morphism is a quasi-isomorphism when the residue field of $K$ is finite. 
\end{proposition}

\begin{proof}
	For $\H^0:$ We have $\H^0(\Cc_{\varphi, \tau})=\Ker(\alpha)\subset\Ker(\alpha^{\rad})=\H^0(\Cc_{\varphi, \tau}^{\rad}).$ We have $(\W(\widehat{F_0^{\rad}})\otimes_{\Oo_{\Ee}}D)^{\varphi=1}=D^{\varphi=1}$ (by lemma \ref{lemm0}). Thus $\Ker(\alpha^{\rad})\subset\Ker(\alpha)=\H^0(\Cc_{\varphi, \tau}).$

	\medskip

	Injectivity of $\H^1:$ Let $(x,y)\in\Ker(\beta)$ whose image in $\H^1(\mathcal{C}_{\varphi,\tau}^{\rad})$ is zero: there exists $\delta\in\W(\widehat{F_0^{\rad}})\otimes_{\Oo_{\Ee}}D$ such that $(x,y)=((\varphi-1)\delta,(\tau_D-1)\delta).$ As $(\varphi-1)\delta=x\in D:$ this implies $\delta\in D$ by lemma \ref{lemm0}, so $(x,y)\in\im(\alpha),$ and so that the class of $(x,y)$ is zero in $\H^1(\Cc_{\varphi, \tau}).$ This proves the injectivity of $\H^1(\Cc_{\varphi, \tau})\to \H^1(\mathcal{C}_{\varphi,\tau}^{\rad}).$

	\medskip

	Surjectivity of $\H^1:$ 
	By theorems \ref{thm main result} and \ref{thm main result for C-phi-tau-rad}, the homology groups $\H^1$ of both complexes are isomorphic to $\H^1(\mathscr{G}_K, T),$ hence they have the same finite dimension. Notice the map is a linear transformation, hence injectivity implies surjectivity.

	\medskip

	For $\H^2:$ Since we proved that both $\Cc_{\varphi, \tau}$ and $\Cc_{\varphi, \tau}^{\rad}$ have the correct $\H^2$ by theorems \ref{thm main result} and \ref{thm main result for C-phi-tau-rad}, hence have the same finite dimension. It is enough to show either injectivity or surjectivity. But the surjectivity is direct.
\end{proof}

\begin{corollary}
	The complexes $\Cc_{\varphi, \tau},$ $\Cc_{\varphi, \tau}^{\rad}$ and $\Cc_{\TR}$ are all quasi-isomorphic when the residue field of $K$ is finite.
\end{corollary}
In particular, we have proved theorem \ref{thm map from HC to TR}.

%% file: Chapter3.tex
\chapter{Complexes with \texorpdfstring{$\psi$}{psi}-operator}\label{chapter Complex with psi-operator}

In this chapter, we construct a complex with $\psi$ operator (similar as in \cite[\S 3]{Her98}) that computes the continuous Galois cohomology.

\medskip

As $F_{\tau}$ is perfect, we cannot have such an operator on $(\varphi, \tau)$-modules over $(\Oo_{\Ee}, \Oo_{\Ee_{\tau}})$ (\cf remark \ref{rem why we need unperfected version}): we have to use a refinement of $(\varphi, \tau)$-module theory developped in \cite[1.2.2]{Car13}. More precisely, we will work with coefficients $\Oo_{\Ee_{u, \tau}},$ whose residue field is not perfect ($\cf$ notation \ref{not dictionary}).

\medskip

We construct a complex $\Cc^{u}_{\varphi, \tau}$ in section \ref{section C varphi tau u} and we show that it computes the continuous Galois cohomology. Replacing the operator $\varphi$ with $\psi$ in $\Cc^{u}_{\varphi,\tau}$ provides another complex $\Cc^{u}_{\psi,\tau}$ and we show that these two complexes are quasi-isomorphic. Hence $\Cc^{u}_{\psi,\tau}$ computes the continuous Galois cohomology.

\medskip

To prove the results mentioned above, we will as usual start with $\FF_p$-representations and then pass to $\ZZ_p$-representations by d\'evissage.

\section{The \texorpdfstring{$(\varphi, \tau)$}{(phi,tau)}-modules over partially unperfected coefficients}\label{section modules over unperfected coefficients}

\medskip

For simplicity, we will denote by $u$ and $\eta$ the elements $\widetilde{\pi}$ and $\varepsilon-1$ in $C^\flat$ (this is a little abuse of notation since, strictly speaking, $u$ is a variable that maps to $\widetilde{\pi}$ under the injective map $F_0 \to C^\flat$ and similarly for $\eta$) or the elements $[\widetilde{\pi}]$ and $[\varepsilon]-1$ in $\W(C^\flat)$ under the injective map $\Oo_{\Ee}\to \W(C^{\flat}).$

\begin{notation}($\cf$ \cite[\S 1.2.2]{Car13})\label{rem fixed finite degree}
	\item (1) We put \[ F_{u, \tau}:= k(\!(u, \eta^{1/p^{\infty}})\!)=k(\!(u, \eta )\!)\big[\eta^{1/p^{\infty}}\big]=\bigcup_{n\in \NN}k(\!(u, \eta)\!)[\eta^{1/p^n}]\subset C^{\flat}. \]\index{$F_{u, \tau}$}
	
	\medskip
	
	\item (2) By an abuse of notation we denote 
	\[  C_{u-\np}^{\flat}=F_{u,\tau}^{\sep}\subset C^{\flat}  \]\index{$C_{u-\np}^{\flat}$}
	the separable closure of $F_{u, \tau}$ in $C^{\flat}.$   
	
	\medskip
	
	Note that $C^{\flat}_{u-\np}$ is \emph{not} the tilt of a perfectoid field, though ambiguously it carries a superscript $\flat$ in the notation. 
	
	\item (3) We put \[\mathsf{G}:=\Gal(F_{u,\tau}^{\sep}/ F_{u, \tau}).\]	  \index{$\mathsf{G}$}
\end{notation}

\begin{lemma}\label{lemm for injectivity of lambda}
	The group $\mathsf{G}$ acts isometrically over $F_{u, \tau}^{\sep}.$
\end{lemma}	

\begin{proof}
	Denote $v$ the valuation of $C^{\flat}$ that is normalized by $v(\widetilde{\pi})=1/e.$ We show that $v_{|_{F_{u, \tau}^{\sep}}}$ is the unique valuation of $F_{u, \tau}^{\sep}$ that extends $v_{|_{F_{u, \tau}}}.$
	
	\medskip
	
	If $\alpha\in F_{u, \tau}^{\sep},$ then $\alpha$ is algebraic over $F_{u, \tau}=\bigcup\limits_n k(\!(u, \eta^{1/p^n})\!).$ This implies that there exists $n\in \NN$ such that $\alpha$ is algebraic over $k(\!(u, \eta^{1/p^n})\!).$ Notice that $k(\!(u, \eta^{1/p^n})\!)$ is complete for the valuation $v$ and hence there exists a unique valuation over $k(\!(u, \eta^{1/p^n})\!)[\alpha]$ that extends $v_{|_{F_{u, \tau}}}$ ($\cf$ \cite[Chapter II, \S 4, Theorem 4.8]{Neu13}). This implies that for any $g\in \mathsf{G},$ $v\circ g=v$ over $F_{u, \tau}^{\sep}.$
\end{proof}

\begin{remark}
	Notice that the action of $\mathsf{G}\simeq \Gal(F_{u, \tau}^{\sep}/F_{u, \tau})$ on $F_{u, \tau}^{\sep}$ is continuous for the discrete topology.
\end{remark}

\begin{proposition}\label{prop G}
	We have \[ \mathsf{G}\simeq \mathscr{G}_L.  \]
\end{proposition}

\begin{proof}
	We first prove that there are injective maps $\mathscr{G}_L\to \mathsf{G} \to \mathscr{G}_{K_{\pi}}$ whose composite is the inclusion. 
	
	\medskip
	
	Let $\alpha\in F_{u, \tau}^{\sep}$ and $P(X)\in F_{u, \tau}[X]$ its minimal polynomial of $\alpha$ over $F_{u, \tau}.$ Then for any $g\in \mathscr{G}_L,$ we have 
	\[  g(P(X))=P(X) \quad  \Longrightarrow  \quad  P(g(\alpha))=g(P(\alpha))=0 \quad  \Longrightarrow  \quad g(\alpha)\in F_{u, \tau}^{\sep}.\] This shows that $F_{u, \tau}^{\sep}$ is stable under $\mathscr{G}_L.$ As $\mathscr{G}_L$ fixes $F_{u, \tau},$ this implies that we have a morphism of groups:
	\[  \mathscr{G}_L \xrightarrow{\quad \rho \quad } \mathsf{G}. \]
	Similarly, let $\alpha\in F_0^{\sep}$ and $Q(X)\in F_0[X]$ its minimal polynomial over $F_0.$ Then for any $g\in \mathsf{G},$ we have 
	\[  g(Q(X))=Q(X) \quad \Longrightarrow \quad  Q(g(\alpha))=g(Q(\alpha))=0 \quad  \Longrightarrow g(\alpha)\in F_{0}^{\sep}.\]
	This shows that $F_{0}^{\sep}$ is stable under $\mathsf{G}.$ As $\mathsf{G}$ fixes $F_0,$ this implies that we have a morphism of groups:
	\[ \mathsf{G} \xrightarrow{\quad \lambda \quad} \Gal(F_0^{\sep}/F_0)\simeq \mathscr{G}_{K_{\pi}}. \]
	Hence finally we have the diagram: 
	\[ \xymatrix{   \mathscr{G}_L \ar[r]^{\rho} \ar@{_(->}[rd] & \mathsf{G} \ar[d]^{\lambda} \\
		& \mathscr{G}_{K_{\pi}}.     } \]
	The map $\rho$ is thus injective and we are left to prove that $\lambda$ is injective. Take any $g\in \mathsf{G}$ that acts trivially on $F_0^{\sep}.$ Recall that $F_0^{\sep}$ is dense in $C^{\flat}$ for the valuation topology ($\cf$ \cite[Proof of Proposition 1.8]{Car13}) and hence also dense in $F_{u, \tau}^{\sep}.$ By lemma \ref{lemm for injectivity of lambda}, we have $g=\id_{F_{u, \tau}^{\sep}}$ and hence $\lambda$ is injective. 
	
	\medskip
	
	Now we can see $\mathsf{G}/\rho(\mathscr{G}_L)$ as a subgroup of $\mathscr{G}_{K_{\pi}}/\mathscr{G}_L\simeq \overline{\la  \gamma \ra}.$ Suppose $z\in \ZZ_p$ is such that $\gamma^{z}$ acts trivially on $F_{u, \tau},$ then we have
	\[\varepsilon-1 =\gamma^{z}(\varepsilon-1) =\varepsilon^{\chi(\gamma)^z}-1,\ \ie   \varepsilon^{\chi(\gamma)^z}=\varepsilon,\ \text{ thus }  \gamma^{z}=\id.\]
	This shows that $\rho$ is surjective, and we conclude that $\mathsf{G}\simeq \mathscr{G}_L.$
\end{proof}

\begin{corollary}\label{coro F u tau = C GL}
	We have $F_{u, \tau}=(C_{u-\np}^{\flat})^{\mathscr{G}_L}.$
\end{corollary}
\begin{proof}
	By definition we have $F_{u, \tau}=(C_{u-\np}^{\flat})^{\mathsf{G}},$ which is $(C_{u-\np}^{\flat})^{\mathscr{G}_L}$  by proposition \ref{prop G}.
\end{proof}

\begin{remark} (1) Recall that
	$$\Oo_{\Ee}=\Big\{ \sum\limits_{i\in \ZZ} a_iu^i;\ a_i\in W(k),\ \Lim{i\to -\infty}a_i=0 \Big\}$$ is a Cohen ring for $k(\!(u)\!),$ and is equipped with the lift $\varphi$ of the Frobenius of $k(\!(u)\!),$ such that $\varphi(u)=u^p.$ It embeds in $\W(C^{\flat})$ by sending $u$ to $[\widetilde{\pi}].$ Similarly we have a Cohen ring $\Oo_{\Ff_u}$\index{$\Oo_{\Ff_u}$} for $k(\!(u, \eta^{1/p^{\infty}})\!)$ which is endowed with a Frobenius $\varphi$ and embeds in $\W(C^{\flat})$ so that 
	\[\Oo_{\Ee}\to \Oo_{\Ff_{u}}\to \W(C^{\flat})  \]
	are compatible with Frobenius maps ($\cf$ \cite[1.3.3]{Car13}).
	\item (2) Note that $F_{u,\tau}$ is stable by the action of $\tau$ on $\big(C^\flat\big)^{\mathscr{G}_L}$, because $\tau(u)=u(\eta+1)$ and $\tau(\eta^{1/p^n})=\eta^{1/p^n}$ for all $n\in\NN$. Similarly, $\mathcal{O}_{\Ff_u}$ is stable under the action of $\tau$ on $\W(C^\flat)^{\mathscr{G}_L}$.
\end{remark}

\begin{notation}\label{def F ur Frob}
	Let $\Ff_u=\Frac(\Oo_{\Ff_u})$\index{$\Ff_u$} and $\Ff^{\ur}_u$\index{$\Ff^{\ur}_u$} the maximal unramified extension in $\W(C^{\flat})[1/p]$ and $\Oo_{\Ff^{\ur}_u}$\index{$\Oo_{\Ff^{\ur}_u}$} its ring of integers. We denote $\Oo_{\widehat{\Ff^{\ur}_{u}}}$\index{$\Oo_{\widehat{\Ff^{\ur}_{u}}}$} its $p$-adic completion and put $\widehat{\Ff^{\ur}_{u}}=\Oo_{\widehat{\Ff^{\ur}_{u}}}[1/p].$\index{$\widehat{\Ff^{\ur}_{u}}$}
	
	The ring $\Oo_{\Ff^{\ur}_{u}}$ is endowed with a Frobenius map that is compatible with the Frobenius map in $\W(C^{\flat})$ by our construction above. By continuity, it extends into a Frobenius map on $\Oo_{\widehat{\Ff^{\ur}_{u}}}$ and $\widehat{\Ff^{\ur}_{u}}.$
	
	We put 
	\[\Oo_{\Ee_{u,\tau}}:=(\Oo_{\widehat{\Ff^{\ur}_{u}}})^{\mathscr{G}_L},\quad   \Ee_{u,\tau}=\Oo_{\Ee_{u,\tau}}[1/p]. \]\index{$\Oo_{\Ee_{u,\tau}}$}
\end{notation}

\begin{remark}
	We summarize the notations by the following diagram: 
	\[ \xymatrix{
		\W(C^{\flat}) \ar@{->>}[r] & C^{\flat}\ar@{-}[d]\\
		\Oo_{\Ff^{\ur}_u}\ar[u]\ar@{->>}[r]& F_{u,\tau}^{\sep}\ar@{-}[d]\\
		\Oo_{\Ff_{u}} \ar@{->>}[r]\ar[u]& F_{u,\tau}\\
		\ZZ_p\ar[u] \ar@{->>}[r]& \FF_p.\ar[u]
	} \]
\end{remark}

\begin{theorem}\label{thm notation of Caruso}
	We have $\Oo_{\Ff_u}=\Oo_{\Ee_{u, \tau}}.$
\end{theorem}

\begin{proof}
	We have the following diagram:
	\[ \xymatrix{
		& \Oo_{\widehat{\Ff_u^{\ur}}}  &\\
		\Oo_{\Ff_u}\ar@{-}[ru]	\ar@{->>}[d]&&   \Oo_{\Ee_{u, \tau}} \ar@{-}[lu] \ar@{->>}[d]\\
		F_{u,\tau}\ar@{=}[rr]^{\text{Proposition }\ref{prop G}}&&(F_{u,\tau}^{\sep})^{\mathscr{G}_L}.	
	}  \]
	Notice that $\Oo_{\Ff_u}$ is fixed by $\mathscr{G}_L,$ hence $\Oo_{\Ff_u}\subset \Oo_{\Ee_{u, \tau}}.$ Both $\Oo_{\Ff_u}$ and $\Oo_{\Ee_{u, \tau}}$ have the same residue field $F_{u, \tau}.$ Indeed, $\Oo_{\Ee_{u, \tau}}$ has residue field $(F_{u,\tau}^{\sep})^{\mathscr{G}_L},$ which is $F_{u,\tau}$ by proposition \ref{prop G}. Hence $\Oo_{\Ff_u}\subset \Oo_{\Ee_{u, \tau}}$ are two Cohen rings for $F_{u, \tau}$ and they must be equal as $\Oo_{\Ff_u}$ is dense and closed inside $\Oo_{\Ee_{u, \tau}}$ for the $p$-adic topology. 
\end{proof}

\begin{remark}
	The theorem \ref{thm notation of Caruso} can be rewritten as $\Ff^{\mathsf{int}}_{u-\mathsf{np}}=\Ee_{u-\mathsf{np}, \tau}^{\mathsf{int}}$ with Caruso's notations in \cite[1.3.3]{Car13}.
\end{remark}

\begin{definition}\label{def phi tau Fp u}
	A \emph{$(\varphi, \tau)$-module over $(F_0, F_{u,\tau})$} is the data:
	\item (1) an \'etale $\varphi$-module $D$ over $F_0$; 
	\item (2) a $\tau$-semi-linear endomorphism $\tau_D$ over $D_{u, \tau}:= F_{u,\tau} \otimes_{F_0} D$ which commutes with $\varphi_{F_{u,\tau}}\otimes \varphi_D$ (where $\varphi_{F_{u,\tau}}$ is the Frobenius map on $F_{u,\tau}$ and $\varphi_D$ the Frobenius map on $D$) such that 
	\[ \big(\forall x \in D\big)\  (g\otimes 1)\circ \tau_D (x) = \tau_D^{\chi(g)}(x), \] 
	for all $g\in \mathscr{G}_{K_{\pi}}/\mathscr{G}_L$ such that $\chi(g)\in \NN.$\\
	
	We denote $\Mod_{F_0,F_{u,\tau}}(\varphi,\tau)$\index{$\Mod_{F_0,F_{u,\tau}}(\varphi,\tau)$} the corresponding category.
\end{definition}

\begin{theorem}\label{thm cat equi Fp u}
	The functors 
	\begin{align*}
	\Rep_{\FF_p}(\mathscr{G}_K) &\to \Mod_{F_0,F_{u,\tau}}(\varphi,\tau)\\
	T &\mapsto \mathcal{D}(T)= (F_0^{\sep}\otimes_{\FF_p} T)^{\mathscr{G}_{K_{\pi}}}\\
	\mathcal{T}(D)=(F_0^{\sep}\otimes_{F_0} D)^{\varphi=1}  &\mapsfrom  D
	\end{align*}	
	establish quasi-inverse equivalences of categories,	where the $\tau$-semilinear endomorphism $\tau_D$ over $\mathcal{D}(T)_{u,\tau}:=F_{u,\tau}\otimes_{F_0} \mathcal{D}(T)$ is induced by $\tau\otimes\tau$ on $C_{u-\np}^\flat\otimes T$ using the following lemma \ref{lemm tau action natural}. 
\end{theorem}

\begin{proof}
	$\cf$ \cite[Th\'eor\`eme 1.14]{Car13}.
\end{proof}

\begin{lemma}\label{lemm tau action natural}
	The natural map $F_{u,\tau} \otimes_{F_0} \mathcal{D}(T)\to (C^{\flat}_{u-\np}\otimes T)^{\mathscr{G}_L}$ is an isomorphism.
\end{lemma}
\begin{proof}
	$\cf$ \cite[Lemma 1.12]{Car13}.
\end{proof}

More generally, we have the integral analogue of theorem \ref{thm cat equi Fp u}. 

\begin{definition}\label{def phi tau over O e Oe u tau}
	A \emph{$(\varphi, \tau)$-module over $(\Oo_{\Ee}, \mathcal{O}_{\mathcal{E}_{u,\tau}})$} is the data:
	\item (1) an \'etale $\varphi$-module $D$ over $\Oo_{\Ee}$; 
	\item (2) a $\tau$-semi-linear endomorphism $\tau_D$ over $D_{u, \tau}:=\mathcal{O}_{\mathcal{E}_{u,\tau}} \otimes_{\Oo_{\Ee}} D$ which commutes with $\varphi_{\mathcal{O}_{\mathcal{E}_{u,\tau}}}\otimes \varphi_D$ (where $\varphi_{\mathcal{O}_{\mathcal{E}_{u,\tau}}}$ is the Frobenius map on $\mathcal{O}_{\mathcal{E}_{u,\tau}}$ and $\varphi_D$ the Frobenius map on $D$) such that 
	\[ (\forall x \in D)\  (g\otimes 1)\circ \tau_D (x) = \tau_D^{\chi(g)}(x), \]
	for all $g\in \mathscr{G}_{K_{\pi}}/\mathscr{G}_L$ such that $\chi(g)\in \NN.$ \\
	
	We denote $\Mod_{\Oo_{\Ee}, \mathcal{O}_{\mathcal{E}_{u,\tau}}}(\varphi,\tau)$\index{$\Mod_{\Oo_{\Ee}, \mathcal{O}_{\mathcal{E}_{u,\tau}}}(\varphi,\tau)$} the corresponding category. 
\end{definition}

\begin{theorem}
	The functors 
	\begin{align*}
	\Rep_{\ZZ_p}(\mathscr{G}_K) &\to \Mod_{\Oo_{\Ee},\Oo_{\Ee_{u,\tau}}}(\varphi,\tau)\\
	T &\mapsto \mathcal{D}(T)= (\Oo_{\widehat{\Ee^{\ur}}}\otimes_{\ZZ_p} T)^{\mathscr{G}_{K_{\pi}}}\\
	\mathcal{T}(D)=(\Oo_{\widehat{\Ee^{\ur}}}\otimes_{\Oo_{\Ee}} D)^{\varphi=1}  &\mapsfrom  D
	\end{align*}
	establish quasi-inverse equivalences of categories, where the $\tau$-semilinear endomorphism $\tau_D$ over $\mathcal{D}(T)_{u,\tau}:=(\Oo_{\widehat{\Ff^{\ur}_{u}}}\otimes_{\ZZ_p} T)^{\mathscr{G}_L}\simeq\Oo_{\Ee_{u,\tau}}\otimes\mathcal{D}(T)$ is induced by $\tau\otimes\tau$ on $\Oo_{\widehat{\Ff^{\ur}_{u}}}\otimes_{\ZZ_p} T.$	
\end{theorem}

\begin{remark}\label{rem 3.1.16 e e u tau}
	For $V\in \Rep_{\QQ_p}(\mathscr{G}_K),$ we can similarly define $\Mod_{\Ee, \mathcal{E}_{u,\tau}}(\varphi,\tau)$\index{$\Mod_{\Ee, \mathcal{E}_{u,\tau}}(\varphi,\tau)$}: the category of $(\varphi, \tau)$-modules over $(\Ee, \mathcal{E}_{u,\tau})$ and establish an equivalence of category.
\end{remark}

\section{The complex \texorpdfstring{$\Cc_{\varphi, \tau}^{u}$}{C\textpinferior\texthinferior\textiinferior, \texttinferior\textainferior\textuinferior\unichar{"1D58}}}\label{section C varphi tau u}

\begin{notation}\label{not D u tau normla version at the beginning}
	%Let $T\in \Rep_{\ZZ_p}(\mathscr{G}_K)$ and $(D, D_{u,\tau})$ be the $(\varphi, \tau)$-modules over $(\Oo_{\Ee},\Oo_{\Ee_{u,\tau}})$ attached to it. We put
	Let $(D, D_{u,\tau})\in \Mod_{\Oo_{\Ee}, \mathcal{O}_{\mathcal{E}_{u,\tau}}}(\varphi,\tau),$ we put
	\[D_{u, \tau, 0}:=\big\{ x\in D_{u, \tau};\ (\forall g\in \mathscr{G}_{K_{\pi}})\ \chi(g)\in \ZZ_{>0} \Rightarrow (g\otimes 1)(x)=x+\tau_D(x)+\tau_D^{2}(x)+\cdots +\tau_D^{\chi(g)-1}(x) \big\}. \]\index{$D_{u, \tau, 0}$}
	By similar arguments as that of lemma \ref{Replace g witi gamma}, we have 
	
	\[D_{u,\tau, 0}=\big\{ x\in D_{u,\tau} ;\  (\gamma\otimes 1)x=(1+\tau_D+\tau_D^{2}+\cdots + \tau_D^{\chi(\gamma)-1})(x)   \big\}.\]
\end{notation}

\begin{lemma}\label{lemm complex u well defined}
	Let $(D, D_{u,\tau})\in \Mod_{\Oo_{\Ee}, \mathcal{O}_{\mathcal{E}_{u,\tau}}}(\varphi,\tau),$ then $\varphi-1$ and $\tau_D-1$ induce maps $\varphi-1: D_{u, \tau, 0}\to D_{u, \tau, 0}$ and $\tau_D-1: D\to D_{u, \tau, 0}.$ 
\end{lemma}

\begin{proof}
	$\cf$ lemma \ref{lemm 1.1.11 well-defined tau-1 varphi-1}.
\end{proof}

\begin{definition}\label{def complex varphi tau}
	%	Let $T\in \Rep_{\ZZ_p}(\mathscr{G}_K)$ (resp. $V\in \Rep_{\QQ_p}(\mathscr{G}_K)$) with $(D, D_{u,\tau})$ its $(\varphi, \tau)$-module over $(\Oo_{\Ee},\Oo_{\Ee_{u,\tau}})$ (resp. over $(\Ee, \mathcal{E}_{u,\tau})$). We define a complex $\Cc_{\varphi, \tau}^{u}(D)$ as follows:
	Let $(D, D_{\tau})\in \Mod_{\Oo_{\Ee}, \mathcal{O}_{\mathcal{E}_{u,\tau}}}(\varphi,\tau)$ (resp. $(D, D_{\tau})\in \Mod_{\Ee, \mathcal{E}_{u,\tau}}(\varphi,\tau)$). We define a complex $\Cc_{\varphi, \tau}^{u}(D)$\index{$\Cc_{\varphi, \tau}^{u}(D)$} as follows:
	\[  \xymatrix{    0 \ar[rr] && D \ar[r] & D\oplus D_{u,\tau, 0} \ar[r] & D_{u,\tau, 0} \ar[r] & 0\\
		&&x \ar@{|->}[r] &  ((\varphi-1)(x), (\tau_D-1)(x)) &&\\
		&&              &    (y,z) \ar@{|->}[r]& (\tau_D-1)(y)-(\varphi-1)(z).&
	}
	\]
	If $T\in \Rep_{\ZZ_p}(\mathscr{G}_K)$ (resp. $V\in \Rep_{\QQ_p}(\mathscr{G}_K)$), we have in particular the complex $\Cc_{\varphi, \tau}^{u}(\mathcal{D}(T))$ (resp. $\Cc_{\varphi, \tau}^{u}(\mathcal{D}(V))$), which will also be simply denoted $\Cc_{\varphi, \tau}^{u}(T)$\index{$\Cc_{\varphi, \tau}^{u}(T)$} (resp. $\Cc_{\varphi, \tau}^{u}(V)$).
\end{definition}

\begin{theorem}\label{coro complex over non-perfect ring works also Zp-case}
	Let $T\in \Rep_{\ZZ_p}(\mathscr{G}_K),$ then the complex $\Cc_{\varphi, \tau}^{u}(T)$ computes the continuous Galois cohomology of $T,$ $\ie \H^i(\mathscr{G}_K,\Cc_{\varphi, \tau}^{u}(T))\simeq \H^i(\mathscr{G}_K, T)$ for $i\in \NN.$ 
\end{theorem}

\begin{remark}\label{rem diagram non-perfect complex} (1) Let $T\in \Rep_{\ZZ_p}(\mathscr{G}_K)$ and $(D, D_{u,\tau})$ be its $(\varphi, \tau)$-module over $(\Oo_{\Ee},\Oo_{\Ee_{u,\tau}}).$ We have the following diagram of complexes

	\[\xymatrix{
		\Cc_{\varphi, \tau}^{u}(D)&	0 \ar[r] & D \ar@{=}[d]\ar[rr]^-{(\varphi-1, \tau_D-1)}&& D \oplus D_{u,\tau, 0} \ar@{^(->}[d]\ar[rr]^-{ (\tau_D-1)\ominus(\varphi-1)}&& D_{u,\tau,0} \ar@{^(->}[d]\ar[r] & 0\\
		\Cc_{\varphi, \tau}(D)&	0 \ar[r] & D \ar[d]\ar[rr]^-{(\varphi-1, \tau_D-1)}&& D \oplus D_{\tau,0} \ar@{->>}[d]\ar[rr]^-{(\tau_D-1)\ominus(\varphi-1)}&& D_{\tau,0} \ar@{->>}[d]\ar[r] & 0\\
		\Cc(D)	&	0 \ar[r] & 0 \ar[rr]^{}&& D_{\tau,0}/D_{u,\tau, 0} \ar[rr]^-{\varphi-1}&& D_{\tau, 0}/D_{u,\tau, 0} \ar[r] & 0.\\
	}\]
	
	\item (2) Recall that for $T\in \Rep_{\ZZ_p}(\mathscr{G}_K),$ the complex $\Cc_{\varphi, \tau}(D)$ computes the continuous Galois cohomology by theorem \ref{thm main result}. We will show in the following that $\Cc_{\varphi, \tau}^{u}(D)$ and $\Cc_{\varphi, \tau}(D)$ are quasi-isomorphic, and hence the complex $\Cc_{\varphi, \tau}^{u}(D)$ also computes the continuous Galois cohomology. 
\end{remark}

\subsection{Proof of theorem \ref{coro complex over non-perfect ring works also Zp-case}: the quasi-isomorphism}

Let $T\in \Rep_{\ZZ_p}(\mathscr{G}_K),$ and denote by $(D, D_{u, \tau})$ its $(\varphi, \tau)$-module over $(\Oo_{\Ee}, \Oo_{\Ee_{u,\tau}}).$

\begin{lemma}\label{lemm D tau D u tau inj}
	The map $D_{\tau}/D_{u,\tau}\xrightarrow{\varphi-1}D_{\tau}/D_{u,\tau}$ is injective. 
\end{lemma}

\begin{proof}
	For any $x\in D_{\tau}=(\W(C^{\flat})\otimes_{\ZZ_p}T)^{\mathscr{G}_L}$ such that $(\varphi-1)x\in D_{u,\tau}=(\Oo_{\widehat{\Ff^{\ur}_{u}}}\otimes_{\ZZ_p} T)^{\mathscr{G}_L},$ we have to show that $x\in D_{u,\tau}.$ Obviously it suffices to show that for any element $x\in \W(C^{\flat})\otimes_{\ZZ_p}T,$ the relation $(\varphi-1)x\in \Oo_{\widehat{\Ff^{\ur}_{u}}}\otimes_{\ZZ_p} T$ implies $x\in \Oo_{\widehat{\Ff^{\ur}_{u}}}\otimes_{\ZZ_p} T.$ By d\'evissage we can reduce to the case $pT=0.$ For any $x\in C^{\flat}\otimes_{\FF_p} T,$ we claim that $(\varphi-1)x\in k(\!( u, \eta^{1/p^{\infty}})\!)^{\sep}\otimes_{\FF_p}T$ implies $x\in k(\!( u, \eta^{1/p^{\infty}})\!)^{\sep}\otimes_{\FF_p}T.$ We have 
	$C^{\flat}\otimes T\simeq (C^{\flat})^d$ and 
	$$k(\!( u, \eta^{1/p^{\infty}})\!)^{\sep}\otimes T\simeq (k(\!( u, \eta^{1/p^{\infty}})\!)^{\sep})^d$$ as $\varphi$-modules. Hence it suffices to show that for any $x\in C^{\flat},$ $x^p-x\in k(\!( u, \eta^{1/p^{\infty}})\!)^{\sep}$ will imply that $x\in k(\!( u, \eta^{1/p^{\infty}})\!)^{\sep}.$ Put $P(X)=X^p-X-(x^p-x)\in k(\!(u, \eta^{1/p^{\infty}})\!)[X]:$ it is separable as $P^{\prime}(X)=-1$ so that $x$ is separable over $k(\!( u, \eta^{1/p^{\infty}})\!)$ and hence $x\in k(\!( u, \eta^{1/p^{\infty}})\!)^{\sep}.$
\end{proof}

\begin{lemma}\label{lemm 3.2.8 assume T is killed by p} Assume $T$ is killed by $p,$ then there are exact sequences of abelian groups
	\begin{align*}
	0&\to T^{\mathscr{G}_L} \to D_{u, \tau} \xrightarrow{\varphi-1} D_{u, \tau} \to \H^1(\mathscr{G}_L, T) \to 0\\
	0&\to T^{\mathscr{G}_L} \to D_{\tau} \xrightarrow{\varphi-1} D_{\tau} \to \H^1(\mathscr{G}_L, T) \to 0.
	\end{align*}
\end{lemma}

\begin{proof}
	The sequence of $\mathscr{G}_L$-modules 
	\[ 0\to \FF_p \to k(\!( u, \eta^{1/p^{\infty}} )\!)^{\sep} \xrightarrow{\varphi-1} k(\!( u, \eta^{1/p^{\infty}} )\!)^{\sep} \to 0 \] and 
	\[ 0\to \FF_p \to C^{\flat} \xrightarrow{\varphi-1} C^{\flat} \to 0 \] are exact (here we endow $k(\!( u, \eta^{1/p^{\infty}} )\!)^{\sep}$ with the discrete topology and $C^{\flat}$ with its valuation topology). Tensoring with $T$ and taking continuous cohomology (the first for the discrete topology, the second for valuation topology) gives exact sequences:
	\begin{align*}
	0&\to T^{\mathscr{G}_L} \to D_{u, \tau} \xrightarrow{\varphi-1} D_{u, \tau} \to \H^1(\mathscr{G}_L, T) \to \H^1(\mathscr{G}_L, k(\!( u, \eta^{1/p^{\infty}} )\!)^{\sep}\otimes_{\FF_p} T) \\
	0&\to T^{\mathscr{G}_L} \to D_{\tau} \xrightarrow{\varphi-1} D_{\tau} \to \H^1(\mathscr{G}_L, T) \to \H^1(\mathscr{G}_L, C^{\flat}\otimes_{\FF_p} T).
	\end{align*}
	The lemma follows from the vanishing of $\H^1(\mathscr{G}_L, k(\!( u, \eta^{1/p^{\infty}} )\!)^{\sep}\otimes_{\FF_p} T)$ (by Hilbert 90) and the vanishing of $\H^1(\mathscr{G}_L, C^{\flat}\otimes_{\FF_p} T)$ (that follows from the fact that $\widehat{L}$ is perfectoid.)
\end{proof}

\begin{lemm}\label{lemmcoho0 Zp tor non-perfect case}
	If $T\in\Rep_{\ZZ_p, \tors}(\mathscr{G}_K),$ then
	\[\H^1(\mathscr{G}_L,\Oo_{\widehat{\Ff^{\ur}_{u}}}\otimes_{\ZZ_p}T)=0\]
	where $\Oo_{\widehat{\Ff^{\ur}_{u}}}$ is endowed with the $p$-adic topology.
\end{lemm}

\begin{proof}
	By d\'evissage we may assume that $T$ is killed by $p,$ in which case this reduces to the vanishing of $\H^1(\mathscr{G}_L, C_{u-\np}^{\flat}\otimes_{\FF_p}T),$ which follows from Hilbert 90 since $\mathscr{G}_L$ acts continuously on $C_{u-\np}^{\flat}=k(\!(u, \eta^{1/p^{\infty}})\!)^{\sep}$ (endowed with the discrete topology).
	%\ft{\tc{blue}{Why we can assume $k(\!(u, \eta^{1/p^{\infty}})\!)^{\sep}$ is with the discrete topology ? $\mathscr{G}_L$ acts trivially over $k(\!(u, \eta^{1/p^{\infty}})\!)$ and $k(\!(u, \eta^{1/p^{\infty}})\!)^{\sep}$ is the union of finite dimensional extensions over $k(\!(u, \eta^{1/p^{\infty}})\!),$ hence $k(\!(u, \eta^{1/p^{\infty}})\!)^{\sep}\otimes_{\FF_p} T...$? }} 
\end{proof}

\begin{corollary}\label{lemmcoho0 Zp non-perfect case}
	If $T\in\Rep_{\ZZ_p}(\mathscr{G}_K),$ then
	\begin{align*}
	\H^1(\mathscr{G}_L,\Oo_{\widehat{\Ff^{\ur}_{u}}}\otimes_{\ZZ_p}T)=0,
	\end{align*}
	where $\Oo_{\widehat{\Ff^{\ur}_{u}}}$ is endowed with the $p$-adic topology.
	
\end{corollary}

\begin{proof}
	Denote $\Oo_{\widehat{\Ff^{\ur}_{u}}, n}:=\Oo_{\widehat{\Ff^{\ur}_{u}}}/p^n.$
	By \cite[Theorem 2.3.4]{NSW13}, we have the exact sequence
	$$0\to \R^1\pLim{n}\H^{0}(\mathscr{G}_L,\Oo_{\widehat{\Ff^{\ur}_{u}},n}\otimes_{\ZZ_p}T)\to
	\H^1(\mathscr{G}_L,\Oo_{\widehat{\Ff^{\ur}_{u}}}\otimes_{\ZZ_p}T)\to\pLim{n} \H^1(\mathscr{G}_L,\Oo_{\widehat{\Ff^{\ur}_{u}},n}\otimes_{\ZZ_p}T).$$
	We have $\H^1(\mathscr{G}_L,\Oo_{\widehat{\Ff^{\ur}_{u}},n}\otimes_{\ZZ_p}T)=0$ by lemma \ref{lemmcoho0 Zp tor non-perfect case} and $\R^1\pLim{n}\H^{0}(\mathscr{G}_L,\Oo_{\widehat{\Ff^{\ur}_{u}},n}\otimes_{\ZZ_p}T)=0$ from the observation that $\{\H^0(\mathscr{G}_L,\Oo_{\widehat{\Ff^{\ur}_{u}},n}\otimes_{\ZZ_p}T)\}_n$ has the Mittag-Leffler property. Hence we get $\H^1(\mathscr{G}_L,\Oo_{\widehat{\Ff^{\ur}_{u}}}\otimes_{\ZZ_p}T)=0.$ 
\end{proof}

\begin{coro}\label{Exactness D-tau-0 Zp non-perfect case}\label{coro D u tau Zp case}
	If $0\to T^\prime\to T\to T^{\prime\prime}\to0$ is an exact sequence in $\Rep_{\ZZ_p}(\mathscr{G}_K),$ then the sequence
	$$0\to\mathcal{D}(T^\prime)_{u,\tau}\to\mathcal{D}(T)_{u,\tau}\to\mathcal{D}(T^{\prime\prime})_{u,\tau}\to 0$$
	is exact.
\end{coro}

\begin{proof}
	As $\Oo_{\widehat{\Ff^{\ur}_{u}}}$ is torsion-free, we have the exact sequence
	$$0\to\Oo_{\widehat{\Ff^{\ur}_{u}}}\otimes_{\ZZ_p}T^\prime\to\Oo_{\widehat{\Ff^{\ur}_{u}}}\otimes_{\ZZ_p}T\to\Oo_{\widehat{\Ff^{\ur}_{u}}}\otimes_{\ZZ_p}T^{\prime\prime}\to0$$
	which induces the exact sequences
	$$0\to\mathcal{D}(T^\prime)_{u,\tau}\to\mathcal{D}(T)_{u,\tau}\to\mathcal{D}(T^{\prime\prime})_{u,\tau}\to\H^1(\mathscr{G}_L,\Oo_{\widehat{\Ff^{\ur}_{u}}}\otimes_{\ZZ_p}T^\prime).$$
	By corollary \ref{lemmcoho0 Zp non-perfect case}, we get the exact sequence.
\end{proof}

\begin{proposition}
	The map $D_{\tau}/D_{u,\tau}\xrightarrow{\varphi-1}D_{\tau}/D_{u,\tau}$ is bijective.
\end{proposition}

\begin{proof}
	\begin{enumerate}
	\item [(1)]	We first assume $T$ is killed by $p.$
		By lemma \ref{lemm D tau D u tau inj}, it suffices to show the surjectivity.
		By lemma \ref{lemm 3.2.8 assume T is killed by p}, the inclusion $k(\!( u, \eta^{1/p^{\infty}} )\!)^{\sep}\subset C^{\flat}$ induces a commutative diagram:
		\[
		\xymatrix{
			0 \ar[r] & T^{\mathscr{G}_L} \ar[r] \ar@{=}[d]& D_{u, \tau} \ar[r]^{\varphi-1} \ar[d]& D_{u, \tau} \ar[r] \ar[d]& \H^1(\mathscr{G}_L, T) \ar[r]\ar@{=}[d] & 0\\
			0 \ar[r] & T^{\mathscr{G}_L} \ar[r] & D_{\tau} \ar[r]^{\varphi-1} & D_{\tau} \ar[r]  & \H^1(\mathscr{G}_L, T) \ar[r] & 0.
		}
		\]
		
		If $y\in D_{\tau},$ there exists $z\in D_{u, \tau}$ having the same image in $\H^1(\mathscr{G}_L, T):$ hence $y-z$ maps to 0 in $\H^1(\mathscr{G}_L, T),$ so there exists $x\in D_{\tau}$ such that $(\varphi-1)x=y-z$ and thus $(\varphi-1)(x+D_{u, \tau})=y+D_{u, \tau}.$ This finishes the proof.
	\item [(2)] We now use d\'evissage for $\ZZ_p$-representations. Notice that $\mathcal{D}(T/p^n)=\mathcal{D}(T)/p^n$ and $T, \mathcal{D}(T)$ are both $p$-adically complete. Hence it suffices to deal with the case where $T$ is killed by $p^n$ with $n\in \NN_{\geq0}.$ We use induction over $n.$ Suppose $T$ is killed by $p^n$, we put $T^{'}=pT, T^{''}=T/pT$ and consider the following exact sequence 
	\[ 0 \to T^{'} \to T \to T^{''} \to 0  \] 
	in $\Rep_{\ZZ_p}(\mathscr{G}_K).$ Then we have the following diagram of exact sequences by corollary \ref{coro D u tau Zp case}
		\[ \xymatrix{ 0 \ar[r] & \mathcal{D}(T^{\prime})_{u, \tau} \ar[r] \ar@{^(->}[d]& \mathcal{D}(T)_{u, \tau} \ar[r] \ar@{^(->}[d]& \mathcal{D}(T^{\prime\prime})_{u, \tau} \ar@{^(->}[d] \ar[r]& 0 \\
		0 \ar[r] & \mathcal{D}(T^{\prime})_{\tau}\ar[r]  & \mathcal{D}(T)_{\tau} \ar[r] & \mathcal{D}(T^{\prime\prime})_{\tau} \ar[r]  & 0.\\
	}   \]
By snake lemma we have the exact sequence
\[ 0 \to \mathcal{D}(T^{\prime})_{\tau}/\mathcal{D}(T^{\prime})_{u, \tau} \to \mathcal{D}(T)_{\tau}/\mathcal{D}(T)_{u,\tau} \to \mathcal{D}(T^{\prime\prime})_{\tau}/ \mathcal{D}(T^{\prime\prime})_{u, \tau} \to 0 \]
and we consider the following diagram
	\[ \xymatrix{  0 \ar[r] & \mathcal{D}(T^{\prime})_{\tau}/\mathcal{D}(T^{\prime})_{u, \tau} \ar[r] \ar[d]_{\varphi-1}& \mathcal{D}(T)_{\tau}/\mathcal{D}(T)_{u,\tau} \ar[r] \ar[d]_{\varphi-1}& \mathcal{D}(T^{\prime\prime})_{\tau}/ \mathcal{D}(T^{\prime\prime})_{u, \tau} \ar[d]_{\varphi-1} \ar[r]& 0 \\
	0 \ar[r] & \mathcal{D}(T^{\prime})_{\tau}/\mathcal{D}(T^{\prime})_{u, \tau} \ar[r]& \mathcal{D}(T)_{\tau}/\mathcal{D}(T)_{u,\tau} \ar[r]& \mathcal{D}(T^{\prime\prime})_{\tau}/ \mathcal{D}(T^{\prime\prime})_{u, \tau} \ar[r] & 0.
}   \]
Notice that the first and the third vertical maps are isomorphisms by induction hypothesis, hence the middle one is an isomorphism. We then conclude by passing to the limit.
	\end{enumerate}	
\end{proof}

\begin{corollary}\label{coro prepare complex over non-perfect ring works also}
	The map $D_{\tau,0}/D_{u,\tau,0}\xrightarrow{\varphi-1}D_{\tau,0}/D_{u,\tau,0}$ is bijective.
\end{corollary}

\begin{proof}
	Consider the following morphism of short exact sequences:
	\[
	\xymatrix{  
		& D_{\tau, 0}/D_{u, \tau, 0} \ar[r]^{\varphi-1} \ar@{_(->}[d] & D_{\tau, 0}/D_{u, \tau, 0}  \ar@{_(->}[d] &\\
		0 \ar[r] & D_{\tau}/D_{u, \tau} \ar[r]^{\varphi-1} \ar@{->>}[d]^{\delta-\gamma} & D_{\tau}/D_{u, \tau} \ar[r] \ar@{->>}[d]^{\delta-\gamma} & 0\\
		0 \ar[r] & D_{\tau}/D_{u, \tau} \ar[r]^{\varphi-1} & D_{\tau}/D_{u, \tau}  \ar[r] & 0.
	}
	\]
	We then have the claimed isomorphism using the snake lemma.
\end{proof}

\begin{remark}
By remark \ref{rem diagram non-perfect complex}, this finishes the proof of theorem \ref{coro complex over non-perfect ring works also Zp-case}. 
\end{remark}

\section{The complex \texorpdfstring{$\Cc_{\psi,\tau}^{u}$}{C\textpinferior\textsinferior\textiinferior,\texttinferior\textainferior\textuinferior\unichar{"1D58}}}

In this section we will define a $\psi$ operator for $(\varphi, \tau)$-modules over $\Oo_{\widehat{\Ff^{\ur}_{u}}},$ and then construct a complex $\Cc_{\psi, \tau}^{u}.$ At the end of this section, we will show that this complex  $\Cc_{\psi, \tau}^{u}$ computes the continuous Galois cohomology: we will first prove the case of $\FF_p$-representations and then pass to the case of $\ZZ_p$-representations by d\'evissage.\\

Recall that we have defined $\Oo_{\Ff^{\ur}_{u}}$ in definition \ref{def F ur Frob} and denote $\Oo_{\widehat{\Ff^{\ur}_{u}}}$ its $p$-adic completion. The ring $\Oo_{\widehat{\Ff^{\ur}_{u}}}$ has a Frobenius map that lifts that of the residue field $k(\!(u, \eta^{1/p^{\infty}})\!)^{\sep}$ ($\cf$ \cite[Theorem 29.2]{Mat89}).

\begin{lemma}\label{lemm Fp-version degree of sep clo}
	Let $M$ be a field of characteristic $p,$ such that $[M:\varphi(M)]<\infty,$ then $\varphi(M^{\sep})\otimes_{\varphi(M)}M\simeq M^{\sep}$ and in particular $[M^{\sep}: \varphi(M^{\sep})]=[M:\varphi(M)].$ 
\end{lemma}

\begin{proof}
	It suffices to show that for any separable and algebraic extension $L/M$ (we pass to inductive limit for the case of separable extension) we have an isomorphism $M\otimes_{\varphi(M)}\varphi(L)\simeq L.$ By \cite[Theorem 26.4]{Mat89}, the natural map $M\otimes_{\varphi(M)}\varphi(L)\to M\varphi(L)$ is an isomorphism. This implies that $[M\varphi(L):M]=[\varphi(L):\varphi(M)]=[L:M].$ As $M\subset M\varphi(L)\subset L,$ then $[L:M]=[L:M\varphi(L)][M\varphi(L):M]$ and hence $[L:M\varphi(L)]=1,$ $\ie$ $M\varphi(L)=L.$
\end{proof}

\begin{corollary}\label{coro Fp case degree p}
	The extension $C_{u-\np}^\flat/\varphi(C_{u-\np}^\flat)$ has degree $p.$
\end{corollary}

\begin{proof}
	By definition $C_{u-\np}^{\flat}=k(\!(u, \eta^{1/p^{\infty}})\!)^{\sep},$ and we apply lemma \ref{lemm Fp-version degree of sep clo}.
\end{proof}

\begin{lemma}\label{lemm Zp-version degree of sep clo}
	The extension $\widehat{\Ff^{\ur}_{u}}/\varphi(\widehat{\Ff^{\ur}_{u}})$ has degree $p.$
\end{lemma}

\begin{proof}
	We have the diagram 
	\[ \xymatrix{
		\Oo_{\widehat{\Ff^{\ur}_{u}}} \ar@{->>}[r] \ar[d]^{\varphi}& C_{u-\np}^\flat \ar[d]^{\varphi}\\
		\Oo_{\widehat{\Ff^{\ur}_{u}}} \ar@{->>}[r]& C_{u-\np}^\flat
	}  \]
	with $[C_{u-\np}^\flat: \varphi(C_{u-\np}^\flat)]=p.$ Hence there exist $a_1, \dots, a_p\in \Oo_{\widehat{\Ff^{\ur}_{u}}}$ such that whose image   $(\overline{a_1}, \dots, \overline{a_p})$ modulo $p$ forms a basis of $C_{u-\np}^\flat$ over $\varphi(C_{u-\np}^\flat).$  The following map is surjective
	
	\begin{align*}
	\rho\colon	\big( \varphi(\Oo_{\widehat{\Ff^{\ur}_{u}}})\big)^{p} &\to  \Oo_{\widehat{\Ff^{\ur}_{u}}}\\
	(\lambda_1, \dots, \lambda_p) & \mapsto \sum\limits_{i=1}^{p}\lambda_i a_i.
	\end{align*}
	
	Indeed, for any $a\in \Oo_{\widehat{\Ff^{\ur}_{u}}}$ there exists $(\lambda_1^{(1)}, \dots,  \lambda_p^{(1)})\in \big( \varphi(\Oo_{\widehat{\Ff^{\ur}_{u}}})\big)^{p}$ such that 
	\[ a-\sum\limits_{i=1}^{p} \lambda_i^{(1)}a_i  \in p\Oo_{\widehat{\Ff^{\ur}_{u}}}. \]
	Hence there exists $(\lambda_1^{(2)}, \dots,  \lambda_p^{(2)})\in 	\big( \varphi(\Oo_{\widehat{\Ff^{\ur}_{u}}})\big)^{p}$ such that 
	\[ a-\sum\limits_{i=1}^{p} \lambda_i^{(1)}a_i-p\sum\limits_{i=1}^{p} \lambda_i^{(2)}a_i  \in p^2 \Oo_{\widehat{\Ff^{\ur}_{u}}}. \]
	By induction, for any $n\in \NN$ there exists $(\lambda_1^{(n)}, \dots,  \lambda_p^{(n)})\in 	\big( \varphi(\Oo_{\widehat{\Ff^{\ur}_{u}}})\big)^{p}$ such that 
	\[ a-\sum\limits_{i=1}^{p} \bigg(\sum\limits_{j=1}^n p^{j-1} \lambda_i^{(j)}\bigg)a_i  \in p^{n}\Oo_{\widehat{\Ff^{\ur}_{u}}}. \]
	As $\Oo_{\widehat{\Ff^{\ur}_{u}}}$ is $p$-adically complete, we have 
	\[ a= \sum\limits_{i=1}^{p} \bigg(\sum\limits_{j=1}^{\infty} p^{j-1} \lambda_i^{(j)}\bigg)a_i. \]
	Notice that $\sum\limits_{j=1}^{\infty} p^{j-1} \lambda_i^{(j)}\in \varphi(\Oo_{\widehat{\Ff^{\ur}_{u}}})$ and hence we proved the surjectivity of the map $\rho$ defined above. Now $\Oo_{\widehat{\Ff^{\ur}_{u}}}$ is a $\varphi(\Oo_{\widehat{\Ff^{\ur}_{u}}})$-module of finite type and we can apply Nakayama's lemma: the map is an isomorphism since it is so modulo $p$ by corollary \ref{coro Fp case degree p}.
\end{proof}

\begin{definition}\label{def psi operator}
	For any $x\in \widehat{\Ff^{\ur}_{u}},$ we put
	\[ \psi(x)=\frac{1}{p}\varphi^{-1}(\Tr_{\widehat{\Ff^{\ur}_{u}}/\varphi(\widehat{\Ff^{\ur}_{u}})}(x)). \]\index{$\psi$}
	
	In particular we have $\psi\circ \varphi=\id_{\widehat{\Ff^{\ur}_{u}}}$ on $\widehat{\Ff^{\ur}_{u}}.$ Applying lemma \ref{lemm Zp-version degree of sep clo} to $\widehat{\Ee^{\ur}},$ we see that the operator $\psi$ induces an operator $\psi\colon \widehat{\Ee^{\ur}}\to \widehat{\Ee^{\ur}}.$ Note that $\psi(\Oo_{\widehat{\Ff^{\ur}_{u}}})\subset \Oo_{\widehat{\Ff^{\ur}_{u}}}$ and $\psi(\Oo_{\widehat{\Ee^{\ur}}})\subset \Oo_{\widehat{\Ee^{\ur}}}.$
\end{definition}

\begin{remark}\label{rem Zp-case psi commutes with g}
	We have $\psi\circ g=g\circ\psi$ for all $g\in\mathscr{G}_K.$ Indeed, we have the following commutative square
	
	\[  \begin{tikzcd}
	\widehat{\Ff^{\ur}_{u}} \arrow[r, "g"] \arrow[dd, "\varphi","\simeq"'] & \widehat{\Ff^{\ur}_{u}}\arrow[dd, "\varphi"',"\simeq"] \\
	\\
	\varphi(\widehat{\Ff^{\ur}_{u}})  \arrow[r, "g"] \arrow[uu, bend left, "\varphi^{-1}"]&  \varphi(\widehat{\Ff^{\ur}_{u}}). \arrow[uu, bend right, "\varphi^{-1}"']
	\end{tikzcd}  \]
	
	This implies $\varphi^{-1}$ commutes with $g\in \mathscr{G}_K$ over $\varphi(\widehat{\Ff^{\ur}_{u}}).$ As $\Tr_{\widehat{\Ff^{\ur}_{u}}/\varphi(\widehat{\Ff^{\ur}_{u}})}\colon \widehat{\Ff^{\ur}_{u}} \to \varphi(\widehat{\Ff^{\ur}_{u}})$ commutes with $g\in \mathscr{G}_K,$ so does $\psi$ on $\widehat{\Ff^{\ur}_{u}}.$	
\end{remark}

\begin{proposition}
	\label{prop Zp-case psi surjective}
	Let $(D, D_{\tau})\in \Mod_{\Oo_{\Ee}, \Oo_{\Ee_{\tau}}}(\varphi, \tau).$ There exists a unique additive map
	\[ \psi_D\colon D\to D, \]
	satisfying
	\item (1) \[ (\forall a\in \Oo_{\Ee})  ( \forall x\in D )\quad  \psi_D(a \varphi_D (x))=\psi_{\Oo_{\Ee}}(a)x, \]
	\item (2) \[ (\forall a\in \Oo_{\Ee}) (\forall x\in D) \quad  \psi_D(\varphi_{\Oo_{\Ee}}(a)x)=a\psi_D(x). \]  
	This map is surjective and satisfies $\psi_D \circ \varphi_D =\id_D$.\\

	There also exists a unique additive map $\psi_{D_{u,\tau}}: D_{u,\tau} \to D_{u,\tau}$ that satisfies similar conditions as above and extends the additive map $\psi_D$. 
\end{proposition}

\begin{proof} Let $T\in \Rep_{\ZZ_p}(\mathscr{G}_K)$ be such that $D=\mathcal{D}(T).$ We have defined $\psi$ on $\widehat{\Ff^{\ur}_{u}},$ hence it is defined over $\Ee_{u,\tau}=(\widehat{\Ff^{\ur}_{u}})^{\mathscr{G}_L}.$ The operator $\psi\otimes 1$ on $\Oo_{\widehat{\Ee^{\ur}}}\otimes_{\ZZ_p}T$ and $\Oo_{\widehat{\Ff^{\ur}_{u}}}\otimes_{\ZZ_p}T$ induces operators $\psi$ on $D$ and $D_{u,\tau}.$ One easily verifies the above properties: this shows the existence. The unicity follows from the fact that $D$ is \'etale, $\ie  \varphi(D)$ generates $D$ as an $\Oo_{\Ee}$-module.
\end{proof}

\begin{remark}\label{rem why we need unperfected version} (1)	Let $(D, D_{\tau})\in \Mod_{\Oo_{\Ee}, \Oo_{\Ee_{\tau}}}(\varphi, \tau).$ Suppose there exists an operator $\psi_{D_{\tau}}$ over $D_{\tau}$ that extends $\psi_D$ and $\psi_{D_{\tau}}\circ \varphi_{D_{\tau}}=\id_{D_{\tau}}$, then $\varphi_{D_{\tau}}$ being bijective over $D_{\tau}$ (as $D$ is \'etale and $\varphi_{\Oo_{\Ee_{\tau}}}$ is bijective over $\Oo_{\Ee_{\tau}}$) will imply that $\psi_{D_{\tau}}$ is bijective, which contradicts the fact $\psi_D$ is not injective.
\item (2) When there is no confusion, we will simply denote $\psi$ instead of $\psi_{\Oo_{\Ee}}, \psi_{D}$ and $\psi_{D_{u,\tau}}$.
\end{remark}

\begin{lemma}\label{lemm psi D u tau 0}
	Let $(D, D_{u, \tau})\in \Mod_{\Oo_{\Ee}, \Oo_{\Ee_{u,\tau}}}(\varphi, \tau)$. We have a map $\psi\colon D_{u, \tau, 0}\to D_{u, \tau, 0}.$
\end{lemma}

\begin{proof}
	We have the $\psi$ operator on $D_{u, \tau}$ by proposition \ref{prop Zp-case psi surjective}. Notice that $\psi$ commutes with $g\in \mathscr{G}_K$ by remark \ref{rem Zp-case psi commutes with g}, hence $\psi$ induces a $\ZZ_p$-linear endomorphism on $D_{u, \tau, 0}.$ Indeed, if $x\in D_{u, \tau, 0},$ then we have
	\[(\gamma\otimes 1)x=(1+\tau_D+\cdots +\tau_D^{\chi(\gamma)-1})x.\]
	Applying $\psi$ to both sides and by the commutativity we have
	\[(\gamma\otimes 1)(\psi(x))=(1+\tau_D+\cdots +\tau_D^{\chi(\gamma)-1})\psi(x).\]
	This implies $\psi(x)\in D_{u,\tau, 0}.$
\end{proof}

\begin{remark}
By proposition \ref{prop Zp-case psi surjective} and lemma \ref{lemm psi D u tau 0}, $\psi$ is surjective on $D$ and $D_{u,\tau,0}.$
\end{remark}

We now define the following complex:

\begin{definition}\label{def complex psi tau}
	%	Let $T\in \Rep_{\ZZ_p}(\mathscr{G}_K)$ and $(D, D_{u,\tau})$ be its $(\varphi, \tau)$-module over $(\Oo_{\Ee},\Oo_{\Ee_{u,\tau}}).$ We define a complex $\Cc_{\psi, \tau}^{u}(D)$ as follows
	
	Let $(D, D_{u,\tau})\in \Mod_{\Oo_{\Ee},\Oo_{\Ee_{u,\tau}}}(\varphi, \tau).$ We define a complex $\Cc_{\psi, \tau}^{u}(D)$\index{$\Cc_{\psi, \tau}^{u}(D)$} as follows:
	\[ \xymatrix{0\ar[rr] && D \ar[r] & D \oplus D_{u,\tau, 0} \ar[r] &   D_{u, \tau, 0}  \ar[r] & 0   \\
		&&  x  \ar@{|->}[r] & ((\psi-1)(x), (\tau_D-1)(x)) &&\\
		&&& (y,z) \ar@{|->}[r]  & (\tau_D-1)(y)-(\psi-1)(z).
	}  \]
	If $T\in \Rep_{\ZZ_p}(\mathscr{G}_K),$ we have in particular the complex $\Cc_{\psi, \tau}^{u}(\mathcal{D}(T)),$ which will also be simply denoted $\Cc_{\psi, \tau}^{u}(T).$\index{$\Cc_{\psi, \tau}^{u}(T)$}
\end{definition}

\begin{theorem}\label{prop quais-iso psi} The morphism of complexes 
	\[ 
	\xymatrix{
		\Cc^{u}_{\varphi,\tau}\colon & 0\ar[rr]  && D \ar[rr]^-{(\varphi-1, \tau_D-1)} \ar@{=}[d] && D\oplus D_{u, \tau,0} \ar[rr]^-{(\tau_D-1)\ominus(\varphi-1)} \ar[d]^{-\psi\ominus \id} && D_{u, \tau, 0}\ar[rr] \ar[d]^{-\psi} && 0 \\
		\Cc^{u}_{\psi,\tau}\colon & 0\ar[rr]  && D \ar[rr]^-{(\psi-1, \tau_D-1)}  && D\oplus D_{u, \tau,0} \ar[rr]^-{(\tau_D-1)\ominus(\psi-1)} && D_{u, \tau, 0}\ar[rr] && 0
	} \]
	is a quasi-isomorphism.
\end{theorem}

\begin{remark}\label{rem psi embedding Zp case} (1) The diagram in theorem \ref{prop quais-iso psi} is indeed a morphism of complexes.
	
	\begin{proof}
		As $\psi$ commutes with the Galois action, it induces a map $\psi\colon D_{u,\tau,0}\to D_{u,\tau,0}.$ We claim that ${\psi_{D_{u,\tau}}}_{|_{D}}=\psi_{D},$ and hence the diagram in theorem \ref{prop quais-iso psi} commutes. Indeed, we have the following commutative square
		\[ \xymatrix{
			\Oo_{\widehat{\Ee^{\ur}}} \ar@{^(->}[r] & \Oo_{\widehat{\Ff^{\ur}_{u}}}\\
			\Oo_{\widehat{\Ee^{\ur}}} \ar@{^(->}[r] \ar[u]^{\varphi} & \Oo_{\widehat{\Ff^{\ur}_{u}}} \ar[u]^{\varphi}. }
		\]
		By lemma \ref{lemm Zp-version degree of sep clo}, $[	\Oo_{\widehat{\Ee^{\ur}}}:\varphi(	\Oo_{\widehat{\Ee^{\ur}}})]=[\Oo_{\widehat{\Ff^{\ur}_{u}}}:\varphi(\Oo_{\widehat{\Ff^{\ur}_{u}}})]=p,$ hence we conclude by the construction of $\psi.$
		
	\end{proof}
	
	\item (2) We have diagram
	\[ 
	\xymatrix{
		&0 \ar[r] & 0 \ar[rr]\ar[d] && D^{\psi=0} \ar[rr]^-{\tau_D-1} \ar[d] && D_{u,\tau, 0}^{\psi=0 }   \ar[r]\ar[d]& 0\\
		\Cc^{u}_{\varphi,\tau}\colon & 0\ar[r]  & D \ar[rr]^-{(\varphi-1, \tau_D-1)} \ar@{=}[d] && D\oplus D_{u, \tau,0} \ar[rr]^-{(\tau_D-1)\ominus(\varphi-1)} \ar[d]^{-\psi\ominus \id} && D_{u, \tau, 0}\ar[r] \ar[d]^{-\psi} & 0 \\
		\Cc^{u}_{\psi,\tau}\colon & 0  \ar[r] & D \ar[rr]^-{(\psi-1, \tau_D-1)}  \ar[d]&& D\oplus D_{u, \tau,0} \ar[rr]^-{(\tau_D-1)\ominus(\psi-1)} \ar[d]&& D_{u, \tau, 0}\ar[r]  \ar[d]& 0\\
		& 0 \ar[r]& 0 \ar[rr] &&0 \ar[rr] &&0 \ar[r] &0.	
	} \]
	As $\psi$ is surjective, the cokernel complex is trivial and it suffices to show that the kernel complex is acyclic, $\ie$the map
	\[ D^{\psi=0} \xrightarrow{\tau_D-1} D_{u, \tau, 0}^{\psi=0}  \]
	is an isomorphism. 
\end{remark}

To prove the theorem \ref{prop quais-iso psi}, we will start with the case of $\FF_p$-representations and then pass to the $\ZZ_p$-representations  by d\'evissage.

\subsection{The case of \texorpdfstring{$\FF_p$}{\unichar{"0046}_p}-representations}

We assume in this subsection that $T\in \Rep_{\FF_p}(\mathscr{G}_K).$

\begin{lemma}\label{lemm point fixed by tau p n}
	For all $r\in \NN,$ we have $(C^{\flat}_{u-\np})^{\mathscr{G}_{K_\zeta}}=F_{u,\tau}^{\tau^{p^r}}=k(\!(\eta^{1/p^\infty})\!).$
\end{lemma}

\begin{proof}
	Let $z\in F_{u, \tau},$ write $z=\sum\limits_{i=m}^{\infty}f_i(\eta)u^i,$ where $m\in \ZZ$ and $f_i(X)\in k(\!(X^{1/p^{n_i}})\!)$ for some $n_i\in \NN.$ We have $\tau^{p^n}(z)=\sum\limits_{i=m}^{\infty}\varepsilon^{ip^n}f_i(\eta)u^i$ so that $\tau^{p^n}(z)=z$ implies that $\varepsilon^{ip^n}f_i(\eta)=f_i(\eta)$ and hence $f_i=0$ for $i\not=0.$ We conclude $z=f_0(\eta)\subset \bigcup\limits_{n\in \NN} k(\!(\eta^{1/p^n})\!)$ and hence $F_{u,\tau}^{\tau^{p^r}} \subset k(\!(\eta^{1/p^\infty})\!).$ Conversely, we have $ k(\!(\eta^{1/p^\infty})\!) \subset F_{u,\tau}^{\tau^{p^r}},$ as $k(\!(\eta^{1/p^\infty})\!)\subset F_{u,\tau}$ and $\tau$ acts trivially over $K_{\zeta}.$ Notice that $\mathscr{G}_{K_{\zeta}}=\la \mathscr{G}_L, \tau \ra$ and  $F_{u,\tau}=(C^{\flat}_{u-\np})^{\mathscr{G}_L}$ by corollary \ref{coro F u tau = C GL}, so that $(C^{\flat}_{u-\np})^{\mathscr{G}_{K_\zeta}}=F_{u,\tau}^{\tau},$ which is $k(\!(\eta^{1/p^\infty})\!)$ by similar computations as above.
\end{proof}

\begin{lemma}\label{lemma varphi is surjective over non-perfect ring modified}
	Put $\mathbf{D}(T)=(C^{\flat}_{u-\np}\otimes T)^{\mathscr{G}_{K_{\zeta}}}.$ We have a $\mathscr{G}_{K_{\zeta}}$-equivariant isomorphism 
	\[  C^{\flat}_{u-\np}\otimes_{\FF_p}T\simeq C^{\flat}_{u-\np}\otimes_{k(\!(\eta^{1/p^\infty})\!)}\mathbf{D}(T). \]
	In particular, we have $\dim_{k(\!(\eta^{1/p^{\infty}})\!)}\mathbf{D}(T)=\dim_{\FF_p}T.$	 
\end{lemma}

\begin{proof}
		Denote $\DD(T)=\big(k(\!( \eta^{1/p^{\infty}} )\!)^{\sep} \otimes_{\FF_p} T \big)^{\mathscr{G}_{K_{\zeta}}}$, then from the field of norm theory (of perfect fields) ($\cf$ \cite[Proposition 1.2.4]{Fon90}) we have the following $\mathscr{G}_{K_\zeta}$-equivariant isomorphism
	\[  k(\!( \eta^{1/p^{\infty}} )\!)^{\sep} \otimes_{k(\!(\eta^{1/p^{\infty}})\!)} \DD(T) \simeq k(\!( \eta^{1/p^{\infty}} )\!)^{\sep} \otimes_{\FF_p} T. \]
	Tensoring with $C_{u-\np}^{\flat}$ over $k(\!(\eta^{1/p^{\infty}})\!)^{\sep}$, we have 
	\begin{equation}\label{equ remark varphi is surjective over non-perfect ring modified}
	C_{u-\np}^{\flat}\otimes_{k(\!(\eta^{1/p^{\infty}})\!)} \DD(T) \simeq C_{u-\np}^{\flat}\otimes_{\FF_p} T.
	\end{equation}
	Taking the points fixed by $\mathscr{G}_{K_{\zeta}}$ on both sides gives
	\[ \big( C_{u-\np}^{\flat}\otimes_{k(\!(\eta^{1/p^{\infty}})\!)} \DD(T) \big)^{\mathscr{G}_{K_{\zeta}}} \simeq  \big( C_{u-\np}^{\flat}\otimes_{\FF_p} T \big)^{\mathscr{G}_{K_{\zeta}}} = \mathbf{D}(T). \]
	As $\DD(T)$ is fixed by ${\mathscr{G}_{K_{\zeta}}}$ from definition, the left hand side is 
	\[  \big( C_{u-\np}^{\flat}\otimes_{k(\!(\eta^{1/p^{\infty}})\!)} \DD(T) \big)^{\mathscr{G}_{K_{\zeta}}}= ( C_{u-\np}^{\flat} )^{\mathscr{G}_{K_{\zeta}}} \otimes_{k(\!(\eta^{1/p^{\infty}})\!)} \DD(T) = \DD(T) \]
	by lemma \ref{lemm point fixed by tau p n}.
	This proves that $\DD(T)=\mathbf{D}(T),$ hence equation \ref{equ remark varphi is surjective over non-perfect ring modified} gives what we want.
\end{proof}

\begin{lemma}\label{lemm x in D(T)}
	Let $r\in \NN_{>0},$ we then have $\mathcal{D}(T)_{u, \tau}^{\tau_D^{p^r}}\subset \mathbf{D}(T).$
\end{lemma}

\begin{proof}
	We have $\mathcal{D}(T)_{u,\tau}^{\tau_D^{p^r}}= (C^{\flat}_{u-\np}\otimes_{\FF_p} T)^{\mathscr{G}_{K_{\zeta}(\pi^{1/p^r})}}.$
	Notice that by lemma \ref{lemma varphi is surjective over non-perfect ring modified}, we have 
	
	\begin{align*}
	(C^{\flat}_{u-\np}\otimes_{\FF_p} T)^{\mathscr{G}_{K_{\zeta}(\pi^{1/p^r})}}&=(C^{\flat}_{u-\np})^{\mathscr{G}_{K_{\zeta}(\pi^{1/p^r})}}\otimes \mathbf{D}(T)\\
	&=((C^{\flat}_{u-\np})^{\mathscr{G}_L})^{\tau^{p^r}}\otimes \mathbf{D}(T)\\
	&=F_{u,\tau}^{\tau^{p^r}}\otimes \mathbf{D}(T)\\
	&=\mathbf{D}(T).
	\end{align*}
	
	The last step follows from lemma \ref{lemm point fixed by tau p n}, as $\bigcup\limits_{n\in \NN} k(\!(\eta)\!)[\eta^{1/p^n}]\subset (C^{\flat}_{u-\np})^{\mathscr{G}_{K_{\zeta}}}.$
	
\end{proof}

\begin{proposition}\label{prop tau-1 high power inj}
	The map $\frac{\tau_D^{p^r}-1}{\tau_D-1}\colon D_{u,\tau,0}^{\psi=0}\to D_{u,\tau^{p^r},0}^{\psi=0}$ is injective.
\end{proposition}

\begin{proof}
	Take any $x\in D_{u, \tau, 0}^{\psi=0}$ with $\frac{\tau_D^{p^r}-1}{\tau_D-1}(x)=0,$ then in particular $(\tau_D^{p^r}-1)x=0.$ Hence $x\in D_{u, \tau}^{\tau_D^{p^r}, \psi=0}$ and  $x\in \mathbf{D}(T)^{\psi=0}$ by lemma \ref{lemm x in D(T)}. Lemma \ref{lemm point fixed by tau p n} and lemma \ref{lemma varphi is surjective over non-perfect ring modified} imply that $x=\varphi(x^{\prime})$ for some $x^{\prime}\in \mathbf{D}(T)$ (as $\mathbf{D}(T)$ is \'etale and the base field $k(\!(\eta^{1/p^\infty})\!)$ is perfect) and hence $0=\psi(x)=\psi(\varphi(x^{\prime}))=x^{\prime}$ implies $x=0.$
\end{proof}
Recall that by remark \ref{rem psi embedding Zp case} (2), to prove theorem \ref{prop quais-iso psi} it suffices to prove $D^{\psi=0} \xrightarrow{\tau_D-1} D_{u, \tau, 0}^{\psi=0}$ being isomorphic: we firstly prove the case of $\FF_p$-representations, in several steps.

\subsubsection{The injectivity}

Recall that $F_0=k(\!(u)\!)$ embeds into $C^{\flat}$ by $u\mapsto \tilde{\pi}.$

\begin{lemma}\label{lemma I.3.1} We have $F_0^{\tau=1}=k.$
\end{lemma}
\begin{proof}
	Recall that $\tau(u)=\varepsilon u= (\eta+1)u$ and $\tau(u^i)=(\eta +1)^iu^i.$ Suppose $x=\sum\limits_{i=i_0}^{\infty}\lambda_i u^i\in F_0^{\tau=1}$ with $\lambda_i\in k,$ then $\tau(x)=\sum\limits_{i=i_0}^{\infty}\lambda_i (\eta+1)^i u^i.$
	% If $i\not =0,$
	%	 $(\eta+1)^i=\sum\limits_{k=0}^{\infty}\binom ik\eta^k$ hence $\tau(x)=\sum\limits_{i=i_0}^{\infty}\sum\limits_{j\geq 0}\lambda_i \binom ij \eta^ju^i\in k(\!(u, \eta)\!).$ This implies that $\lambda_i \binom ij =0$ for all $j>0.$ Then $\tau(x)=\sum\limits_{i=i_0}^{\infty}\lambda_i(\eta+1)^iu^i$ 
	Hence $\lambda_i=\lambda_i(\eta+1)^i,$ so that $\lambda_i((\eta+1)^i-1)=0$ for all $i\geq i_0.$ If $i\neq 0,$ then $(\eta+1)^i\neq 1,$ hence $(\eta+1)^i-1\in k(\!(\eta)\!)^{\times}$ so that $\lambda_i=0.$ In particular, $x=\lambda_0\in k.$
	%\ft{Actually, this already imply $\lambda_i=0$ for $i\neq0.$}  Let $j=1$ we have $i \lambda_i=0,$ so that $\lambda_i=0$ when $p\nmid i.$ Hence $x\in \varphi(F_0).$ Suppose $x=\varphi(x_1),$ then $\tau x=x$ implies that $\varphi(\tau(x_1))=\varphi(x_1)$ and hence $\tau(x_1)=x_1,$ then by induction we have $x\in \bigcap_n \varphi^n(F_0)=k.$ 
\end{proof}

\begin{lemma}\label{lemm Brinon saves}
	We have $k(\!(\eta^{1/p^{\infty}})\!)^{\gamma=1}=k.$
\end{lemma}

\begin{proof}
	We have $k(\!(\eta^{1/p^{\infty}})\!)\subset (C^{\flat})^{\mathscr{G}_{K_{\zeta}}},$ hence $k(\!(\eta^{1/p^{\infty}})\!)^{\gamma=1}\subset (C^{\flat})^{\mathscr{G}_{K}}.$ As $C^{\flat}=\pLim{x\mapsto x^p} C,$ we have 
	\[(C^{\flat})^{\mathscr{G}_K}=\pLim{x\mapsto x^p}K.\]
	Let $x=(x_n)_n\in \pLim{x\mapsto x^p}K,$ we have $x_0=x_n^{p^n}$ for all $n\in \NN.$ In particular, we have $v(x_0)=p^nv(x_n).$ If $x\neq0,$ this implies $v(x_n)=0$ for all $n\in \NN.$ Let $\overline{x_0}$ be the image of $x_0$ in $k=\Oo_K/(\pi)$ and $y=\big([\overline{x_0}^{p{^{-n}}}] \big)_n\in \pLim{x\mapsto x^p}K,$ then $y^{-1}x=(z_n)_n\in \pLim{x\mapsto x^p}K$ satisfies $z_n\equiv 1$ mod $\pi\Oo_K$ for all $n\in \NN.$ This implies $z_n=z_{n+m}^{p^m}\equiv 1$ mod $\pi^{m+1}\Oo_K$ for all $m\in \NN,$ hence $z_n=1$ for all $n\in\NN.$ This shows that $x=y,$ and that the map 
	\begin{align*}
	k &\to \pLim{x\mapsto x^p} K=(C^{\flat})^{\mathscr{G}_K}\\
	\alpha &\mapsto \big([\alpha^{p^{-n}}] \big)_n
	\end{align*}
	is a ring isomorphism. As $k\subset k(\!(\eta^{1/p^{\infty}})\!)^{\gamma=1}\subset (C^{\flat})^{\mathscr{G}_K}$, this shows that $k=k(\!(\eta^{1/p^{\infty}})\!)^{\gamma=1}=(C^{\flat})^{\mathscr{G}_K}.$
\end{proof}

\begin{proposition}
	The natural map $k(\!(\eta^{1/p^{\infty}})\!)\otimes_k \mathbf{D}(T)^{\gamma=1}\to \mathbf{D}(T)$ is injective.
\end{proposition}

\begin{proof}
	We use the standard argument: assume it is not and let $x=\sum\limits_{i=1}^r\lambda_i\otimes \alpha_i,$ with $\lambda_i\in k(\!(\eta^{1/p^{\infty}})\!)$ and $\alpha_i\in \mathbf{D}(T)^{\gamma=1}$ for all $i,$ be a nonzero element in the kernel such that $r$ is minimal. Dividing by $\lambda_r$ we may assume that $\lambda_r=1.$ As $\sum\limits_{i=1}^{r}\gamma(\lambda_i)\otimes \alpha_i$ maps to $\gamma(\sum\limits_{i=1}^{r}\lambda_i\alpha_i)=0,$ the element $\sum\limits_{i=1}^{r-1}(\gamma-1)(\lambda_i)\otimes \alpha_i$ lies in the kernel, by the minimality of $r,$ we have $\gamma(\lambda_i)=\lambda_i,$ $\ie$$\lambda_i\in k$ for all $i\in \{ 1, \ldots, r\}$ ($\cf$ lemma \ref{lemm Brinon saves}). Then $x=1\otimes(\sum\limits_{i=1}^r\lambda_i\alpha_i)=0,$ which is a contradiction to the assumption. Hence the map is injective.
\end{proof}

\begin{corollary}\label{coro dim Brinon}
	We have $\dim_k\mathbf{D}(T)^{\gamma=1}\leq \dim_{\FF_p}T.$
\end{corollary}

\begin{lemma}
	\label{lemm 1-tau inj Cherbonnier}
	If $x\in D^{\psi=0}$ is such that $\tau_D(x)=x,$ then $x=0.$
\end{lemma}

\begin{proof}
	Let $x\in D^{\psi=0}$ be such that $\tau_D(x)=x.$ This shows that $x,$ seen as an element of $C^{\flat}_{u-\np}\otimes_{\FF_p}T,$ is fixed by $\mathscr{G}_K:$ it belongs to $\mathbf{D}(T)^{\gamma=1}.$ By corollary \ref{coro dim Brinon}, the latter is a finite dimensional $k$-vector space. It is endowed with the restriction $\varphi$ of the Frobenius map, and this restriction is injective. As $k$ is perfect, this implies that $\varphi\colon \mathbf{D}(T)^{\gamma=1}\to \mathbf{D}(T)^{\gamma=1}$ is bijective: there exists $y\in \mathbf{D}(T)^{\gamma=1}$ such that $x=\varphi(y).$ Then $y=\psi\varphi(y)=\psi(x)=0,$ hence $x=\varphi(y)=0.$
\end{proof}

\begin{corollary}\label{cor injective before "the trivial case"}
	The map $D^{\psi=0} \xrightarrow{\tau_D-1} D_{u, \tau, 0}^{\psi=0}$ is injective.
\end{corollary}

\subsubsection{The trivial case}

Recall that $\psi\colon F_0\to F_0$ is given by the formula $\psi(x)=\frac{1}{p}\varphi^{-1}(\Tr_{F_0/\varphi(F_0)}(x)).$ The elements $1, u, \ldots, u^{p-1}$ form a basis of $F_0$ over $\varphi(F_0).$

\begin{lemma}\label{lemm psi mini poly}
	Let $x=\sum\limits_{i=0}^{p-1}x_iu^i$ be an element of $F_0$ with $x_i\in \varphi(F_0).$ Then we have $\psi(x)=\varphi^{-1}(x_0).$ 
\end{lemma}	

\begin{proof}
	For $1\leq i<p,$ we have $\psi(u^i)=0$ as $\Tr_{F_0/\varphi(F_0)}(u^{i})=0.$ Indeed, the minimal polynomial of $u^i$ is $f(X)=X^p-u^{pi}\in \varphi(F_0)[X]$ when $1\leq i <p.$
\end{proof}

\begin{corollary}\label{coro psi=0 3.3.23}
	An element $x=\sum\limits_{i=0}^{p-1}x_iu^i\in F_0$ with $x_i\in \varphi(F_0)$ is killed by $\psi$ if and only if $x_0=0.$
\end{corollary}

\begin{proof}
	By the lemma \ref{lemm psi mini poly} and notice that $\varphi$ is injective.
\end{proof}

Let $z=\sum\limits_{i=1}^{p-1}u^iz_i\in {F_0}^{\psi=0}$ with $z_i\in \varphi(F_0)=k(\!(u^p)\!).$ More precisely, write $z_i=\sum\limits_{j=n_i}^{+\infty}b_{ij}u^{pj}$ with $b_{ij}\in k.$ Then $z=\sum\limits_{i=1}^{p-1}\sum\limits_{j=n_i}^{+\infty}b_{ij}u^{i+pj}$ with $b_{ij}\in k$ so that

\begin{equation}\label{eq tau-1}
(\tau-1)z=\sum\limits_{i=1}^{p-1}\sum\limits_{j=n_i}^{\infty}b_{ij}u^{i+pj}(\varepsilon^{i+pj}-1).
\end{equation}

\begin{lemma}\label{lemm inverse of tau-1}\label{coro inverse almost}
	The map $\tau-1\colon F_0^{\psi=0}\to F_{u,\tau,0}^{\psi=0}$ is surjective.
\end{lemma}

\begin{proof}
	Let $x \in F_{u,\tau,0}^{\psi=0}\subset k(\!(u, \eta^{1/p^{\infty}} )\!)^{\psi=0}:$ by corollary \ref{coro psi=0 3.3.23} we can write uniquely $x=\sum\limits_{i=1}^{p-1}u^i x_i$ with \[x_i\in \varphi\big(k(\!(u, \eta^{1/p^{\infty}} )\!)^{\psi=0}\big).\] 
	
	We can write 
	\[x_i=\sum\limits_{j=m_i}^{\infty}u^{pj}f_{ij}(\eta)\]  with $m_i\in \ZZ$ and $f_{ij}(X)\in k(\!(X^{1/p^{n_{ij}}})\!)$ for some $n_{ij}\in \NN$, hence \[x=\sum\limits_{i=1}^{p-1}\sum\limits_{j=m_i}^{\infty}u^{i+pj} f_{ij}(\eta).\] 
	By definition, $x\in F_{u, \tau, 0}$ implies 
	\[\gamma(x)=(1+\tau+\cdots + \tau^{\chi(\gamma)-1})(x).\]
	The left hand side is (recall that $\eta=\varepsilon-1$)
	\[ \gamma(x)=\sum\limits_{i=1}^{p-1}\sum\limits_{j=m_i}^{\infty}u^{i+pj} f_{ij}(\gamma(\eta))=\sum\limits_{i=1}^{p-1}\sum\limits_{j=m_i}^{\infty}u^{i+pj} f_{ij}(\varepsilon^{\chi(\gamma)}-1), \]
	and the right hand side is 
	\[ (1+\tau+\cdots+\tau^{\chi(\gamma)-1})(x)=\sum\limits_{i=1}^{p-1}\sum\limits_{j=m_i}^{\infty}u^{i+pj}\cdot(1+\varepsilon^{i+pj}+ \varepsilon^{2(i+pj)} + \cdots + \varepsilon^{(\chi(\gamma)-1)(i+pj)})\cdot f_{ij}(\eta).\] This implies that for all $i, j,$ we have
	\[ f_{ij}(\gamma(\eta))=f_{ij}(\varepsilon^{\chi(\gamma)}-1)=(1+\varepsilon^{i+pj}+ \varepsilon^{2(i+pj)}+ \cdots + \varepsilon^{(\chi(\gamma)-1)(i+pj)})\cdot f_{ij}(\eta).\]
	Put $l=\chi(\gamma)$ and $m=i+pj$ (notice $m\neq 0$ as $p\nmid i$), the condition translates into
	\[\gamma(f_{ij}(\eta))=\frac{\varepsilon^{lm}-1}{\varepsilon^m-1}f_{ij}(\eta),\]
	$\ie$
	\[\gamma\Big( \frac{f_{ij}(\eta)}{\varepsilon^m-1}  \Big) = \frac{f_{ij}(\eta)}{\varepsilon^m-1}. \] 
	%	This rational function must be constant by lemma \ref{lemm raising power}, $\ie$ an element in $k.$ 
	As $\frac{f_{ij}(\eta)}{\varepsilon^m-1}\in k(\!(\eta^{1/p^{\infty}})\!),$ we have $\frac{f_{ij}(\eta)}{\varepsilon^m-1}\in k$ by lemma \ref{lemm Brinon saves}: there exist $b_{i,j}\in k$ such that \[x=\sum\limits_{i=1}^{p-1}\sum\limits_{j=m_i}^{\infty}b_{ij}u^{i+pj}(\varepsilon^{i+pj}-1).\]
	
	By equation (\ref{eq tau-1}), an inverse image of $x=\sum\limits_{i=1}^{p-1}\sum\limits_{j=m_i}^{\infty}b_{ij}u^{i+pj}(\varepsilon^{i+pj}-1)$ is \[ (\tau-1)^{-1}(x):=\sum\limits_{i=1}^{p-1}\sum\limits_{j=m_i}^{\infty}b_{ij}u^{i+pj}.\]
\end{proof}

\begin{corollary}\label{coro form of F u tau 0}
An element $x\in \Oo_{\Ee_{u, \tau, 0}}^{\psi=0}$ can be written in the form 
\[ x=\sum\limits_{i=1}^{p-1}\sum\limits_{j\in \ZZ}c_{ij}u^{i+pj}([\varepsilon]^{i+pj}-1)   \]
with $c_{ij}\in \W(k)$ such that $\Lim{j \to -\infty}c_{ij}=0$ for all $i\in \{1, \ldots, p-1\}.$
\end{corollary}
\begin{proof}
By lemma \ref{coro inverse almost}, elements of $F_{u,\tau,0}^{\psi=0}$ can be written in the form 
\[x=\sum\limits_{i=1}^{p-1}\sum\limits_{j=m_i}^{\infty}b_{ij}u^{i+pj}(\varepsilon^{i+pj}-1) \text{ with } b_{ij}\in k.\] We then conclude by d\'evissage.
\end{proof}

\begin{corollary}
	The map $\tau-1\colon F_0^{\psi=0}\to F_{u,\tau, 0}^{\psi=0}$ is bijective.
\end{corollary}
\begin{proof}
	This follows from corollary \ref{cor injective before "the trivial case"} and lemma \ref{coro inverse almost}.
\end{proof}

We have proved the bijection for trivial representations, and we now prove the general case.

\subsubsection{The general case}

Recall that $D=\mathcal{D}(T)$ is an \'etale $(\varphi,\tau)$-module; the natural map $\varphi^\ast\colon F_0\otimes_{\varphi,F_0}D\to D$ is an isomorphism. Let $(e_1,\ldots,e_d)$ be a basis of $D$ over $F_0,$ then $(\varphi(e_1),\ldots,\varphi(e_d))$ is again a basis, $\ie$$D=\oplus_{i=1}^dF_0\varphi(e_i).$

\begin{lemm}\label{lemm psi semi-linearity in another sense}
	If $x=\sum\limits_{i=1}^d\lambda_i\varphi(e_i)\in D,$ then $\psi(x)=\sum\limits_{i=1}^d\psi(\lambda_i)e_i.$
\end{lemm}

\begin{proof}
	We have $D=(F_0^{\sep}\otimes_{\FF_p}T)^{\mathscr{G}_{K_\pi}}.$ Let $(v_1,\ldots,v_d)$ be a basis of $T$ over $\FF_p:$ we can write $e_i=\sum\limits_{j=1}^d\alpha_j\otimes v_j,$ so that $\varphi(e_i)=\sum\limits_{j=1}^d\alpha_j^p\otimes v_j,$ hence $\psi(\lambda_i\varphi(e_i))=\sum\limits_{j=1}^d\psi(\lambda_i)\alpha_j\otimes v_j=\psi(\lambda_i)e_i$ for all $i\in\{1,\ldots,d\}.$
\end{proof}

\begin{coro}
	We have $\sum\limits_{i=1}^d\lambda_i\varphi(e_i)\in D^{\psi=0}$ if and only if $\lambda_i\in F_0^{\psi=0}$ for all $i\in\{1,\ldots,d\}.$
\end{coro}
\begin{proof}
	This follows from lemma \ref{lemm psi semi-linearity in another sense}.
\end{proof}

\begin{coro}\label{lemm psi semi-linearity in another sense: tau case}
	If $\mu_1,\ldots,\mu_d\in F_{u,\tau},$ we have $\sum\limits_{i=1}^d\mu_i\varphi(e_i)\in D_{u,\tau}^{\psi=0}$ if and only if $\mu_i\in F_{u,\tau}^{\psi=0}$ for all $i\in\{1,\ldots,d\}.$
\end{coro}

\begin{proof}
	The proof is similar to that of lemma \ref{lemm psi semi-linearity in another sense}.
\end{proof}

\begin{lemm}\label{lemm5.18}
	Let $n\in\ZZ$ and $f=\sum\limits_{j=n}^\infty\lambda_ju^j\in F_{u,\tau},$ where $\lambda_j\in k(\!(\eta^{1/p^\infty})\!):=\bigcup\limits_{n=0}^\infty k(\!(\eta^{1/p^n})\!).$ Then $\psi(f)=0$ if and only if $p\mid j\Rightarrow\lambda_j=0.$
\end{lemm}

\begin{proof}
	If $\lambda\in k(\!(\eta^{1/p^{\infty}})\!)$ and $i\in \ZZ$, we have
	\[ \psi(\lambda u^i)=\psi\big(\varphi(\varphi^{-1}(\lambda))u^i)=\varphi^{-1}(\lambda)\psi(u^i),  \]
	and 
	\[ \psi(u^i)=\begin{cases}
	u^{i/p},\  \text{ if } p|i\\
	0,\ \text{ else.}
	\end{cases}  \]
	We have thus 
	\[  \psi(f)=\sum\limits_{j\geq n,\, p|j}\varphi^{-1}(\lambda_j)u^{j/p}, \]
	so that $\psi(f)=0$ if and only if $\varphi^{-1}(\lambda_j)=0$, $\ie$$\lambda_j=0$ for all $j$ such that $p|j$.
\end{proof}

\begin{notation}
	We have $k[\![u,\eta]\!]\subset k(\!(\eta)\!)[\![u]\!],$ so that there is an inclusion $k(\!(u,\eta)\!)\subset k(\!(\eta)\!)(\!(u)\!),$ hence an inclusion $F_{u,\tau}=k(\!(u,\eta^{1/p^\infty})\!)\subset\bigcup\limits_{n=0}^\infty k(\!(\eta)\!)(\!(u)\!)[\eta^{1/p^n}]$ ($\cf$ notation \ref{rem fixed finite degree}). Put $$\mathscr{D}=\bigoplus_{i=1}^dk[\![u]\!]\varphi(e_i)$$\index{$\mathscr{D}$} and $$\mathscr{D}_{u,\tau}=\bigoplus_{i=1}^dk(\!(\eta)\!)[\![u]\!]\varphi(e_i).$$\index{$\mathscr{D}_{u,\tau}$} By construction, we have $D_{u,\tau}\subset\bigcup\limits_{n=0}^\infty\mathscr{D}_{u,\tau}\big[\eta^{1/p^n},\frac{1}{u}\big]$ ($\cf$ definition \ref{def phi tau Fp u}) and $\mathscr{D}_{u,\tau}$ is $u$-adically separated and complete.
\end{notation}

\begin{lemma}\label{lem cf-1 proof of thm 3.32}
	Assume $n, k\in \NN$ are such that $\tau_D^{k}(\varphi(e_i))\in\mathscr{D}_{u,\tau}[\eta^{1/p^n}].$ Then for any $m\in \NN,$ $y\in \frac{1}{u^m}\mathscr{D}_{u,\tau}[\eta^{1/p^n}]$ implies $\tau_D^{k}(y)\in\frac{1}{u^m}\mathscr{D}_{u,\tau}[\eta^{1/p^n}].$
\end{lemma}

\begin{proof}
	Write $y=\sum_{i=1}^d\lambda_i\varphi(e_i)$ with $\lambda_i\in\frac{1}{u^m}k(\!(\eta)\!)[\![u]\!].$ Then $\tau_D^{k}(y)=\sum_{i=1}^d\tau^{k}(\lambda_i)\tau_D^{k}(\varphi(e_i)).$ By assumption we know that $\tau_D^{k}(\varphi(e_i))\in\mathscr{D}_{u,\tau}[\eta^{1/p^n}]:$ it remains to control $\tau^{k}(\lambda_i).$
	Recall that $\tau(u)=\varepsilon u=u+\eta u.$ Write $\lambda_i=\frac{1}{u^m}\mu_i$ with $\mu_i\in k(\!(\eta)\!)[\![u]\!]:$ we have $\tau^{k}(\lambda_i)=\frac{1}{u^m\varepsilon^{km}}\tau^{k}(\mu_i)\in \frac{1}{u^m}k(\!(\eta)\!)[\![u]\!]$ as $\varepsilon=1+\eta\in k[\![\eta]\!]^{\times}.$
\end{proof}

\begin{proof}[Proof of theorem \ref{prop quais-iso psi} for $\FF_p$-representations:]
	
	Let $y\in D_{u,\tau,0}^{\psi=0}:$ there exist $n,m\in\NN$ such that $y\in\frac{1}{u^m}\mathscr{D}_{u,\tau}[\eta^{1/p^n}].$ By continuity, there exists $r\in\NN$ such that $\tau_D^{p^r}(e_i)\equiv e_i\mod u^m\mathscr{D}_{u,\tau}[\eta^{1/p^n}]$ (making $n$ larger if necessary), whence $\tau_D^{p^r}(\varphi(e_i))\equiv\varphi(e_i)\mod u^{pm}\mathscr{D}_{u,\tau}[\eta^{1/p^n}]$ (recall that $\tau_D$ and $\varphi$ commute). Put \ft{\cf corollary \ref{coro precise computation of map between different tau power -1} for more details.}
	$$z=\tfrac{\tau_D^{p^r}-1}{\tau_D-1}(y)=(1+\tau_D+\cdots+\tau_D^{p^r-1})(y)\in D_{u,\tau^{p^r},0}^{\psi=0}$$
	There exists $l\in \NN$, such that $z\in\frac{1}{u^l}\mathscr{D}_{u,\tau}[\eta^{1/p^n}]^{\psi=0}$ (making $n$ larger if necessary) by lemma  \ref{lem cf-1 proof of thm 3.32} and the fact that $\psi$ commutes with $\tau_D.$ By lemma \ref{lemm5.18}, we can write $z=\sum\limits_{i=1}^d\sum\limits_{\substack{j\geq-l\\ p\nmid j}}f_{i,j}(\eta)u^j\varphi(e_i)$ with $f_{i,j}(\eta)\in k(\!(\eta^{1/p^n})\!).$ An argument similar to the proof of lemma \ref{lemm inverse of tau-1} shows that for all $i,j,$ there exists $c_{i,j}\in k$ such that $f_{i,j}(\eta)=c_{i,j}(\varepsilon^{jp^r}-1)\ \mod u^{pm}\mathscr{D}_{u,\tau}[\eta^{1/p^n}].$ 
	
	Indeed $z\in D_{u,\tau^{p^r},0}^{\psi=0}$ implies that
	$$(\gamma\otimes 1)z=(1+\tau_D^{p^r}+\tau_D^{2p^r}+\cdots +\tau_D^{( \chi{(\gamma)}-1) p^r } )(z).$$
	The left hand side is
	$$(\gamma\otimes 1)z=(\gamma\otimes 1)\big(\sum\limits_{i=1}^d\sum\limits_{\substack{j\geq-l\\ p\nmid j}}f_{i,j}(\eta)u^j\varphi(e_i)\big)=\sum\limits_{i=1}^d\sum\limits_{\substack{j\geq-l\\ p\nmid j}}f_{i,j}(\eta^{\chi(\gamma)})u^j\varphi(e_i)$$
	as $\varphi(e_i)$ are all fixed by $\mathscr{G}_{K_{\pi}},$ hence in particular fixed by $\gamma.$ For the right hand side, we have
	\begin{align*}
	&(1+\tau_D^{p^r}+\tau_D^{2p^r}+\cdots +\tau_D^{( \chi{(\gamma)}-1) p^r } )(z)\\
	\equiv& \sum\limits_{i=1}^d\sum\limits_{\substack{j\geq-l\\ p\nmid j}}f_{i,j}(\eta)u^j(1+\varepsilon^{jp^r}+\varepsilon^{2jp^r}+\cdots + \varepsilon^{(\chi(\gamma)-1)jp^r})\varphi(e_i)\  \mod u^{pm}\mathscr{D}_{u,\tau}[\eta^{1/p^n}]\\
	=&  \sum\limits_{i=1}^d\sum\limits_{\substack{j\geq-l\\ p\nmid j}}f_{i,j}(\eta)u^j\big( \frac{\varepsilon^{\chi(\gamma)j^{p^r}-1}}{\varepsilon^{jp^r}-1} \big) \varphi(e_i)
	\end{align*} 
	as $\tau_D^{p^r}(\varphi(e_i))\equiv\varphi(e_i)\mod u^{pm}\mathscr{D}_{u,\tau}[\eta^{1/p^n}].$ Hence for all $i\in\{1, \ldots, d\}$ and $j<pm,$ there exists $c_{i,j}\in k$ such that $f_{i,j}(\eta)=c_{i,j}(\varepsilon^{jp^r}-1)$ ($\cf$ lemma \ref{lemm inverse of tau-1}). Hence 
	\[z=\sum\limits_{i=1}^d\sum\limits_{\substack{p\nmid j \\ -l\leq j\leq pm}}c_{i,j}(\varepsilon^{jp^r}-1)u^j\varphi(e_i)\ \mod u^{pm}\mathscr{D}_{u,\tau}[\eta^{1/p^n}].\]
	
	Put
	$$x_0=\sum_{i=1}^d\Big(\sum_{\substack{p\nmid j\\-l\leq j\leq pm}} c_{i,j}u^j\Big)\varphi(e_i)\in\tfrac{1}{u^m}\mathscr{D}.$$
	For all $i\in\{1,\ldots,d\},$ we can write $\tau_D^{p^r}(\varphi(e_i))=\varphi(e_i)+u^{pm}g_{i,j}$ with $g_{i,j}\in\mathscr{D}_{u,\tau}[\eta^{1/p^n}]:$ we have
	\[  (\tau_D^{p^r}-1)(x_0)=\sum\limits_{i=1}^d\sum\limits_{\substack{p\nmid j\\-l\leq j\leq pm}} c_{i,j}\Big(\varepsilon^{jp^r}u^j(\varphi(e_i)+u^{pm}g_{i,j})-u^j\varphi(e_i)\Big)\equiv z-z_1 \ \mod u^{pm}\mathscr{D}_{u,\tau}[\eta^{1/p^n}], \]
	where \[ z_1=\sum\limits_{i=1}^d\sum\limits_{\substack{j\geq pm\\ p\nmid j}}f_{i,j}(\eta)u^j\varphi(e_i) -\sum\limits_{i=1}^d\sum\limits_{\substack{p\nmid j\\-l\leq j\leq pm}} c_{i,j}\varepsilon^{jp^r}u^{j+pm}g_{i,j}\in u^{pm-l}\mathscr{D}_{u,\tau}[\eta^{1/p^n}].\]
	By construction, we have $\psi(x_0)=0,$ which implies $\psi\big((\tau_D^{p^r}-1)(x_0)\big)=0:$ as $\psi(z)=0,$ we have $\psi(z_1)=0$ as well. Note also that $\tau_D^{p^r}-1$ maps $D$ into $D_{u,\tau^{p^r},0}:$ as $z\in D_{u,\tau^{p^r},0},$ we also have $z_1\in D_{u,\tau^{p^r},0}.$ This shows that we can carry on the preceding construction, and build sequences $(z_s)_{s\in\NN}$ in $D_{u,\tau^{p^r},0}$ and $(x_s)_{s\in\NN}$ in $D$ such that $z_0=z,$ $z_s\in u^{spm-l}\mathscr{D}_{u,\tau}[\eta^{1/p^n}],$ $x_s\in u^{spm-l}\mathscr{D}[\eta^{1/p^n}]$  and $(\tau_D^{p^r}-1)(x_\ell)=z_s-z_{s+1}$ for all $s\in\NN.$ The series $x=\sum\limits_{s=0}^\infty x_s$ converges in $\frac{1}{u^{l}}\mathscr{D},$ and summing all the equalities gives $(\tau_D^{p^r}-1)(x)=z.$
	
	This shows that $\frac{\tau_D^{p^r}-1}{\tau_D-1}\big((\tau_D-1)(x)\big)=(\tau_D^{p^r}-1)(x)=\frac{\tau_D^{p^r}-1}{\tau_D-1}(y):$ as $\frac{\tau_D^{p^r}-1}{\tau_D-1}\colon D_{u,\tau,0}^{\psi=0}\to D_{u,\tau^{p^r},0}^{\psi=0}$ is injective by proposition \ref{prop tau-1 high power inj}, we get $y=(\tau_D-1)(x),$ showing that $y$ belongs to the image of $\tau_D-1.$ As this holds for all $y\in D_{u,\tau,0}^{\psi=0},$ this proves the surjectivity of $\tau_D-1: D^{\psi=0}\to D_{u,\tau,0}^{\psi=0}$. Together with corollary \ref{cor injective before "the trivial case"}, we finish the proof of theorem \ref{prop quais-iso psi} for $\FF_p$-representations.
\end{proof}

\begin{corollary}\label{coro Fp bijection thm} For all $r\in\NN_{\geq0},$ the maps $\tau_D^{p^r}-1\colon D^{\psi=0}\to D_{u,\tau^{p^r},0}^{\psi=0}$ and $\frac{\tau_D^{p^r}-1}{\tau_D-1}\colon D_{u,\tau,0}^{\psi=0}\to D_{u,\tau^{p^r},0}^{\psi=0}$ are bijective.
\end{corollary}

\begin{proof}
	Replacing $T$ by its restriction to $\mathscr{G}_{K_r}$ and replacing its $(\varphi,\tau)$-module by the corresponding $(\varphi,\tau^{p^r})$-module, $\ie\tau_D$ by $\tau_D^{p^r},$ the $\FF_p$-case of theorem \ref{prop quais-iso psi} thus implies that the map $\tau_D^{p^r}-1\colon D^{\psi=0}\to D_{u,\tau^{p^r},0}^{\psi=0}$ is bijective. The statement about $\frac{\tau_D^{p^r}-1}{\tau_D-1}$ follows.
\end{proof}

\bigskip

\subsection{Proof of theorem \ref{prop quais-iso psi}: the quasi-isomophism}

\begin{lemma}\label{lemm tue par psi is exact}
	Let $0\to T^{\prime}\to T\to T^{\prime\prime}\to 0$ be a short exact sequence in $\Rep_{\ZZ_p}(\mathscr{G}_K),$ we then have a short exact sequence
	\[ 0\to \mathcal{D}(T^{\prime})^{\psi=0}\to \mathcal{D}(T)^{\psi=0}\to \mathcal{D}(T^{\prime\prime})^{\psi=0}\to 0.  \]
\end{lemma}

\begin{proof}
	We know by lemma \ref{Exactness D-tau-0 Zp} that $\mathcal{D}(-)$ is an exact functor and hence we have the short exact sequence $0\to \mathcal{D}(T^{\prime})\to \mathcal{D}(T)\to \mathcal{D}(T^{\prime\prime})\to 0.$ Now we consider the following diagram of complexes:
	
	\[\xymatrix{ & \mathcal{D}(T^{\prime})^{\psi=0} \ar[r] \ar[d]& \mathcal{D}(T)^{\psi=0} \ar[r] \ar[d]& \mathcal{D}(T^{\prime\prime})^{\psi=0} \ar[d]& \\
		0  \ar[r] & \mathcal{D}(T^{\prime}) \ar[r] \ar[d]^{\psi}& \mathcal{D}(T) \ar[r] \ar[d]^{\psi}& \mathcal{D}(T^{\prime\prime}) \ar[r] \ar[d]^{\psi}& 0\\
		0  \ar[r] & \mathcal{D}(T^{\prime}) \ar[r] & \mathcal{D}(T) \ar[r] & \mathcal{D}(T^{\prime\prime}) \ar[r] & 0.
	}\]
	Since $\psi$ is surjective, by snake lemma we get the following short exact sequence: \[ 0\to \mathcal{D}(T^{\prime})^{\psi=0} \to \mathcal{D}(T)^{\psi=0} \to \mathcal{D}(T^{\prime\prime})^{\psi=0}\to 0.  \]
\end{proof}

\begin{proposition}\label{prop Zp-case kernel complex trivial}
	Let $T\in \Rep_{\ZZ_p}(\mathscr{G}_{K}),$ we then have a bijection $\mathcal{D}(T)^{\psi=0}\xrightarrow{\tau_D-1} \mathcal{D}(T)_{u,\tau, 0}^{\psi=0}.$
\end{proposition}

\begin{proof}
	Notice that $\mathcal{D}(T/(p^n))=\mathcal{D}(T)/(p^n)$ and $T,$ $\mathcal{D}(T)$ are both $p$-adically complete. Hence it suffices to prove the cases when $T$ is killed by $p^n$ with $n\in \NN.$ We use induction on $n,$ the case $n=1$ being corollary \ref{coro Fp bijection thm} with $r=1.$ Suppose $T$ is killed by $p^n.$ Put $T^{\prime}=p^{n-1}T,$ $T^{\prime\prime}=T/T^{\prime}$ and consider the exact sequence in $\Rep_{\ZZ_p}(\mathscr{G}_K)$ \[0\to T^{\prime}\to T\to T^{\prime\prime}\to 0. \]
	We then have the following commutative diagram
	
	\[ \xymatrix{  0 \ar[r] & \mathcal{D}(T^{\prime})^{\psi=0} \ar[r] \ar[d]_{\tau_D-1}& \mathcal{D}(T)^{\psi=0} \ar[r] \ar[d]_{\tau_D-1}& \mathcal{D}(T^{\prime\prime})^{\psi=0} \ar[d]_{\tau_D-1} \ar[r]& 0 \\
		0 \ar[r] & \mathcal{D}(T^{\prime})_{u, \tau, 0}^{\psi=0} \ar[r]& \mathcal{D}(T)_{u, \tau, 0}^{\psi=0} \ar[r]& \mathcal{D}(T^{\prime\prime})_{u, \tau, 0}^{\psi=0} &
	}   \]
	
	The first line is exact by lemma \ref{lemm tue par psi is exact} and the second from the fact $\mathcal{D}(-)_{u, \tau, 0}$ is left exact and then we apply the functor $(-)^{\psi=0}$ which is also left exact. As the first and third vertical maps are isomorphisms by induction hypothesis, so is that in the middle. We then finish the proof by passing to the limit. 
\end{proof}

\begin{remark}
	The theorem \ref{prop quais-iso psi} follows from proposition \ref{prop Zp-case kernel complex trivial}.
\end{remark}

\section{The \texorpdfstring{$(\varphi, \tau^{p^r})$}{(varphi, tau p powers)}-modules}\label{section 3.4 phi tau}

\begin{notation}\label{not Kr} 	
	We put $K_r=K(\pi_{r})$\index{$K_r$} for $r\in \NN$ and we have the following diagram:
	\[  \begin{tikzcd}[every arrow/.append style={dash}]
	&&\overline{K} \arrow[ldd, bend right, "\mathscr{G}_{K_{\pi}}"] \arrow[llddd, bend right, "\mathscr{G}_{K_{r}}"'] \arrow[d, "\mathscr{G}_L"]\arrow[rrddd, bend left, "\mathscr{G}_{K_{\zeta}}"]&&\\
	&&L\arrow[ld, "\overline{\la \gamma \ra}"]\arrow[ddd]\arrow[rd,"\overline{\la \tau^{p^r} \ra}"']&&\\		
	&K_{\pi}\arrow[ld]& &  K_{\zeta}K_r\arrow[rd]&\\
	K_r \arrow[rrd]&&&&K_{\zeta}\arrow[lld, "\Gamma=\overline{\la\gamma \ra}"]\\
	&& K &
	\end{tikzcd}
	\]
\end{notation}

\begin{definition}\label{def (phi, tau_p^r)-modules normal version}
	Let $r\in \NN$ and then a \emph{$(\varphi, \tau^{p^r})$-module over $(F_0, F_{\tau})$} is the data:
	\item (1) an \'etale $\varphi$-module $D$ over $F_0$;
	\item (2) a $\tau^{p^r}$-semi-linear map $\tau_D^{p^r}$ on $D_{\tau}:=F_{\tau}\otimes_{F_0}D$ which commutes with $\varphi_{F_{\tau}}\otimes \varphi_D$ (where $\varphi_{F_{\tau}}$ is the Frobenius map on $F_{\tau}$ and $\varphi_D$ the Frobenius map on $D$) and such that 
	\[ (\forall x\in M)  \,   (g\otimes 1)\circ \tau_D^{p^r}(x)=(\tau_D^{p^r})^{\chi(g)}(x), \] 
	for all $g\in \mathscr{G}_{K_{\pi}}/\mathscr{G}_L$ such that $\chi(g)\in\NN.$\\
	
	We denote $\Mod_{F_0, F_{\tau}}(\varphi,\tau^{p^r})$\index{$\Mod_{F_0, F_{\tau}}(\varphi,\tau^{p^r})$} the corresponding category. 
\end{definition}

By \cite[Remark 1.15]{Car13}, we have an equivalence of categories between $\Rep_{\FF_p}(\mathscr{G}_{K_r})$ and the category of $(\varphi, \tau^{p^r})$-modules over $(F_0, F_{\tau}).$

\begin{notation}\label{not D u tau normla version}
	Let $(D, D_{\tau})\in \Mod_{F_0, F_{\tau}}(\varphi,\tau^{p^r}).$ We put
	\[D_{\tau^{p^r}, 0}:=\big\{ x\in D_{\tau};\ (\forall g\in \mathscr{G}_{K_{\pi}})\ \chi(g)\in \ZZ_{>0} \Rightarrow (g\otimes 1)(x)=x+\tau_D^{p^r}(x)+\tau_D^{2p^r}(x)+\cdots +\tau_D^{p^r(\chi(g)-1)}(x) \big\}. \]\index{$D_{\tau^{p^r}, 0}$}
	By similar arguments as that of lemma \ref{Replace g witi gamma}, we have 
	\[D_{\tau^{p^r}, 0}=\big\{ x\in D_{\tau} ;\  (\gamma\otimes 1)x=(1+\tau_D^{p^r}+\tau_D^{2p^r}+\cdots + \tau_D^{p^{r}(\chi(\gamma)-1)})(x)   \big\}.\]
\end{notation}

\begin{definition}\label{def complex phi tau n normal version}
	%Let $T\in \Rep_{\FF_p}(\mathscr{G}_K),$  and $(D, D_{\tau})\in \Mod_{F_0, F_{\tau}}(\varphi,\tau^{p^r})$ be its $(\varphi, \tau^{p^r})$-module over $(F_0, F_{\tau}),$ then we define a complex $\mathcal{C}_{\varphi, \tau^{p^r}}(D)$ as follows:
	Let $(D, D_{\tau})\in \Mod_{F_0, F_{\tau}}(\varphi,\tau^{p^r}).$ We define a complex $\mathcal{C}_{\varphi, \tau^{p^r}}(D)$\index{$\mathcal{C}_{\varphi, \tau^{p^r}}(D)$} as follows:
	\[ \xymatrix{
		0\ar[rr] && D \ar[r]  & D\oplus D_{ \tau^{p^r}, 0} \ar[r]  & D_{\tau^{p^r}, 0}\ar[r]  & 0\\
		&&x\ar@{|->}[r] &((\varphi-1)(x), (\tau^{p^r}-1)(x))&&\\
		&&& (y, z) \ar@{|->}[r] &(\tau^{p^r}-1)(y)-(\varphi-1)(z).&}
	\]
	If $T\in \Rep_{\FF_p}(\mathscr{G}_K),$ we have in particular the complex $\mathcal{C}_{\varphi, \tau^{p^r}}(\mathcal{D}(T)),$ which will also be simply denoted $\mathcal{C}_{\varphi, \tau^{p^r}}(T).$\index{$\mathcal{C}_{\varphi, \tau^{p^r}}(T)$}
	
\end{definition}

\begin{proposition}\label{prop tau-puissance complex normal version}
	The complex $\mathcal{C}_{\varphi, \tau^{p^r}}(T)$ computes the continuous Galois cohomology $\H^i(\mathscr{G}_{K_r}, T)$ for $i\in \NN.$
\end{proposition}	

\begin{proof}
	This follows from theorem \ref{thm main result} by replacing $K$ by $K_r$ ($\cf$ \cite[Remarque 1.15]{Car13}).
\end{proof}

\begin{corollary}\label{coro precise computation of map between different tau power -1}
	For any $r\in \NN,$ we have the following morphism of complexes:
	\[ 
	\xymatrix{
		\Cc_{\varphi,\tau^{p^r}}\colon & 0\ar[r]  & D \ar[r] \ar@{=}[d] & D\oplus D_{\tau^{p^r},0} \ar[r] \ar[d]^{\id\oplus \frac{\tau_D^{p^{r+1}}-1}{\tau_D^{p^{r}}-1}} & D_{\tau^{p^r}, 0}\ar[r] \ar[d]^{\frac{\tau_D^{p^{r+1}}-1}{\tau_D^{p^{r}}-1}} & 0 \\
		\Cc_{\varphi,\tau^{p^{r+1}}}\colon & 0\ar[r]  & D \ar[r]  & D\oplus D_{\tau^{p^{r+1}}, 0} \ar[r]& D_{\tau^{p^{r+1}}, 0}\ar[r] & 0.
	} \] 
\end{corollary}

\begin{proof}
	This follows from direct computations. Recall that for any $g\in \mathscr{G}_K$ and $T\in \Rep_{\ZZ_p}(\mathscr{G}_K),$ the action $g\otimes 1$ over $\mathcal{D}(T)_{\tau}$ is induced by that of $g\otimes g$ over $\W(C^{\flat})\otimes_{\ZZ_p}T.$ For any $x\in D_{\tau^{p^r},0}:$ we have $(\gamma\otimes 1)x=\big(1+\tau_D^{p^r}+\tau_D^{2p^r}+\cdots + \tau_D^{p^{r}(\chi(\gamma)-1)}\big)(x).$ Now we verify that $y:=\frac{\tau_D^{p^{r+1}}-1}{\tau_D^{p^{r}}-1}(x)$ is in $D_{\tau^{p^{r+1}}, 0}.$ Indeed, by the relation $\gamma \tau=\tau^{\chi(\gamma)}\gamma$ we have 
	\begin{align*}
	\gamma(y)&=\gamma\Big(\frac{\tau_D^{p^{r+1}}-1}{\tau_D^{p^{r}}-1}\Big)(x)\\
	&=\frac{\tau_D^{\chi(\gamma)p^{r+1}}-1}{\tau_D^{\chi(\gamma)p^r}-1}(\gamma(x))\\
	&=\frac{\tau_D^{\chi(\gamma)p^{r+1}}-1}{\tau_D^{\chi(\gamma)p^r}-1}\big(\big(1+\tau_D^{p^r}+\tau_D^{2p^r}+\cdots + \tau_D^{p^{r}(\chi(\gamma)-1)}\big)(x)\big)\\
	&=\frac{\tau_D^{\chi(\gamma)p^{r+1}}-1}{\tau_D^{\chi(\gamma)p^r}-1}\cdot\frac{\tau_D^{\chi(\gamma)p^r}-1}{\tau_D^{p^r}-1}(x)\\
	&=\frac{\tau_D^{\chi(\gamma)p^{r+1}}-1}{\tau_D^{p^r}-1}(x).
	\end{align*}
	On the other hand
	\begin{align*}
	&\big(1+\tau_D^{p^{r+1}}+\tau_D^{2p^{r+1}}+\cdots + \tau_D^{p^{r+1}(\chi(\gamma)-1)}\big)(y)\\
	=&\big(1+\tau_D^{p^{r+1}}+\tau_D^{2p^{r+1}}+\cdots + \tau_D^{p^{r+1}(\chi(\gamma)-1)}\big) \Big(\frac{\tau_D^{p^{r+1}}-1}{\tau_D^{p^{r}}-1}\Big)(x)\\
	=&\frac{\tau_D^{\chi(\gamma)p^{r+1}}-1}{\tau_D^{p^{r+1}}-1}\cdot \frac{\tau_D^{p^{r+1}}-1}{\tau_D^{p^{r}}-1}(x)\\
	=&\frac{\tau_D^{\chi(\gamma)p^{r+1}}-1}{\tau_D^{p^{r}}-1}(x).\\
	\end{align*} 
	This shows that
	\[\gamma(y)=\big(1+\tau_D^{p^{r+1}}+\tau_D^{2p^{r+1}}+\cdots + \tau_D^{p^{r+1}(\chi(\gamma)-1)}\big)(y), \]
	and we conclude that $y\in D_{\tau^{p^{r+1}}, 0}.$
\end{proof}

\begin{lemma}\label{lemm map for different tau powers normal version}
	Let $n, m\in \NN$ with $n\leq m.$ There is a natural  map $\frac{\tau_D^{p^{m}}-1}{\tau_D^{p^{n}}-1} \colon D_{\tau^{p^n},0}\to D_{\tau^{p^{m}}, 0}.$ 
\end{lemma}

\begin{proof}
	By direct computations as in the proof of corollary \ref{coro precise computation of map between different tau power -1}.
\end{proof}

\subsubsection{The \texorpdfstring{$(\varphi, \tau^{p^r})$}{(phi, tau p powers)}-modules over partially unperfected coefficients}

Similarly, we have results for $(\varphi, \tau^{p^r})$-modules over partially unperfected coefficients.

\begin{definition}\label{def (phi, tau_p^r)-modules}
	Let $r\in \NN$ and then a \emph{$(\varphi, \tau^{p^r})$-module over $(F_0, F_{u, \tau})$} is the data:
	\item (1) an \'etale $\varphi$-module $D$ over $F_0$;
	\item (2) a $\tau^{p^r}$-semilinear map $\tau_D^{p^r}$ on $D_{u, \tau}:=F_{u, \tau}\otimes_{F_0}D$ which commutes with $\varphi_{F_{u, \tau}}\otimes \varphi_D$ (where $\varphi_{F_{u, \tau}}$ is the Frobenius map on $F_{u, \tau}$ and $\varphi_D$ the Frobenius map on $D$) and such that 
	\[ (\forall x\in M)  \quad   (g\otimes 1)\circ \tau_D^{p^r}(x)=(\tau_D^{p^r})^{\chi(g)}(x), \] 
	for all $g\in \mathscr{G}_{K_{\pi}}/\mathscr{G}_L$ such that $\chi(g)\in\NN.$\\
	
	We denote $\Mod_{F_0, F_{u, \tau}}(\varphi,\tau^{p^r})$\index{$\Mod_{F_0, F_{u, \tau}}(\varphi,\tau^{p^r})$} the corresponding category. 
\end{definition}

\begin{remark}
	We have an equivalence of categories between $\Rep_{\FF_p}(\mathscr{G}_{K_r})$ and $\Mod_{F_0, F_{u, \tau}}(\varphi,\tau^{p^r})$ ($\cf$ \cite[Remark 1.15]{Car13}). 
\end{remark}

\begin{notation}\label{not D u tau }
	Let $(D, D_{u, \tau})\in \Mod_{F_0, F_{u, \tau}}(\varphi,\tau^{p^r}).$ We put
	\[D_{u, \tau^{p^r}, 0}:=\big\{ x\in D_{u, \tau};\ (\forall g\in \mathscr{G}_{K_{\pi}})\ \chi(g)\in \ZZ_{>0} \Rightarrow (g\otimes 1)(x)=x+\tau_D^{p^r}(x)+\tau_D^{2p^r}(x)+\cdots +\tau_D^{p^r(\chi(g)-1)}(x) \big\}. \]\index{$D_{u, \tau^{p^r}, 0}$}
	By similar arguments as that of lemma \ref{Replace g witi gamma}, we see that 
	
	\[D_{u, \tau^{p^r}, 0}=\big\{ x\in D_{u,\tau} ;\  (\gamma\otimes 1)x=(1+\tau_D^{p^r}+\tau_D^{2p^r}+\cdots + \tau_D^{p^{r}(\chi(\gamma)-1)})(x)  \big\}.\]
\end{notation}

\begin{definition}\label{def complex phi tau n}
	%Let $T\in \Rep_{\FF_p}(\mathscr{G}_K)$ and $(D, D_{u, \tau})\in\Mod_{F_0, F_{u, \tau}}(\varphi,\tau^{p^r})$ be its $(\varphi, \tau^{p^r})$-module over $(F_0, F_{u, \tau}),$ then we define a complex $\mathcal{C}^{u}_{\varphi, \tau^{p^r}}(D)$ as follows:
	Let $(D, D_{u, \tau})\in\Mod_{F_0, F_{u, \tau}}(\varphi,\tau^{p^r}).$ We define a complex $\mathcal{C}^{u}_{\varphi, \tau^{p^r}}(D)$\index{$\mathcal{C}^{u}_{\varphi, \tau^{p^r}}(D)$} as follows:
	\[ \xymatrix{
		0\ar[rr] && D \ar[r]  & D\oplus D_{u, \tau^{p^r}, 0} \ar[r]  & D_{u, \tau^{p^r}, 0}\ar[r]  & 0\\
		&&x\ar@{|->}[r] &((\varphi-1)(x), (\tau^{p^r}-1)(x))&&\\
		&&& (y, z) \ar@{|->}[r] &(\tau^{p^r}-1)(y)-(\varphi-1)(z).&}
	\]
	If $T\in \Rep_{\FF_p}(\mathscr{G}_K),$ we have in particular the complex $\mathcal{C}^u_{\varphi, \tau^{p^r}}(\mathcal{D}(T)),$ which will also be simply denoted $\mathcal{C}^u_{\varphi, \tau^{p^r}}(T).$\index{$\mathcal{C}^u_{\varphi, \tau^{p^r}}(T)$}
\end{definition}

\begin{proposition}\label{prop tau-puissance complex}
	The complex $\mathcal{C}^{u}_{\varphi, \tau^{p^r}}(T)$ computes the continuous Galois cohomology $\H^i(\mathscr{G}_{K_r}, T)$ for $i\in \NN.$
\end{proposition}	

\begin{proof}
	This follows from theorem \ref{coro complex over non-perfect ring works also Zp-case} replacing $K$ by $K_r.$
\end{proof}

\begin{lemma}\label{lemm map for different tau powers}
	Let $n, m\in \NN$ with $n<m.$ There is a natural  map $D_{u, \tau^{p^n},0}\xrightarrow{\frac{\tau_D^{p^{m}}-1}{\tau_D^{p^{n}}-1}} D_{u, \tau^{p^{m}}, 0}.$ 
\end{lemma}

\begin{proof}
	By direct computations similar as in corollary \ref{coro precise computation of map between different tau power -1}.
\end{proof}

\begin{corollary}
	There is a morphism between complexes:
	\[ 
	\xymatrix{
		\Cc^{u}_{\varphi,\tau^{p^r}}\colon & 0\ar[r]  & D \ar[r] \ar@{=}[d] & D\oplus D_{u, \tau^{p^r},0} \ar[r] \ar[d]^{\id\oplus \frac{\tau_D^{p^{r+1}}-1}{\tau_D^{p^{r}}-1}} & D_{u, \tau^{p^r}, 0}\ar[r] \ar[d]^{\frac{\tau_D^{p^{r+1}}-1}{\tau_D^{p^{r}}-1}} & 0 \\
		\Cc^{u}_{\varphi,\tau^{p^{r+1}}}\colon & 0\ar[r]  & D \ar[r]  & D\oplus D_{u, \tau^{p^{r+1}}, 0} \ar[r]& D_{u, \tau^{p^{r+1}}, 0}\ar[r] & 0.
	} \] 
\end{corollary}

\begin{proof}
	This follows from direct computations similar as in corollary \ref{coro precise computation of map between different tau power -1}.
\end{proof}

\begin{remark}
	In subsection \ref{section 3.4 phi tau}, $(\varphi, \tau^{p^r})$-modules with $r\in\NN$ are only defined for $\Rep_{\FF_p}(\mathscr{G}_K).$ One can easily define the categories $\Mod_{\Oo_{\Ee}, \Oo_{\Ee_{\tau}}}(\varphi,\tau^{p^r}),$\index{$\Mod_{\Oo_{\Ee}, \Oo_{\Ee_{\tau}}}(\varphi,\tau^{p^r})$} $\Mod_{\Oo_{\Ee}, \Oo_{\Ee_{u,\tau}}}(\varphi,\tau^{p^r}),$\index{$\Mod_{\Oo_{\Ee}, \Oo_{\Ee_{u,\tau}}}(\varphi,\tau^{p^r})$} $\Mod_{\Ee, \Ee_{\tau}}(\varphi,\tau^{p^r}),$\index{$\Mod_{\Ee, \Ee_{\tau}}(\varphi,\tau^{p^r})$}  $\Mod_{\Ee, \Ee_{u,\tau}}(\varphi,\tau^{p^r}),$\index{$\Mod_{\Ee, \Ee_{u,\tau}}(\varphi,\tau^{p^r})$} and generalize the related notions to $\Rep_{\ZZ_p}(\mathscr{G}_K)$ and $\Rep_{\QQ_p}(\mathscr{G}_K).$ The results given for $\Mod_{F_0, F_{\tau}}(\varphi,\tau^{p^r})$ and $\Mod_{F_0, F_{u, \tau}}(\varphi,\tau^{p^r})$ hold similarly for the other categories.
\end{remark}

%% file: Chapter4.tex
\chapter{Complexes over overconvergent rings}

In this chapter, we introduce $(\varphi, \tau)$-modules over overconvergent rings $(\Ee^\dagger, \Ee_{\tau}^\dagger)$ and $(\Ee^\dagger, \Ee_{u,\tau}^\dagger).$ Then we will define complexes $\Cc_{\varphi, \tau}^{u}(D^{\dagger})$ and $\Cc_{\psi, \tau}^{u}(D^{\dagger}),$ which embed respectively into the complexes $\Cc_{\varphi, \tau}^{u}(D)$ and $\Cc_{\psi, \tau}^{u}(D)$ (defined in chapter \ref{chapter Complex with psi-operator}). We will show that they are quasi-isomorphic and calculate the correct $\H^0$ and $\H^1$.

\section{Locally analytic vectors}

In this section, we will use results of Poyeton ($\cf$ \cite{Poy19}), hence also some notations of \textit{loc. cit.}  However we made some modifications for the sake of consistency with our notations.

\medskip

Since we concentrate on $(\varphi, \tau)$-modules, we remove subscripts $\tau$ in the notations of Poyeton when it is used to distinguish $(\varphi, \tau)$-modules from $(\varphi, \Gamma)$-modules. We also replace subscripts $K$ in the notations of Poyeton with $K_{\pi}$ when it corresponds to invariants under $\mathscr{G}_{K_{\pi}}$ (this is also consistent with our subscripts $K_{\zeta},$ $\cf$ notation \ref{not dictionary} for details).

\begin{notation}\label{not dictionary}
	We have the following dictionary between our notations used in the first three chapters and the notations from Poyeton ($\cf$ \cite{Poy19}). 
	
	\begin{align*}
	&\widetilde{\EE}^{+}=\Oo_{C^\flat},\ \widetilde{\EE}=C^\flat,\  \widetilde{\EE}_u=C^{\flat}_{u-\np}   ,\  \EE_{K_{\pi}}=F_0,\ \EE=F_0^{\sep},\ \widetilde{\EE}_L=F_\tau, \ \widetilde{\EE}_{u,L}=F_{u,\tau}\\
	&\widetilde{\AA}^{+}=\W(\Oo_{C^\flat}),\ \widetilde{\AA}=\W(C^\flat),\  \widetilde{\AA}_{u}=\Oo_{\widehat{\Ff^{\ur}_{u}}},\ \AA_{K_{\pi}}^{+}=\W(k)[\![u]\!],\ \AA_{K_{\pi}}=\mathcal{O}_{\mathcal{E}},\ \AA=\mathcal{O}_{\widehat{\mathcal{E}^{\ur}}},\ \widetilde{\AA}_L=\mathcal{O}_{\mathcal{E}_\tau},\ \widetilde{\AA}_{u, L}=\mathcal{O}_{\mathcal{E}_{u,\tau}}\\
	&\widetilde{\BB}^{+}=\W(\Oo_{C^\flat})[1/p],\ \widetilde{\BB}=\W(C^\flat)[1/p],\ \widetilde{\BB}_u=\widehat{\Ff^{\ur}_{u}},\ \BB_{K_{\pi}}^{+}=\AA_{K_{\pi}}^{+}[1/p],\ \BB_{K_{\pi}}=\mathcal{E},\ \BB=\widehat{\mathcal{E}^{\ur}},\ \widetilde{\BB}_L=\mathcal{E}_\tau, \ \widetilde{\BB}_{u,L}=\mathcal{E}_{u,\tau}
	\end{align*}\index{$\widetilde{\EE}^{+}$}\index{$\widetilde{\EE}$}\index{$\widetilde{\EE}_u$}\index{$\EE_{K_{\pi}}$}\index{$\widetilde{\EE}_L$}
	\index{$\widetilde{\EE}_{u,L}$}\index{$\widetilde{\AA}^{+}$}\index{$\widetilde{\AA}$}\index{$\widetilde{\AA}_{u}$}\index{$\AA_{K_{\pi}}^{+}$}\index{$\AA_{K_{\pi}}$} \index{$\AA$} \index{$\widetilde{\AA}_L$} \index{$ \widetilde{\AA}_{u, L}$} \index{$\widetilde{\BB}^{+}$} \index{$\widetilde{\BB}$} \index{$\widetilde{\BB}_u$} \index{$\BB_{K_{\pi}}^{+}$} \index{$\BB_{K_{\pi}}$} \index{$\BB$} \index{$\widetilde{\BB}_L$} \index{$\widetilde{\BB}_{u,L}$}
	Notice that we can read information from the notation itself: if $A$ is an algebra endowed with an action of $\mathscr{G}_L$ (resp. $\mathscr{G}_{K_{\pi}}$), we put $A_{L}=A^{\mathscr{G}_L}$ (resp. $A_{K_{\pi}}=A^{\mathscr{G}_{K_{\pi}}}$). If $A$ is a perfect ring of characteristic $p,$ $A_{u}$ (or $A_{u-\np}$) is the partially unperfected subring of it, and we also put subscript $u$ for a Cohen ring of $A_u$ and also the fraction field of this Cohen ring (for example $ \widetilde{\EE}_u,$ $\widetilde{\AA}_{u}$ and $\widetilde{\BB}_u$). Notice that the above rule works also for double-subscripts, for example $\widetilde{\AA}_{u, L}$ is an unperfected version of $\widetilde{\AA}_{L},$ while the latter is $\widetilde{\AA}^{\mathscr{G}_L}.$  
\end{notation}

\begin{definition}($\cf$ \cite[\S 1.1]{Poy21}\label{def BB dagger r})
	For $r> 0,$ we define the set of \emph{overconvergent elements of $\widetilde{\BB}$ of radius $r$} by
	\[ \widetilde{\BB}^{\dagger, r}= \Big\{	\sum\limits_{n\gg-\infty}p^n[x_n]\in \widetilde{\BB}, \quad   \Lim{n \to + \infty}v^{\flat}(x_n) + \frac{pr}{p-1}n = +\infty \Big\} \]\index{$\widetilde{\BB}^{\dagger, r}$}
	and we denote $\widetilde{\BB}^{\dagger} =\bigcup\limits_{r> 0} \widetilde{\BB}^{\dagger, r}\subset\widetilde{\BB}$\index{$\widetilde{\BB}^{\dagger}$} the set of overconvergent elements (we have $r_1\leq r_2\Rightarrow\widetilde{\BB}^{\dagger,r_1}\subset\widetilde{\BB}^{\dagger,r_2}$). We put $\BB^\dagger=\BB\cap\,\widetilde{\BB}^{\dagger},$\index{$\BB^\dagger$} and similarly $\widetilde{\AA}^{\dagger}=\widetilde{\BB}^{\dagger}\cap\, \widetilde{\AA}$\index{$\widetilde{\AA}^{\dagger}$} and $\AA^{\dagger}=\widetilde{\AA}^{\dagger}\cap \AA.$\index{$\AA^{\dagger}$} We put $\Ee^{\ur, \dagger}:=\BB^{\dagger}$\index{$\Ee^{\ur, \dagger}$} and $\Oo_{\Ee^{\ur, \dagger}}:=\AA^{\dagger}.$\index{$\Oo_{\Ee^{\ur, \dagger}}$}
\end{definition}

\begin{remark}
	The subrings $\widetilde{\BB}^{\dagger,r}$ and $\widetilde{\BB}^\dagger$ of $\widetilde{\BB}$ are stable by $\mathscr{G}_K,$ and $\varphi(\widetilde{\BB}^{\dagger,r})\subset\widetilde{\BB}^{\dagger,pr}$ for all $r>0,$ so that $\widetilde{\BB}^\dagger$ is stable by $\varphi$ in $\widetilde{\BB}.$ Recall that $\widetilde{\BB}^\dagger$ and $\BB^\dagger$ are fields ($\cf$ \cite[Proposition 3.2]{Mat95}). 
\end{remark}

We recall that ($\cf$ definition \ref{Phi-tau-module}) for $V\in \Rep_{\QQ_p}(\mathscr{G}_K),$ the $(\varphi, \tau)$-module associated to $V$ over $(\BB_{K_{\pi}}, \widetilde{\BB}_{L})$ $(\ie  (\Ee, \Ee_{\tau}))$ is the $\varphi$-module 
\[\mathcal{D}(V)=(\BB\otimes_{\QQ_p}V)^{\mathscr{G}_{K_{\pi}}}\] and a semi-linear $\tau$-action over 
\[ \mathcal{D}(V)_{\tau}=(\widetilde{\BB}\otimes_{\QQ_p}V)^{\mathscr{G}_L}.\]

\begin{definition}($\cf$ \cite[D\'efinition 4.1.16]{Poy19}\label{def D dagger V})
	For $r>0,$ we define 
	\[\mathcal{D}^{\dagger, r}(V)_{\tau}=(\widetilde{\BB}^{\dagger, r}\otimes_{\QQ_p}V)^{\mathscr{G}_L}\]\index{$\mathcal{D}^{\dagger, r}(V)_{\tau}$} and \[\mathcal{D}^{\dagger,r}(V)=\mathcal{D}(V)\cap \mathcal{D}^{\dagger,r}(V)_{\tau}=(\mathbf{B}^{\dagger,r}\otimes_{\QQ_p}V)^{\mathscr{G}_{K_\pi}},\]\index{$\mathcal{D}^{\dagger,r}(V)$}
	where $\mathbf{B}^{\dagger,r}=\widetilde{\mathbf{B}}^{\dagger,r}\cap\BB.$\index{$\mathbf{B}^{\dagger,r}$} We put \[\mathcal{D}^{\dagger}(V)_{\tau}=(\widetilde{\BB}^{\dagger}\otimes_{\QQ_p}V)^{\mathscr{G}_L}=\bigcup\limits_{r>0}\mathcal{D}^{\dagger,r}(V)_{\tau}\]\index{$\mathcal{D}^{\dagger}(V)_{\tau}$} and \[ \mathcal{D}^{\dagger}(V)=(\BB^\dagger\otimes_{\QQ_p}V)^{\mathscr{G}_{K_\pi}}=\mathcal{D}(V)\cap \mathcal{D}^\dagger(V)_{\tau}=\bigcup\limits_{r>0}\mathcal{D}^{\dagger,r}(V).\]\index{$\mathcal{D}^{\dagger}(V)$}

	We say that a $(\varphi, \tau)$-module $D$ associated to $V$ is \emph{overconvergent} if there exists $r>0$ such that we have
	\[ \mathcal{D}(V)=\BB_{K_{\pi}}\otimes_{\BB_{K_{\pi}}^{\dagger, r}}\mathcal{D}^{\dagger, r}(V) \]
	and 
	\[ \mathcal{D}(V)_{\tau}= \widetilde{\BB}_{L}\otimes_{\widetilde{\BB}_{L}^{\dagger, r}} \mathcal{D}^{\dagger, r}(V)_{\tau} \]
	where $\BB_{K_{\pi}}^{\dagger,r}=\BB_{K_{\pi}}\cap\, \widetilde{\BB}^{\dagger, r}$\index{$\BB_{K_{\pi}}^{\dagger,r}$} and $\widetilde{\BB}_{L}^{\dagger, r}=\widetilde{\BB}_{L}\cap \widetilde{\BB}^{\dagger, r}.$\index{$\widetilde{\BB}_{L}^{\dagger, r}$} 
	
	\medskip
	
	When this holds, we have in particular 
	\[\mathcal{D}(V)=\BB_{K_{\pi}}\otimes_{\BB_{K_{\pi}}^\dagger}\mathcal{D}^\dagger(V)\]
	and
	\[\mathcal{D}(V)_{\tau}=\widetilde{\BB}_{L}\otimes_{\widetilde{\BB}_{L}^{\dagger}}\mathcal{D}^\dagger(V)_{\tau},\]
	where $\BB_{K_{\pi}}^\dagger=\bigcup\limits_{r>0}\BB_{K_{\pi}}^{\dagger,r}$\index{$\Ee^{\dagger}$} \index{$\BB_{K_{\pi}}^\dagger$} and $\widetilde{\BB}_{L}^{\dagger}=\bigcup\limits_{r>0}\widetilde{\BB}_{L}^{\dagger, r}.$ \index{$\Ee_{\tau}^{\dagger}$} \index{$\widetilde{\BB}_{L}^{\dagger}$} We put $\Ee^{\dagger}:=\BB_{K_{\pi}}^\dagger$ and $\Ee_{\tau}^{\dagger}:=\widetilde{\BB}_{L}^{\dagger}$ to match to the notations used in the first three chapters ($\cf$ notation \ref{not dictionary}).
\end{definition}

\begin{remark}
	(1)	As we mentioned, our notations are slightly different from those of Poyeton in \cite{Poy19} as we do not work with $(\varphi,\Gamma)$-modules: we denote by $\mathcal{D}(V)$ and $\mathcal{D}^\dagger(V)$ for what are denoted by $\mathcal{D}_\tau(V)$ and $\mathcal{D}^\dagger_\tau(V)$ in \it{loc. cit.}, while our $\mathcal{D}(V)_{\tau}$ comes from the pair $(\mathcal{D}(V), \mathcal{D}(V)_{\tau}),$ which is part of the data of a $(\varphi, \tau)$-module.
	\item (2) We have $\mathcal{D}^{\dagger}(V)_{\tau} =\widetilde{\BB}_{L}^{\dagger}\otimes \mathcal{D}^{\dagger}(V) \simeq (\widetilde{\BB}^{\dagger}\otimes V)^{\mathscr{G}_L}.$ 
\end{remark}

\begin{theorem}\label{thm surconvergence phi-tau version}
	Let $V$ be a $\QQ_p$-representation of $\mathscr{G}_K.$ Then the associated $(\varphi, \tau)$-module over $(\BB_{K_{\pi}}, \widetilde{\BB}_{L})$ is overconvergent.
\end{theorem}
\begin{proof}
	$\cf$ \cite[Th\'eor\`eme 4.3.29]{Poy19}. 
\end{proof}

\begin{definition}\label{def overconvergent phi tau}
	A \emph{$(\varphi, \tau)$-module over $(\BB_{K_{\pi}}^\dagger,\widetilde{\BB}_{L}^\dagger)$} consists of
	\item (i) an \'etale $\varphi$-module $D^{\dagger}$ over $\BB_{K_{\pi}}^\dagger$;
	\item (ii) a $\tau$-semi-linear endomorphism $\tau_D$ on $D^{\dagger}_\tau:=\widetilde{\BB}_{L}^\dagger\otimes_{\BB_{K_{\pi}}^\dagger}D^{\dagger}$ which commutes with $\varphi_{\widetilde{\BB}_{L}^\dagger}\otimes \varphi_D$ (where $\varphi_{\widetilde{\BB}_{L}^\dagger}$ is the Frobenius map on $\widetilde{\BB}_{L}^\dagger$ and $\varphi_D$ is the Frobenius map on $D$) and that satisfies
	\[ (\forall x\in D^{\dagger})\ (g\otimes 1)\circ \tau_D(x) = \tau_D^{\chi(g)}(x),  \]
	for all $g\in \mathscr{G}_{K_{\pi}}/\mathscr{G}_L$ such that $\chi(g)\in \ZZ_{>0}.$ The corresponding category is denoted $\Mod_{\BB_{K_{\pi}}^\dagger,\widetilde{\BB}_{L}^\dagger}(\varphi,\tau).$\index{$\Mod_{\BB_{K_{\pi}}^\dagger,\widetilde{\BB}_{L}^\dagger}(\varphi,\tau)$}
	
	\medskip
	
	One defines similarly the notion of $(\varphi,\tau)$-module over $(\AA_{K_{\pi}}^\dagger,\widetilde{\AA}_{L}^\dagger)=:(\Oo_{\Ee^{\dagger}}, \Oo_{\Ee_{\tau}^{\dagger}}),$\index{$\Oo_{\Ee^{\dagger}}$} \index{$\Oo_{\Ee_{\tau}^{\dagger}}$} and the corresponding category is denoted $\Mod_{\AA_{K_{\pi}}^\dagger,\widetilde{\AA}_{L}^\dagger}(\varphi,\tau).$\index{$\Mod_{\AA_{K_{\pi}}^\dagger,\widetilde{\AA}_{L}^\dagger}(\varphi,\tau)$} 
\end{definition}

\begin{theorem}\label{prop equivalence cat overconvergent: normal case}
	The functors
	\begin{align*}
	\mathcal{D}^\dagger\colon\Rep_{\QQ_p}(\mathscr{G}_K)&\simeq\Mod_{\BB_{K_{\pi}}^\dagger,\widetilde{\BB}_{L}^\dagger}(\varphi,\tau)\\
	V&\mapsto \mathcal{D}^{\dagger}(V)=(\BB^\dagger\otimes_{\QQ_p}V)^{\mathscr{G}_{K_\pi}}\\
	\mathcal{V}(D^{\dagger})=(\BB^{\dagger} \otimes_{\BB_{K_{\pi}}^\dagger} D^{\dagger})^{\varphi=1}   & \mapsfrom D^{\dagger}
	\end{align*}
	$($with the natural $\tau$-semi-linear endomorphism $\tau_D$ over $\mathcal{D}^{\dagger}(V)_{\tau}=\widetilde{\BB}_{L}^{\dagger}\otimes_{\BB_{K_{\pi}}^\dagger} \mathcal{D}^{\dagger}(V))$ establish quasi-inverse equivalences of categories, which are refinements of the equivalences of \cite[Th\'eor\`eme 1.14]{Car13}.
\end{theorem}

\begin{proof}
	Recall that we have the category equivalence
	\[ \mathcal{D}\colon \Rep_{\QQ_p}(\mathscr{G}_K)\simeq \Mod_{\BB_{K_{\pi}},\widetilde{\BB}_{L}}(\varphi,\tau). \]
	For any $V_1, V_2\in \Rep_{\QQ_p}(\mathscr{G}_K),$ we have $\Hom_{\Rep_{\QQ_p}(\mathscr{G}_K)}(V_1, V_2)\simeq\Hom_{\Mod_{\BB_{K_{\pi}},\widetilde{\BB}_{L}}(\varphi,\tau)}(\mathcal{D}(V_1), \mathcal{D}(V_2))$ by \cite[Th\'eor\`eme 1.14]{Car13}. By theorem \ref{thm surconvergence phi-tau version} we have  $$\Hom_{\Mod_{\BB_{K_{\pi}},\widetilde{\BB}_{L}}(\varphi,\tau)}(\mathcal{D}(V_1),\mathcal{D}(V_2))=\Hom_{\Mod_{\BB_{K_{\pi}},\widetilde{\BB}_{L}}(\varphi,\tau)}(\BB_{K_{\pi}}\otimes_{\BB_{K_{\pi}}^{\dagger}}\mathcal{D}^{\dagger}(V_1), \BB_{K_{\pi}}\otimes_{\BB_{K_{\pi}}^{\dagger}}\mathcal{D}^{\dagger}(V_2)).$$
	We prove this is isomorphic to $\Hom_{\Mod_{\BB_{K_{\pi}}^\dagger,\widetilde{\BB}_{L}^\dagger}(\varphi,\tau)}(\mathcal{D}^{\dagger}(V_1), \mathcal{D}^{\dagger}(V_2)),$ whence the fully-faithfullness. We have the commutative diagram
	
	\[ \xymatrix{
		\Hom_{\Rep_{\QQ_p}(\mathscr{G}_K)}(V_1, V_2) \ar[rr]_{\sim}^{\mathcal{D}} \ar@{^(->}[rdd]^{\mathcal{D}^{\dagger}}&& \Hom_{\Mod_{\BB_{K_{\pi}},\widetilde{\BB}_{L}}(\varphi,\tau)}(\mathcal{D}(V_1),\mathcal{D}(V_2))\\
		&&\\
		& \Hom_{\Mod_{\BB_{K_{\pi}}^\dagger,\widetilde{\BB}_{L}^\dagger}(\varphi,\tau)}(\mathcal{D}^{\dagger}(V_1), \mathcal{D}^{\dagger}(V_2)) \ar@{->>}[ruu]^{(\ast)}_{\otimes\BB_{K_{\pi}}, \, \otimes\widetilde{\BB}_{L}}  &
	}   \]
	It is enough to show the map $(\ast)$ is injective, and hence an isomorphism. For any map $f\colon D_1^{\dagger} \to D_2^{\dagger}$ in $\Mod_{\BB_{K_{\pi}}^\dagger,\widetilde{\BB}_{L}^\dagger}(\varphi,\tau),$ we have the commutative square:
	
	\[ \xymatrix{
		D_1^{\dagger} \ar[r]^{f} \ar@{_(->}[d] & D_2{^\dagger} \ar@{_(->}[d]\\
		\BB_{K_{\pi}}\otimes_{\BB_{K_{\pi}}^{\dagger}}D_1^{\dagger} \ar[r]^{1\otimes f} & \BB_{K_{\pi}}\otimes_{\BB_{K_{\pi}}^{\dagger}}D^{\dagger}_2
	}  \]
	which shows that $1\otimes f=0$ implies that $f=0.$
	
	\medskip
	
	Let $(D^{\dagger}, D_{\tau}^{\dagger})\in \Mod_{\BB_{K_{\pi}}^\dagger,\widetilde{\BB}_{L}^\dagger}(\varphi,\tau).$ Tensoring with $\BB_{K_{\pi}}$ and $\widetilde{\BB}_{L}$ respectively, we have
	\[(\BB_{K_{\pi}}\otimes_{\BB_{K_{\pi}}^{\dagger}} D^{\dagger}, \widetilde{\BB}_{L} \otimes_{\widetilde{\BB}_{L}^{\dagger}} D_{\tau}^{\dagger})\in \Mod_{\BB_{K_{\pi}},\widetilde{\BB}_{L}}(\varphi,\tau)\]
	by theorem \ref{thm surconvergence phi-tau version}. By the equivalence induced by the functor $\mathcal{D},$ the $(\varphi, \tau)$ module $(\BB_{K_{\pi}}\otimes_{\BB_{K_{\pi}}^{\dagger}} D^{\dagger}, \widetilde{\BB}_{L} \otimes_{\widetilde{\BB}_{L}^{\dagger}} D_{\tau}^{\dagger})$ corresponds to a $p$-adic representation $V\in \Rep_{\QQ_p}(\mathscr{G}_K).$ We show that $\mathcal{D}^{\dagger}(V)$ (more precisely $(\mathcal{D}^{\dagger}(V), \mathcal{D}^{\dagger}(V)_{\tau})$) equals the original object $(D^{\dagger}, D_{\tau}^{\dagger})\in \Mod_{\BB_{K_{\pi}}^\dagger,\widetilde{\BB}_{L}^\dagger}(\varphi,\tau).$ Indeed, $\mathcal{D}^{\dagger}(V)$ and $(D^{\dagger}, D_{\tau}^{\dagger})$ map to the same $(\varphi, \tau)$-module in $\Mod_{\BB_{K_{\pi}},\widetilde{\BB}_{L}}(\varphi,\tau),$ and hence they are isomorphic in $\Mod_{\BB_{K_{\pi}}^\dagger,\widetilde{\BB}_{L}^\dagger}(\varphi,\tau)$ as the map $(\ast)$ in the diagram above is an isomorphism. This shows that $\mathcal{D}^{\dagger}$ is essentially surjective and hence estabilishes the equivalence of categories.
\end{proof}

\begin{proposition}\label{prop Zp surconvergence}
	Let $T\in \Rep_{\ZZ_p}(\mathscr{G}_K)$, then its $(\varphi, \tau)$-module $(D, D_{\tau})\in \Mod_{\AA_{K_{\pi}},\widetilde{\AA}_{L}}(\varphi,\tau)$ is overconvergent.
\end{proposition}
\begin{proof}
	\begin{enumerate}
		\item[(1)]	We first suppose $T$ is $p$-torsion free. Recall that $\mathcal{D}^{\dagger}(T)=(\AA^{\dagger}\otimes_{\ZZ_p}T)^{\mathscr{G}_{K_{\pi}}}$ and we have a natural map
		\[ \AA_{K_{\pi}}\otimes_{\AA_{K_{\pi}}^{\dagger}} \mathcal{D}^{\dagger}(T) \xrightarrow{f} \mathcal{D}(T), \]
		such that tensoring $\otimes_{\ZZ_p} \QQ_p$ is an isomorphism ($\cf$ \cite[Th\'eor\`eme 4.3.29]{Poy19}). This implies that $f$ is injective and $\mathcal{D}^{\dagger}(T)$ is free of rank $d:=\rank_{\ZZ_p}T$ over $\AA_{K_{\pi}}^{\dagger}.$ In particular, there are integers $(n_i)_{1\leq i \leq d}$ such that 
		\[ \Coker(f)\simeq \bigoplus_{i=1}^d \AA_{K_{\pi}}/(p^{n_i}). \]
		We have the commutative diagram 
		\[ \xymatrix{
			0\ar[r] & \AA_{K_{\pi}}\otimes_{\AA_{K_{\pi}}^{\dagger}}\mathcal{D}^{\dagger}(T) \ar[r]^{p} \ar[d]^{f}& \AA_{K_{\pi}}\otimes_{\AA_{K_{\pi}}^{\dagger}}\mathcal{D}^{\dagger}(T) \ar[r] \ar[d]^{f} & \AA_{K_{\pi}}\otimes_{\AA_{K_{\pi}}^{\dagger}}\mathcal{D}^{\dagger}(T/pT)\ar[d]^{\sim}\ar[r] & 0\\
			0 \ar[r] & \mathcal{D}(T) \ar[r]^{p} & \mathcal{D}(T) \ar[r] & \mathcal{D}(T/pT)\ar[r] & 0
		} \] 
		which shows that multiplication by $p$ is bijective over $\Coker(f)$ and hence bijective over $\AA_{K_{\pi}}/(p^{n_i})$ for $i\in \{ 1, \ldots, d \}.$ This implies $n_i=0$ for all $i\in \{ 1, \ldots, d \}$, and hence $\Coker(f)=0$. Thus
		\[ \AA_{K_{\pi}}\otimes_{\AA_{K_{\pi}}^{\dagger}} \mathcal{D}^{\dagger}(T) \xrightarrow{\simeq} \mathcal{D}(T) \]		
		is an isomorphism and the $\varphi$-module $\mathcal{D}^{\dagger}(T)$ is overconvergent. Recall that $\mathcal{D}(T)_{\tau}=\widetilde{\AA}_{L}\otimes_{\AA_{K_{\pi}}}\mathcal{D}(T)$ and  $\mathcal{D}^{\dagger}(T)_{\tau}=\widetilde{\AA}^{\dagger}_{L}\otimes_{\AA_{K_{\pi}}^{\dagger}}\mathcal{D}^{\dagger}(T):$ we have \[ \mathcal{D}(T)_{\tau}=\widetilde{\AA}_{L}\otimes_{\widetilde{\AA}^{\dagger}_{L}}\mathcal{D}^{\dagger}(T)_{\tau}. \]
		\item [(2)] In general, for any $T\in \Rep_{\ZZ_p}(\mathscr{G}_K)$, we have the short exact sequence:
		\[ 0 \to T^{\prime} \to T \to T^{\prime\prime}\to 0,   \]
		where $T^{\prime}$  is the $p$-torsion submodule of $T$ and $T^{\prime\prime}=T/T^{\prime}$ is $p$-torsion free. Notice that $\mathcal{D}^{\dagger}(T^{\prime})=\mathcal{D}(T^{\prime})$, and we have the exact sequence
		\[ 0\to \mathcal{D}(T^{\prime}) \to \mathcal{D}^{\dagger}(T) \to \mathcal{D}^{\dagger}(T^{\prime\prime}) \to \H^1(\mathscr{G}_{K_{\pi}}, \AA\otimes T^{\prime})=0,   \]
		where the last equality follows from Hilbert 90 ($\AA$ is endowed with $p$-adic topology). We thus have the commutative diagram:
		\[ \xymatrix{
			0\ar[r] & \mathcal{D}(T^{\prime}) \ar[r] \ar@{=}[d]& \AA_{K_{\pi}}\otimes_{\AA_{K_{\pi}}^{\dagger}}\mathcal{D}^{\dagger}(T) \ar[r] \ar[d]^{f} & \AA_{K_{\pi}}\otimes_{\AA_{K_{\pi}}^{\dagger}}\mathcal{D}^{\dagger}(T^{\prime\prime})\ar[d]^{\sim}\ar[r] & 0\\
			0 \ar[r] &\mathcal{D}(T^{\prime}) \ar[r] & \mathcal{D}(T) \ar[r] & \mathcal{D}(T^{\prime\prime})\ar[r] & 0
		} \] 
		which shows that $f$ is an isomorphism.
	\end{enumerate}		 
\end{proof}

\begin{corollary}
	The functors
	\begin{align*}
	\mathcal{D}^\dagger\colon\Rep_{\ZZ_p}(\mathscr{G}_K)&\simeq\Mod_{\AA_{K_{\pi}}^\dagger,\widetilde{\AA}_{L}^\dagger}(\varphi,\tau)\\
	T&\mapsto \mathcal{D}^{\dagger}(T)=(\AA^\dagger\otimes_{\ZZ_p}T)^{\mathscr{G}_{K_\pi}}\\
	\mathcal{V}(D^{\dagger})=(\AA^{\dagger} \otimes_{\AA_{K_{\pi}}^\dagger} D^{\dagger})^{\varphi=1}   & \mapsfrom D^{\dagger}
	\end{align*}
	$($with the natural $\tau$-semi-linear endomorphism $\tau_D$ over $\mathcal{D}^{\dagger}(T)_{\tau}=\widetilde{\AA}_{L}^{\dagger}\otimes_{\AA_{K_{\pi}}^\dagger} \mathcal{D}^{\dagger}(T))$ establish quasi-inverse equivalences of categories.
\end{corollary}
\begin{proof}
	The proof is almost the same as that of theorem \ref{prop equivalence cat overconvergent: normal case}, with the following modifications. For any map $f\colon D_1^{\dagger} \to D_2^{\dagger}$ in $\Mod_{\AA_{K_{\pi}}^\dagger,\widetilde{\AA}_{L}^\dagger}(\varphi,\tau)$ we consider the commutative square:
	\[ \xymatrix{
		D_1^{\dagger} \ar[r]^{f} \ar[d] & D_2{^\dagger} \ar[d]\\
		\AA_{K_{\pi}}\otimes_{\AA_{K_{\pi}}^{\dagger}}D_1^{\dagger} \ar[r]^{1\otimes f} & \AA_{K_{\pi}}\otimes_{\AA_{K_{\pi}}^{\dagger}}D^{\dagger}_2.
	}  \]
	As $\AA_{K_{\pi}}$ is faithfully flat over $\AA_{K_{\pi}}^{\dagger}$ and $\widetilde{\AA}_{L}$ is faithfully flat over $\widetilde{\AA}_{L}^\dagger$ ($\cf$ \cite[Theorem 7.2 (3)]{Mat89}), $1\otimes f=0$ implies $f=0.$ For essential surjectivity, we use proposition \ref{prop Zp surconvergence} instead of theorem \ref{thm surconvergence phi-tau version}.
\end{proof}

\begin{notation}
	Let $(D^{\dagger}, D^{\dagger}_{\tau})\in \Mod_{\BB_{K_{\pi}}^\dagger,\widetilde{\BB}_{L}^\dagger}(\varphi,\tau).$ Put 
	\[ D^{\dagger}_{\tau, 0}:=\big\{ x\in D^{\dagger}_{\tau};\ (\forall g\in \mathscr{G}_{K_{\pi}})\ \chi(g)\in \ZZ_{>0} \Rightarrow (g\otimes 1)(x)=\big(1+\tau_D+\cdots +\tau_D^{\chi(g)-1}\big)(x) \big\}. \]\index{$D^{\dagger}_{\tau, 0}$}
	By similar arguments as that of lemma \ref{Replace g witi gamma}, we see that 
	
	\[D^{\dagger}_{\tau, 0}:=\big\{ x\in D^{\dagger}_{\tau} ;\  (\gamma\otimes 1)x=\big(1+\tau_D+\cdots +\tau_D^{\chi(g)-1}\big)(x)  \big\}.\]
\end{notation}

\subsection{Locally analytic vectors with partially unperfected coefficients}

\begin{definition}($\cf$ \cite[\S 1.1]{Poy21})
	We define for $r> 0,$ $\widetilde{\BB}_u^{\dagger, r}=\widetilde{\BB}^{\dagger, r}\cap\, \widetilde{\BB}_u$\index{$\widetilde{\BB}_u^{\dagger, r}$} and we denote $\widetilde{\BB}_u^{\dagger}=\bigcup\limits_{r>0} \widetilde{\BB}_u^{\dagger, r}\subset\widetilde{\BB}_u$\index{$\widetilde{\BB}_u^{\dagger}$} the set of overconvergent elements (we have $r_1\leq r_2\Rightarrow\widetilde{\BB}_u^{\dagger,r_1}\subset\widetilde{\BB}_u^{\dagger,r_2}$). We put $\BB^\dagger_u=\BB\cap\widetilde{\BB}_u^{\dagger}.$\index{$\BB^\dagger_u$}
\end{definition}

\begin{remark}\label{rem dagger has a GK action}
	The subrings $\widetilde{\BB}_u^{\dagger,r}$ and $\widetilde{\BB}_u^\dagger$ of $\widetilde{\BB}_u$ are stable by $\mathscr{G}_K,$ and $\varphi(\widetilde{\BB}_u^{\dagger,r})\subset\widetilde{\BB}_u^{\dagger,pr}$ for all $r>0,$ so that $\widetilde{\BB}_u^\dagger$ is stable by $\varphi$ in $\widetilde{\BB}_u.$ Remark that $\widetilde{\BB}_u^\dagger$ and $\BB^\dagger_u$ are fields ($\cf$ \cite[Proposition 3.2]{Mat95}).
\end{remark}

Recall that (\cf remark \ref{rem 3.1.16 e e u tau}) if $V$ is a $\QQ_p$-representation of $\mathscr{G}_K,$ then its associated $(\varphi, \tau)$-module over $(\BB_{K_{\pi}}, \widetilde{\BB}_{u, L})$ $(\ie  (\Ee, \Ee_{u,\tau}))$ is the $\varphi$-module $\mathcal{D}(V)=(\BB\otimes_{\QQ_p}V)^{\mathscr{G}_{K_{\pi}}}$ together with a $\tau$-semi-linear action over $\mathcal{D}(V)_{u,\tau}=\widetilde{\BB}_{u, L}\otimes_{\BB_{K_{\pi}}}\mathcal{D}(V).$

\begin{definition}($\cf$ \cite[D\'efinition 4.1.16]{Poy19})
	
	For $r>0,$ we define 
	\[\mathcal{D}^{\dagger, r}(V)_{u,\tau}=(\widetilde{\BB}_u^{\dagger, r}\otimes_{\QQ_p}V)^{\mathscr{G}_L}\]\index{$\mathcal{D}^{\dagger, r}(V)_{u,\tau}$}
	and 
	\[\mathcal{D}^{\dagger,r}(V)=\mathcal{D}(V)\cap \mathcal{D}^{\dagger,r}(V)_{u,\tau}=(\mathbf{B}^{\dagger,r}\otimes_{\QQ_p}V)^{\mathscr{G}_{K_\pi}},\]\index{$\mathcal{D}^{\dagger,r}(V)$}
	observing that $\mathbf{B}^{\dagger,r}=\widetilde{\mathbf{B}}_u^{\dagger,r}\cap\BB.$\index{$\mathbf{B}^{\dagger,r}$} We put \[\mathcal{D}^{\dagger}(V)_{u,\tau}=(\widetilde{\BB}_u^{\dagger}\otimes_{\QQ_p}V)^{\mathscr{G}_L}=\bigcup\limits_{r>0}\mathcal{D}^{\dagger,r}(V)_{u, \tau}\]\index{$\mathcal{D}^{\dagger}(V)_{u,\tau}$}
	and 
	\[\mathcal{D}^{\dagger}(V)=(\BB^\dagger\otimes_{\QQ_p}V)^{\mathscr{G}_{K_\pi}}=\mathcal{D}(V)\cap \mathcal{D}^\dagger(V)_{u,\tau}=\bigcup\limits_{r>0}\mathcal{D}^{\dagger,r}(V).\]\index{$\mathcal{D}^{\dagger}(V)$}

	We say that a $(\varphi, \tau)$-module $D$ associated to $V$ is \emph{overconvergent} if there exists $r>0$ such that we have
	\[ \mathcal{D}(V)=\BB_{K_{\pi}}\otimes_{\BB_{K_{\pi}}^{\dagger, r}}\mathcal{D}^{\dagger, r}(V) \]
	and 
	\[ \mathcal{D}(V)_{u,\tau}= \widetilde{\BB}_{u, L}\otimes_{\widetilde{\BB}_{u, L}^{\dagger, r}} \mathcal{D}^{\dagger, r}(V)_{u,\tau} \]
	where $\widetilde{\BB}_{u, L}^{\dagger, r}=\widetilde{\BB}_{u, L}\cap \widetilde{\BB}_u^{\dagger, r}.$\index{$\widetilde{\BB}_{u, L}^{\dagger, r}$} 
	
	\medskip
	
	When this holds, we have in particular 
	\[\mathcal{D}(V)=\BB_{K_{\pi}}\otimes_{\BB_{K_{\pi}}^\dagger}\mathcal{D}^\dagger(V)\] and
	\[\mathcal{D}(V)_{u,\tau}=\widetilde{\BB}_{u, L}\otimes_{\widetilde{\BB}_{u, L}^{\dagger}}\mathcal{D}^\dagger(V)_{u,\tau},\]
	where $\widetilde{\BB}_{u, L}^{\dagger}=\bigcup\limits_{r>0}\widetilde{\BB}_{u, L}^{\dagger, r}.$\index{$\widetilde{\BB}_{u, L}^{\dagger}$} 
\end{definition}

\begin{remark}
	(1) We have $\mathcal{D}^{\dagger}(V)_{u,\tau} = (\widetilde{\BB}_u^{\dagger}\otimes_{\QQ_p} V)^{\mathscr{G}_L}\simeq \widetilde{\BB}_{u, L}^{\dagger}\otimes_{\BB_{K_{\pi}}^{\dagger}}\mathcal{D}^{\dagger}(V).$
	\item (2) Notice that for any $r>0$ we have:
	$$\mathbf{B}^{\dagger,r}=\widetilde{\mathbf{B}}_u^{\dagger,r}\cap\BB=\widetilde{\mathbf{B}}^{\dagger,r}\cap\BB,$$
	and hence $$\mathbf{B}^{\dagger}=\widetilde{\mathbf{B}}_u^{\dagger}\cap\BB=\widetilde{\mathbf{B}}^{\dagger}\cap\BB.$$
	This is simply because we have
	\[ \BB\subset \widetilde{\mathbf{B}}_u \subset \widetilde{\mathbf{B}}. \]
\end{remark}

\begin{definition}\label{def non-perfect phi tau}
	
	A \emph{$(\varphi, \tau)$-module over} $(\BB_{K_{\pi}}^\dagger,\widetilde{\BB}_{u, L}^\dagger)=:(\Ee^{\dagger}, \Ee_{u, \tau}^{\dagger})$ consists of
	\item (i) an \'etale $\varphi$-module $D^{\dagger}$ over $\BB_{K_{\pi}}^\dagger$;
	\item (ii) a $\tau$-semi-linear endomorphism $\tau_D$ on $D^{\dagger}_{u,\tau}:=\widetilde{\BB}_{u, L}^\dagger\otimes_{\BB_{K_{\pi}}^\dagger}D^{\dagger}$ which commutes with $\varphi_{\widetilde{\BB}_{u,L}^\dagger}\otimes \varphi_{D^{\dagger}}$ (where $\varphi_{\widetilde{\BB}_{u,L}^\dagger}$ is the Frobenius map on $\widetilde{\BB}_{u,L}^\dagger$ and $\varphi_{D^{\dagger}}$ is the Frobenius map on $D^{\dagger}$) and which satisfies
	\[ (\forall x\in D^{\dagger})\ (g\otimes 1)\circ \tau_D(x) = \tau_D^{\chi(g)}(x),  \]
	for all $g\in \mathscr{G}_{K_{\pi}}/\mathscr{G}_L$ such that $\chi(g)\in \ZZ_{>0}.$ The corresponding category is denoted $\Mod_{{\BB_{K_{\pi}}^\dagger,\widetilde{\BB}_{u, L}^\dagger}}(\varphi,\tau).$\index{$\Mod_{{\BB_{K_{\pi}}^\dagger,\widetilde{\BB}_{u, L}^\dagger}}(\varphi,\tau)$}
	
	\medskip
	
	One defines similarly the notion of $(\varphi,\tau)$-module over $(\AA_{K_{\pi}}^\dagger,\widetilde{\AA}_{u,L}^\dagger)=:(\Oo_{\Ee^{\dagger}}, \Oo_{\Ee_{u, \tau}^{\dagger}}),$ and the corresponding category is denoted $\Mod_{\AA_{K}^\dagger,\widetilde{\AA}_{u, L}^\dagger}(\varphi,\tau).$\index{$\Mod_{\AA_{K}^\dagger,\widetilde{\AA}_{u, L}^\dagger}(\varphi,\tau)$} 
	
\end{definition}

\subsection{The overconvergence for \texorpdfstring{$(\varphi, \tau)$}{(varphi, tau)}-modules over partially unperfected coefficients}

For $(\varphi, \tau)$-modules over $(\Ee, \Ee_{\tau}),$ we have seen the overconvergence result in theorem \ref{thm surconvergence phi-tau version}. In this subsection we will prove the overconvergence result for $(\varphi, \tau)$-modules over $(\BB_{K_{\pi}},\widetilde{\BB}_{u,L}) (\ie (\Ee, \Ee_{u,\tau})).$

\begin{theorem}
	\label{thm overconvergence with non-perfect rings}
	Let $V\in \Rep_{\QQ_p}(\mathscr{G}_K),$ then its $(\varphi, \tau)$-module $(D, D_{u, \tau})\in \Mod_{\BB_{K_{\pi}},\widetilde{\BB}_{u,L}}(\varphi,\tau)$ is overconvergent.
\end{theorem}

\begin{proof}
	Let $(D, D_{u,\tau})\in \Mod_{\Ee, \Ee_{u,\tau}}(\varphi, \tau)$ be the $(\varphi, \tau)$-module over $(\Ee, \Ee_{u,\tau})$ associated to $V\in \Rep_{\QQ_p}(\mathscr{G}_K).$ By definition, it suffices to show that 	
	\begin{equation}\label{equ surconvergence varphi}
	D=\Ee\otimes_{\Ee^{\dagger}}D^{\dagger}
	\end{equation} 
	and 
	\begin{equation}\label{equ surconvergence varphi tau}
	D_{u,\tau}= \Ee_{u, \tau}\otimes_{\Ee_{u, \tau}^{\dagger}} D^{\dagger}_{u, \tau}.
	\end{equation}
	
	The first equality (\ref{equ surconvergence varphi}) follows directly from theorem \ref{thm surconvergence phi-tau version}. For the second equality (\ref{equ surconvergence varphi tau}), it suffices to show that the following map is an isomorphism
	\begin{equation}\label{equ surconvergence f}
	f \colon \Ee^{\dagger}_{u,\tau}\otimes_{\Ee^{\dagger}} D^{\dagger}  \to D^{\dagger}_{u, \tau}.
	\end{equation}
	Indeed, this impies that 
	\[ \Ee_{u,\tau}\otimes_{\Ee^{\dagger}}D^{\dagger} \simeq \Ee_{u,\tau}\otimes_{\Ee_{u,\tau}^{\dagger}}D_{u,\tau}^{\dagger}, \]
	$\ie$ 
	\[ \Ee_{u,\tau}\otimes_{\Ee}(\Ee\otimes_{\Ee^{\dagger}}D^{\dagger}) \simeq \Ee_{u,\tau}\otimes_{\Ee_{u,\tau}^{\dagger}}D_{u,\tau}^{\dagger}. \]
	By equality (\ref{equ surconvergence varphi}) the left hand side is exactly $\Ee_{u,\tau}\otimes_{\Ee}D=D_{u,\tau},$ hence this gives equality (\ref{equ surconvergence varphi tau}).
	To prove that (\ref{equ surconvergence f}) is an isomorphism, it suffices to prove it is so after tensoring $\Ee^{\dagger}_{\tau}\otimes_{\Ee^{\dagger}_{u,\tau}}, \ie$that
	\[\Ee^{\dagger}_{\tau}\otimes_{\Ee^{\dagger}_{u,\tau}} f\colon \Ee^{\dagger}_{\tau}\otimes_{\Ee^{\dagger}} D^{\dagger}\to \Ee_{\tau}^{\dagger}\otimes_{\Ee^{\dagger}_{u,\tau}}D^{\dagger}_{u, \tau} \] is an isomorphic. We consider the commutative diagram:
	
	\[\xymatrix{ \Ee^{\dagger}_{\tau}\otimes_{\Ee^{\dagger}} D^{\dagger} \ar[rr]^{\Ee^{\dagger}_{\tau}\otimes f}\ar[rd]_{\text{ theorem } \ref{thm surconvergence phi-tau version}}^{\simeq} && \Ee_{\tau}^{\dagger}\otimes_{\Ee^{\dagger}_{u,\tau}}D^{\dagger}_{u, \tau}\ar[ld] \\
		&D_{\tau}^{\dagger}&	}	\]
	Notice that this is a digram of linear-transformations of $\Ee^{\dagger}_{\tau}$-vector spaces: it suffices to use dimensions. Notice that by theorem \ref{thm surconvergence phi-tau version} the arrow on the left is an isomorphism, and
	\[\dim \Ee_{\tau}^{\dagger}\otimes_{\Ee^{\dagger}_{u,\tau}}D^{\dagger}_{u, \tau}\leq \dim \Ee^{\dagger}_{\tau}\otimes D^{\dagger}. \] We hence conclude that $\Ee^{\dagger}_{\tau}\otimes f$ is an isomoprhism.
\end{proof}

\begin{proposition}
	Let $T\in \Rep_{\ZZ_p}(\mathscr{G}_K)$, then its $(\varphi, \tau)$-module $(D, D_{u,\tau})\in \Mod_{\AA_{K_{\pi}},\widetilde{\AA}_{u, L}}(\varphi,\tau)$ is overconvergent.
\end{proposition}
\begin{proof}
	A similar proof as that of proposition \ref{prop Zp surconvergence} works.
\end{proof}

\begin{remark}
	One might also prove proposition \ref{prop Zp surconvergence} using similar arguments as Gao and Poyeton ($\cf$ \cite{GP21}), based on Tate-Sen's method ($\cf$ \cite[\S 3]{BC08}).
\end{remark}

\subsection{The categorical equivalence}

\begin{proposition}\label{prop equivalence cat}
	The functors 	
	\begin{align*}
	\mathcal{D}^\dagger\colon\Rep_{\QQ_p}(\mathscr{G}_K)&\simeq\Mod_{\BB_{K_{\pi}}^\dagger,\widetilde{\BB}_{u, L}^\dagger}(\varphi,\tau)	\\
	V&\mapsto \mathcal{D}^{\dagger}(V)=(\BB^\dagger\otimes_{\QQ_p}V)^{\mathscr{G}_{K_\pi}}\\
	\mathcal{V}(D^{\dagger})=(\BB^{\dagger} \otimes_{\BB_{K_{\pi}}^\dagger} D^{\dagger})^{\varphi=1}   & \mapsfrom D^{\dagger}
	\end{align*}
	$($with the $\tau$-semi-linear endomorphism $\tau_D=\tau\otimes \tau$ over $\mathcal{D}^{\dagger}(V)_{u,\tau}=(\widetilde{\BB}_u^{\dagger}\otimes_{\QQ_p}V)^{\mathscr{G}_L})$ establish quasi-inverse equivalences of categories, which refines the equivalence of \cite[Th\'eor\`eme 1.14]{Car13}. 
\end{proposition}
\begin{proof}
	The same proof as that of theorem \ref{prop equivalence cat overconvergent: normal case} works, except that for the overconvergence result we have to use theorem \ref{thm overconvergence with non-perfect rings}.
\end{proof}

\begin{notation}
	Let $(D^{\dagger}, D_{u,\tau}^{\dagger})\in \Mod_{{\BB_{K_{\pi}}^\dagger,\widetilde{\BB}_{u, L}^\dagger}}(\varphi,\tau).$ We put 
	\[D^{\dagger}_{u, \tau, 0}:=\big\{ x\in D^{\dagger}_{u, \tau};\ (\forall g\in \mathscr{G}_{K_{\pi}})\ \chi(g)\in \ZZ_{>0} \Rightarrow (g\otimes 1)(x)=x+\tau_D(x)+\cdots +\tau_D^{\chi(g)-1}(x) \big\}. \]\index{$D^{\dagger}_{u, \tau, 0}$} Let $n\in \NN$ and we put 
	\[ D^{\dagger}_{u, \tau^{p^n}, 0}:=\big \{ x\in D^{\dagger}_{u, \tau};\ (\forall g\in \mathscr{G}_{K_{\pi}})\ \chi(g)\in \ZZ_{>0} \Rightarrow (g\otimes 1)(x)=x+\tau_D^{p^n}(x)+ \tau_D^{2p^n}(x) + \cdots +\tau_D^{(\chi(g)-1)p^n}(x) \big \}. \]\index{$D^{\dagger}_{u, \tau^{p^n}, 0}$}
	Notice that these are all subgroups of $D^{\dagger}_{u,\tau}.$
\end{notation}

\section{The complexes \texorpdfstring{$\Cc_{\varphi, \tau}^{u, \dagger}$}{C\textpinferior\textsinferior\textiinferior , \texttinferior\textainferior\textuinferior \unichar{"1D58}, \unichar{"207A}} and \texorpdfstring{$\Cc_{\psi, \tau}^{u, \dagger}$}{C\textpinferior\texthinferior\textiinferior , \texttinferior\textainferior\textuinferior \unichar{"1D58}, \unichar{"207A}}}
\begin{lemma}\label{lemma well-defined psi in the complex}
		Let $(D^{\dagger}, D_{u,\tau}^{\dagger})\in \Mod_{{\AA_{K_{\pi}}^\dagger,\widetilde{\AA}_{u, L}^\dagger}}(\varphi, \tau)$ $($resp. $\Mod_{{\BB_{K_{\pi}}^\dagger,\widetilde{\BB}_{u, L}^\dagger}}(\varphi, \tau))$. Then there are maps $\tau_D-1\colon D^{\dagger} \to D^{\dagger}_{u, \tau, 0},$ $\varphi-1\colon D^{\dagger}_{u,\tau, 0}\to D^{\dagger}_{u,\tau, 0}$ and $\psi\colon D_{u,\tau, 0}^{\dagger}\to D_{u,\tau, 0}^{\dagger}.$ 
\end{lemma}

\begin{proof}
	Notice that for any $(D, D_{u,\tau})\in \Mod_{\BB_u, \widetilde{\BB}_{u, L}}(\varphi, \tau),$ we have $\tau_D-1\colon D \to D_{u, \tau, 0},$  $\varphi-1\colon D_{u,\tau, 0}\to D_{u,\tau, 0}$ ($\cf$ lemma \ref{lemm complex u well defined}), and $\psi\colon D_{u,\tau, 0}\to D_{u,\tau, 0}$ ($\cf$ lemma \ref{lemm psi D u tau 0}). The lemma then follows from the fact that operator $\tau_D$ on $D_{\tau}$ preserves overconvergence in $D_{u,\tau}.$ Same arguments work for $\varphi$ and $\psi.$
\end{proof}

\begin{definition}
	Let $(D^{\dagger}, D_{u,\tau}^{\dagger})\in \Mod_{{\AA_{K_{\pi}}^\dagger,\widetilde{\AA}_{u, L}^\dagger}}(\varphi, \tau)$ (resp.  $\Mod_{{\BB_{K_{\pi}}^\dagger,\widetilde{\BB}_{u, L}^\dagger}}(\varphi, \tau)$). 
	\item (1) We define the complex $\Cc_{\varphi, \tau}^{u, \dagger}(D^{\dagger})$\index{$\Cc_{\varphi, \tau}^{u, \dagger}(D^{\dagger})$} as follows:
	\[ \xymatrix{
		0\ar[rr]  &&  D^{\dagger} \ar[r]  &  D^{\dagger} \oplus D^{\dagger}_{u,\tau, 0}  \ar[r]& \ar[r]  D^{\dagger}_{u, \tau, 0}  \ar[r] & 0   \\
		&&x  \ar@{|->}[r] &((\varphi-1)(x), (\tau_D-1)(x))  &{}\\
		&&& (y,z) \ar@{|->}[r]& (\tau_D-1)(y)-(\varphi-1)(z).
	}    \]

	If $T\in \Rep_{\ZZ_p}(\mathscr{G}_K)$ (resp. $V\in \Rep_{\QQ_p}(\mathscr{G}_K)$), we have in particular the complex $\Cc_{\varphi, \tau}^{u, \dagger}(\mathcal{D}^{\dagger}(T))$ (resp. $\Cc_{\varphi, \tau}^{u, \dagger}(\mathcal{D}^{\dagger}(V))$), which will also be simply denoted $\Cc_{\varphi, \tau}^{u, \dagger}(T)$\index{$\Cc_{\varphi, \tau}^{u, \dagger}(T)$} (resp. $\Cc_{\varphi, \tau}^{u, \dagger}(V)$).
	
	\item (2) Similarly, we define the complex $\Cc_{\psi, \tau}^{u, \dagger}(D^{\dagger})$\index{$\Cc_{\psi, \tau}^{u, \dagger}(D^{\dagger})$} as follows:
	\[  \xymatrix{    0 \ar[rr] && D^{\dagger} \ar[r] & D^{\dagger}\oplus D^{\dagger}_{u,\tau, 0} \ar[r] & D^{\dagger}_{u,\tau, 0} \ar[r] & 0\\
		&&x \ar@{|->}[r] &  ((\psi-1)(x), (\tau_D-1)(x)) &&\\
		&&              &    (y,z) \ar@{|->}[r]& (\tau_D-1)(y)-(\psi-1)(z).&
	}
	\]
	If $T\in \Rep_{\ZZ_p}(\mathscr{G}_K)$ (resp. $V\in \Rep_{\QQ_p}(\mathscr{G}_K)$), we have in particular the complex $\Cc_{\psi, \tau}^{u, \dagger}(\mathcal{D}^{\dagger}(T))$\index{$\Cc_{\psi, \tau}^{u, \dagger}(\mathcal{D}^{\dagger}(T))$} (resp. $\Cc_{\psi, \tau}^{u, \dagger}(\mathcal{D}^{\dagger}(V))$), which will also be simply denoted $\Cc_{\psi, \tau}^{u, \dagger}(T)$\index{$\Cc_{\psi, \tau}^{u, \dagger}(T)$} (resp. $\Cc_{\psi, \tau}^{u, \dagger}(V)$)\index{$\Cc_{\psi, \tau}^{u, \dagger}(V)$}.
\end{definition}

\begin{remark}
	Notice that for $T\in \Rep_{\ZZ_p}(\mathscr{G}_K),$ we use superscripts $\dagger$ in $\Cc_{\psi, \tau}^{u, \dagger}(T)$ and $\Cc_{\varphi, \tau}^{u, \dagger}(T)$ to distinguish from $\Cc_{\psi, \tau}^{u}(T)$ and $\Cc_{\varphi, \tau}^{u}(T)$ ($\cf$ definition \ref{def complex varphi tau} and definition \ref{def complex psi tau}).
\end{remark}

%\begin{definition}
%We put
%	\begin{align*}
%	\Ff^i\colon \Rep_{\QQ_p}(\mathscr{G}_K) &\to \Ab  \\
%	V &\mapsto \H^i(\Cc_{\varphi, \tau}^{u, \dagger}(V))			
%	\end{align*}\index{$\Ff^i$}   
%\end{definition}

%\tc{red}{\begin{remark}
%		Here we use an abuse of notation, the functor $\{\Ff^i\}_i$ just defined is different from before ($\cf$ definition \ref{def Fi 1} and definition \ref{def Fi 2}).
%\end{remark}}

\subsection{Quasi-isomorphism between \texorpdfstring{$\Cc_{\psi, \tau}^{u, \dagger}$}{C\textpinferior\textsinferior\textiinferior , \texttinferior\textainferior\textuinferior \unichar{"1D58}, \unichar{"207A}} and \texorpdfstring{$\Cc_{\varphi, \tau}^{u, \dagger}$}{C\textpinferior\texthinferior\textiinferior , \texttinferior\textainferior\textuinferior \unichar{"1D58}, \unichar{"207A}}}

Let $T\in\Rep_{\ZZ_p}(\mathscr{G}_K)$ and $(D,D_{u,\tau})$ (resp. $(D^\dagger,D^\dagger_{u,\tau})$) the associated $(\varphi,\tau)$-module over $(\Oe,\Oo_{\Ee_{u,\tau}})$ (resp. $(\Oo_{\Ee^\dagger},\Oo_{\Ee^\dagger_{u,\tau}}))$. We consider the following morphism of complexes:

\[ 
\xymatrix{
	\Cc^{u, \dagger}_{\varphi,\tau}(T) \colon & 0\ar[rr]  && D^{\dagger} \ar[rr]^-{(\varphi-1, \tau_D-1)} \ar@{=}[d] && D^{\dagger}\oplus D^{\dagger}_{u, \tau, 0} \ar[rr]^-{(\tau_D-1)\ominus(\varphi-1)} \ar[d]^{-\psi\ominus \id} && D^{\dagger}_{u, \tau, 0}\ar[rr] \ar[d]^{-\psi} && 0 \\
	\Cc^{u, \dagger}_{\psi,\tau}(T) \colon & 0\ar[rr]  && D^{\dagger} \ar[rr]^-{(\psi-1, \tau_D-1)}  && D^{\dagger} \oplus D^{\dagger}_{u, \tau,0} \ar[rr]^-{(\tau_D-1)\ominus(\psi-1)} && D^{\dagger}_{u, \tau, 0}\ar[rr] && 0.
} \]

\begin{theorem}\label{thm quasi-iso psi varphi u dagger}
	The morphism just defined is a quasi-isomorphism.
\end{theorem}

\begin{remark}
	We have the following diagram:
	
	\[ 
	\xymatrix{
		&0 \ar[r] & 0 \ar[rr]\ar[d] && (D^{\dagger})^{\psi=0} \ar[rr]^{\tau_D-1} \ar[d] && (D^{\dagger}_{u,\tau, 0})^{\psi=0 }   \ar[r]\ar[d]& 0\\
		\Cc^{u, \dagger}_{\varphi,\tau}(T) \colon & 0\ar[r]  & D^{\dagger} \ar[rr]^-{(\varphi-1, \tau_D-1)} \ar@{=}[d] && D^{\dagger}\oplus D^{\dagger}_{u, \tau,0} \ar[rr]^{(\tau_D-1)\ominus(\varphi-1)} \ar[d]^{-\psi\ominus \id} && D^{\dagger}_{u, \tau, 0}\ar[r] \ar[d]^{-\psi} & 0 \\
		\Cc^{u, \dagger}_{\psi,\tau}(T) \colon & 0  \ar[r] & D^{\dagger} \ar[rr]^-{(\psi-1, \tau_D-1)}  \ar[d]&& D^{\dagger}\oplus D^{\dagger}_{u, \tau,0} \ar[rr]^-{(\tau_D-1)\ominus(\psi-1)} \ar[d]&& D^{\dagger}_{u, \tau, 0}\ar[r]  \ar[d]& 0\\
		& 0 \ar[r]& 0 \ar[rr] &&0 \ar[rr] &&0 \ar[r] &0.	
	} \]
	
	Hence to prove theorem \ref{thm quasi-iso psi varphi u dagger}, it suffices to show that $\tau_D-1\colon (D^{\dagger})^{\psi=0}\to (D^{\dagger}_{u, \tau, 0})^{\psi=0}$ is bijective. 
\end{remark}

\subsubsection{Proof of theorem \ref{thm quasi-iso psi varphi u dagger}}

As the operator $\psi$ is surjective (since $\psi\circ\varphi=\Id$), there is an exact sequence of complexes
$$0\to\Cc^{u, \dagger}(T)\to\Cc^{u, \dagger}_{\varphi,\tau}(T)\to\Cc^{u, \dagger}_{\psi,\tau}(T)\to0$$
where $\Cc^{u, \dagger}(T)$ is the complex $[(D^\dagger)^{\psi=0}\xrightarrow{\tau_D-1}(D^\dagger_{u,\tau,0})^{\psi=0}]$ in which the first term is in degree $1$. 

Similarly, we denote $\Cc(D)$ the complex $[D^{\psi=0}\xrightarrow{\tau_D-1}D_{u,\tau,0}^{\psi=0}]$, which is acyclic by proposition \ref{prop Zp-case kernel complex trivial}. There is a natural morphism of complexes $\Cc^{u, \dagger}(T)\to\Cc(D)$ induced by the commutative square
$$\xymatrix{(D^\dagger)^{\psi=0}\ar[r]^-{\tau_D-1}\ar[d] & (D^\dagger_{u,\tau,0})^{\psi=0}\ar[d]\\
	D^{\psi=0}\ar[r]^-{\tau_D-1}_-\sim & D_{u,\tau,0}^{\psi=0}}$$
which implies the injectivity of $\tau_D-1\colon(D^\dagger)^{\psi=0}\to(D^\dagger_{u,\tau,0})^{\psi=0}$.

\medskip

If $x=\sum\limits_{m=0}^\infty p^m[x_m]\in\W(C^\flat)$ and $n\in\NN$, we put $w_n(x)=\inf\limits_{0\leq m\leq n}v^\flat(x_m)$. Reacll that if $x,y\in\W(C^\flat)$, we have $w_n(x+y)\geq\inf\{w_n(x),w_n(y)\}$ and $w_n(xy)\geq\inf\limits_{n_1+n_2\leq n}\big(w_{n_1}(x)+w_{n_2}(y)\big)$ (\cf \cite[p. 584]{CC98}). Also, an element $x\in\W(C^\flat)$ is overconvergent if there exists $r\in\RR_{\geq0}$ such that the sequence $\big(\big(w_n(x)+\frac{pr}{p-1}n\big)\big)_{n\in\NN}$ is bounded below.

\begin{lemma}\label{lemmaultildepsi0}
	The map $\tau-1\colon\widetilde{\AA}_{u,L}^{\psi=0}\to\widetilde{\AA}_{u,L}^{\psi=0}$ is bijective. If $y\in\big(\widetilde{\AA}^\dagger_{u,L}\big)^{\psi=0}$, then $(\tau-1)^{-1}(y)\in\big(\widetilde{\AA}^\dagger_{u,L}\big)^{\psi=0}$.
\end{lemma}

\begin{proof}
	By d\'evissage, the first assertion can be checked by modulo $p$: we have $\widetilde{\AA}_{u,L}/(p)\simeq k(\!(u,\eta^{1/p^\infty})\!)$, so that $\big(\widetilde{\AA}_{u,L}\big)^{\psi=0}/(p)\simeq k(\!(u,\eta^{1/p^\infty})\!)^{\psi=0}=\bigoplus\limits_{i=1}^{p-1}u^ik(\!(u^p,\eta^{1/p^\infty})\!)$. As element $x\in k(\!(u,\eta^{1/p^\infty})\!)^{\psi=0}$ can be written in a unique way
$$x=\sum\limits_{i=1}^{p-1}\sum\limits_{j\in\ZZ} u^{i+pj}x_{i,j}$$
where $x_{i,j}\in k(\!(\eta^{1/p^\infty})\!)$ is zero when $j\ll0$. Then we have
$$(\tau-1)(x)=\sum\limits_{i=1}^{p-1}\sum\limits_{j\in\ZZ} u^{i+pj}(\varepsilon^{i+pj}-1)x_{i,j}.$$
As $\varepsilon^{i+pj}\neq1$ for all $i\in\{1,\ldots,p-1\}$ and $j\in\ZZ$, this implies that $\tau-1$ is injective. If $y=\sum\limits_{i=1}^{p-1}\sum\limits_{j\in\ZZ} u^{i+pj}y_{i,j}$ is an element in $k(\!(u,\eta^{1/p^\infty})\!)^{\psi=0}$, then
$$(\tau-1)^{-1}(y)=\sum\limits_{i=1}^{p-1}\sum\limits_{j\in\ZZ} u^{i+pj}\tfrac{y_{i,j}}{\varepsilon^{i+pj}-1}.$$ 

For the second assertion, let $y\in\big(\widetilde{\AA}^\dagger_{u,L}\big)^{\psi=0}$. We can write
$$y=\sum\limits_{i=1}^{p-1}\sum\limits_{j\in\ZZ}u^{i+pj}y_{i,j}$$ (where $u=[\tilde{\pi}]$ here)
with $y_{i,j}\in\W\big(k(\!(\eta^{1/p^\infty})\!)\big)$, where $\Lim{j\to-\infty}y_{i,j}=0$ (for the $p$-adic topology) for all $i\in\{1,\ldots,p-1\}$. Then $(\tau-1)^{-1}(y)=\sum\limits_{i=1}^{p-1}\sum\limits_{j\in\ZZ}u^{i+pj}\frac{y_{i,j}}{[\varepsilon]^{i+pj}-1}$. As $y$ is overconvergent, there exists $r\in\RR_{\geq0}$ and $c\in\RR$ such that $w_n(u^{i+pj}y_{i,j})+\frac{pr}{p-1}n\geq c$ for all $i\in\{1,\ldots,p-1\}$, $j\in\ZZ$ and $n\in\NN$. If $n\in\NN$ we have
\begin{align*}
w_n\big(u^{i+pj}\frac{y_{i,j}}{[\varepsilon]^{i+pj}-1}\big)+\tfrac{pr}{p-1}n &= w_n\big(\tfrac{u^{i+pj}}{[\varepsilon^{i+pj}-1]}y_{i,j}\tfrac{[\varepsilon^{i+pj}-1]}{[\varepsilon]^{i+pj}-1}\big)+\tfrac{pr}{p-1}n\\
&= -\tfrac{p}{p-1}+w_n\big(u^{i+pj}y_{i,j}\tfrac{[\varepsilon^{i+pj}-1]}{[\varepsilon]^{i+pj}-1}\big)+\tfrac{pr}{p-1}n\\
&\geq -\tfrac{p}{p-1}+\inf\limits_{n_1+n_2\leq n}\Big(w_{n_1}\big(u^{i+pj}y_{i,j})+w_{n_2}\big(\tfrac{[\varepsilon^{i+pj}-1]}{[\varepsilon]^{i+pj}-1}\big)\Big)+\tfrac{pr}{p-1}n\\
&\geq -\tfrac{p}{p-1}+\inf\limits_{n_1+n_2\leq n}\Big(w_{n_1}\big(u^{i+pj}y_{i,j})+\tfrac{pr}{p-1}n_1+\big(\tfrac{pr}{p-1}-1)n_2\Big)
\end{align*}
since $v^\flat(\varepsilon^{i+pj}-1)=\frac{p}{p-1}$ (because $p\nmid i+pj$) and $w_n\big(\tfrac{[\varepsilon^{i+pj}-1]}{[\varepsilon]^{i+pj}-1}\big)\geq -n$ (by \cite[Corollaire II.1.5]{CC98}). Taking $r$ larger if necessary, we may assume that $\tfrac{pr}{p-1}-1\geq0$. Then $w_n\big(u^{i+pj}\frac{y_{i,j}}{[\varepsilon]^{i+pj}-1}\big)+\tfrac{pr}{p-1}n\geq-\tfrac{p}{p-1}+c$ for all $i\in\{1,\ldots,p-1\}$, $j\in\ZZ$ and $n\in\NN$. This shows that $(\tau-1)^{-1}(y)$ is overconvergent.
\end{proof}

\begin{proposition}
	The map $\tau_D-1\colon(D^\dagger)^{\psi=0}\to(D^\dagger_{u,\tau,0})^{\psi=0}$ is bijective.
\end{proposition}

\begin{proof}

	Put $\DD^\dagger(T)=(\AA^\dagger\otimes_{\ZZ_p}T)^{\mathscr{G}_{K_\zeta}}$ (this is the overconvergent $(\varphi,\Gamma)$-module associated to $T$). We have a $\mathscr{G}_K$-equivariant isomorphism $\AA^\dagger\otimes_{\AA^\dagger_{K_\zeta}}\DD^\dagger(T)\simeq\AA^\dagger\otimes_{\ZZ_p}T$: extending the scalars to $\widetilde{\AA}^\dagger_u:=\Oo_{\widehat{\mathcal{F}_u^{\ur}}} \cap \, \widetilde{\AA}^\dagger$, we deduce a $\mathscr{G}_K$-equivariant isomorphism $\widetilde{\AA}^\dagger_u\otimes_{\AA^\dagger_{K_\zeta}}\DD^\dagger(T)\simeq\widetilde{\AA}^\dagger_u\otimes_{\ZZ_p}T$. Taking $\mathscr{G}_L$ invariants provides a $\mathscr{G}_K/\mathscr{G}_L$-equivariant isomorphism
	$$D^\dagger_{u,\tau}=\big(\widetilde{\AA}^\dagger_u\otimes_{\ZZ_p}T\big)^{\mathscr{G}_L}\simeq\widetilde{\AA}^\dagger_{u,L}\otimes_{\AA^\dagger_{K_\zeta}}\DD^\dagger(T).$$
	
	\begin{enumerate}
    \item [(1)]	Assume $T$ is torsion-free: so is $\DD^\dagger(T)$. Let $(e_1,\ldots,e_d)$ be a basis of $\DD^\dagger(T)$ so that $(\varphi(e_1),\ldots,\varphi(e_d))$ is a basis of $\DD^\dagger(T)$ over $\AA^\dagger_{K_\zeta}$. Then $D^\dagger_{u,\tau}=\bigoplus\limits_{i=1}^d\widetilde{\AA}^\dagger_{u,L}\varphi(e_i)$. If $x=\sum\limits_{i=1}^d\lambda_i\varphi(e_i)\in D^\dagger_{u,\tau}$ with $\lambda_1,\ldots,\lambda_d\in\widetilde{\AA}^\dagger_{u,L}$, we have $\psi(x)=\sum\limits_{i=1}^d\psi(\lambda_i)e_i$. This implies that $x\in(D^\dagger_{u,\tau})^{\psi=0}$ if and only if $\psi(\lambda_i)=0$ \ie $\lambda_i\in(\widetilde{\AA}^\dagger_{u,L})^{\psi=0}$. Lemma \ref{lemmaultildepsi0} implies that $\mu_i:=(\tau-1)^{-1}(\lambda_i)\in(\widetilde{\AA}^\dagger_{u,L})^{\psi=0}$. As $\tau_D(\varphi(e_i))=\varphi(e_i)$, this implies that $x=(\tau_D-1)y$ with $y=\sum\limits_{i=1}^d\mu_i\varphi(e_i)\in D^\dagger_{u,\tau}$. As we know that $\tau_D-1\colon D^{\psi=0}\to D_{u,\tau,0}^{\psi=0}$ is bijective, this implies that in fact $y\in D\cap D^\dagger_{u,\tau}=D^\dagger$, finishing the proof in that case.
	
   \item [(2)]	In the general case, there is an exact sequence of representations
	$$0\to T^\prime\to T\to T^{\prime\prime}\to0$$
	where $T^\prime$ is finite and $T^{\prime\prime}$ torsion-free. Tensoring with $\AA^\dagger$ and taking $\mathscr{G}_{K_\pi}$ invariants provides the exact sequence
	$$0\to\mathcal{D}(T^\prime)\to\mathcal{D}^\dagger(T)\to\mathcal{D}^\dagger(T^{\prime\prime})\to\H^1(\mathscr{G}_{K_\pi},\AA\otimes_{\ZZ_p}T^\prime)=0$$
	(because $\AA^\dagger\otimes_{\ZZ_p}T^\prime=\AA\otimes_{\ZZ_p}T^\prime$ since $T^\prime$ is torsion). As the maps $\psi$ are surjective, the snake lemma applied to the commutative diagram
	$$\xymatrix{0\ar[r] &\mathcal{D}(T^\prime)\ar[r]\ar[d]^\psi & \mathcal{D}^\dagger(T)\ar[r]\ar[d]^\psi & \mathcal{D}^\dagger(T^{\prime\prime})\ar[r]\ar[d]^\psi & 0\\
		0\ar[r] &\mathcal{D}(T^\prime)\ar[r] & \mathcal{D}^\dagger(T)\ar[r] & \mathcal{D}^\dagger(T^{\prime\prime})\ar[r] & 0}$$
	shows that the sequence
	$$0\to\mathcal{D}(T^\prime)^{\psi=0}\to\mathcal{D}^\dagger(T)^{\psi=0}\to\mathcal{D}^\dagger(T^{\prime\prime})^{\psi=0}\to0$$
	is exact. Similarly, tensoring the short exact sequence $0\to\mathcal{D}(T^\prime)\to\mathcal{D}^\dagger(T)\to\mathcal{D}^\dagger(T^{\prime\prime})\to0$ by $\widetilde{\AA}^\dagger_{u,L}$ provides the short exact sequence
	$$0\to\mathcal{D}(T^\prime)_{u,\tau}\to\mathcal{D}^\dagger(T)_{u,\tau}\to\mathcal{D}^\dagger(T^{\prime\prime})_{u,\tau}\to0$$
	(by flatness of $\widetilde{\AA}^\dagger_{u,L}$ over the discrete valuation ring $\AA^\dagger_{K_\pi}$). Consider the following commutative diagram
	$$\xymatrix{0\ar[r] & \mathcal{D}(T^\prime)_{u,\tau}\ar[r]\ar[d]^{\delta-\gamma\otimes1} & \mathcal{D}^\dagger(T)_{u,\tau}\ar[r]\ar[d]^{\delta-\gamma\otimes1} & \mathcal{D}^\dagger(T^{\prime\prime})_{u,\tau}\ar[r]\ar[d]^{\delta-\gamma\otimes1} & 0\\
		0\ar[r] & \mathcal{D}(T^\prime)_{u,\tau}\ar[r] & \mathcal{D}^\dagger(T)_{u,\tau}\ar[r] & \mathcal{D}^\dagger(T^{\prime\prime})_{u,\tau}\ar[r] & 0,}$$
	 and we have the exact sequence
	$$0\to\mathcal{D}(T^\prime)_{u,\tau,0}\to\mathcal{D}^\dagger(T)_{u,\tau,0}\to\mathcal{D}^\dagger(T^{\prime\prime})_{u,\tau,0}.$$
 Now consider the following commutative diagram
		$$\xymatrix{0\ar[r] & \mathcal{D}(T^\prime)_{u,\tau,0}\ar[r]\ar[d]^{\psi} & \mathcal{D}^\dagger(T)_{u,\tau,0}\ar[r]^{f}\ar[d]^{\psi} & \mathcal{D}^\dagger(T^{\prime\prime})_{u,\tau,0}\ar[d]^{\psi}\\
		0\ar[r] & \mathcal{D}(T^\prime)_{u,\tau,0}\ar[r] & \mathcal{D}^\dagger(T)_{u,\tau,0}\ar[r]^{f} & \mathcal{D}^\dagger(T^{\prime\prime})_{u,\tau,0}}.$$ The composition $\mathcal{D}(T^\prime)_{u,\tau,0}^{\psi=0} \subset \mathcal{D}(T^\prime)_{u,\tau,0}\to \mathcal{D}^\dagger(T)_{u,\tau,0}$ is injective, hence $\mathcal{D}(T^\prime)_{u,\tau,0}^{\psi=0}\to \mathcal{D}^{\dagger}(T)_{u,\tau,0}^{\psi=0}$ is injective. Suppose $x\in \mathcal{D}^\dagger(T)_{u,\tau,0}^{\psi=0}$ is mapped to $0$ in $\mathcal{D}^\dagger(T^{\prime\prime})_{u,\tau,0}^{\psi=0},$ then $x\in \Ker(f)\cap \mathcal{D}^\dagger(T)_{u,\tau,0}^{\psi=0}$ and hence $x$ is in the image of $\mathcal{D}(T^\prime)_{u,\tau,0}^{\psi=0}.$
	This implies that the sequence
	$$0\to\mathcal{D}(T^\prime)_{u,\tau,0}^{\psi=0}\to\mathcal{D}^\dagger(T)_{u,\tau,0}^{\psi=0}\to\mathcal{D}^\dagger(T^{\prime\prime})_{u,\tau,0}^{\psi=0}$$
	is exact.
	Now we consider the following commutative diagram
	$$\xymatrix{0\ar[r] & \mathcal{D}(T^\prime)^{\psi=0}\ar[r]\ar[d]^{\tau_{D^\prime}-1} & \mathcal{D}^\dagger(T)^{\psi=0}\ar[r]\ar[d]^{\tau_D-1} & \mathcal{D}^\dagger(T^{\prime\prime})^{\psi=0}\ar[r]\ar[d]^{\tau_{D^{\prime\prime}}-1} & 0\\
		0\ar[r] & \mathcal{D}(T^\prime)_{u,\tau,0}^{\psi=0}\ar[r] & \mathcal{D}^\dagger(T)_{u,\tau,0}^{\psi=0}\ar[r] & \mathcal{D}^\dagger(T^{\prime\prime})_{u,\tau,0}^{\psi=0} & }.$$
As the maps $\tau_{D^\prime}-1$ and $\tau_{D^{\prime\prime}}-1$ are bijective (by proposition \ref{prop Zp-case kernel complex trivial} for the first, and by the special case above for the second), so is $\tau_D-1$ in the middle.
\end{enumerate}

\end{proof}

This shows in particular that the complex $\Cc^{u, \dagger}(T)$ is acyclic, so that the complexes $\Cc^{u, \dagger}_{\varphi,\tau}(T)$ and $\Cc^{u, \dagger}_{\psi,\tau}(T)$ are quasi-isomorphic.

\section{ \texorpdfstring{$\H^i(\Cc^{u, \dagger}_{\psi,\tau}(T))$}{H\unichar{"2071}(C\textpinferior\textsinferior\textiinferior , \texttinferior\textainferior\textuinferior \unichar{"1D58}, \unichar{"207A}(T))} and \texorpdfstring{$\H^i(\Cc^{u, \dagger}_{\varphi,\tau}(T))$}{H\unichar{"2071}(C\textpinferior\texthinferior\textiinferior , \texttinferior\textainferior\textuinferior \unichar{"1D58}, \unichar{"207A}(T))} for \texorpdfstring{$i\in \{0, 1\}$}{i=0, 1}}

\begin{proposition}\label{prop H1 varphi tau u dagger}
Let $T\in \Rep_{\ZZ_p}(\mathscr{G}_k).$ We have $\H^i(\Cc_{\varphi, \tau}^{u, \dagger}(T))\simeq \H^i(\mathscr{G}_{K}, T)$ for $i\in \{0, 1\}.$
\end{proposition}

\begin{proof}
	We have $\H^0(\Cc_{\varphi, \tau}^{u, \dagger}(T))=(\AA^{\dagger}\otimes_{\ZZ_p}T)^{\mathscr{G}_{K_{\pi}}, \varphi=1, \tau_D=1}=T^{\la \mathscr{G}_{K_{\pi}}, \tau\ra}=T^{\mathscr{G}_K}=\H^0(\mathscr{G}_{K}, T).$
	
	\medskip
	
	We claim that the natural $\ZZ_p$-linear map $\H^1(\Cc_{\varphi, \tau}^{u, \dagger}(T))\to \H^1(\Cc_{\varphi, \tau}(T))$ is an isomorphism. Recall that $\H^1(\Cc_{\varphi, \tau}(T))$ classifies all extensions of $\mathcal{D}(T)$ by $\AA_{K_{\pi}}$ in the category $\Mod_{\AA_{K_{\pi}}, \widetilde{\AA}_{L}}(\varphi, \tau).$ Similarly, $\H^1(\Cc_{\varphi, \tau}^{u, \dagger}(T))$ classifies all extensions of $\mathcal{D}^{\dagger}(T)$ by $\AA_{K_{\pi}}^{\dagger}$ in the category $\Mod_{\AA_{K_{\pi}}^{\dagger}, \widetilde{\AA}_{u, L}^{\dagger}}(\varphi, \tau).$ Since both categories are equivalent to $\Rep_{\ZZ_p}(\mathscr{G}_K),$ we conclude that $\H^1(\Cc_{\varphi, \tau}^{u, \dagger}(T))\simeq \H^1(\Cc_{\varphi, \tau}(T)),$ which is isomorphic to $\H^1(\mathscr{G}_K, T)$ by proposition \ref{prop Caruso H1 Zp}.
\end{proof}

\begin{corollary}
We have $\H^i(\Cc_{\psi, \tau}^{u, \dagger}(T))\simeq \H^i(\mathscr{G}_{K}, T)$ for $i\in \{0, 1\}.$
\end{corollary}

\begin{proof}
	By theorem \ref{thm quasi-iso psi varphi u dagger}, we have $\H^i(\Cc_{\psi, \tau}^{u, \dagger}(T))\simeq \H^i(\Cc_{\varphi, \tau}^{u, \dagger}(V))$ for $i\in \NN:$ the result follows from proposition \ref{prop H1 varphi tau u dagger}.
\end{proof}

\section{Some remarks on the morphism between \texorpdfstring{$\Cc_{\psi, \tau}^{u, \dagger}$}{C\textpinferior\textsinferior\textiinferior , \texttinferior\textainferior\textuinferior \unichar{"1D58}, \unichar{"207A}} and \texorpdfstring{$\Cc_{\psi, \tau}^{u}$}{C\textpinferior\textsinferior\textiinferior , \texttinferior\textainferior\textuinferior \unichar{"1D58}}}
Let $T\in \Rep_{\ZZ_p}(\mathscr{G}_K)$ and $(D, D_{u, \tau})\in \Mod_{\AA_{K_{\pi}},\widetilde{\AA}_{u, L}}(\varphi, \tau)$  $($resp. $(D^{\dagger}, D^{\dagger}_{u, \tau})\in \Mod_{{\AA_{K_{\pi}}^\dagger,\widetilde{\AA}_{u, L}^\dagger}}(\varphi, \tau))$ its associated $(\varphi, \tau)$-modules. We then have a natural map from $\Cc_{\psi, \tau}^{u, \dagger}(T)$ to $\Cc_{\psi, \tau}^{u}(T):$
\[ 
\xymatrix{
	\Cc_{\psi, \tau}^{u, \dagger}(T)\colon & 0\ar[r]  & D^{\dagger} \ar[rr]^-{(\psi-1, \tau_D-1)} \ar@{^(->}[d] && D^{\dagger}\oplus D^{\dagger}_{u, \tau,0} \ar[rr]^-{(\tau_D-1)\ominus(\psi-1)} \ar@{^(->}[d] && D^{\dagger}_{u, \tau, 0}\ar[r] \ar@{^(->}[d] & 0 \\
	\Cc_{\psi, \tau}^{u}(T) \colon & 0  \ar[r] & D \ar[rr]^-{(\psi-1, \tau_D-1)}  && D\oplus D_{u, \tau,0} \ar[rr]^-{(\tau_D-1)\ominus(\psi-1)} && D_{u, \tau, 0}\ar[r]  & 0.\\
} \]

\begin{remark}
(1)	We summarize the relations between the different complexes:
	\[\xymatrix{ \Cc_{\varphi, \tau}^{u, \dagger}(T) \ar@{<->}[rr]^{\text{ theorem }\ref{thm quasi-iso psi varphi u dagger}}_{\text{quasi-iso}} && \Cc_{\psi, \tau}^{u, \dagger}(T) \ar[r] & \Cc_{\psi, \tau}^{u}(T) \ar@{<->}[rr]^{\text{ theorem }\ref{prop quais-iso psi}}_{\text{quasi-iso}} && \Cc_{\varphi, \tau}^{u}(T) \ar@{<->}[rr]^{\text{ theorem }\ref{coro complex over non-perfect ring works also Zp-case}}_{\text{quasi-iso}} && \Cc_{\varphi, \tau}(T)}.   \]
	In the first chapter we have showed that $\Cc_{\varphi, \tau}(T)$ computes the continuous Galois cohomology of $T$ ($\cf$ theorem \ref{thm main result}). It is hence natural to study whether the morphism between $\Cc_{\psi, \tau}^{u, \dagger}$ and $\Cc_{\psi, \tau}^{u}$ is a quasi-isomorphism (we expect it is).
\item (2) The complex  $\Cc_{\psi, \tau}^{u}(D)$ is the total complex of the double complex
\[ \xymatrix{  
	D \ar[r]^{\psi-1} \ar[d]_{\tau_D-1} & D\ar[d]^{\tau_D-1}\\
	D_{u,\tau, 0} \ar[r]_{\psi-1} &D_{u,\tau, 0}
	}   \]

\medskip

Similarly,  the complex $\Cc^{u, \dagger}_{\psi, \tau}(D^{\dagger})$ is the total complex of the double complex 
\[ \xymatrix{  
	D^{\dagger} \ar[r]^{\psi-1} \ar[d]_{\tau_D-1} & D^{\dagger}\ar[d]^{\tau_D-1}\\
	D^{\dagger}_{u,\tau, 0} \ar[r]_{\psi-1} &D^{\dagger}_{u,\tau, 0}
}  \]
\end{remark}

\begin{lemma}
	The complex $\Cc_{\psi, \tau}^{u}(D)$ is quasi-isomorphic to $\Cc^{u, \dagger}_{\psi, \tau}(D^{\dagger})$ if the following two morphisms of complexes are both quasi-isomorphisms:
	
	\[ \xymatrix{    [D^{\dagger} \ar[r]^{\psi-1} \ar@{^(->}[d]&   D^{\dagger}] \ar@{^(->}[d]   && [D_{u, \tau, 0}^{\dagger} \ar[r]^{\psi-1} \ar@{^(->}[d]&   D_{u,\tau, 0}^{\dagger}] \ar@{^(->}[d] \\
		[D \ar[r]^{\psi-1} &   D] &&  [D_{u, \tau, 0} \ar[r]^{\psi-1} &   D_{u, \tau, 0}].
	}  \]
\end{lemma}

\begin{proof}
	$\cf$ \cite[\S 1.3]{Wei95}. 
\end{proof}

\begin{remark}  (1) When the residue field $k$ is finite. The morphism 
\[ \xymatrix{    [D^{\dagger} \ar[r]^{\psi-1} \ar@{^(->}[d]&   D^{\dagger}]\ar@{^(->}[d]\\
	[D \ar[r]^{\psi-1} &   D] }  \]
is a quasi-isomorphism. Indeed, this can be proved using classical methods: for $h^0,$  it suffices to show that $x\in D^{\psi=1}$ implies $x\in D^{\dagger}$ (\cf \cite[Lemma I.6.4; Proposition III.2.1 (ii)]{CC99}); for $h^1,$ it suffices to show that $D^{\dagger}/(\psi-1)\simeq D/(\psi-1)$ (\cf \cite[Lemma 2.6]{RLiu08}, \cite[Corollaire I.7.4]{CC99} and \cite[Proposition 3.6 (2)]{Her98}).
\item (2) As $k(\!(\eta)\!)$ is never finite, it seems hard to use the same method to show that 
	\[ \xymatrix{  [ D_{u,\tau, 0}^{\dagger} \ar[r]^{\psi-1} \ar@{^(->}[d]&   D_{u, \tau, 0}^{\dagger}]\ar@{^(->}[d] \\
	[D_{u, \tau, 0} \ar[r]^{\psi-1} &   D_{u, \tau, 0}] } \]
is a quasi-isomorphism. In particular the structure of $D_{u, \tau, 0}/(\psi-1)$ is not clear.
\end{remark}

\subsubsection{Some remarks on \texorpdfstring{$\H^2(\Cc^{u, \dagger}_{\psi, \tau}(T))$}{H\textonesuperior(C\textpinferior\textsinferior\textiinferior , \texttinferior\textainferior\textuinferior \unichar{"1D58}, \unichar{"207A}(T))} and \texorpdfstring{$\H^2(\Cc^{u, \dagger}_{\varphi, \tau}(T))$}{H\textonesuperior(C\textpinferior\texthinferior\textiinferior , \texttinferior\textainferior\textuinferior \unichar{"1D58}, \unichar{"207A}(T))}} 

Let $T\in \Rep_{\ZZ_p}(\mathscr{G}_K)$ and $(D, D_{u, \tau})\in \Mod_{\Oo_{\Ee}, \Oo_{\Ee_{u, \tau}}}(\varphi, \tau)$ its associated $(\varphi, \tau)$-module. Notice that by theorem \ref{thm quasi-iso psi varphi u dagger}, $\H^2(\Cc^{u, \dagger}_{\psi, \tau}(V))\simeq \H^2(\Cc^{u, \dagger}_{\varphi, \tau}(V)).$ We have 
\[ \H^2(\Cc^{u,\dagger}_{\psi, \tau}(V))\simeq  D^{\dagger}_{u, \tau, 0}\big/ \big( (\psi-1)D^{\dagger}_{u, \tau, 0}+(\tau_D-1)D^{\dagger} \big), \]
and 
\[ \H^2(\Cc_{\psi, \tau}(V)) \simeq D_{u, \tau, 0}\big/ \big( (\psi-1)D_{u, \tau, 0}+(\tau_D-1)D \big)\simeq \H^2(\mathscr{G}_{K}, V). \]

\begin{lemma}\label{lemm structure for D u tau 0}
Let $(D, D_{u, \tau})\in \Mod_{\Oo_{\Ee}, \Oo_{\Ee_{u, \tau}}}(\varphi, \tau)$ $($resp. $\Mod_{\Ee, \Ee_{u, \tau}}(\varphi, \tau)),$ we then have 
	\begin{align*}
		D&=\varphi(D)\oplus D^{\psi=0}, \\
	 D_{u, \tau}&=\varphi(D_{u, \tau})\oplus D_{u, \tau}^{\psi=0},\\
 D_{u, \tau, 0}&=\varphi(D_{u, \tau, 0})\oplus D_{u, \tau, 0}^{\psi=0}.
	\end{align*}
There are similar equalities with $D^{\dagger}, D^{\dagger}_{u, \tau}$ and $D_{u, \tau, 0}^{\dagger}.$
\end{lemma}

\begin{proof}
For any $z\in D_{u, \tau},$ put $z_0=\varphi(\psi(z))$ and write $z=z_0+(z-z_0).$ Notice that $z_0\in \varphi(D_{u, \tau})$ and $z-z_0\in D_{u, \tau}^{\psi=0}.$ Suppose $x\in \varphi(D_{u, \tau})\cap D_{u, \tau}^{\psi=0},$ then $x=\varphi(y)$ and $0=\psi(z)=y,$ hence $x=0.$ This proves $ D_{u, \tau}=\varphi(D_{u, \tau})+D_{u, \tau}^{\psi=0},$ and similarly we have $D=\varphi(D)\oplus D^{\psi=0}.$ 

\medskip

For any $z\in D_{u, \tau,0},$ again put $z_0=\varphi(\psi(z))$ and write $z=z_0+(z-z_0).$ Then $z_0\in \varphi(D_{u, \tau, 0})$ and $z-z_0\in D_{u, \tau, 0}^{\psi=0}.$ Indeed, as $\psi$ commutes with the action of $\mathscr{G}_K,$ $z\in D_{u, \tau, 0}$ implies $\psi(z)\in D_{u, \tau, 0}.$

\medskip

The last statement follows from the fact that the operator $\varphi$ and $\psi$ respect overconvergence.
\end{proof}

\begin{remark}\label{rem for H2 dagger not yet clear}
	Let $f\colon D_{u, \tau, 0}\to \H^2(\Cc^{u}_{\psi, \tau}(T))$ and $f^{\dagger}\colon D_{u, \tau, 0}^{\dagger}\to \H^2(\Cc_{\psi, \tau}^{u, \dagger}(T))$ be the canonical surjections. If $x\in D_{u, \tau, 0},$ we can write $x=\varphi(y)+z$ with $y\in D_{u, \tau, 0}$ and $z\in D_{u, \tau, 0}^{\psi=0}$ by lemma \ref{lemm structure for D u tau 0}. By proposition \ref{prop Zp-case kernel complex trivial}, we have $D_{u, \tau, 0}^{\psi=0}=(\tau_D-1)(D^{\psi=0}),$ so that the image of $z$ in $\H^2(\Cc_{\psi,\tau}^{u}(T))$ is zero, $\ie$$f(x)=f(\varphi(y)).$ Iterating, this shows that $f$ induces a surjective map 
	\[ f\colon \bigcap\limits_{n=0}^{\infty} \varphi^n(D_{u, \tau, 0})\to \H^2(\Cc_{\psi, \tau}^{u}(T)).  \]
	We have a similar statement for $f^{\dagger}\colon \bigcap\limits_{n=0}^{\infty} \varphi^n(D_{u, \tau, 0}^{\dagger})\to \H^2(\Cc_{\psi, \tau}^{u, \dagger}(T)).$ A strategy to prove the bijectivity of 
	\[ \H^2(\Cc_{\psi, \tau}^{u}(T)) \to \H^2(\Cc_{\psi, \tau}^{u, \dagger}(T)) \]
	could be to compare $\bigcap\limits_{n=0}^{\infty} \varphi^n(D_{u, \tau, 0})$ and $\bigcap\limits_{n=0}^{\infty} \varphi^n(D_{u, \tau, 0}^{\dagger}),$ or more precisely the cokernels of $\psi-1$ on these.
\end{remark}

%% file: Chapter5.tex
\chapter{Complex over the Robba ring with \texorpdfstring{$(\varphi, N_{\nabla})$}{(varphi, N nabla)}{}-modules}

In this chapter, for any Galois representation $V\in \Rep_{\QQ_p}(\mathscr{G}_K),$ we construct a three-term complex $\Cc_{\varphi, N_{\nabla}}(V),$ using its corresponding $(\varphi,N_\nabla)$-module over the Robba ring. We show (\cf proposition \ref{prop H0 H1 of complex C_R(V)}) that its $\H^i$ is isomorphic to $\iLim{n}\H^i(\mathscr{G}_{K_n},V)$ for $i\in \{0,1\},$ and we construct a pairing analogous to the one which gives rise to the Tate duality when $k$ is finite.
	
	\medskip
If $V$ is a crystalline representation with Hodge-Tate weights $\HT(V)\subset\{2,\ldots,h\}$ for some $h\in\NN_{\geq2}$, we construct similarly a three-term complex $\Cc_{\varphi, \partial_{\tau}}(V)$ from its corresponding $(\varphi,\partial_{\tau})$-module over the Robba ring, which has similar results for $\H^i, i\in \{ 0, 1\}$ (\cf proposition \ref{prop H0 H1 of complex C_R(V) partial tau}) and construct similarly a pairing when $k$ is finite.

\section{The Robba ring}

\subsection{The rings \texorpdfstring{$\widetilde{\BB}^{I}$}{BI} and \texorpdfstring{$\BB^{I}$}{BI tilde}}\label{subsection BI BI tilde}

\begin{definition}(\cf \cite[\S 2.1]{Ber02})
Let $r$ and $s$ be two elements of $\NN_{\geq 0}\big[1/p\big]\cup \{+\infty\}$ such that $r\leq s$, we put
	\begin{align*}
	\widetilde{\AA}^{[r, s]}&= \widetilde{\AA}^{+} \bigg\{ \frac{p}{[\eta]^r}, \frac{[\eta]^s}{p}  \bigg\} \\
	&=\widetilde{\AA}^{+}\big\{ X, Y  \big\}\big/ \big( [\eta]^rX-p, pY-[\eta]^s, XY-[\eta]^{s-r}\big)\\
	 \widetilde{\BB}^{[r, s]}&=\widetilde{\AA}^{[r, s]}\big[1/p\big],
	\end{align*}

 where $[x]^r$ is defined to be $[x^r]$ when $r$ is not an integer and $\widetilde{\AA}^{+}\big\{ X, Y  \big\}$ is the $p$-adic completion of $\widetilde{\AA}^{+}[X, Y].$ As convention, we put $\frac{p}{[\eta]^{+\infty}}=\frac{1}{[\eta]}$ and $\frac{[\eta]^{+\infty}}{p}=0$. \ft{Recall that $\eta=\varepsilon -1\in \Oo_{C^{\flat}}$.}
\end{definition}

\begin{remark} 
\item (1) (\cf \cite[Lemma 2.4]{Ber02})
All elements of $\widetilde{\AA}^{[r, s]}$ can be written in the form 
$$ \sum_{k\geq 0}\Bigg( \frac{p}{[\eta]^{r}} \Bigg)^k a_k + \sum_{k\geq 0} \Bigg( \frac{[\eta]^s}{p} \Bigg)^k b_k.$$
        with $(a_k), (b_k)$ two sequences in $\widetilde{\AA}^{+}$ that tend to $0$.
\item (2) (\cf \cite[Lemma 2.5]{Ber02})
For any $r_1 \leq r_2 \leq s_2 \leq s_2$, we have an injective map $\widetilde{\AA}^{[r_1, s_1]}\to \widetilde{\AA}^{[r_2, s_2]}$. 
\end{remark}

\begin{definition} (\cf \cite[\S 2.1]{Ber02})
For any interval $I$ of $\RR \cup \, \{ +\infty \}$, we define 
\[\widetilde{\BB}^{I} :=\bigcap_{[r, s]\subset I}  \widetilde{\BB}^{[r, s]}. \]
\end{definition}

\begin{remark}
\item (1) (\cf \cite[\S 2.1]{Ber02}) For any two closed intervals $I \subset J$, we define a $p$-adic valuation $V_I$ over $\widetilde{\BB}^{J}$ (notice $\widetilde{\BB}^{J}\subset \widetilde{\BB}^{I}$) by putting $V_I(x)=0$ if and only if $x\in \widetilde{\AA}^{I}\setminus p\widetilde{\AA}^{I}$ and that $\im V_I=\ZZ$. Remark that $\widetilde{\BB}^I$ is a $p$-adic Banach space and the completion of $\widetilde{\BB}^J$ with respect to $V_I$ is $\widetilde{\BB}^I$.
\item (2) (\cf \cite[\S 1.3]{Poy21})
	For $r\geq 0,$ we define a valuation $V(\cdot, r)$\index{$V(\cdot, r)$} on $\widetilde{\BB}^{+}\big[1/[\eta]\big]$ by 
	\[ V(x, r)=\inf \limits_{k\in \ZZ}\Big(k+\frac{p-1}{pr}v^{\flat}(x_k)\Big) \]
	for $x=\sum\limits_{\substack{{k\gg -\infty}}}p^k[x_k].$ If now $I$ is a closed subinterval of $[0, +\infty),$ we put $V(x, I)=\inf \limits_{\substack{{r\in I}}}V(x, r).$\index{$V(\cdot, I)$} Then $\widetilde{\BB}^{I}$\index{$\widetilde{\BB}^{I}$} is also the completion of $\widetilde{\BB}^{+}\big[1/[\eta]\big]$ for the valuation $V(\cdot, I)$ if $0\not\in I,$ and the completion of $\widetilde{\BB}^{+}$ for $V(\cdot, I)$ if $I=[0,r].$ In both cases, we can recover $\widetilde{\AA}^{I}$ as the ring of integer of $\widetilde{\BB}^{I}$ with respect to the valuation $V(\cdot, I)$.
	\item (3) (\cf \cite[Lemma 2.7]{Ber02}) If $I\subset [r, +\infty]$ then $\widetilde{\BB}^{\dagger, r}\subset \widetilde{\BB}^{I}$ and if $x\in \widetilde{\BB}^{\dagger, r}$ is written $x=\sum p^k[x_k]$, then $V_I(x)=\lfloor {V(x, I)} \rfloor$.
\end{remark}

\begin{definition}($\cf$ \cite[Definition 2.1.8]{GP21})
	Suppose $r\in \ZZ_{\geq 0}[1/p],$ and let $x=\sum\limits_{i\geq i_0}p^i[x_i]\in \widetilde{\BB}^{[r, +\infty]}$ with $i_0\in \ZZ$ and $x_i\in C^\flat$ for all $i\geq i_0$. Denote $w_k(x)=\inf \limits_{i<k}\{ v^{\flat}(x_i) \}.$\index{$w_k$} Put
	\[ W^{[s, s]}(x)=\inf \limits_{k\geq i_0} \Big\{ k, \frac{p-1}{ps}\cdot v^{\flat}(x_k) \Big\} =\inf \limits_{k\geq i_0}\Big\{ k, \frac{p-1}{ps}\cdot w_k(x)  \Big\};  \]\index{$W^{[s, s]}$}this is a well-defined valuation ($\cf$ \cite[Proposition 5.4]{Col08}). For $I\subset [r, +\infty)$ a non-empty closed interval such that $I\neq [0, 0],$ let 
	\[ W^I(x):=\inf \limits_{\alpha\in I, \alpha\neq 0}\big\{ W^{[\alpha, \alpha]}(x)  \big\}. \]\index{$W^I$}
\end{definition}

\begin{remark}
	Notice that $W^{[s, s]}$ is a variant of $p$-adic valuation, as the latter is $\inf \limits_{\{k\in \ZZ; x_k\neq 0 \}}\{k\}.$
\end{remark}

\begin{definition} (1) When $r\in \ZZ_{\geq 0}\big[1/p\big],$ we put
	\[ \AA^{[r, +\infty]}:=\AA \cap   \widetilde{\AA}^{[r, +\infty]}, \]\index{$\AA^{[r, +\infty]}$}
	\[ \BB^{[r, +\infty]}:=\BB \cap\ \widetilde{\BB}^{[r, +\infty]}. \]\index{$\BB^{[r, +\infty]}$}
	\item (2) When $r, s\in \ZZ_{\geq0}\big[1/p\big]$ and $s\neq 0,$ put $\BB^{[r, s]}$\index{$\BB^{[r, s]}$} the closure of $\BB^{[r, +\infty]}$ in $\widetilde{\BB}^{[r, s]}$ with respect to the topology given by the valuation $W^{[r, s]}.$ Put $\AA^{[r, s]}:=\BB^{[r, s]}\cap\ \widetilde{\AA}^{[r, s]},$ this is the ring of integers of $\BB^{[r, s]}.$
	\item (3) When $r\in \ZZ_{\geq 0}\big[1/p\big],$ put 
	\[ \BB^{[r, +\infty)}=\bigcap_{n\geq 0} \BB^{[r, s_n]},  \]\index{$\BB^{[r, +\infty)}$}
	where $s_n\in \ZZ_{>0}\big[1/p\big]$ and $\Lim{n\to +\infty}s_n = +\infty$ (\cf \cite[Definition 2.1.4]{GP21}).
\end{definition}

\begin{notation}($\cf$ \cite[\S 1.3]{Poy21} and \cite[Definition 3.4.6]{GP21})
	\item (1)	Let $I$ be an interval. When $\widetilde{\BB}^{I}$ (resp. $\widetilde{\AA}^{I}$) is defined, put $\widetilde{\BB}^{I}_{K_{\pi}}:=(\widetilde{\BB}^{I})^{\mathscr{G}_{K_{\pi}}}$\index{$\widetilde{\BB}^{I}_{K_{\pi}}$} (resp. $\widetilde{\AA}^{I}_{K_{\pi}}:=(\widetilde{\AA}^{I})^{\mathscr{G}_{K_{\pi}}}$).\index{$\widetilde{\AA}^{I}_{K_{\pi}}$}
	Similarly, put $\BB^{I}_{K_{\pi}}:=(\BB^{I})^{\mathscr{G}_{K_{\pi}}}$\index{$\BB^{I}_{K_{\pi}}$} (resp. $\AA^{I}_{K_{\pi}}:=(\AA^{I})^{\mathscr{G}_{K_{\pi}}}$).\index{$\AA^{I}_{K_{\pi}}$}
	\item (2) We put
	\[\widetilde{\BB}_{\rig}^{\dagger, r}=\widetilde{\BB}^{[r,+\infty)}, \quad \widetilde{\BB}^{\dagger}_{\rig}=\bigcup\limits_{r\geq 0}\widetilde{\BB}_{\rig}^{\dagger, r},  \]\index{$\widetilde{\BB}_{\rig}^{\dagger, r}$} \index{$\widetilde{\BB}^{\dagger}_{\rig}$}
	\[ \BB^{\dagger, r}_{\rig}:=\BB^{[r, +\infty)}, \quad  \BB^{\dagger}_{\rig}:=\bigcup\limits_{r\geq 0}\BB^{\dagger, r}_{\rig}, \]\index{$\BB^{\dagger, r}_{\rig}$} \index{$\BB^{\dagger}_{\rig}$}
	\[ \BB_{\rig, K_{\pi}}^{\dagger, r}=\BB_{K_{\pi}}^{ [r, +\infty )},  \quad  \BB_{\rig, K_{\pi}}^{\dagger}=\bigcup\limits_{r\geq0}\BB_{\rig, K_{\pi}}^{\dagger, r}. \]\index{$\BB_{\rig, K_{\pi}}^{\dagger, r}$} \index{$\BB_{\rig, K_{\pi}}^{\dagger}$}
	 
\end{notation}

\begin{remark}\label{rem robba algebraic def}	
	We have \[ \BB^{\dagger}_{\rig, K_{\pi}}=\bigcup\limits_{r\geq 0}\bigcap_{\substack{s>r\\ s\in \ZZ[1/p]}}\BB_{K_{\pi}}^{ [r, s ] }. \]	
\end{remark}

\subsection{Relation with Laurent series}

\begin{definition}\label{def Laurent series with T}\label{def Robba ring power series}
(1)	Let $I$ be as in the last section. We put
	\[ \BB^{I}(K_{\pi})=\Big\{ \sum\limits_{k\in\ZZ} a_kT^k ;\,  a_k\in\W(k)\big[1/p\big], \,  \big(\forall \rho\in I \big)\ \,  \Lim{|k|\to \infty}v_p(a_k)+\frac{p-1}{pe}\cdot\frac{k}{\rho} =+\infty    \Big\}.   \]\index{$\BB^{I}(K_{\pi})$}
\item (2) Let $r\geq 0$ and we put	$$ \mathcal{R}_r=\bigg\{  \sum\limits_{i\in \mathbf{Z}}a_iu^i; \, a_i\in \W(k)\big[1/p\big], \, \big(\forall  \rho\in  [p^{-r}, 1) \big) \, \Lim{i\to \pm \infty} |a_i|\rho^i=0   \bigg\}. $$\index{$\mathcal{R}_r$}
		In other words, elements of $\mathcal{R}_{r}$ are Laurent series $ \sum\limits_{i\in \mathbf{Z}}a_iu^i$ that converge on the annulus $p^{-r}\leq |u|<1.$ We define \emph{Robba ring} to be $$ \mathcal{R}=\bigcup\limits_{r>0}\mathcal{R}_r.  $$\index{$\mathcal{R}$}
	In other words, elements of $\mathcal{R}$ are Laurent series $ \sum\limits_{i\in \mathbf{Z}}a_iu^i$ that converge on the annulus $p^{-r}\leq |u|<1$ for some $r\geq0.$ 
\item (3) For any $0< \rho < 1$ and $x= \sum\limits_{i\in \mathbf{Z}}a_iu^i\in \Rr,$ we define the \emph{$\rho$-Gauss norm} over $\mathcal{R}$ as follows:
	$$  \big|  x \big|_{\rho}:=\sup_i\big\{ |a_i|\rho^i \big\}.  $$\index{$\big|  \cdot  \big|_{\rho}$}
\end{definition}

\begin{remark}
	\item (1) The ring $\mathcal{R}_r$ carries a Fréchet topology, in which a sequence converges if and only if it converges under the $\rho$-Gauss norm for all $\rho \in {[}p^{-r}, 1{)}.$ (It is complete for this topology.)
	\item (2) The ring $\mathcal{R}$ carries a \emph{limit-of-Fréchet topology}, or \emph{LF topology}. This topology is defined on $\mathcal{R}$ by taking the locally convex direct limit of the $\mathcal{R}_{r}$ (each equipped with the Fréchet topology). In particular, a sequence converges in $\mathcal{R}$ if it is a convergent sequence in $\mathcal{R}_r$ for some $r>0.$
\end{remark}

\begin{proposition}
	If the endpoints of an interval $I$ lie in $\ZZ_{>0}\big[1/p\big],$ the map $T\mapsto u$ induces an isomorphism 
	\[ \BB^{I}(K_{\pi})\simeq \BB_{K_{\pi}}^{I}  \]
	hence we have an equality $\Rr=\BB^{\dagger}_{\rig, K_{\pi}}$ by remark \ref{rem robba algebraic def}. 
\end{proposition}
\begin{proof}
$\cf$ \cite[Lemma 2.2.7 (2)]{GP21}.
\end{proof}

\begin{notation}($\cf$ \cite[Proposition 2.2.5]{Poy21})\label{not R-tau}	
	We put $\Rr_{\tau}=(\widetilde{\BB}^{\dagger}_{\rig})^{\mathscr{G}_L}=\widetilde{\BB}^{\dagger}_{\rig, L}.$ \index{$\Rr_{\tau}$} Notice that we 
	have 
		$$\BB_{K_{\pi}}^\dagger\subset\mathcal{R}\subset\mathcal{R}_\tau\subset\widetilde{\BB}_{\rig}^\dagger.$$
\end{notation}

\section{The \texorpdfstring{$(\varphi, \tau)$}{(phi,tau)}-modules over \texorpdfstring{$(\Rr, \Rr_{\tau})$}{(R,R\texttinferior\textainferior\textuinferior)}}

\begin{definition}\label{def R,R-rau}
	A \emph{$(\varphi, \tau)$-module over $(\Rr, \Rr_{\tau})$} consists of
	\item (i) an \'etale $\varphi$-module $D^{\dagger}_{\rig}$ over $\Rr$;
	\item (ii) a $\tau$-semi-linear endomorphism $\tau_D$ on $D^{\dagger}_{\rig,\tau}:=\Rr_{\tau}\otimes_{\Rr} D^{\dagger}_{\rig}$ which commutes with $\varphi_{\Rr_{\tau}}\otimes \varphi_{D^{\dagger}_{\rig}}$ (where $\varphi_{\Rr_{\tau}}$ is the Frobenius map on $\Rr_{\tau}$ and $\varphi_{D^{\dagger}_{\rig}}$ is the Frobenius map on $D^{\dagger}_{\rig}$) and which satisfies:
	\[ (\forall x\in D^{\dagger}_{\rig})\quad  (g\otimes 1)\circ \tau_D(x) = \tau_D^{\chi(g)}(x), \]
	for all $g\in \mathscr{G}_{K_{\pi}}/\mathscr{G}_L$ such that $\chi(g)\in \ZZ_{>0}.$
	The corresponding category is denoted $\Mod_{\Rr, \Rr_{\tau}}(\varphi, \tau).$ \index{$\Mod_{\Rr, \Rr_{\tau}}(\varphi, \tau)$}
\end{definition}

\begin{theorem}\label{thm D dagger rig}
		The functors 
		\begin{align*}
		\mathcal{D}^{\dagger}_{\rig}\colon	\Rep_{\QQ_p}(\mathscr{G}_K) &\to \Mod_{\Rr, \Rr_{\tau}}(\varphi, \tau)\\
		V   & \mapsto \mathcal{D}^{\dagger}_{\rig}(V)=\Rr \otimes_{\Ee^{\dagger}} \mathcal{D}^{\dagger}(V)\\
		\mathcal{V}(D^{\dagger}_{\rig})=(\BB_{\rig}^{\dagger} \otimes_{\Rr} D_{\rig}^{\dagger})^{\varphi=1}& \mapsfrom D^{\dagger}_{\rig}
		\end{align*}
		$($with the natural $\tau$-semi-linear endomorphism $\tau_D$ over $\mathcal{D}^{\dagger}_{\rig}(V)_{\tau}=\Rr_{\tau}\otimes_{\Rr} \mathcal{D}^{\dagger}_{\rig}(V))$ establish quasi-inverse equivalences of categories.
\end{theorem}

\begin{proof}
	By theorem \ref{prop equivalence cat overconvergent: normal case}, it is enough to show that the functor
	\begin{align*}
		\Mod_{\Ee^{\dagger}, \Ee^{\dagger}_{\tau}}(\varphi, \tau) &\to \Mod_{\Rr, \Rr_{\tau}}(\varphi, \tau)\\
		(D^{\dagger}, D_{\tau}^{\dagger}) &\mapsto (\Rr\otimes_{\Ee^{\dagger}}D^{\dagger}, \Rr_{\tau}\otimes_{\Ee^{\dagger}}D^{\dagger})
	\end{align*}
	is an equivalence of category. By \cite[Theorem 6.33]{Ked05}, the functor
	\begin{align*}
		\Mod_{\Ee^{\dagger}}^{\text{\'et}}(\varphi) &\xrightarrow{(\ast)} \Mod_{\Rr}^{\text{\'et}}(\varphi)\\
		D^{\dagger} &\mapsto \Rr\otimes_{\Ee^{\dagger}}D^{\dagger}
	\end{align*}
	is an equivalence of categories, where $\Mod_{\Ee^{\dagger}}^{\text{\'et}}(\varphi)$ is the category of \'etale $\varphi$-modules over $\Ee^{\dagger}$ and $\Mod_{\Rr}^{\text{\'et}}(\varphi)$ is the category of \'etale $\varphi$-modules over $\Rr$ ($\ie$pure of slope $0$). Let $(D_1^{\dagger}, D_{1, \tau}^{\dagger}), (D_2^{\dagger}, D_{2, \tau}^{\dagger})\in \Mod_{\Ee^{\dagger}, \Ee^{\dagger}_{\tau}}(\varphi, \tau)$ and 
	\[f: \Rr\otimes_{\Ee^{\dagger}}D_1^{\dagger} \to \Rr\otimes_{\Ee^{\dagger}}D_2^{\dagger}\] be a morphism in $\Mod_{\Rr, \Rr_{\tau}}(\varphi, \tau).$ By fully-faithfullness of $(\ast),$ $f$ comes from a unique map $f_0: D_1^{\dagger}\to D_2^{\dagger}$ in $\Mod_{\Ee^{\dagger}, \Ee^{\dagger}_{\tau}}(\varphi, \tau)$ by extension of scalars. The latter induces a map $1\otimes f_0: D_{1, \tau}^{\dagger}\to D_{2, \tau}^{\dagger}$ that inserts in the commutative square:
	\[ \xymatrix{ D_{1, \tau}^{\dagger} \ar[r]^{1\otimes f_0}\ar@{^(->}[d] & D_{2, \tau}^{\dagger}\ar@{^(->}[d]\\
		\Rr_{\tau}\otimes_{\Ee^{\dagger}}D_{1}^{\dagger} \ar[r]^{1\otimes f}& 	\Rr_{\tau}\otimes_{\Ee^{\dagger}}D_{2}^{\dagger},
	}  \]
which shows that $1\otimes f_0$ is compatible with the operator $\tau_{D_1}$ and $\tau_{D_2},$ $\ie$that $f_0$ is a morphism in $\Mod_{\Ee^{\dagger}, \Ee^{\dagger}_{\tau}}(\varphi, \tau).$ 

\medskip

Let $(D_{\rig}^{\dagger}, D_{\rig, \tau}^{\dagger})\in \Mod_{\Rr, \Rr_{\tau}}(\varphi, \tau).$ By the essential surjectivity of $(\ast),$ there exists $D^{\dagger}\in\Mod_{\Ee^{\dagger}}^{\text{\'et}}(\varphi),$ such that $D_{\rig}^{\dagger}\simeq \Rr\otimes_{\Ee^{\dagger}}D^{\dagger}$ as a $\varphi$-module. We have 
\[D_{\tau}^{\dagger}=\Ee_{\tau}^{\dagger}\otimes_{\Ee^{\dagger}} D^{\dagger} \subset D_{\rig, \tau}^{\dagger}.  \]
Notice that $D_{\tau}^{\dagger}$ is the unique $\varphi$-stable $\Ee_{\tau}^{\dagger}$-lattice of $D_{\rig, \tau}^{\dagger}$ (as $\Rr_{\tau}$ is faithfully flat over $\Ee^{\dagger}$): this implies that the $\tau_D$ map over $D_{\rig, \tau}^{\dagger}$ maps $D_{\tau}^{\dagger}$ into itself, so that $(D^{\dagger}, D_{\tau}^{\dagger})\in \Mod_{\Ee^{\dagger}, \Ee^{\dagger}_{\tau}}(\varphi, \tau).$ 
\end{proof}

\begin{remark}
		We can define similarly the category $\Mod_{\Rr, \Rr_{\tau}}(\varphi, \tau^{p^r})$\index{$\Mod_{\Rr, \Rr_{\tau}}(\varphi, \tau^{p^r})$} for $r\in \NN.$ Notice that for any $(D_{\rig}^{\dagger}, D^{\dagger}_{\rig, \tau})\in \Mod_{\Rr, \Rr_{\tau}}(\varphi, \tau^{p^r})$, there is a $\tau^{p^r}$-semilinear endomorphism $\tau_D^{p^r}$ on $D^{\dagger}_{\rig,\tau}$ rather than a $\tau$-semilinear endomorphism $\tau_D$ (which is the case in $\Mod_{\Rr, \Rr_{\tau}}(\varphi, \tau)).$
\end{remark}

\begin{definition}(\cf \cite[\S 1.2]{Poy21}) Let $G$ be a $p$-adic Lie group and let $W$ be a Banach $\QQ_p$-representation of $G.$ Let $H$ be an open subgroup of $G$ that admits a system of coordinates that induces an analytic bijection $H\to \ZZ_p^{d}.$
\item (1) We say that $x\in W$ is \emph{$H$-analytic} if the orbit map $H\to W; \, g\mapsto g(x)$ is analytic (\cf \cite[\S2]{Ber16}). 
\item (2) We say that $x\in W$ is \emph{locally analytic} $($for $G)$, if the orbit map $G\to W; \, g\mapsto g(x)$ is locally analytic, $\ie$there exists some open $H\leqslant G$ such that $x$ is $H$-analytic. We denote by $W^{\textsf{la}}$\index{$W^{\textsf{la}}$} the set of locally analytic elements in $W.$
\item (3) Let $W$ be a Fr\'echet space, whose topology is defined by a sequence of semi-norms $(p_i)_{i\geq 1}.$ We then denote $W_i$ the completion of $W$ for $p_i,$ hence we have $W=\pLim{i\geq 1}W_i.$ We say that $w\in W$ is \emph{pro-analytic} if its image $\pi_i(w)\in W_i$ is locally analytic for all $i.$ We denote by $W^{\textsf{pa}}$\index{$W^{\textsf{pa}}$} the set of pro-analytic elements in $W.$
\end{definition}

\begin{notation} 
\item (1)
($\cf$ \cite{Kis06})\label{nota O}
	Consider the open rigid analytic disc $D{[} 0, 1 {)}$ over $\W(k)\big[1/p \big]$ of outer radius 1 with coordinate $u.$ For any interval $I$ in  ${[} 0, 1 {)},$ we consider the admissible open subspace $D(I)\subset D{[} 0, 1 {)}$. We denote 
	\[ \Oo_{I}=\Gamma (D(I), \Oo_{D(I)} ),\]
	 the ring of rigid analytic functions. Notice this is a $\W(k)\big[1/p \big]$-subalgebra of $\W(k)\big[1/p \big][\![u]\!]$. We put \[ \Oo=\Oo_{{[} 0, 1 {)}},\] \index{$\Oo$}\index{$\Oo_{[0, 1)}$} the ring of rigid analytic functions on the disc. Remark that we have $\Oo_{[0, 1)}=\BB_{K_{\pi}}^{[0, +\infty)}$ (\cf \cite[\S 3.4]{GP21}).
\item (2) ($\cf$ \cite{Kis06}) Let $E(u)\in \W(k)[u]$ be the Eisenstein polynomial of $\pi.$ We put 
	\[ c:=p\frac{E(u)}{E(0)}\in \W(k)[u]   \]\index{$c$}
	\[\lambda:=\prod_{n=0}^{\infty}\varphi^n(E(u)/E(0))\in \Oo \]\index{$\lambda$}
	\[t=\log([\varepsilon]):=\sum\limits_{n=1}^{\infty}(-1)^{n-1}\frac{([\varepsilon]-1)^{n}}{n}\in \widetilde{\BB}^{\dagger}_{\rig, L}.\]\index{$t$}
\item (3) For any locally analytic $\mathscr{G}_{K_{\pi}}$-representation ($\cf$ \cite[D\'efinition 1.2.12]{Poy21}), we define the following operators:
	\[ \log (\tau^{p^n}) := -\sum \limits_ {k = 1}^\infty \frac{(1-\tau^{p^n})^k} {k}, \text{ for } n\in \NN \]\index{$\log (\tau^{p^r})$}
	\[ \nabla_\tau := \frac {1} {p^n} \log (\tau^{p^n}), \text{ for } n\gg 0. \] \index{$\nabla_\tau$}
	Remark that the operator $\nabla_{\tau}$ is well-defined on $(\widetilde{\BB}^{\dagger}_{\rig, L})^{\text{pa}}$ (\cf \cite[\S 2.2]{Poy21}). 
\item (4) (\cf \cite[\S 2.2]{Poy21}) We put 
	\[ N_{\nabla}:=\frac{-\lambda}{t}\nabla_\tau. \]\index{$N_{\nabla}$}
	It is showed in \textit{loc. cit.} that $N_{\nabla}$ induces an operator on $\Rr$ (whereas $\nabla_{\tau}$ does not). Remark that later we will extend it to certain $\Rr$-modules, and then to $\mathcal{D}^{\dagger}_{\rig}(V)$ for $V\in \Rep_{\QQ_p}(\mathscr{G}_K).$  
\end{notation}

\begin{remark} (1)  We have $\lambda=\frac{E(u)}{E(0)}\varphi(\lambda).$
	\item (2) The element $c$ we just defined is different from that in \cite{Car13}, there $c=\frac{E(0)}{p}\in \W(k)^{\times}.$
\item (3)
	For any $n\in \NN,$ the element $\prod\limits_{i=0}^n\varphi^i(\frac{E(u)}{E(0)})$ is invertible in $\Rr.$
\begin{proof}
	Notice that $\frac{E(u)}{E(0)}=\frac{E(u)}{pa}$ with $a\in W(k)^{\times}:$ it is hence in $\Ee^{\dagger}.$ Hence $\frac{E(u)}{E(0)}\in \Rr$ is a unit of the Robba ring (\cf \cite[diagram (10.4.3)]{BC09}), thus $\varphi^{n}(\frac{E(u)}{E(0)})$ are units of the Robba ring  for all $n\in\NN,$ and so is $\prod\limits_{i=0}^n\varphi^i(\frac{E(u)}{E(0)}).$
\end{proof}
\item (4) \label{rmq N nabla }($\cf$ \cite[\S 2.2]{Poy21})
	We have $N_{\nabla}=\frac{-\lambda}{t}\nabla_\tau=-u\lambda \frac{d}{du}$ as operators over $\Rr:$ the operator $N_{\nabla}$ we defined coincides with that in \cite[\S  1]{Kis06}.
	\begin{proof}
		By direct computation we have 
		\[ -u\lambda \frac{d}{du}(u^n)=-n\lambda u^n. \]
	On the other hand $\exp(p^r\nabla_{\tau})(u^n)=\tau^{p^r}(u^n)$ for $r\gg 0$ and $\tau^{p^r}(u^n)=[\varepsilon]^{np^r}u^n.$ Using the expansion $\exp(p^r\nabla_{\tau})(u^n)=u^n+p^r\nabla_{\tau}(u^n)+O(p^{2r})$ for $r\gg 0,$ we have  
		\begin{align*}
		\nabla_{\tau}(u^n)&=\Lim{r\to \infty}\frac{[\varepsilon]^{np^r}-1}{p^r}u^n\\
		&=\Lim{r\to \infty}\frac{e^{tnp^r}-1}{p^r}u^n\\
		&=ntu^n.
		\end{align*}
		Hence we have 
		\[	
		\frac{-\lambda}{t}\nabla_\tau(u^n)=-\frac{\lambda}{t}ntu^n=-n\lambda u^n.
		\]
		This finishes the proof.
	\end{proof} 
\end{remark}

\section{The complex \texorpdfstring{$\Cc_{\varphi, N_\nabla}$}{C \textpinferior\texthinferior\textiinferior, N\textninferior\textainferior}}

\begin{definition}(\cf \cite[2.2.1]{Poy21})\label{def N nabla module}
	We define a \emph{$(\varphi, N_{\nabla})$-module over $\Rr$} to be a free $\Rr$-module $D$ endowed with a Frobenius map $\varphi$ and a connection $N_{\nabla}\colon D_{}\to D$ over $N_{\nabla}\colon \Rr\to \Rr,$ $\ie$an additive map such that 
	\[(\forall m\in D)\ (\forall x\in \Rr)\quad  N_{\nabla}(x\cdot m)=N_{\nabla}(x)\cdot m+x\cdot N_{\nabla}(m),\]
	that satisfies $N_{\nabla}\circ \varphi= c\varphi N_{\nabla}.$ The corresponding category is denoted $\Mod_{\Rr}(\varphi, N_{\nabla}).$\index{$\Mod_{\Rr}(\varphi, N_{\nabla})$}
\end{definition}

\begin{proposition}\textup{(\cf \cite[Proposition 2.2.2]{Poy21})}
	Let $V\in \Rep_{\QQ_p}(\mathscr{G}_K),$ then 
	\[ N_{\nabla}(\mathcal{D}^{\dagger}_{\rig}(V))\subset \mathcal{D}^{\dagger}_{\rig}(V).  \]
\end{proposition}

This implies that if $V$ is in $\Rep_{\QQ_p}(\mathscr{G}_K),$ the operator $N_{\nabla}$ associated to its $(\varphi, \tau)$-module $\mathcal{D}^{\dagger}_{\rig}(V)$ provides a $(\varphi, N_{\nabla})$-module structure.

\begin{definition}\label{def complex N nabla 1st}
		Let $D_{\rig}^{\dagger}\in \Mod_{\Rr}(\varphi, N_{\nabla}),$ we define the complex $\Cc_{\varphi, N_{\nabla}}(D_{\rig}^{\dagger})$\index{$\Cc_{\varphi, N_{\nabla}}(D_{\rig}^{\dagger})$} as follows:	
		\[ \xymatrix{
			0\ar[rr] && D_{\rig}^{\dagger} \ar[r] & D_{\rig}^{\dagger} \oplus D_{\rig}^{\dagger} \ar[r] & D_{\rig}^{\dagger} \ar[r] & 0   \\
			&& x  \ar@{|->}[r] & ((\varphi-1)(x), N_{\nabla}(x))  &{}\\
			&&& {} (y,z) \ar@{|->}[r] & N_{\nabla}(y)-(c\varphi-1)(z).
		}    \]	
		If $V\in \Rep_{\QQ_p}(\mathscr{G}_K),$  we have in particular the complex $\Cc_{\varphi, N_{\nabla}}(\mathcal{D}^{\dagger}_{\rig}(V)),$ which will also be simply denoted  $\Cc_{\varphi, N_{\nabla}}(V).$\index{$\Cc_{\varphi, N_{\nabla}}(V)$}
\end{definition}

\begin{remark}
	The complex $\Cc_{\varphi, N_{\nabla}}(V)$ above is well-defined. Indeed we have 
	\[ \Big(N_{\nabla}\ominus(c\varphi-1)\Big)\circ (\varphi-1, N_{\nabla})=N_{\nabla}\circ (\varphi-1) - (c\varphi-1)\circ N_{\nabla}=0, \]
	as	$N_{\nabla}\circ \varphi= c\varphi N_{\nabla}$ by definition of a $(\varphi, N_{\nabla})$-module over $\Rr.$
\end{remark}

\begin{lemma}\label{lemm H1 Ext N nabla}
	Let $V\in \Rep_{\QQ_p}(\mathscr{G}_K)$ and $D=\mathcal{D}^{\dagger}_{\rig}(V).$ We have $\Ext^1_{\Mod_{\Rr}(\varphi, N_{\nabla})}(D, \Rr) \simeq \H^1(\Cc_{\varphi, N_{\nabla}}(D)).$
\end{lemma}

\begin{proof}
	To give an extension
	\[ 0\to D \to E \xrightarrow{\epsilon} \Rr \to 0 \]
	of $\mathcal{R}$ by $D$ in $\Mod_{\mathcal{R}}(\varphi,N_\nabla)$ is equivalent to giving a $(\varphi, N_{\nabla})$-module structure to the $\Rr$-module 
	\[E=D\oplus \Rr\cdot e, \]
	where $e\in E$ is a preimage of 1 under $\epsilon.$ Since $D$ is already a $(\varphi, N_{\nabla})$-module, it suffices to specify the image of $e$ by $\varphi$ and $N_{\nabla}.$ Since $\epsilon(\varphi(e)-e)=0,$ we must have $\varphi(e)-e\in D$ and write: 
	\[ \varphi(e)=e+\lambda \text{ with } \lambda\in D.\]
	For $N_{\nabla},$ we have $\epsilon(N_{\nabla}(e))=N_{\nabla}(1)=0:$ put $\mu=N_\nabla(e)\in D$ (this completely defines $N_\nabla$ on $E$).
	In order to have an extension, the only condition that $(\lambda, \mu)\in D\oplus D$ has to satisfy is $N_{\nabla}\circ\varphi(e)=c\varphi\circ N_{\nabla}(e),$ which by construction is $N_{\nabla}(\lambda)+\mu=c\varphi(\mu), \ie$
	\[N_{\nabla}(\lambda)=(c\varphi-1)(\mu).\]
	Notice that the submodule of $D\oplus D$
	\[M:=\big\{ (\lambda, \mu)\in D\oplus D;  N_{\nabla}(\lambda)=(c\varphi-1)(\mu) \big\} \]
	is exactly $\Ker \beta$ in the following complex (\cf definition \ref{def complex N nabla 1st})
	\[\Cc_{\varphi, N_{\nabla}}(D) \ \colon  0\to D \xrightarrow{\alpha} D\oplus D \xrightarrow{\beta} D \to 0. \]
	One checks that $(\lambda, \mu)$ corresponds to the trivial extension if and only if it lies in the $\im \alpha,$ $\ie$there exists $d\in D$ such that $\lambda= (\varphi-1)(d)$ and $\mu= N_{\nabla}(d).$ We define another submodule of $M$  \[N:= \Big\{ \big((\varphi-1)d, (\tau_D-1)d\big);\ d\in D \Big\}. \]
	Then there is an isomorphism  \[ \Ext(\Oo_{\Ee}, D)\simeq M/N \simeq \Ker \beta/ \im \alpha \simeq \H^1(\Cc_{\varphi, N_{\nabla}}(D)). \]
\end{proof}

\begin{notation}
	Let $(D_{\rig}^{\dagger}, D_{\rig, \tau}^{\dagger})\in \Mod_{\Rr, \Rr_{\tau}}(\varphi, \tau).$  We put
	\[ D^{\dagger}_{\rig, \tau^{p^n},0}=\big\{  x\in D^{\dagger}_{\rig,\tau} ;  (\gamma\otimes 1)x=(1+\tau_D^{p^n}+ \tau_D^{2p^n}\cdots +\tau_D^{(\chi(\gamma)-1)p^n})(x) \big \}  \]
	for any $n\in \NN.$\index{$D^{\dagger}_{\rig, \tau^{p^n},0}$}
\end{notation}

\begin{definition}\label{def complex phi tau n modified with R tau} 
Let $n\in \NN$ and $(D^{\dagger}_{\rig}, D^{\dagger}_{\rig, \tau})\in \Mod_{\Rr, \Rr_{\tau}}(\varphi, \tau^{p^n}).$ We define the complex $\mathcal{C}_{\varphi, \tau^{p^n}}^{\rig}(D^{\dagger}_{\rig})$\index{$\mathcal{C}_{\varphi, \tau^{p^n}}^{\rig}(D^{\dagger}_{\rig})$} as follows:	\[ \xymatrix{
			0\ar[rr] && D^{\dagger}_{\rig} \ar[r]  & D^{\dagger}_{\rig} \oplus D^{\dagger}_{\rig,\tau^{p^n}, 0}  \ar[r] &  D^{\dagger}_{\rig,\tau^{p^n}, 0} \ar[r] & 0   \\
			&&x  \ar@{|->}[r] &((\varphi-1)(x), (\tau_D^{p^n}-1)(x))  &{}\\
			&&&{} (y,z) \ar@{|->}[r] & (\tau_D^{p^n}-1)(y)-(\varphi-1)(z).
		} \]
		If $V\in \Rep_{\QQ_p}(\mathscr{G}_K),$  we have in particular the complex $\mathcal{C}_{\varphi, \tau^{p^n}}^{\rig}(\mathcal{D}^{\dagger}_{\rig}(V)),$ which will also be simply denoted  $\mathcal{C}_{\varphi, \tau^{p^n}}^{\rig}(V).$\index{$\mathcal{C}_{\varphi, \tau^{p^n}}^{\rig}(V)$}	
\end{definition}

\begin{lemma}\label{lemm complex C R r computes H0 H1 }
	Let $r\in \NN$ and $V\in \Rep_{\QQ_p}(\mathscr{G}_K),$ then $\H^i(\mathcal{C}^{\rig}_{\varphi, \tau^{p^r}}(V))\simeq \H^i(\mathscr{G}_{K_r}, V)$ for $i\in \{0, 1 \}.$
\end{lemma}

\begin{proof}
	For $i=0$ this follows by direct computation. For $i=1,$ similar argument as in proposition \ref{prop Caruso H1 Zp} works, as we have an equivalence of categories between $\Mod_{\Rr, \Rr_{\tau}}(\varphi, \tau^{p^r})$ and $\Rep_{\QQ_p}(\mathscr{G}_{K_r})$ (the proof is similar as that of theorem \ref{thm D dagger rig}). 
\end{proof}

\begin{proposition}\label{prop N Nabla connection R-tau}
	Let $D\in \Mod_{\Rr}(\varphi, N_{\nabla}).$ Then there exists $n\geq 0$ and $D^{\dagger}_{\rig}\in \Mod_{\Rr, \Rr_{\tau}}(\varphi, \tau^{p^n})$ such that $D\simeq D^{\dagger}_{\rig}$ as $(\varphi, N_{\nabla})$-modules over $\Rr.$
\end{proposition}

\begin{proof}
	$\cf$ \cite[Proposition 2.2.5]{Poy21}.
\end{proof}

\begin{lemma}\label{lemm Ext 1 nabla bij limit Ext 1}
		Let $V\in \Rep_{\QQ_p}(\mathscr{G}_K).$  We have an isomorphism of groups 
		\[\iLim{n} \Ext^1_{\Mod_{\mathcal{R},\mathcal{R}_\tau}(\varphi, \tau^{p^n})}(\mathcal{D}^{\dagger}_{\rig}(V), \mathcal{R}) \to \Ext^1_{\Mod_{\Rr}(\varphi, N_{\nabla})}(\mathcal{D}_{\rig}^{\dagger}(V), \mathcal{R}).\]
\end{lemma}

\begin{proof}
		Let $V\in \Rep_{\QQ_p}(\mathscr{G}_K),$ and we associate its $(\varphi, N_{\nabla})$-module $\mathcal{D}_{\rig}^{\dagger}(V)$ and $(\varphi, \tau^{p^n})$-module $\mathcal{D}^{\dagger}_{\rig}(V)_n$ (notice that $\mathcal{D}^{\dagger}_{\rig}(V)_n=\mathcal{D}^{\dagger}_{\rig}(V)$ as $\varphi$-modules, here we just use the subscript $n$ to indicate it is an object of $\Mod_{\Rr, \Rr_{\tau}}(\varphi, \tau^{p^n})$). If $n\in\NN,$ an extension of its associated $(\varphi,\tau^{p^n})$-module over $\mathcal{R}$ is naturally equipped with a $(\varphi,N_\nabla)$-module structure. This provides a map from $\Ext^1_{\Mod_{\mathcal{R},\mathcal{R}_\tau}(\varphi,\tau^{p^n})}(\mathcal{D}(V)_n,\mathcal{R})$ to $\Ext^1_{\Mod_{\mathcal{R}}(\varphi,N_\nabla)}(\mathcal{D}^{\dagger}_{\rig}(V),\mathcal{R}).$ Those maps are compatible as $n$ grows: this provides the map \[\iLim{n}\Ext^1_{\Mod_{\mathcal{R},\mathcal{R}_\tau}(\varphi,\tau^{p^n})}(\mathcal{D}^{\dagger}_{\rig}(V)_n,\mathcal{R})\to \Ext^1_{\Mod_{\mathcal{R}}(\varphi,N_\nabla)}(\mathcal{D}_{\rig}^{\dagger}(V),\mathcal{R}).\]
		An extension $E$ of $\mathcal{D}_{\rig}^{\dagger}(V)$ in $\Mod_{\mathcal{R}}(\varphi, N_{\nabla})$ has a $(\varphi, \tau^{p^n})$-module structure for some $n\gg 0$ by proposition \ref{prop N Nabla connection R-tau}. This shows that the extension $E$ comes from an extension of $\mathcal{D}^{\dagger}_{\rig}(V)$ in the category of $\Mod_{\mathcal{R}, \Rr_{\tau}}(\varphi, \tau^{p^n})$ for some $n\gg 0,$ which implies the map is surjective. Moreover, it is also injective by \cite[remark 2.2.4]{Poy21}. 
\end{proof}

\begin{lemma}
Let $n\in \NN$ and $(D_{\rig}^{\dagger}, D_{\rig, \tau}^{\dagger})\in \Mod_{\Rr, \Rr_{\tau}}(\varphi, \tau).$ We have an operator \[\frac{N_{\nabla}}{\tau_D^{p^n}-1}: (D_{\rig, \tau}^{\dagger})^{\pa}\to (D_{\rig, \tau}^{\dagger})^{\pa},\]
such that $N_{\nabla}=\frac{N_{\nabla}}{\tau_D^{p^n}-1}\circ (\tau_D^{p^n}-1)=(\tau_D^{p^n}-1)\circ \frac{N_{\nabla}}{\tau_D^{p^n}-1}$ on $(D_{\rig, \tau}^{\dagger})^{\pa}.$
\end{lemma}

\begin{proof}
We have

\begin{align*}
\frac{N_{\nabla}}{\tau_D^{p^n}-1}&= -\frac{\lambda}{t}\cdot \frac{\nabla_{\tau}}{\tau_D^{p^n}-1}\\
                                 &= -\frac{\lambda}{p^m t}\cdot \frac{\log(\tau_D^{p^m})}{\tau_D^{p^n}-1} \quad (\text{for } m\gg 0)   \\
                                 &= \frac{\lambda}{p^mt}\cdot \sum\limits_{i=1}^{\infty}\frac{(1-\tau_D^{p^m})^{i}}{i(\tau_D^{p^n}-1)}\\
  &= -\frac{\lambda}{p^m t}\cdot \sum\limits_{i=1}^{\infty}\frac{(1-\tau_D^{p^m})^{i-1}}{i}\cdot\bigg( \frac{\tau_D^{p^m}-1}{\tau_D^{p^n}-1} \bigg) \quad (\text{we can assume } n|m)\\
  &= -\frac{\lambda}{p^m t}\cdot \sum\limits_{i=1}^{\infty}\frac{(1-\tau_D^{p^m})^{i-1}}{i}\cdot (1+\tau_D^{p^{n}}+\cdots + \tau_D^{(\frac{m}{n}-1)p^n})
\end{align*}
which shows it is well defined over $(D_{\rig, \tau}^{\dagger})^{\pa}$ ($\cf$ \cite[\S 2.2]{Poy21}).
\end{proof}

\begin{definition}
Let $(D_{\rig}^{\dagger}, D_{\rig, \tau}^{\dagger})\in \Mod_{\Rr, \Rr_{\tau}}(\varphi, \tau).$ For any $n\in \NN,$ we put 
\[D_{\rig, \tau^{p^n}, 0}^{\dagger, \pa}=D_{\rig, \tau^{p^n}, 0}^{\dagger} \cap (D_{\rig, \tau}^{\dagger})^{\pa}.\]\index{$D_{\rig, \tau^{p^n}, 0}^{\dagger, \pa}$}
We then define a complex 
	\[ \Cc^{\pa}_{\varphi, \tau^{p^n}}(D^{\dagger}_{\rig})\colon \quad	0\to D^{\dagger}_{\rig} \xrightarrow{(\varphi-1, \tau_D^{p^n}-1)} D^{\dagger}_{\rig} \oplus D^{\dagger, \pa}_{\rig, \tau^{p^n}, 0} \xrightarrow{(\tau_D^{p^n}-1)\ominus(\varphi-1)}  D^{\dagger, \pa}_{\rig, \tau^{p^n}, 0}  \to  0.\]\index{$\Cc^{\pa}_{\varphi, \tau^{p^n}}(D^{\dagger}_{\rig})$}
If $V\in \Rep_{\QQ_p}(\mathscr{G}_{K}),$ we have in particular the complex $\Cc^{\pa}_{\varphi, \tau^{p^n}}(\mathcal{D}(V)^{\dagger}_{\rig}),$ which will also be simply denoted $\Cc^{\pa}_{\varphi, \tau^{p^n}}(V).$\index{$\Cc^{\pa}_{\varphi, \tau^{p^n}}(V)$}
	
\end{definition}

\begin{remark}
Notice that $D^{\dagger}_{\rig}\subset (D_{\rig, \tau}^{\dagger})^{\pa}$ ($\cf$ \cite[Lemma 1.3.4]{Poy21}), and the complex $\Cc^{\pa}_{\varphi, \tau^{p^n}}(D^{\dagger}_{\rig})$ is well-defined as $\varphi$ and $\tau_D$ preserve pro-analyticity.
\end{remark}

\begin{lemma}\label{lemm gamma commutes N nabla}
We have $\gamma\circ N_{\nabla}=N_{\nabla}\circ \gamma.$
\end{lemma}

\begin{proof}
We have $\gamma\circ\tau=\tau^{\chi(\gamma)}\circ\gamma,$ hence $\gamma\circ\tau^{p^m}=\tau^{\chi(\gamma)p^m}\circ\gamma.$ We have
\begin{align*}
	\gamma\circ\log(\tau^{p^m})&=\log(\tau^{\chi(\gamma)p^m})\circ\gamma\\
	                      &=\chi(\gamma)\log(\tau^{p^m})\circ\gamma,
\end{align*}
hence $\gamma\circ\nabla_{\tau}=\chi(\gamma)\nabla_{\tau}\circ\gamma.$ We then have
\begin{align*}
	\gamma\circ N_{\nabla} &= \gamma \big( -\frac{\lambda}{t}\nabla_{\tau} \big)\\
                           &= -\frac{\lambda}{\gamma(t)}\gamma\circ \nabla_{\tau}\\
                           &= -\frac{\lambda}{\chi(\gamma)t}\chi(\gamma)\nabla_{\tau}\circ\gamma\\
                           &=N_{\nabla}\circ \gamma.	
\end{align*}
\end{proof}

\begin{proposition}\label{prop N over tau-1}
Let $n\in \NN,$ then the operator $\frac{N_{\nabla}}{\tau_D^{p^n}-1}$ induces a map $\frac{N_{\nabla}}{\tau_D^{p^n}-1}: D_{\rig, \tau^{p^n}, 0}^{\dagger, \pa}\to D_{\rig}^{\dagger}.$
\end{proposition}

\begin{proof}
Let $x\in D_{\rig, \tau^{p^n}, 0}^{\dagger, \pa},$ then by lemma \ref{lemm gamma commutes N nabla} we have 
\begin{align*}
	\gamma\bigg(\frac{N_{\nabla}}{\tau_D^{p^n}-1}(x)  \bigg)&=\frac{N_{\nabla}}{\tau_D^{p^n\chi(\gamma)}-1}(\gamma(x))\\
	                                                        &=\frac{N_{\nabla}}{\tau_D^{p^n\chi(\gamma)}-1}\big(1+\tau_D^{p^n} +\cdots + \tau_D^{p^n(\chi(\gamma)-1)} \big)(x)\\
	                                                      &=\frac{N_{\nabla}}{\tau_D^{p^n\chi(\gamma)}-1}\cdot \frac{\tau_D^{p^n\chi(\gamma)}-1}{\tau_D^{p^n}-1}(x)\\
	                                                      &= \frac{N_{\nabla}}{\tau_D^{p^n}-1}(x).
\end{align*}
Hence $\gamma\Big(\frac{N_{\nabla}}{\tau_D^{p^n}-1}(x)  \Big)\in (D_{\rig, \tau}^{\dagger})^{\gamma=1}=D_{\rig}^{\dagger}.$
\end{proof}

\begin{remark}
	Let $V\in \Rep_{\QQ_p}(\mathscr{G}_K)$ and $(D^{\dagger}_{\rig}, D^{\dagger}_{\rig, \tau})$ be its $(\varphi, \tau^{p^n})$-module over $(\Rr, \Rr_{\tau}),$ then in particular $D^{\dagger}_{\rig}$ is equipped with a $(\varphi, N_{\nabla})$-module structure. When $n\gg 0,$ we have the following morphism of complexes:
\[ \xymatrix{ 
	\mathcal{C}_{\rig, \tau^{p^n}}^{\pa}(V)\colon & 0 \ar[r] & D^{\dagger}_{\rig} \ar[rr]^-{(\varphi-1, \tau_D^{p^n}-1)} \ar@{=}[d]&& D^{\dagger}_{\rig} \oplus D^{\dagger, \pa}_{\rig, \tau^{p^n}, 0} \ar[rr]^-{(\tau_D^{p^n}-1)\ominus (\varphi-1)} \ar[d]_{\id\oplus \frac{N_{\nabla}}{\tau_D^{p^n}-1}}&& D^{\dagger, \pa}_{\rig, \tau^{p^n}, 0} \ar[d]^{\frac{N_{\nabla}}{\tau_D^{p^n}-1}} \ar[r]& 0 \\
	\mathcal{C}_{\varphi, N_{\nabla}}(V)\colon & 0 \ar[r] & D^{\dagger}_{\rig} \ar[rr]^-{(\varphi-1, N_{\nabla})}&& D^{\dagger}_{\rig}\oplus D^{\dagger}_{\rig} \ar[rr]^-{N_{\nabla}\ominus (\frac{pE(u)}{E(0)}\varphi-1)}&& D^{\dagger}_{\rig} \ar[r]&0.
}   \]	
\end{remark}

\begin{proposition}\label{prop H0 H1 of complex C_R(V)}
	We have a group isomorphism $\H^i(\Cc_{\varphi, N_{\nabla}}(V))\simeq \iLim{n}\H^i(\mathscr{G}_{K_n}, V)$ for $i\in \{0, 1 \}.$
\end{proposition}

\begin{proof}
	When $i=0,$ it suffices to show that $N_{\nabla}(x)=0$ if and only if $\tau^{p^n}(x)=x$ for some $n\gg 0.$ This is indeed the case as  $\exp(p^n\frac{t}{-\lambda}N_{\nabla})=\tau^{p^n}$ for $n\gg 0.$ When $i=1,$ this follows from the bijectivity of $\iLim{m}\Ext^1_{(\varphi,\tau^{p^m})}(\mathcal{D}^{\dagger}_{\rig}(V),\mathcal{R})\to\Ext^1_{(\varphi,N_\nabla)}(\mathcal{D}(V),\mathcal{R})$ ($\cf$ lemmas \ref{lemm H1 Ext N nabla} and \ref{lemm Ext 1 nabla bij limit Ext 1}), noting that $\Ext^1_{(\varphi,\tau^{p^n})}(\mathcal{D}^{\dagger}_{\rig}(V),\mathcal{R})\simeq \H^1(K_n, V)$ for all $n\in \NN.$ 
\end{proof}

\begin{proposition}
	If $D_{\rig, \tau^{p^n}, 0}^{\dagger}\subset D^{\dagger, \pa}_{\rig, \tau}$ for all $n\in \NN,$ then the map in proposition \ref{prop H0 H1 of complex C_R(V)} is a natural $\QQ_p$-linear isomorphism.
\end{proposition}

\begin{proof}
If $D_{\rig, \tau^{p^n}, 0}^{\dagger}\subset D^{\dagger, \pa}_{\rig, \tau},$ then we have $\frac{N_{\nabla}}{\tau_D^{p^n}-1}: D_{\rig, \tau^{p^n}, 0}^{\dagger}\to D_{\rig}^{\dagger}$ by proposition \ref{prop N over tau-1}. We then have the following map of complexes
\[ \xymatrix{ 
	\mathcal{C}_{\varphi, \tau^{p^n}}^{\rig}(V)\colon & 0 \ar[r] & D^{\dagger}_{\rig} \ar[rr]^-{(\varphi-1, \tau_D^{p^n}-1)} \ar@{=}[d]&& D^{\dagger}_{\rig} \oplus D^{\dagger}_{\rig, \tau^{p^n}, 0} \ar[rr]^-{(\tau_D^{p^n}-1)\ominus (\varphi-1)} \ar[d]_{\id\oplus \frac{N_{\nabla}}{\tau_D^{p^n}-1}}&& D^{\dagger}_{\rig, \tau^{p^n}, 0} \ar[d]^{\frac{N_{\nabla}}{\tau_D^{p^n}-1}} \ar[r]& 0 \\
	\mathcal{C}_{\varphi, N_{\nabla}}(V)\colon &	0 \ar[r] & D^{\dagger}_{\rig} \ar[rr]^-{(\varphi-1, N_{\nabla})}&& D^{\dagger}_{\rig}\oplus D^{\dagger}_{\rig} \ar[rr]^-{N_{\nabla}\ominus (\frac{pE(u)}{E(0)}\varphi-1)}&& D^{\dagger}_{\rig} \ar[r]&0
}   \]
this induces a natural $\QQ_p$-linear map $\iLim{m} \H^1(\mathscr{G}_{K_m},V)\to \H^1(\mathcal{C}_{\varphi, N_{\nabla}}(V)),$ which identifies with the bijective map $\iLim{m}\Ext^1_{(\varphi,\tau^{p^m})}(\mathcal{D}^{\dagger}_{\rig}(V),\mathcal{R})\to\Ext^1_{(\varphi,N_\nabla)}(\mathcal{D}(V),\mathcal{R})$ mentioned in proposition \ref{prop H0 H1 of complex C_R(V)}.
\end{proof}

\begin{remark}
	As $V$ has finite dimension, the increasing sequence $(\H^0(\mathscr{G}_{K_n},V))_{n\in\NN}$ is stationary, so that $\H^0(\mathcal{C}_{\varphi, N_{\nabla}}(V))=\H^0(\mathscr{G}_{K_n},V)$ for $n\gg0.$
	The analogue is not true for $\H^1.$ For instance, (suppose $k$ is finite) if $V=\QQ_p(1),$ we have $\dim_{\QQ_p}(\H^1(\mathscr{G}_{K_n},V))\geq1+[K_n:\QQ_p]=1+p^n[K:\QQ_p],$ so that the sequence $(\H^1(\mathscr{G}_{K_n},V))_{n\in\NN}$ is not stationary.
\end{remark}

\subsection{Examples: the \texorpdfstring{$(\varphi, \tau)$}{(phi, tau)}-module of \texorpdfstring{$\ZZ_p(n)$}{Zp(n)}}\label{section D of tate twist}

Let $\mathfrak{t}\in \W(\Oo_{C^{\flat}})$ be an element that is not divisible by $p$ and satisfies $\varphi(\mathfrak{t})=c\mathfrak{t}.$ In fact $\mathfrak{t}$ is unique up to multiplication by an element in $\ZZ_p^{\times}.$ By \cite[Propositions 1.3.5, 1.3.6]{Poy21}, we can normalize $\mathfrak{t}$ such that $\mathfrak{t}=\frac{t}{p\lambda}\in (\widetilde{\AA}^{+})^{\mathscr{G}_L}$\index{$\mathfrak{t}$} (\cf also \cite[Example 3.2.3]{Liu08}, this element is denoted $b_{\gamma}$ in \textit{loc. cit.}).

By \cite[1.3.4]{Car13} we see that the $(\varphi, \tau)$-module structure over $(\Oo_{\Ee}, \Oo_{\Ee_{\tau}})$ associated to the representation $\ZZ_p(n)$ is given by:
\begin{align*}
&\mathcal{D}(\ZZ_p(n))=\Oo_{\Ee}\mathfrak{t}^{-n}\\
&\varphi(\mathfrak{t}^{-n})=c^{-n}\mathfrak{t}^{-n}\\
&\tau(\mathfrak{t}^{-n})=\Big(\frac{\lambda}{\tau(\lambda)}\Big)^{-n}\mathfrak{t}^{-n}. 
\end{align*}
Remark that the minus sign in the first equality follows from the fact that our constructions are covariant (compare with \cite[1.3.4]{Car13}).

\begin{remark}\label{rem galois over t}
	We have $g(\mathfrak{t})=\chi(g)\frac{\lambda}{g(\lambda)}\mathfrak{t}$ for $g\in \mathscr{G}_{K},$ and in particular $g(\mathfrak{t})=\chi(g)\mathfrak{t}$ for $g\in \mathscr{G}_{K_{\pi}}.$ 
\end{remark}
\begin{proof}
	$\cf$ \cite[Proposition 1.3.6]{Poy21}.
\end{proof}

\begin{lemma}\label{lemm N nabla t =0}
	We have $N_{\nabla}(t)=0.$
\end{lemma}

\begin{proof}
	This follows directly from the definition of the element $t$ and the operator $N_{\nabla}$ over $(\widetilde{\BB}^{\dagger}_{\rig, L})^{\pa}.$ 
\end{proof}

\begin{lemma}\label{lemm N nabla t^-1}
	Let $n\in \NN, $ then we have $N_{\nabla}(\mathfrak{t}^{-n})=n\mathfrak{t}^{-n}\frac{N_{\nabla}(\lambda)}{\lambda}.$ 
\end{lemma}
\begin{proof}
	
	Recall that $t=p\lambda\mathfrak{t}.$ By lemma \ref{lemm N nabla t =0} we have $N_{\nabla}(t)=0.$ By Leibniz's rule we have \[ N_{\nabla}(\lambda)\mathfrak{t}+\lambda N_{\nabla}(\mathfrak{t})=0, \] \ie
	\[ N_{\nabla}(\mathfrak{t})=-\frac{\mathfrak{t}}{\lambda}N_{\nabla}(\lambda). \]
	Again, by Leibniz's rule we have 
	\begin{align*}
	N_{\nabla}(\mathfrak{t}^{-n})&=-n\mathfrak{t}^{-n-1}N_{\nabla}(\mathfrak{t})\\
	&=-n\mathfrak{t}^{-n-1}\big(-\frac{\mathfrak{t}}{\lambda}\big)N_{\nabla}(\lambda)\\
	&=n\mathfrak{t}^{-n}\frac{N_{\nabla}(\lambda)}{\lambda}.
	\end{align*}
\end{proof}

\begin{remark}
Similarly, the $(\varphi, \tau)$-module structure over $(\Rr, \Rr_{\tau})$ associated to the representation $\QQ_p(n)$ is given by:
\[ \mathcal{D}_{\rig}^{\dagger}(\QQ_p(n))=\Rr\mathfrak{t}^{-n}, \]
with the operators $\varphi$ and $N_{\nabla}$ as mentioned above. 
\end{remark}

\subsection{Construction of a pairing}

\begin{lemma}\label{lemm internal hom and tensor product}
	Let $T_1, T_2\in \Rep_{\ZZ_p}(\mathscr{G}_K),$ then we have 
	\begin{align*}
	&\mathcal{D}^{\dagger}(T_1\otimes_{\ZZ_p} T_2)\simeq \mathcal{D}^{\dagger}(T_1)\otimes_{\Oo_{\Ee}^{\dagger}}\mathcal{D}^{\dagger}(T_2)\\
	&\mathcal{D}^{\dagger}(\Hom_{\ZZ_p}(T_1, T_2))\simeq \Hom_{\Oo_{\Ee}^{\dagger}}(\mathcal{D}^{\dagger}(T_1), \mathcal{D}^{\dagger}(T_2)).
	\end{align*}
\end{lemma}
\begin{proof}
	The proof is similar to that of proposition \ref{prop internal hom and tensor product}.
\end{proof}

\begin{proposition}\label{prop pairing}
	Suppose the residue field $k$ is finite. We have a pairing of groups
	\[\H^0(\Cc_{\varphi, N_{\nabla}}(V^{\vee}(1))) \times \H^2(\Cc_{\varphi, N_{\nabla}}(V)) \to \QQ_p.  \]
\end{proposition}

\begin{proof}
	By lemma \ref{lemm internal hom and tensor product}, we have $\mathcal{D}^{\dagger}(V^{\vee}(1))=\Hom_{\Ee^{\dagger}}(\mathcal{D}^{\dagger}(V), \mathcal{D}^{\dagger}(\QQ_p(1))).$ Hence we have
	\begin{align*}
	\H^0(\Cc_{\varphi, N_{\nabla}}(V^{\vee}(1)))&=(\Rr\otimes_{\Ee^{\dagger}} (\mathcal{D}^{\dagger}(V^{\vee}(1)))^{\varphi=1, N_{\nabla}=0}\\
	&=(\Rr\otimes_{\Ee^{\dagger}}\Hom_{\Ee^{\dagger}}(\mathcal{D}^{\dagger}(V), \mathcal{D}^{\dagger}(\QQ_p(1))))^{\varphi=1, N_{\nabla}=0}\\
	&=(\Hom_{\Rr}(\Rr\otimes_{\Ee^{\dagger}}\mathcal{D}^{\dagger}(V), \Rr \mathfrak{t}^{-1}))^{\varphi=1, N_{\nabla}=0}.
	\end{align*}
	Indeed, we used the  $\mathscr{G}_K$-equivariant isomorphism of $\Rr$-modules $\Rr\otimes_{\Ee^{\dagger}}{\mathcal{D}}^{\dagger}(\QQ_p(1))\simeq \Rr \mathfrak{t}^{-1}$ (\cf section \ref{section D of tate twist}).
	We start with the pairing:
	
	\begin{align*}
	\H^0(\Cc_{\varphi, N_{\nabla}}(V^{\vee}(1))) \times (\mathcal{R}\otimes_{\Ee^{\dagger}} \mathcal{D}^{\dagger}(V)) &\to \Rr \mathfrak{t}^{-1} \\
	(f,  \ x) &\mapsto f(x).
	\end{align*}	
	
	Recall that elements in $\Rr$ can be written uniquely as series $\sum\limits_{n\in\ZZ} a_n u^n$ with $a_n\in \W(k)\big[1/p\big].$ We define a residue map  over $\Rr \mathfrak{t}^{-1}$ as follows:
	\begin{align*}
	\res\colon \Rr \mathfrak{t}^{-1}&\to \W(k)\big[1/p\big]\\
	\mathfrak{t}^{-1}\sum\limits_{n\in\ZZ} a_n u^n&\mapsto a_0. 
	\end{align*}
	
	Composing with the trace map $\Tr=\Tr_{\W(k)\big[1/p\big]/\QQ_p}\colon \W(k)\big[1/p\big]\to \QQ_p,$ we obtain a map

	\begin{align*}
	\H^0(\Cc_{\varphi, N_{\nabla}}(V^{\vee}(1))) \times (\mathcal{R}\otimes_{\Ee^{\dagger}} \mathcal{D}^{\dagger}(V)) &\to \QQ_p\\
	(f,  \ x) &\mapsto \Tr (\res f(x)).
	\end{align*}	
	
	Recall that $\H^2(\Cc_{\varphi, N_{\nabla}}(V))\simeq (\mathcal{R}\otimes_{\Ee^{\dagger}} \mathcal{D}^{\dagger}(V))/\im((c\varphi-1)\oplus N_{\nabla})).$ 
	To construct the claimed pairing, we must show that the above map factors through $\im((c\varphi-1)\oplus N_{\nabla})\subset \mathcal{R}\otimes_{\Ee^{\dagger}} \mathcal{D}^{\dagger}(V).$ In other words, for any $x\in\im((c\varphi-1)\oplus N_{\nabla}),$ we have to show that $\Tr (\res f(x))=0.$
	
	Notice that for any $f\in\Hom_{\Rr}(\Rr\otimes_{\Ee^{\dagger}}\mathcal{D}^{\dagger}(V), \Rr\otimes_{\Ee^{\dagger}}\mathcal{D}^{\dagger}(\QQ_p(1)))^{ \varphi=1, N_{\nabla}=0},$ the condition $N_{\nabla}(f)=0$ means that  $N_{\nabla}\circ f=f\circ N_{\nabla},$ and $\varphi(f)=f$ means that $\varphi\circ f=f\circ \varphi.$ Indeed, the actions are $N_{\nabla}(f)=N_{\nabla}\circ f-f\circ N_{\nabla}$ and $\varphi(f)=\varphi\circ f\circ \varphi^{-1}.$
	
	For any $x\in \mathcal{R}\otimes_{\Ee^{\dagger}} \mathcal{D}^{\dagger}(V)$ and $f\in \H^0(\Cc_{\varphi, N_{\nabla}}(V^{\vee}(1))),$ write $f(x)=\mathfrak{t}^{-1} r=\mathfrak{t}^{-1}\sum\limits_{n\in\ZZ} a_n u^n \in \Rr \mathfrak{t}^{-1}$ with $a_n\in \W(k)\big[1/p\big].$ By remark \ref{rmq N nabla } and lemma \ref{lemm N nabla t^-1},  we have
	\begin{align*}
	\Tr\circ \res (f(N_{\nabla}(x)))&=\Tr\circ \res \big(N_{\nabla}(f(x))\big)\\
	&=\Tr\circ \res \big(N_{\nabla}(\mathfrak{t}^{-1} r)\big)\\
	&=\Tr\circ \res \big(N_{\nabla}({\mathfrak{t}^{-1}})r+N_{\nabla}(r)\mathfrak{t}^{-1}\big)\\
	&=\Tr\circ \res \big(\frac{N_{\nabla}(\lambda)}{\lambda}{\mathfrak{t}^{-1}}r+N_{\nabla}(r)\mathfrak{t}^{-1}\big)\\
	&=\Tr\circ \res\big(\frac{N_{\nabla}(\lambda r)}{\lambda}\mathfrak{t}^{-1}\big)\\
	&=  \Tr\circ \res\big(-u\frac{d}{du}(\lambda r)\mathfrak{t}^{-1}\big)\\
	&=\Tr(0)\\
	&=0.      
	\end{align*} 
	
	By section \ref{section D of tate twist} we have $\varphi(\mathfrak{t}^{-1})=c^{-1}\mathfrak{t}^{-1},$ hence 
	
	\begin{align*}
	\Tr\circ \res (f((c\varphi-1)x))&=\Tr\circ \res ((c\varphi-1)f(x))\\
	&=\Tr\circ \res \big((c\varphi-1)(\mathfrak{t}^{-1}\sum\limits_{n\in\ZZ} a_n u^n)\big)\\
	&=\Tr\circ \res \big(cc^{-1}\mathfrak{t}^{-1} \varphi(\sum\limits_{n\in\ZZ} a_n u^n)-\mathfrak{t}^{-1}\sum\limits_{n\in\ZZ} a_n u^n\big)\\
	&=\Tr\circ \res \big(\mathfrak{t}^{-1}(\varphi-1)(\sum\limits_{n\in\ZZ} a_n u^n)\big) \\
	&=\Tr\circ \res \big(\mathfrak{t}^{-1}( \sum\limits_{n\in\ZZ} \varphi(a_n) u^{pn}-\sum\limits_{n\in\ZZ} a_n u^n)\big)\\
	&=\Tr(\varphi(a_0)-a_0)\\
	&=(\varphi-1)\Tr(a_0)\\
	&=0.
	\end{align*} 
	
	Hence we have the desired pairing 
	\begin{align*}	\H^0(\Cc_{\varphi, N_{\nabla}}(V^{\vee}(1))) \times \H^2(\Cc_{\varphi, N_{\nabla}}(V)) &\to \QQ_p\\
	(f, x) &\mapsto B(f, x)=\Tr (\res f(x)).
	\end{align*} 
	
\end{proof}

\begin{proposition}\label{prop perfect pairing}
The pairing just constructed is nondegenerate on the left.
\end{proposition}

\begin{proof}
	Recall that $\H^0(\Cc_{\varphi, N_{\nabla}}(V^{\vee}(1)))=\Hom_{\Rr,\varphi, N_{\nabla}}(\Rr\otimes \mathcal{D}^{\dagger}(V), \Rr \mathfrak{t}^{-1})$ ($\cf$ the proof of proposition \ref{prop pairing}).
	Suppose $\Rr\otimes \mathcal{D}^{\dagger}(V)$ has rank $d\in \NN_{>0},$ and fix a $\Rr$-basis $\{e_1, \ldots, e_d\}.$ Then $f\in \Hom_{\Rr,\varphi, N_{\nabla}}(\Rr\otimes \mathcal{D}^{\dagger}(V), \Rr \mathfrak{t}^{-1})$ is determined by the image of the basis (by $\Rr$-linearity), $\ie$it is determined by a $d\times 1$ matrix with coefficients in $\Rr \mathfrak{t}^{-1}$ (it also satisfies certain conditions arising from the compatibility with $\varphi$ and $N_{\nabla}$). 
	Denote
	\begin{align*}
	B\colon	\H^0(\Cc_{\varphi, N_{\nabla}}(V^{\vee}(1))) \times \H^2(\Cc_{\varphi, N_{\nabla}}(V)) &\to \QQ_p\\
	(f, x) &\mapsto B(f, x)=\Tr (\res f(x))
	\end{align*} 
	the pairing constructed in proposition \ref{prop pairing}. Let $f\in \Hom_{\Rr,\varphi, N_{\nabla}}(\Rr\otimes \mathcal{D}^{\dagger}(V), \Rr \mathfrak{t}^{-1})$ be such that $B(f, z)=0$ for all $z\in \H^2(\Cc_{\varphi, N_{\nabla}}(V)),$ we claim that $f=0.$
	
	\medskip
	
	Suppose $f$ corresponds to $(r_1\mathfrak{t}^{-1}, \ldots, r_d\mathfrak{t}^{-1})\in \mathsf{M}_{d\times 1}(\Rr \mathfrak{t}^{-1})$ under the fixed basis with $r_i\in \Rr$ for $i\in \{1, \ldots, d\}.$ Write $r_i=\sum\limits_{j\in \ZZ}a_{i, j}u^j$ with $a_{i, j}\in \W(k)\big[1/p\big].$ We first prove $r_1=0$ and the others follow similarly. For any $x\in \mathcal{R}\otimes_{\Ee^{\dagger}} \mathcal{D}^{\dagger}(V)$ with its image  $\overline{x}\in \H^2(\Cc_{\varphi, N_{\nabla}}(V))\simeq (\mathcal{R}\otimes_{\Ee^{\dagger}} \mathcal{D}^{\dagger}(V))/\im((c\varphi-1)\oplus N_{\nabla}),$ we have $B(f, \overline{x})=0$ by assumption. To prove $r_1=0,$ it suffices to prove $a_{1,j}=0$ for any $j\in \ZZ.$ For any $j\in \ZZ,$ consider $x$ whose coordinates under the fixed basis is $(\alpha u^{-j}, 0, \ldots, 0)\in \Rr^d,$ with $\alpha\in \W(k)\big[1/p\big].$  Then we have 
	\[f(\overline{x})=\alpha u^{-j}\mathfrak{t}^{-1}r_1=\mathfrak{t}^{-1}\alpha u^{-j}\sum\limits_{m\in \ZZ}a_{1, m}u^{m}=\mathfrak{t}^{-1}\sum\limits_{m\in \ZZ}\alpha a_{1, m}u^{m-j}.\]
	We have $\res(f(\overline{x}))=\alpha a_{1, j}.$ If $a_{1,j}\ne 0,$ then we can always find $\alpha\in \W(k)\big[1/p\big]$ such that $\Tr(\alpha a_{1, j})\ne 0,$ which then contradicts the hypothesis that $B(f, \overline{x})=0$ for any $x\in \mathcal{R}\otimes_{\Ee^{\dagger}} \mathcal{D}^{\dagger}(V).$ Hence $a_j^{1}=0$ for all $j\in \ZZ$ and hence $r_1=0.$ Similarly we can show that $r_i=0$ for all $i\in \{1, \ldots, d\}$ and hence $f=0.$ This shows that $B$ is nondegenerate on the left.
\end{proof}

\begin{remark} Assume the previous pairing is perfect, then for any $V\in \Rep_{\QQ_p}(\mathscr{G}_{K}),$ we have 
	\[  \H^2(\mathcal{C}_{\varphi, N_{\nabla}}(V))=\H^2(\mathscr{G}_{K_{n}}, V),\ n\gg 0. \]

	\begin{proof}
		The perfect pairing 
		\[\H^0(\Cc_{\varphi, N_{\nabla}}(V^{\vee}(1))) \times \H^2(\Cc_{\varphi, N_{\nabla}}(V)) \to \QQ_p \]
	implies
		\[ \H^2(\mathcal{C}_{\varphi, N_{\nabla}}(V)) = \Hom_{\QQ_p}(\iLim{m} \H^{0}(\mathscr{G}_{K_{m}},V^{\vee}(1)), \QQ_p).  \]
		Notice that the sequence $ \big( \H^{0}(\mathscr{G}_{K_{m}},V^{\vee}(1)) \big)_{m\geq 0}$ is stationary (it is an increasing sequence of sub $\QQ_p$-vector spaces of $V^{\vee}(1)$), hence for $n\gg 0,$ we have
		\[  \H^2(\mathcal{C}_{\varphi, N_{\nabla}}(V)) = \Hom_{\QQ_p}(\H^{0}(\mathscr{G}_{K_{n}},V^{\vee}(1)), \QQ_p)=\H^{0}(\mathscr{G}_{K_{n}},V^{\vee}(1))^{\vee}. \]
		Since $\H^{0}(\mathscr{G}_{K_{n}},V^{\vee}(1))^{\vee}$ is isomorphic to $\H^2(\mathscr{G}_{K_{n}}, V)$  by Tate duality,  we hence conclude.	
	\end{proof}
\end{remark}

\subsection{Examples with \texorpdfstring{$\QQ_p(n)$}{Qp(n)}}

Assume $k$ is finite.

\begin{remark}\label{rem Tate duality}
	Let $n, r\in \NN,$ then by Tate duality we have 
	
	\[  \H^2(\mathscr{G}_{K_{r}}, \QQ_p(n))\simeq \H^0(\mathscr{G}_{K_r}, \QQ_p(1-n)).  \]
	
	\item As
	
	\[ \H^0(\mathscr{G}_{K_{r}}, \QQ_p(n))\simeq
	\begin{cases}
	\QQ_p &\text{if}\ n=0\\
	0  &\text{if}\ n\neq 0,
	\end{cases}
	\]
	
	\item we have 
	\[ \H^2(\mathscr{G}_{K_{r}}, \QQ_p(n))\simeq 
	\begin{cases}
	\QQ_p &\text{if}\ n=1\\
	0  &\text{if}\ n\neq 1.
	\end{cases}
	\]
\end{remark}

We have computed $\H^i(\Cc_{\varphi, N_{\nabla}}(V))$ for $i\in \{ 0, 1 \}$ in proposition \ref{prop H0 H1 of complex C_R(V)}. Let's see some examples for $\H^2(\Cc_{\varphi, N_{\nabla}}(V)).$

\begin{example}\label{example Robba for H2=0} (1) For $V=\QQ_p(n)$ we have $\mathcal{D}^{\dagger}_{\rig}(\QQ_p(n))\simeq \Rr\mathfrak{t}^{-n}$ and the complex $\mathcal{C}_{\varphi, N_{\nabla}}(\QQ_p)$ is 
	
	\[  \xymatrix{
		0\ar[rr] && \Rr\mathfrak{t}^{-n} \ar[r] & \Rr\mathfrak{t}^{-n} \oplus \Rr\mathfrak{t}^{-n}  \ar[r]&  \Rr\mathfrak{t}^{-n} \ar[r] &0   \\
		&&x  \ar@{|->}[r] &((\varphi-1)(x), N_{\nabla}(x))  &{}\\
		&&&{} (y,z)  \ar@{|->}[r]& N_{\nabla}(y)-(c\varphi-1)(z).
	}
	\]
For any $f, g\in \Rr,$ we have 
\begin{align*}
	N_{\nabla}(f\mathfrak{t}^{-n})&=N_{\nabla}(f)\mathfrak{t}^{-n}+fn\mathfrak{t}^{-n}\frac{N_{\nabla}(\lambda)}{\lambda}=(N_{\nabla}+n\frac{N_{\nabla}(\lambda)}{\lambda})(f)\cdot\mathfrak{t}^{-n},\\
	(c\varphi-1)(g\mathfrak{t}^{-n})&=c\varphi(g)c^{-n}\mathfrak{t}^{-n}-g\mathfrak{t}^{-n}=(c^{1-n}\varphi-1)(g)\cdot\mathfrak{t}^{-n}.
\end{align*}	
Hence \[ \H^2(\Cc_{\varphi, N_{\nabla}}(\QQ_p(n)))\simeq \Rr\big/\big( \im(N_{\nabla}+n\frac{N_{\nabla}(\lambda)}{\lambda})+\im(c^{1-n}\varphi-1)  \big).  \]
\item (2) In particular, for $V=\QQ_p,$ the complex $\mathcal{C}_{\varphi, N_{\nabla}}(\QQ_p)$ is 
\[  \xymatrix{
	0\ar[rr] && \Rr \ar[r] & \Rr \oplus \Rr  \ar[r]&  \Rr \ar[r] &0   \\
&&x  \ar@{|->}[r] &((\varphi-1)(x), N_{\nabla}(x))  &{}\\
&&&{} (y,z)  \ar@{|->}[r]& N_{\nabla}(y)-(c\varphi-1)(z),
}
\]
and we have $\H^2(\Cc_{\varphi, N_{\nabla}}(\QQ_p))\simeq \Rr\big/\big( \im(N_{\nabla})+\im(c\varphi-1)  \big). $
\end{example}

\begin{notation}
By definition \ref{def Robba ring power series}, any element of $\Rr$ can be written in the form $f(u)=\sum\limits_{i\in \ZZ}a_i u^i$ with $a_i\in \W(k)\big[1/p\big].$ We put $f^{+}(u)=\sum\limits_{i\geq 0}a_i u^i$ and $f^{-}(u)=\sum\limits_{i<0}a_i u^i,$ then $f(u)=f^{+}(u)+f^{-}(u).$ 
\end{notation}

\begin{lemma}\label{lemm topology in O}\label{lemm topology in O N-operator}
	With notations as above, $\sum\limits_{n=0}^{\infty}(c\varphi)^n(f^{+}(u))$ converges in $\Oo.$
\end{lemma}

\begin{proof}
	Recall that any $h(u)=\sum\limits_{i=0}^{\infty}\lambda_i u^i\in \W(k)\big[1/p\big][\![u]\!]$ belongs to $\Oo$ if and only if for any $r\in [0, 1).$ $\Lim{i\to \infty} |\lambda_i|r^i = 0$ ($|\cdot|$ denotes the $p$-adic absolute value), which is equivalent to $|h|_r:=\sup_i | \lambda_i|r^i$ being finite for any $r\in [0, 1)$. Recall $\frac{p}{E(0)}\in \W(k)^{\times}$ and $c=\frac{p E(u)}{E(0)},$ we hence have $|c|_{r}=\max\{ 1/p, r^{e}\}$ and $|\varphi^k (c)|=\max\{1/p, r^{p^k e} \},$ which is $1/p$ when $k\gg 0.$ Notice also that $|\varphi(h(u))|_r=|h(u)|_{r^p}$ . We have (for $k\gg0$)
	\[  \big|c\varphi(c)\varphi^2(c)\cdots \varphi^{k-1}(c)\varphi^k(f^{+}(u))\big|_r\leq O_r(\frac{1}{p^k})\big|f^{+}(u)\big|_{r^{p^k}}=O_r(\frac{1}{p^k}). \]
	Hence $\sum\limits_{n=0}^{\infty}(c\varphi)^n(f^{+}(u))$ converges in $\Oo$ as its $k$-th term $c\varphi(c)\varphi^2(c)\cdots \varphi^{k}(c)\varphi^{k+1}(f^{+}(u))$ tends to $0$ when $k$ tends to $+\infty.$
\end{proof}

\begin{remark}
	To check $\H^2(\Cc_{\varphi, N_{\nabla}}(\QQ_p))=0$ (hence $\Cc_{\varphi, N_{\nabla}}(\QQ_p)$ computes the Galois cohomology by remark \ref{rem Tate duality}), it is equivalent to showing that $\Rr\oplus \Rr\xrightarrow{c\varphi-1\oplus N_{\nabla}} \Rr$ is surjective. Although we cannot prove it, we can see from the following lemma that the image contains a lot of elements.
\end{remark}

\begin{lemma}\label{lemma H2=0 first step}
	 We have
	\[ \Big\{ \sum\limits_{i\geq m}^{+\infty}a_i u^i\in \Rr; \, a_i\in \W(k)\big[1/p\big], m\in \ZZ \Big\} \subset \im(c\varphi-1\oplus N_{\nabla}). \]
\end{lemma}

	\begin{proof}
	Let $f(u)=\sum\limits_{i\in \ZZ}a_i u^i\in \Rr,$ then we claim $f^{+}(u)\in \im(c\varphi-1).$ Put $f_1(u)=\sum\limits_{m\geq 0} (c\varphi)^m(f^{+}(u)),$ it is well-defined in $\Rr$ by lemma \ref{lemm topology in O} and satisfies $(c\varphi-1)(f_1(u))=f^{+}(u).$ 
	
	\medskip
	
	We now study $f^{-}(u).$ Observe that $a_ju^{-j}\in \im(c\varphi-1\oplus N_{\nabla})$ for any $j\in\ZZ_{+}$ and $a_j\in \W(k)\big[1/p\big].$ We prove by induction as follows. As $\lambda\in \Oo,$ we can write $\lambda=\sum\limits_{i=0}^{\infty} \lambda_i u^i$ and then 
	\begin{align*}
	N_{\nabla}(a_1u^{-1})	=a_1\lambda u^{-1}=a_1\sum\limits_{i=0}^{\infty}\lambda_iu^{i-1}=a_1\lambda_0 u^{-1}+  a_1\sum\limits_{i=1}^{\infty}\lambda_iu^{i-1}.
	\end{align*} 
	Observe that $\lambda_0=1$ and $a_1\sum\limits_{i=1}^{\infty}\lambda_iu^{i-1}\in \Oo \subset \im(c\varphi-1),$ hence $a_1u^{-1}\in \im(c\varphi-1\oplus N_{\nabla}).$ For any $j\geq 2,$ we have 
	\begin{align*}
	N_{\nabla}(a_ju^{-j})&=a_jj\lambda u^{-j}\\
	&=a_jj\sum\limits_{i=0}^{\infty}\lambda_iu^{i-j}\\
	&=a_jj u^{-j}+ a_jj\lambda_1 u^{1-j}+\cdots +a_jj\lambda_{j-1} u^{-1}+ a_jj\sum\limits_{i=j}^{\infty}\lambda_iu^{i-j}.
	\end{align*} 
	We have $a_jj\lambda_1 u^{1-j}+\cdots +a_jj\lambda_{j-1} u^{-1}\in\im(c\varphi-1\oplus N_{\nabla})$ by induction hypothesis and $a_jj\sum\limits_{i=j}^{\infty}\lambda_iu^{i-j}\in \Oo,$ hence in $\im(c\varphi-1\oplus N_{\nabla}).$ Hence $a_jj u^{-j}\in \im(c\varphi-1\oplus N_{\nabla})$ and then $a_j u^{-j}\in \im(c\varphi-1\oplus N_{\nabla}).$
\end{proof}

\section{The complex \texorpdfstring{$\Cc_{\varphi, \partial_{\tau}}$}{C\textpinferior\texthinferior\textiinferior, \unichar{"2202}\texttinferior\textainferior\textuinferior}}

\begin{lemma}\label{lemm N preserves surconvergence}
	For any $x=\sum\limits_{n\in\ZZ\setminus\{ 0\}}x_nu^n$ with $x_n\in \W(k)\big[1/p\big],$ we have $u\frac{d}{du}(x)\in \Rr$ if and only if $x\in \Rr.$
\end{lemma}

\begin{proof}
	For light notation, we put $\partial_{\tau}=u\frac{d}{du}$ and we have $\partial_{\tau}(x)=\sum\limits_{ n\in\ZZ\setminus\{0\} }nx_nu^n.$
	
	\medskip
	
	Assume $\partial_{\tau}(x)\in \Rr:$ we show that $x\in \Rr.$ For any $0<r<1,$ we have $|nx_n|r^{n}\xrightarrow[n\to +\infty]{} 0$ and in particular for $1>\rho^{\prime}>\rho$ we have \[|nx_n|(\rho^{\prime})^n=|nx_n|\big(\frac{\rho^{\prime}}{\rho}\big)^n\rho^n \xrightarrow[n\to +\infty]{} 0. \] 
	Notice that $|\frac{1}{n}|\leq \big(\frac{\rho^{\prime}}{\rho}\big)^n$ for $n\gg 0,$ so that $|x_n|\rho^n\leq |nx_n|(\rho^{\prime})^n$ for $n\gg 0.$ Hence $|x_n| \rho^n\to 0$ as $n\to +\infty.$ 
	On the other hand, as $\partial_{\tau}(x)\in \Rr,$ there exists $0<r_0<1$ such that  $|nx_n|r^n\xrightarrow[n\to -\infty]{}0$ for any $r\in (r_0, 1).$ We claim that $|x_n|\rho^n\xrightarrow[n\to -\infty]{}0$ for any $\rho\in (r_0, 1).$ There exists $\rho^{\prime}$ such that $r_0< \rho^{\prime}<\rho$ and \[|nx_n|(\rho^{\prime})^n=|nx_n|\big(\frac{\rho^{\prime}}{\rho} \big)^n(\rho)^n\xrightarrow[n\to -\infty]{} 0. \]
	We then conclude similarly from the observation $|\frac{1}{n}|\leq \big(\frac{\rho^{\prime}}{\rho} \big)^n$ for $n\ll 0.$
	
	\medskip
	
	To show $x\in \Rr$ implies $\partial_{\tau}(x)\in\Rr$ is direct, as $|nx_n|\leq |x_n|$ for $n\in\ZZ.$ 
\end{proof}

\begin{notation}
	Put $\partial_{\tau}=\frac{1}{t}\nabla_{\tau}=-\frac{1}{\lambda}N_{\nabla}=u\frac{d}{du}$ ($\cf$ remark \ref{rmq N nabla }).\index{$\partial_{\tau}$} Remark that this is an operator over $\Rr$ by lemma \ref{lemm N preserves surconvergence}, and we will extend it to certain $\Rr$-modules.
\end{notation}

	\begin{definition}
		We define a \emph{$(\varphi, \partial_{\tau})$-module over $\Rr$} to be a free $\Rr$-module $D$ endowed with a Frobenius map $\varphi$ and a connection $\partial_{\tau}\colon D\to D$ over $\partial_{\tau}\colon \Rr\to \Rr,$ $\ie$an additive map such that 
		\[ (\forall m\in D)\ (\forall x\in \Rr)\quad \partial_{\tau}(x\cdot m)=\partial_{\tau}(x)\cdot m+x\cdot \partial_{\tau}(m), \]
		that satisfies $\partial_{\tau}\circ \varphi= p\varphi \partial_{\tau}.$ The corresponding category is denoted $\Mod_{\Rr}(\varphi, \partial_{\tau}).$\index{$\Mod_{\Rr}(\varphi, \partial_{\tau})$}
\end{definition}

\begin{definition}($\cf$ \cite[\S 1.3]{Poy21})
\item (1)	Let $I$ be a sub-interval of $[0, +\infty)$ and $V\in \Rep_{\QQ_p}(\mathscr{G}_K)$, we define $\widetilde{D}^I_L(V)$\index{$\widetilde{D}^I_L(V)$} by 
	\[ \widetilde{D}^I_L(V)=(\widetilde{\BB}^{I}\otimes_{\QQ_p}V)^{\mathscr{G}_L}.  \]
\item (2) Let $V\in \Rep_{\QQ_p}(\mathscr{G}_K),$ we put 
	\[ \widetilde{D}^{\dagger, r}_{\rig, L}(V)=(\widetilde{\BB}^{\dagger, r}_{\rig}\otimes_{\QQ_p}V)^{\mathscr{G}_L}. \]\index{$\widetilde{D}^{\dagger, r}_{\rig, L}(V)$}
\end{definition}

\begin{lemma}\label{lemm Poyeton lanalytic lemma 1.3.4}
	Let $V\in \Rep_{\QQ_p}(\mathscr{G}_K)$ with $(D^{\dagger}, D^{\dagger}_{\tau})\in \Mod_{{\BB_{K_{\pi}}^\dagger,\widetilde{\BB}_{L}^\dagger}}(\varphi, \tau)$ its associated $(\varphi, \tau)$-module over $({\BB_{K_{\pi}}^\dagger,\widetilde{\BB}_{L}^\dagger})$, and $(D^{\dagger}_{\rig}, D^{\dagger}_{\rig, \tau})\in \Mod_{\Rr, \Rr_{\tau}}(\varphi, \tau)$ its associated $(\varphi, \tau)$-module over $(\Rr, \Rr_{\tau})$. Suppose $r\geq 0$ is such that $D^{\dagger}=\BB_{K_{\pi}}^{\dagger}\otimes_{\BB_{K_{\pi}}^{\dagger, r}}(\BB^{\dagger, r}\otimes_{\QQ_p}V)^{\mathscr{G}_{K_{\pi}}}$. Then for any compact interval $I$ such that $r\leq \min(I)$, the elements of $D^{\dagger}$ and $D^{\dagger}_{\rig}$, seen as elements of $\widetilde{D}^I_L(V)$ are locally analytic for the group $\Gal(L/K).$ 
\end{lemma}

\begin{proof}
	Notice that for $s\geq r$ we have 
	\[ D^{\dagger, r}\subset D_{\rig}^{\dagger, r}\subset (\widetilde{D}^{\dagger, r}_{\rig, L})^{\mathsf{la}}\subset (\widetilde{D}^{[r, s]}_{L})^{\mathsf{la}} \]
	($\cf$ \cite[Lemma 1.3.4]{Poy21}).
\end{proof}

\begin{theorem}\label{thm Poyeton thm 3.4.10 varphi N module}
	Let ${M}$ be a $(\varphi, \tau)$-module over $\Rr$ whose $N_{\nabla}$-action is locally trivial \textup{(\cf \cite[D\'efinition 3.4.1]{Poy21})}, then there exists a unique $(\varphi, \tau)$-module ${D}\subset {M}[1/\lambda]$ such that ${D}[1/\lambda]\subset {M}[1/\lambda]$ and such that $\partial_{\tau}({D})\subset D.$ 
\end{theorem}

\begin{proof}
	$\cf$ \cite[Th\'eor\`eme 3.4.10]{Poy21}.
\end{proof}

\begin{corollary}\label{coro semistable}
	Let $V\in \Rep_{\QQ_p}(\mathscr{G}_{K})$ be a semi-stable representation and $(D^{\dagger}_{\rig}, D^{\dagger}_{\rig, \tau})\in \Mod_{\Rr, \Rr_{\tau}}(\varphi, \tau)$ the corresponding $(\varphi, \tau)$-module over $(\Rr, \Rr_{\tau})$. Then we have
	\[ \partial_{\tau}(D_{\rig}^{\dagger})\subset D_{\rig}^{\dagger}\subset D_{\rig}^{\dagger}[1/\lambda], \]
 hence there is a $(\varphi, \partial_{\tau})$-module structure on $D_{\rig}^{\dagger}$. 
\end{corollary}

\begin{proof}
	Notice that the $\varphi$-module $D^{\dagger}_{\rig}$ of a semi-stable representation has locally trivial $N_{\nabla}$-action by \cite[Lemma 3.4.5, Th\'eor\`em 3.4.12]{Poy21}, hence it has a connection $\partial_{\tau}$ by theorem \ref{thm Poyeton thm 3.4.10 varphi N module}.
\end{proof}

\begin{definition}\label{def complex N nabla}
	Let $D\in \Mod_{\Rr}(\varphi, \partial_{\tau})$. We define a complex $\Cc_{\varphi, \partial_{\tau}}(D)$\index{$\Cc_{\varphi, \partial_{\tau}}(D)$} as follows:
	\[\xymatrix{	0\ar[rr]&& D \ar[r] & D\oplus D  \ar[r]  & D  \ar[r] &0   \\
		&& x \ar@{|->}[r]  &((\varphi-1)(x), \partial_{\tau}(x))  & &\\
		&&&   (y,z)  \ar@{|->}[r]  &  \partial_{\tau}(y)-(p\varphi-1)(z).&
	}  \]
	If $V\in \Rep_{\QQ_p}(\mathscr{G}_K),$  we have in particular the complex $\Cc_{\varphi, \partial_{\tau}}(\mathcal{D}^{\dagger}_{\rig}(V)),$ which will be simply denoted  $\Cc_{\varphi, \partial_{\tau}}(V).$\index{$\Cc_{\varphi, \partial_{\tau}}(V)$}
\end{definition}

\subsection{\texorpdfstring{$\H^0(\Cc_{\varphi, \partial_{\tau}}(V))$}{H\textzerosuperior(C\textpinferior\texthinferior\textiinferior, \unichar{"2202}\texttinferior\textainferior\textuinferior(V))} and \texorpdfstring{$\H^1(\Cc_{\varphi, \partial_{\tau}}(V))$}{H\textonesuperior(C\textpinferior\texthinferior\textiinferior, \unichar{"2202}\texttinferior\textainferior\textuinferior(V))}}

\begin{proposition}\textup{(\cf \cite[Proposition 2.2.3]{Poy21})}\label{prop injective on the extensino argument partial case}
	Let $V, V^{\prime}\in \Rep_{\QQ_p}(\mathscr{G}_K)$ be crystalline representations, then $\mathcal{D}^{\dagger}_{\rig}(V)$ and $\mathcal{D}^{\dagger}_{\rig}(V^{\prime})$ define the same $(\varphi, \partial_{\tau})$-module in $\Mod_{\Rr}(\varphi, \partial_{\tau})$ if and only if there exist some $n\geq 0$ such that $V$ and $V^{\prime}$ are isomorphic in $\Rep_{\QQ_p}(\mathscr{G}_{K_n})$.
\end{proposition}

\begin{proof}
	The proof is almost the same as that of \cite[Proposition 2.2.3]{Poy21}. If there exists $n\geq 0$ such that $V\simeq V^{\prime}$ in $\Rep_{\QQ_p}(\mathscr{G}_{K_n}),$ then in particular $V\simeq V^{\prime}$ in $\Rep_{\QQ_p}(\mathscr{G}_{K_{\pi}}),$ and we have $\mathcal{D}^{\dagger}_{\rig}(V)=\mathcal{D}^{\dagger}_{\rig}(V^{\prime})$ as $\varphi$-modules. Moreover, the action $\tau^{p^n}$ is the same over $(\widetilde{\BB}^{\dagger}\otimes_{\QQ_p}V)^{\mathscr{G}_L}$ and $(\widetilde{\BB}^{\dagger}\otimes_{\QQ_p}V^{\prime})^{\mathscr{G}_L}$, hence $(\mathcal{D}^{\dagger}(V), \mathcal{D}^{\dagger}(V)_{\tau})$ and $(\mathcal{D}^{\dagger}(V^{\prime}), \mathcal{D}^{\dagger}(V^{\prime})_{\tau})$ are isomorphic in $\Mod_{{\BB_{K_{\pi}}^\dagger,\widetilde{\BB}_{L}^\dagger}}(\varphi, \tau^{p^n}).$ By the definition of the operator $\partial_{\tau},$ we have $\mathcal{D}^{\dagger}_{\rig}(V)=\mathcal{D}^{\dagger}_{\rig}(V^{\prime})$ in $\Mod_{\Rr}(\varphi, \partial_{\tau}).$
	
	\medskip
	
	Conversely, if $V$ and $V^{\prime}$ are two representations such that $\mathcal{D}^{\dagger}_{\rig}(V)=\mathcal{D}^{\dagger}_{\rig}(V^{\prime})$ in $\Mod_{\Rr}(\varphi, \partial_{\tau}),$ we prove that there exists $n\geq 0$ such that $(\mathcal{D}^{\dagger}(V), \mathcal{D}^{\dagger}(V)_{\tau})=(\mathcal{D}^{\dagger}(V^{\prime}), \mathcal{D}^{\dagger}(V^{\prime})_{\tau})$ in $\Mod_{{\BB_{K_{\pi}}^\dagger,\widetilde{\BB}_{L}^\dagger}}(\varphi, \tau^{p^n}).$ Let $r\geq 0$ and $(e_1, \ldots, e_d)\subset\mathcal{D}^{\dagger, r}(V)\cap \mathcal{D}^{\dagger, r}(V^{\prime})$ be such that $(e_1, \ldots, e_d)$ is a basis of the $\varphi$-module $\mathcal{D}^{\dagger}(V)= \mathcal{D}^{\dagger}(V^{\prime}).$ Let $s\geq r$ and let $I=[r, s]$. By lemma \ref{lemm Poyeton lanalytic lemma 1.3.4}, the $e_i$ are locally analytic vectors in $\widetilde{D}^{I}_L(V)=\widetilde{\BB}^{I}_L\otimes_{\BB^{\dagger, r}}\mathcal{D}^{\dagger, r}(V)$ and $\widetilde{D}^{I}_L(V^{\prime})=\widetilde{\BB}^{I}_L\otimes_{\BB^{\dagger, r}}\mathcal{D}^{\dagger, r}(V^{\prime})$ (remark that $\widetilde{D}^{I}_L(V)=\widetilde{D}^{I}_L(V^{\prime})$ as $\varphi$-modules by hypothesis). In particular, there exists $n\geq 0$ such that for all $i,$ $\exp(p^nt\partial_{\tau})(e_i)$ converges in $(\widetilde{D}^{I}_L(V))^{\mathsf{la}}$ and $(\widetilde{D}^{I}_L(V^{\prime}))^{\mathsf{la}}.$ Hence for all $i,$ the action of $\tau^{p^n}$ on $e_i$ inside $\widetilde{D}^{I}_L(V)$ and $\widetilde{D}^{I}_L(V^{\prime})$ are the same. As we have an injection $\widetilde{D}^{\dagger, r}_L(V)\inj \widetilde{D}^{I}_L(V)$ ($\cf$ \cite[Lemma 2.7]{Ber02}), we conclude that the $\tau^{p^n}$-actions over $\widetilde{D}^{\dagger}_{\rig, L}(V)$ and $\widetilde{D}^{\dagger}_{\rig, L}(V^{\prime})$ coincide (hence the $\tau_D^{p^n}$ endomorphisms over $\mathcal{D}^{\dagger}_{\rig}(V)_{\tau}$ and $\mathcal{D}^{\dagger}_{\rig}(V^{\prime})_{\tau}$ coincide by the proof of \ref{lemm Poyeton lanalytic lemma 1.3.4}), which finishes the proof.
\end{proof}

We now show how to construct a $(\varphi, \tau^{p^n})$-module (with $n\gg 0$) from a given $(\varphi, \partial_{\tau})$-module.

\begin{proposition}\label{prop N Nabla connection R-tau partial case} Let $D\in \Mod_{\Rr}(\varphi, \partial_{\tau}).$ Then there exist $n\geq 0$ and $D^{\prime}\in \Mod_{\Rr, \Rr_{\tau}}(\varphi, \tau^{p^n})$ such that $D\simeq D^{\prime}$ in $\Mod_{\Rr}(\varphi, \partial_{\tau}).$
\end{proposition} 

\begin{proof}
	For any $D\in \Mod_{\Rr}(\varphi, \partial_{\tau}),$ it has a natural $(\varphi, N_{\nabla})$-module structure (as $N_{\nabla}=-\lambda\partial_{\tau}).$ Then there exist $n\gg0$ and $D^{'}\in \Mod_{\Rr,\Rr_{\tau}}(\varphi, \tau^{p^n})$ such that $D\simeq D^{'}$ in $\Mod_{\Rr}(\varphi, N_{\nabla})$ ($\cf$ \cite[Proposition 2.2.5]{Poy21}). Hence the operator $N_{\nabla}$  over $D^{'}$ is also divisible by $\lambda$ (since it is for $D$), so that $D^{'}$ has a $(\varphi, \partial_{\tau})$-module structure and $D\simeq D^{'}$ in $\Mod_{\Rr}(\varphi, \partial_{\tau}).$
\end{proof}

\begin{lemma}\label{lemm Ext 1 nabla bij limit Ext 1 partial tau}
		If $V\in \Rep_{\QQ_p}(\mathscr{G}_K)$ is crystalline with Hodge-Tate weights $\HT(V)\subset\{2,\ldots,h\}$ for some $h\in\NN_{\geq2}$, then we have an isomorphism of groups 
			\[\iLim{n} \Ext^1_{\Mod_{\mathcal{R},\mathcal{R}_\tau}(\varphi, \tau^{p^n})}(\mathcal{D}^{\dagger}_{\rig}(V), \mathcal{R}) \to \Ext^1_{\Mod_{\Rr}(\varphi, \partial_{\tau})}(\mathcal{D}_{\rig}^{\dagger}(V), \mathcal{R}).\]
\end{lemma}
\begin{proof}
	We associate $V$ with its $(\varphi, \partial_{\tau})$-module $\mathcal{D}_{\rig}^{\dagger}(V)$ and $(\varphi, \tau^{p^n})$-module $\mathcal{D}^{\dagger}_{\rig}(V)_n$ (notice that $\mathcal{D}^{\dagger}_{\rig}(V)_n=\mathcal{D}^{\dagger}_{\rig}(V)$ as $\varphi$-modules, here we just use the subscript $n$ to indicate it is an object of $\Mod_{\Rr, \Rr_{\tau}}(\varphi, \tau^{p^n})$). If $n\in\NN,$ an extension of $\Rr$ by $\mathcal{D}^{\dagger}_{\rig}(V)_n$ is associated to an extension of $\QQ_p$ by $V$ (as representations of $\mathscr{G}_{K_{n}}$), which is semi-stable by \cite[Lemmas 6.5 and 6.6]{Ber02}, whence is equipped with an operator $\partial_{\tau}$ by corollary \ref{coro semistable}. This provides a map from $\Ext^1_{\Mod_{\mathcal{R},\mathcal{R}_\tau}(\varphi,\tau^{p^n})}(\mathcal{D}_{\rig}^{\dagger}(V)_n,\mathcal{R})$ to $\Ext^1_{\Mod_{\mathcal{R}}(\varphi,\partial_{\tau})}(\mathcal{D}^{\dagger}_{\rig}(V),\mathcal{R}).$ Those maps are compatible as $n$ grows: this provides the map \[\iLim{n}\Ext^1_{\Mod_{\mathcal{R},\mathcal{R}_\tau}(\varphi,\tau^{p^n})}(\mathcal{D}^{\dagger}_{\rig}(V)_n,\mathcal{R})\to \Ext^1_{\Mod_{\mathcal{R}}(\varphi,\partial_{\tau})}(\mathcal{D}_{\rig}^{\dagger}(V),\mathcal{R}).\]
	By proposition \ref{prop N Nabla connection R-tau partial case}, an extension $E$ of $\Rr$ by $\mathcal{D}_{\rig}^{\dagger}(V)$ in $\Mod_{\mathcal{R}}(\varphi, \partial_{\tau})$ has a $(\varphi, \tau^{p^n})$-module structure for some $n\gg 0.$ This shows that the extension $E$ comes from an extension of $\Rr$ by $\mathcal{D}^{\dagger}_{\rig}(V)$ in the category $\Mod_{\mathcal{R}, \Rr_{\tau}}(\varphi, \tau^{p^n})$ for some $n\gg 0.$ This implies the map is surjective. Moreover, it is also injective by proposition \ref{prop injective on the extensino argument partial case}. 
\end{proof}

\begin{proposition}\label{prop H0 H1 of complex C_R(V) partial tau} If $V\in \Rep_{\QQ_p}(\mathscr{G}_K)$ is crystalline with Hodge-Tate weights $\HT(V)\subset\{2,\ldots,h\}$ for some $h\in\NN_{\geq2}$, we have a group isomorphism 
	\[\H^i(\Cc_{\varphi, \partial_{\tau}}(V))\simeq \iLim{n}\H^i(\mathscr{G}_{K_n}, V) \text{ for } i\in \{0, 1 \}.\]
\end{proposition}

\begin{proof}
	When $i=0,$ it suffices to show that $\partial_{\tau}(x)=0$ is equivalent to $\tau^{p^n}(x)=x$ for some $n\gg 0.$ This is indeed the case as  $\exp(p^nt\partial_{\tau})=\tau^{p^n}$ for $n\gg 0.$ When $i=1,$ the isomorphism of lemma \ref{lemm Ext 1 nabla bij limit Ext 1 partial tau}
	identifies with an isomorphism $\H^1(\mathcal{C}_{\varphi, \partial_{\tau}}(V))\simeq \iLim{n} \H^1(\mathscr{G}_{K_n},V).$  
\end{proof}

\subsection{\texorpdfstring{$\H^2(\Cc_{\varphi, \partial_{\tau}}(V))$}{H\texttwosuperior (C\textpinferior\texthinferior\textiinferior, \unichar{"2202}\texttinferior\textainferior\textuinferior(V))}}
We assume $k$ is finite.
\begin{proposition}\label{prop pairing partial}
	Let $V\in \Rep_{\QQ_p}(\mathscr{G}_K),$ then we have a pairing of groups
	\[\H^0(\Cc_{\varphi, \partial_{\tau}}(V^{\vee}(1))) \times \H^2(\Cc_{\varphi, \partial_{\tau}}(V)) \to \QQ_p.  \]
\end{proposition}

\begin{proof}
	We have $\mathcal{D}^{\dagger}(V^{\vee}(1))=\Hom_{\Ee^{\dagger}}(\mathcal{D}^{\dagger}(V), \mathcal{D}^{\dagger}(\QQ_p(1))).$ Hence we have
	\begin{align*}
	\H^0(\Cc_{\varphi, \partial_{\tau}}(V^{\vee}(1)))&= (\mathcal{D}^{\dagger}(V^{\vee}(1)))^{\varphi=1, \partial_{\tau}=0}\\
	&=(\Hom_{\Ee^{\dagger}}(\mathcal{D}^{\dagger}(V), \mathcal{D}^{\dagger}(\QQ_p(1))))^{\varphi=1, \partial_{\tau}=0}\\
	&=(\Hom_{\Rr}(\mathcal{D}^{\dagger}(V), \Ee^{\dagger} \mathfrak{t}^{-1}))^{\varphi=1, \partial_{\tau}=0}
	\end{align*}
	
	We start with the following pairing:
	
	\begin{align*}
	\H^0(\Cc_{\varphi, \partial_{\tau}}(V^{\vee}(1))) \times (\mathcal{R}\otimes_{\Ee^{\dagger}} \mathcal{D}^{\dagger}(V)) &\to \Rr \mathfrak{t}^{-1} \\
	(f,  \ x) &\mapsto f(x).
	\end{align*}	
	
	Recall that elements in $\Rr$ can be written uniquely as series $\sum\limits_{n\in\ZZ} a_n u^n$ with $a_n\in \W(k)\big[1/p\big].$ We define a residue map over $\Rr \mathfrak{t}^{-1}$ as follows: for any $z=\mathfrak{t}^{-1}\sum\limits_{n\in\ZZ} a_n u^n=\frac{p\lambda}{t}\sum\limits_{n\in\ZZ}a_nu^n=\frac{1}{t}\sum\limits_{n\in\ZZ}b_nu^n,$
	\begin{align*}
	\res\colon \Rr \mathfrak{t}^{-1}&\to \W(k)\big[1/p\big]\\
	z&\mapsto b_0. 
	\end{align*}
	Compose with the trace map $\Tr=\Tr_{\W(k)\big[1/p\big]/\QQ_p}\colon \W(k)\big[1/p\big]\to \QQ_p,$ we have the following map	
	\begin{align*}
\H^0(\Cc_{\varphi, \partial_{\tau}}(V^{\vee}(1))) \times (\mathcal{R}\otimes_{\Ee^{\dagger}} \mathcal{D}^{\dagger}(V)) &\to \QQ_p\\
(f,  \ x) &\mapsto \Tr (\res f(x)).
\end{align*}	
	Recall that $\H^2(\Cc_{\varphi, \partial_{\tau}}(V))\simeq (\mathcal{R}\otimes_{\Ee^{\dagger}} \mathcal{D}^{\dagger}(V))/\im((p\varphi-1)\oplus \partial_{\tau})).$ 
	To construct the claimed pairing, we must show that the above map factors through $\im((p\varphi-1)\oplus \partial_{\tau})\subset \mathcal{R}\otimes_{\Ee^{\dagger}} \mathcal{D}^{\dagger}(V).$ In other words, for any $x\in\im((p\varphi-1)\oplus \partial_{\tau},$ we have to show that $\Tr (\res f(x))=0.$
	
	Notice that for any $f\in\Hom_{\Rr}(\Rr\otimes_{\Ee^{\dagger}}\mathcal{D}^{\dagger}(V), \Rr\otimes_{\Ee^{\dagger}}\mathcal{D}^{\dagger}(\QQ_p(1)))^{ \varphi=1, \partial_{\tau}=0},$ the condition $\partial_{\tau}(f)=0$ means that  $\partial_{\tau}\circ f=f\circ \partial_{\tau},$ and $\varphi(f)=f$ means that $\varphi\circ f=f\circ \varphi.$ Indeed, the actions are $\partial_{\tau}(f)=\partial_{\tau}\circ f-f\circ \partial_{\tau}$ and $\varphi(f)=\varphi\circ f\circ \varphi^{-1}.$
	
	For any $x\in \mathcal{R}\otimes_{\Ee^{\dagger}} \mathcal{D}^{\dagger}(V)$ and $f\in \H^0(\Cc_{\varphi, \partial_{\tau}}(V^{\vee}(1))),$ write $f(x)=\mathfrak{t}^{-1} r=\mathfrak{t}^{-1}\sum\limits_{n\in\ZZ} a_n u^n \in \Rr \mathfrak{t}^{-1}$ with $a_n\in \W(k)\big[1/p\big].$ We have
	\begin{align*}
	\Tr\circ \res (f(\partial_{\tau}(x)))&=\Tr\circ \res \big(\partial_{\tau}(f(x))\big)\\
	&=\Tr\circ \res \big(\partial_{\tau}(\mathfrak{t}^{-1} r)\big)\\
	&=\Tr\circ \res \Big(\partial_{\tau}({\mathfrak{t}^{-1}})r+\partial_{\tau}(r)\mathfrak{t}^{-1}\Big)\\
	&=\Tr\circ \res \Big(\frac{\partial_{\tau}(\lambda)}{\lambda}{\mathfrak{t}^{-1}}r+\partial_{\tau}(r)\mathfrak{t}^{-1}\Big)\\
	&=\Tr\circ \res\Big(\frac{\partial_{\tau}(\lambda r)}{\lambda}\mathfrak{t}^{-1}\Big)\\
	&=  \Tr\circ \res\Big(\frac{p}{t} \partial_{\tau}(\lambda r) \Big)\\
	&=  \Tr\circ \res\Big(\frac{p}{t} \cdot u \frac{d}{du}(\lambda r) \Big)\\
	&=\Tr(0)\\
	&=0.      
	\end{align*} 	
	We have $\varphi(\mathfrak{t}^{-1})=c^{-1}\mathfrak{t}^{-1},$ hence 
	\begin{align*}
	\Tr\circ \res (f((p\varphi-1)x))&=\Tr\circ \res ((p\varphi-1)f(x))\\
	&=\Tr\circ \res \bigg((p\varphi-1)(\mathfrak{t}^{-1}\sum\limits_{n\in\ZZ} a_n u^n)\bigg)\\
	&=\Tr\circ \res \bigg(pc^{-1}\mathfrak{t}^{-1} \varphi(\sum\limits_{n\in\ZZ} a_n u^n)-\mathfrak{t}^{-1}\sum\limits_{n\in\ZZ} a_n u^n\bigg)\\
	&=\Tr\circ \res \bigg(pc^{-1}\frac{p\lambda}{t} \varphi(\sum\limits_{n\in\ZZ} a_n u^n)-\frac{p\lambda}{t}\sum\limits_{n\in\ZZ} a_n u^n \bigg)\\
	&=\Tr\circ \res \bigg(pc^{-1}\frac{p\lambda}{t} \varphi(\sum\limits_{n\in\ZZ} a_n u^n)-\frac{p\lambda}{t}\sum\limits_{n\in\ZZ} a_n u^n \bigg)\\
	&=\Tr\circ \res \bigg(\frac{p}{t} \varphi(\lambda)\varphi(\sum\limits_{n\in\ZZ} a_n u^n)-\frac{p}{t}\lambda\sum\limits_{n\in\ZZ} a_n u^n \bigg)\\
	&=\Tr\circ \res \bigg(\frac{p}{t}\cdot (\varphi-1)\big(\lambda\cdot \sum\limits_{n\in\ZZ} a_n u^n\big) \bigg)\\
	&=\Tr(\varphi(b_0)-b_0)\\
	&=(\varphi-1)\Tr(b_0)\\
	&=0.
	\end{align*} 
	where $b_0$ is the constant term of $p\cdot \lambda\cdot \sum\limits_{n\in\ZZ} a_n u^n.$\\
	
	Hence we have the desired pairing 
	\begin{align*}	\H^0(\Cc_{\varphi, \partial_{\tau}}(V^{\vee}(1))) \times \H^2(\Cc_{\varphi, \partial_{\tau}}(V)) &\to \QQ_p\\
	(f, x) &\mapsto B(f, x)=\Tr (\res f(x)).
	\end{align*} 
	
\end{proof}

\begin{remark}
	As $t=p\lambda\mathfrak{t}$ and $\partial_{\tau}(t)=0,$ we have 
	\[ \partial_{\tau}(\lambda)\mathfrak{t}+\lambda \partial_{\tau}(\mathfrak{t})=0.  \]
	Hence \[ \partial_{\tau}(\mathfrak{t})=-\frac{\mathfrak{t}}{\lambda} \partial_{\tau}(\lambda) \]
	and
	\[\partial_{\tau}(\mathfrak{t}^n)=n\mathfrak{t}^{n-1}\partial_{\tau}(\mathfrak{t})=-n\mathfrak{t}^n\frac{\partial_{\tau}(\lambda)}{\lambda}.\]
\end{remark}

\begin{example}\label{example Robba for H2=0 N-operator}
	For $V=\QQ_p,$ we have $\mathcal{D}_{\varphi, \partial_{\tau}}(\QQ_p)\simeq \Rr$ and the complex $\mathcal{C}_{\varphi, \partial_{\tau}}(\QQ_p)$ is 
	
	\[ \xymatrix{ 
		0\ar[rr]&& \Rr \ar[r] & \Rr \oplus \Rr  \ar[r]&  \Rr \ar[r] & 0 \\
		&&x  \ar@{|->}[r] & ((\varphi-1)(x), \partial_{\tau}(x))  &{}\\
		&&&{} (y,z) \ar@{|->}[r] & \partial_{\tau}(y)-(p\varphi-1)(z).
	} \]

\medskip

	Let's check that $\H^2(\Cc_{\varphi, \partial_{\tau}}(\QQ_p))=0.$ For any $x=\sum\limits_{n\in\ZZ} x_n u^n\in \Rr,$ we have $a_0\in \im(c\varphi-1)$ by lemma \ref{lemm topology in O N-operator}. Hence it suffices to show that 
		\[ x^{\prime}:=\sum\limits_{n\in\ZZ\setminus\{0 \}} x_n u^n\in \im(\partial_{\tau}).\]
		As $y:=\sum\limits_{n\in\ZZ\setminus\{ 0\}} \frac{x_n}{n} u^n$ satisfies $\partial_{\tau}(y)=x^{\prime},$ it is enough to show that $y\in \Rr,$ which follows from lemma \ref{lemm N preserves surconvergence}.
\end{example}
	
\begin{remark}
	Let $n\in \ZZ$, we then have 
	\[\H^2(\Cc_{\varphi, \partial_{\tau}}(\QQ_p)(n))=\Rr\mathfrak{t}^{-n}\big/(p\varphi-1)(\Rr\mathfrak{t}^{-n})+ \partial_{\tau}(\Rr\mathfrak{t}^{-n})\simeq \Rr\big/\big(\im(pc^{-n}\varphi-1)+\im(\partial_{\tau}+n\frac{\partial_{\tau}(\lambda)}{\lambda})\big).\]
\end{remark}

%% file: Chapter6.tex
\chapter{Applications to $p$-divisible groups}

In this chapter, we apply the results in the previous chapters to $p$-divisible groups, which allows us to compute the Galois cohomology of (the dual of) the Tate module of some $p$-divisible group using its associated Breuil-Kisin module. 

\section{$p$-divisible groups and Tate modules}

\begin{definition}($\cf$ \cite{Tat67})
	A \emph{Barsotti-Tate group} (or \emph{$p$-divisible group}) of \emph{height} $h$ over a commutative ring $R$ is an inductive system $(G_{n},i_{n})_{i \geq 1}$ in which:
	\item (1) $G_{n}$ is a finite, commutative group scheme over $R$ of order $p^{nh}$;
	\item (2) for each $n$, we have an exact sequence
	\[ 0 \to  G_{n} \xrightarrow{i_{n}}  G_{n+1}\xrightarrow{p^{n}} G_{n+1} \]
	(that is, $i_{n}$ is a closed immersion and identifies $G_{n}$ with the kernel of multiplication by $p^{n}$ on $G_{n+1}$).
	
	\medskip
	
	The corresponding category of Barsotti-Tate groups over $R$ is denoted $\BT_{R}.$\index{$\BT_{R}$}
\end{definition}

\begin{remark}
In particular, $\BT_{\Oo_K}$ is the category of Barsotti-Tate groups over the ring of integers of $K.$ 
\end{remark}

\begin{definition}
Let $G=\iLim{n}G_n$ be a $p$-divisible group. The \emph{Tate module} of $G$ is $\pLim{n}G_n( \overline{K}),$ denoted $\T_pG.$\index{$\T_pG$}	
\end{definition}
It carries a natural continuous $\mathscr{G}_K$-action on it and hence $\T_p G\in \Rep_{\ZZ_p}(\mathscr{G}_K),$ inducing a functor:

\begin{align*}
\T_p\colon \BT_{\Oo_K}&\to \Rep_{\ZZ_p}(\mathscr{G}_K)\\
G &\mapsto\T_p G.
\end{align*}

\begin{theorem} (Fontaine, Kisin, Raynaud, Tate)
	The functor $\T_p$ induces an equivalence of categories between $\BT_{\Oo_K}$ and $\Rep^{\cris, \{0, 1\}}_{\ZZ_p}(\mathscr{G}_K)$\index{$\Rep^{\cris, \{0, 1\}}_{\ZZ_p}(\mathscr{G}_K)$} $($the category of crystalline representations of $\mathscr{G}_K$ over $\ZZ_p$ with Hodge-Tate weights in $\{0,1 \}).$
\end{theorem} 

\begin{proof}
\cf \cite[Theorem 2.2.1]{Liu13}.
\end{proof}

\section{Some inputs from Kisin's work}\label{section N nabla action O times M}

\begin{definition}
(1)	We define a \emph{$(\varphi, N)$-module} over $\Ss=\W(k)[\![u]\!]$\index{$\Ss$} to be a finite free $\Ss$-module $\Mm,$ equipped with a $\varphi$-similinear Frobenius $\varphi\colon \Mm \to \Mm,$ and a linear endomorphism 
	\[N\colon (\Mm/u\Mm) \otimes_{\ZZ_p} \QQ_p \to (\Mm/u\Mm) \otimes_{\ZZ_p} \QQ_p, \]
	such that $N\varphi=p\varphi N$ on $(\Mm/u\Mm) \otimes_{\ZZ_p} \QQ_p.$ 
	
	\medskip
	
	We say that $\Mm$ is of \emph{finite $E$-height} if the cokernel of  $1\otimes \varphi\colon \varphi^{\ast}\Mm \to \Mm$ is killed by some power of $E(u).$ We denote by $\Mod_{\Ss}(\varphi, N)$\index{$\Mod_{\Ss}(\varphi, N)$} the category of $(\varphi, N)$-modules over $\Ss$ of finite $E$-height.
\item (2) ($\cf$ \cite[\S 2.1.3]{Kis06}) We denote by $\Mod_{\Ss}(\varphi)$\index{$\Mod_{\Ss}(\varphi)$} the category of finite free $\Ss$-modules equipped with an $\Ss$-linear map $1\otimes \varphi\colon \varphi^{\ast}\Mm \to \Mm$ whose cokernel is killed by some power of $E(u).$ 
\end{definition}

\begin{remark}
	The category $\Mod_{\Ss}(\varphi)$ is a full subcategory of $\Mod_{\Ss}(\varphi, N)$ by taking the operator $N$ to be 0 on an object of $\Mod_{\Ss}(\varphi).$
\end{remark}

\begin{definition}($\cf$ \cite[\S 2.2.1]{Kis06})
	We denote by $\BT_{\Ss}^{\varphi}$\index{$\BT_{\Ss}^{\varphi}$} the category consisting of objects $\Mm\in \Mod_{\Ss}(\varphi)$ such that $\Mm/\varphi^{\ast}(\Mm)$ is killed by $E(u),$ where $\varphi^{\ast}=1\otimes \varphi\colon \varphi^{\ast}\Mm \to \Mm.$ Objects of $\BT^{\varphi}_{\Ss}$ are also called \emph{Breuil-Kisin modules} (or \emph{Kisin modules of height 1}). 
\end{definition}

\begin{remark}
	The category $\BT^{\varphi}_{\Ss}$ is a full subcategory of $\Mod_{\Ss}(\varphi).$
\end{remark}

\begin{theorem}\label{thm Kisin}
	There exists an equivalence of categories between $\BT^{\varphi}_{\Ss}$ and $\BT_{\Oo_K}.$
\end{theorem}
\begin{proof}
	This is \cite[Theorem 2.2.7]{Kis06} if $p>2$ and \cite[Theorem 1.0.1]{Liu13} when $p=2$.
\end{proof}

\begin{definition}
(1) ($\cf$ \cite[\S 1.1.4]{Kis06}) A \emph{$\varphi$-module over $\Oo$} is a finite free $\Oo$-module $\CMcal{M}$ endowed with $\varphi$-semilinear injective map $\varphi\colon \CMcal{M}\to \CMcal{M}.$ We denote $\Mod_{\Oo}(\varphi)$\index{$\Mod_{\Oo}(\varphi)$} the corresponding category.
\item (2) A \emph{$(\varphi, N_{\nabla})$-module over $\Oo$} is a $\varphi$-module $\CMcal{M}$ equipped with a differential operator $N_{\nabla}\colon \CMcal{M}\to \CMcal{M}$ over $N_{\nabla}\colon \Oo\to \Oo,$ $\ie$a map such that 
\[ (\forall f\in \Oo) (\forall m\in \CMcal{M})\quad N_{\nabla}(fm)=N_{\nabla}(f)m+fN_{\nabla}(m),  \]
satisfying the condition $N_{\nabla}\circ \varphi=c\varphi\circ N_{\nabla}.$ We denote $\Mod_{\Oo}(\varphi, N_{\nabla})$\index{$\Mod_{\Oo}(\varphi, N_{\nabla})$} the corresponding category.	
\item (3) ($\cf$ \cite[\S 1.3.9]{Kis06})
	A \emph{$(\varphi, N)$-module over $\Oo$} is a $\varphi$-module $\CMcal{M}$ over $\Oo$ together with a $\W(k)[1/p]$-linear map 
	\[N\colon \CMcal{M}/u\CMcal{M} \to \CMcal{M}/u\CMcal{M}\]
	which satisfies $N\varphi=p\varphi N,$ where we have written $\varphi$ for the endomorphism of $\CMcal{M}/u\CMcal{M}$ obtained by reducing $\varphi\colon \CMcal{M} \to \CMcal{M}$ modulo $u.$ We denote by $\Mod_{\Oo}(\varphi, N)$\index{$\Mod_{\Oo}(\varphi, N)$} the category of $(\varphi, N)$-modules over $\Oo$ of finite $E$-height.
\item (4) ($\cf$ \cite[\S 1.2.5]{Kis06})
	Let $\CMcal{M}\in\Mod_{\Oo}(\varphi).$ The associated \emph{filtered $\varphi$-module} $\mathcal{D}(\CMcal{M})$ is defined to be the $\W(k)[1/p]$-vector space $\CMcal{M}/u\CMcal{M}$ together with the operator $\varphi$ induced from $\CMcal{M}$ (the filtration is described in \cite[\S 1.2.7]{Kis06}). 
\end{definition}

\begin{lemma}\label{lemm 1.2.6}
	Let $\CMcal{M}$ be a $\varphi$-module over $\Oo.$ There is a unique $\Oo$-linear, $\varphi$-equivariant morphism 
	\[ \xi\colon \mathcal{D}(\CMcal{M})\otimes_{\W(k)[1/p]}\Oo \to \CMcal{M},  \]
	whose reduction modulo $u$ induces the identity on $\mathcal{D}(\CMcal{M}).$ The map $\xi$ is injective and its cokernel is killed by a finite power of $\lambda.$ If $r\in (|\pi|, |\pi|^{1/p}),$ then the image of the map $\xi_{{[}0, r)}$  induced by $\xi$ over $D([0, r])$ coincides with the image of $1\otimes \varphi\colon (\varphi^{\ast}\CMcal{M})_{{[} 0, r)} \to \CMcal{M}_{{[} 0, r)}.$ 
\end{lemma}

\begin{proof}
	$\cf$ \cite[Lemma 1.2.6]{Kis06}.
\end{proof}

\begin{lemma} \label{lemm 1.3.11}
Let $\CMcal{M}\in\Mod_{\Oo}(\varphi, N_{\nabla}),$ the morphism
\[\xi\colon \mathcal{D}(\CMcal{M})\otimes_{\W(k)[1/p]}\Oo \to \CMcal{M}\]
is compatible with the differential operator $N_{\nabla}$ $($the one on the left hand side being given by $1\otimes N_{\nabla})$.
\end{lemma}
\begin{proof}
	This is \cite[Lemma 1.2.12 (3)]{Kis06}, noting that we have $N=0$ thus $\eta=\id$ in our case.
\end{proof}

\begin{remark}\label{rem formule N nabla}
(1) Let $\Mm\in \Mod_{\Ss}(\varphi)$ and $\CMcal{M}=\Mm\otimes_{\Ss}\Oo$ the corresponding object in $\Mod_{\Oo}(\varphi, N).$ Then $\CMcal{M}$ is also an object in $\Mod_{\Oo}(\varphi).$ By lemma \ref{lemm 1.3.11}, we have a map $\mathcal{D}(\CMcal{M})\otimes_{\W(k)[1/p]}\Oo \to \CMcal{M}$ that lifts the identity of $\mathcal{D}(\CMcal{M})\simeq \CMcal{M}/u\CMcal{M},$ and is compatible with $\varphi$ and $N_{\nabla}.$ Here $N_{\nabla}$ acts on $\mathcal{D}(\CMcal{M})\otimes_{\W(k)[1/p]} \Oo$ as $1\otimes-u\lambda\frac{d}{du}$ ($\cf$ \cite[1.2]{Kis06}).
\item (2) As $\CMcal{M}/\varphi^{\ast}\CMcal{M}$ is killed by some power of $E(u),$ say $E(u)^n$ with $n\in \NN,$ then $\CMcal{M}/(\mathcal{D}(\CMcal{M})\otimes_{\W(k)[1/p]} \Oo)$ is killed by $\lambda^n$ (\cf lemma \ref{lemm 1.2.6} and the proof of \cite[Lemma 1.2.6]{Kis06}). Let $m\in \CMcal{M},$ and write $m=\sum\limits_{i=1}^{r}d_i\otimes \lambda^{-n}f_i$ with $d_i\in \mathcal{D}(\CMcal{M})$ and $f_i\in \Oo,$ then we have 
\[  N_{\nabla}(m)= \sum\limits_{i=1}^{r}d_i\otimes -u\lambda\frac{d}{du}(\lambda^{-n}f_i)=\sum\limits_{i=1}^{r}d_i\otimes(un\lambda^{-n}f_i\frac{d\lambda}{du}-u\lambda^{1-n}\frac{df_i}{du}).\]
\end{remark}

\section{Recover $\tau$-action from the $N_{\nabla}$-action}\label{section recover tau from N nabla}

\begin{definition}
(1) ($\cf$ \cite[D\'efinition 2.1]{Car13}) Let $D$ be an \'etale $\varphi$-module over $\Oo_{\Ee}.$ A \emph{$\varphi$-lattice} in $D$ is a finite type sub $\Ss$-module $\Mm\subset D$ which is stable by $\varphi$ and such that $D=\Oo_{\Ee}\otimes_{\Ss}\Mm.$
\item (2)  ($\cf$ \cite[D\'efinition 2.19]{Car13}) Let $(D, D_{\tau})$ be a $(\varphi, \tau)$-module over $(\Oo_{\Ee}, \Oo_{\Ee_{\tau}}).$ A \emph{$(\varphi, \tau)$-lattice} in $(D, D_{\tau})$ is a $\varphi$-lattice $\Mm$ of $D$ such that $\Ss_{\tau}\otimes_{\Ss}\Mm\subset \Oo_{\Ee_{\tau}}\otimes_{\Oo_{\Ee}}D=D_{\tau}$ is stable by $\tau,$ where $\Ss_{\tau}=\Oo_{\Ee_{\tau}}\cap\, \W(\Oo_{C^{\flat}}).$\index{$\Ss_{\tau}$}
\end{definition}

\begin{remark}
	Let $T\in\Rep_{\ZZ_p}(\mathscr{G}_K)$ be such that $V=\QQ_p\otimes_{\ZZ_p}T$ is semistable with non negative Hodge-Tate weights. Denote by $\Mm$ the corresponding $(\varphi, \tau)$-lattice (\cf \cite[Proposition 3.1]{Car13}) in the (contravariant) $(\varphi, \tau)$-module : \[\mathcal{D}^{\ast}(T[1/p]):=\Hom_{\QQ_p[\mathscr{G}_{K_{\pi}}]}(T[1/p], \widehat{\Ee^{\ur}})=\mathcal{D}(T^{\vee}[1/p]).\]\index{$D^{\ast}$}
\item (1) The underlying $\varphi$-lattice is the object in $\Mod_{\Ss}(\varphi)$ associated to $T$ by Kisin in \cite{Kis06} (\cf \cite[\S 3.1.1]{Car13}). 
\item (2) Put $\CMcal{M}=\Oo\otimes_{\Ss}\Mm,$ equipped with its operator $N_{\nabla}.$ By \cite[Proposition 2.23]{Car13}, we have 
	\[ N_{\nabla}(\Mm) \subset \Ss_{\nabla}\otimes_{\Ss}\Mm, \] 
	where $\Ss_{\nabla}$ is the set of sums of the form $\sum\limits_{n\geq 0}\frac{P_n(u)}{p^{n+1}}u^{e\frac{p^n-1}{p-1}}$ and the $P_n(u)$ are polynomials with coefficients in $\W(k).$
\end{remark}

\begin{definition}($\cf$ \cite[\S 3.3.2]{Car13})
	For any nonnegative integer $i,$ we define 
	\[\Rr^{\mathsf{int}}_i = \Big\{  \sum\limits_{n\geq 0}a_n u^n ; \,  v_p(a_n)+i\log_p(n) \text{ are bounded below for } n\geq 1 \Big\}.  \]\index{$\Rr^{\mathsf{int}}_i$}
\end{definition} 

\begin{remark}($\cf$ \cite[\S 3.3.2]{Car13})
	Remark that $\Rr^{\mathsf{int}}_i \subset \Oo.$ If $i$ and $j$ are two integers, the product of one function of $\Rr^{\mathsf{int}}_{i}$ by another function of $\Rr^{\mathsf{int}}_j$ falls into $\Rr^{\mathsf{int}}_{i+j}.$ 
\end{remark}

\begin{definition}($\cf$ \cite[\S 3.3.2]{Car13})
	We define $N_{\nabla}^{(0)}=\id$ and 
	\[ N_{\nabla}^{(i+1)}=iu\cdot\frac{d \lambda}{d u}\cdot N_{\nabla}^{(i)} + N_{\nabla}\circ N_{\nabla}^{(i)}, \]\index{$ N_{\nabla}^{(i+1)}$}which then define $\Ss$-linear maps $N_{\nabla}^{(i)}\colon \Mm \to \Rr_i^{\mathsf{int}}\otimes_{\Ss}\Mm$. The operator $\sum\limits_{i\geq 0}\frac{(p \mathfrak{t})^i}{i!}\cdot N_{\nabla}^{(i)}$ converges to a $\W(k)$-linear map 
	\[\tau\colon \Mm \to \BB^{+}_{\cris}\otimes_{\Ss}\Mm \]
(recall that $\BB_{\cris}^{+}=\AA_{\cris}[1/p]$ where $\AA_{\cris}$ is the $p$-adic completion of the divided power envelope of $\W(\Oo_{C^{\flat}})$ with respect to the ideal generated by $1+[\varepsilon^{1/p}]+[\varepsilon^{1/p}]^{2} +\cdots + [\varepsilon^{1/p}]^{p-1},$ \cf\cite{Fon94}).	 
\end{definition}

\begin{remark}\label{rem tau action} (1) By \cite[Remarque 3.25]{Car13}, we have in fact
	\[ \tau(\Mm)\subset \mathsf{A}\otimes_{\Ss}\Mm, \]
	where $\mathsf{A}=\W(\Oo_{C^{\flat}})[\mathfrak{t}/p]^{\wedge}$ is the $p$-adic completion of $\W(\Oo_{C^{\flat}})[\mathfrak{t}/p].$
\item (2) By \cite[Proposition 3.27]{Car13}, we have $\tau(\Mm)\subset \W(\Oo_{C^{\flat}})\otimes_{\Ss}\Mm,$ which inserts in the following commutative square

\[ \xymatrix{
\Mm \ar[rr]^{\tau} \ar@{^(->}[d]&& \W(\Oo_{C^{\flat}})\otimes_{\Ss}\Mm\\
\Oo_{\Ee}\otimes_{\Ss}\Mm \ar^{\simeq}[d]&&\Oo_{\Ee_{\tau}}\otimes_{\Ss}\Mm\ar[d]^{\simeq}\ar@{^(->}[u]\\
\mathcal{D}^{\ast}(T)\ar[rr]^{\tau_D} && \mathcal{D}^{\ast}(T)_{\tau}.
}\] 	
\item (3)  We have $\tau(ux)=[\varepsilon]\cdot \tau(x)$ for any $x\in \Mm,$ which extends $\tau$ to $\BB_{\cris}^{+}\otimes_{\Ss}\Mm$ by semi-linearity ($\cf$ \cite[\S 3.3.2]{Car13}).

\end{remark}

\section{Relate \texorpdfstring{$(\varphi, \tau)$}{(phi,tau)}-modules with Breuil-Kisin modules}
In the following, let $G$ be a $p$-divisible group over $\Oo_K$ and denote $\T_p G$ (or simply denote $\T$) its Tate module and $\Mm$ its associated Breuil-Kisin module. Recall that $\mathcal{D}^{\ast}(\T):=\Hom_{\ZZ_p[\mathscr{G}_{K_{\pi}}]}(\T, \Oo_{\widehat{\Ee^{\ur}}})\simeq \Mm\otimes_{\Ss}\Oo_{\Ee}$ and $\mathcal{D}^{\ast}(\T[1/p])\simeq \Mm\otimes_{\Ss}\Ee.$

\begin{notation}
Let $T\in \Rep_{\ZZ_p}(\mathscr{G}_K),$ then we put 
\[ \mathcal{D}^{\ast, \dagger}(T)=\Hom_{\ZZ_p[\mathscr{G}_{K_{\pi}}]}(T, \Oo_{\Ee^{\ur, \dagger}}). \]\index{$ \mathcal{D}^{\ast, \dagger}$}
Let $V\in \Rep_{\QQ_p}(\mathscr{G}_K),$ then we put 
\[ \mathcal{D}^{\ast, \dagger}(V)=\Hom_{\QQ_p[\mathscr{G}_{K_{\pi}}]}(V, \Ee^{\ur, \dagger}). \]
\end{notation}

\begin{lemma}\label{lemm breuil kisin Robba ring}
We have 
\[ \mathcal{D}^{\ast, \dagger}(\T[1/p])\simeq \Ee^{\dagger}\otimes_{\Ss}\Mm, \]
hence 
\[ \Rr\otimes_{\Ee^{\dagger}} \mathcal{D}^{\ast, \dagger}(\T[1/p])\simeq \Rr\otimes_{\Ss}\Mm.   \]
\end{lemma}

\begin{proof}
	By \cite[Theorem 2.2.7]{Kis06}, we have
	\[ \T=\Hom_{\Ss, \varphi}(\Mm, \Ss^{\ur}) \]
	(where $\Ss^{\ur}=\Oo_{\Ee^{\ur}}\cap \W(\Oo_{C^{\flat}})).$ \index{$\Ss^{\ur}$} Consider the following diagram
	
	\[ \xymatrix{ \T=\Hom_{\Ss, \varphi}(\Mm, \Ss^{\ur}) \ar[rd]\ar[r]& \Hom_{\Oo_{\Ee}, \varphi}(\Oo_{\Ee}\otimes_{\Ss}\Mm, \Oo_{\widehat{\Ee^{\ur}}}) \simeq \Hom_{\Ss, \varphi}(\Mm, \Oo_{\widehat{\Ee^{\ur}}})\\
			&\Hom_{\Oo_{\Ee^{\dagger}}, \varphi}(\Oo_{\Ee^{\dagger}}\otimes_{\Ss}\Mm, \Oo_{\Ee^{\ur, \dagger}})\simeq \Hom_{\Ss, \varphi}(\Mm, \Oo_{\Ee^{\ur, \dagger}}). \ar@{^(->}[u]
	}  \]
The horizontal map is an isomorphism (\cf \cite[2.1.4]{Kis06}), and the vertical map is injective (as $\Oo_{\Ee^{\ur, \dagger}}\subset \Oo_{\widehat{\Ee^{\ur}}}$), hence all maps in the diagram are isomorphisms. We then have 
\[ \T[1/p]\simeq \Hom_{\Ee^{\dagger}, \varphi}(\Ee^{\dagger}\otimes_{\Ss}\Mm, \Ee^{\ur, \dagger})=\big( (\Ee^{\dagger}\otimes_{\Ss}\Mm)^{\vee}\otimes_{\Ee^{\dagger}}\Ee^{\ur, \dagger} \big)^{\varphi=1}.   \]
In particular  $\T[1/p]\subset  (\Ee^{\dagger}\otimes_{\Ss}\Mm)^{\vee}\otimes_{\Ee^{\dagger}}\Ee^{\ur, \dagger},$  hence we have a $\mathscr{G}_{K_{\pi}}$-equivariant map \begin{equation}\label{equ BK module dfagger}
	\Ee^{\ur, \dagger}\otimes_{\QQ_p}\T[1/p]\to  (\Ee^{\dagger}\otimes_{\Ss}\Mm)^{\vee}\otimes_{\Ee^{\dagger}}\Ee^{\ur, \dagger}
\end{equation}
by $\Ee^{\ur, \dagger}$-linearity. Notice that (\ref{equ BK module dfagger}) is an isomorphism after tensoring with $\widehat{\Ee^{\ur}}$ (as $\mathcal{D}^{\ast}(\T[1/p])\simeq \Ee\otimes_{\Ss}\Mm$), hence (\ref{equ BK module dfagger}) is an isomorphism. We then have 
\begin{align*}
	\mathcal{D}^{\ast, \dagger}(\T[1/p])&=\Hom_{\QQ_p[\mathscr{G}_{K_{\pi}}]}(\T[1/p], \Ee^{\ur, \dagger})\\
	                                   &=\Hom_{\Ee^{\ur, \dagger}}(\Ee^{\ur, \dagger}\otimes_{\QQ_p} \T[1/p], \Ee^{\ur, \dagger})^{\mathscr{G}_{K_{\pi}}}\\
	                                   &=\Hom_{\Ee^{\ur, \dagger}}((\Ee^{\dagger}\otimes_{\Ss}\Mm)^{\vee}\otimes_{\Ee^{\dagger}}\Ee^{\ur, \dagger}, \Ee^{\ur, \dagger})^{\mathscr{G}_{K_{\pi}}}\\
	                                   &=\Hom_{\Ee^{\dagger}}((\Ee^{\dagger}\otimes_{\Ss}\Mm)^{\vee}, \Ee^{\ur, \dagger})^{\mathscr{G}_{K_{\pi}}}\\
	                                   &=\big((\Ee^{\dagger}\otimes_{\Ss}\Mm)\otimes_{\Ee^{\dagger}}\Ee^{\ur, \dagger}\big)^{\mathscr{G}_{K_{\pi}}}\\
	                                   &=\Ee^{\dagger}\otimes_{\Ss}\Mm.
\end{align*}

\end{proof}

\begin{remark}($\cf$ \cite[Proposition 2.2.2]{Kis06})\label{rem N nabla action formula for TpG}
Let $\Mm\in\BT^{\varphi}_{\Ss}$ and $\CMcal{M}=\Mm\otimes_{\Ss}\Oo$ be the corresponding object in $\Mod_{\Oo}(\varphi, N),$ then in particular $\CMcal{M}\in\Mod_{\Oo}(\varphi).$ We have an operator $N_{\nabla}$ on $\CMcal{M}$ by remark \ref{rem formule N nabla}, which extends into a derivation on $\Rr\otimes_{\Oo}(\Oo\otimes_{\Ss}\Mm)\simeq \Rr\otimes_{\Ss}\Mm.$
\end{remark}

\section{Complexes with Breuil-Kisin modules}

For any $V\in \Rep_{\QQ_p}(\mathscr{G}_{K}),$ recall that we have the following complex $\Cc_{\varphi, \tau}(V),$ which computes the continuous Galois cohomology of $V$ by theorem \ref{thm main result}.

\[\xymatrix{ 
	0\ar[rr]&& \mathcal{D}(V) \ar[r]  & \mathcal{D}(V) \oplus  \mathcal{D}(V)_{\tau, 0}   \ar[r] &  \mathcal{D}(V)_{\tau, 0}   \ar[r]& 0   \\
	&& x \ar@{|->}[r]  &((\varphi-1)(x), (\tau_D-1)(x))  &&{}\\
	&&&{} (y,z) \ar@{|->}[r] &(\tau_D-1)(y)-(\varphi-1)(z).&
}  \]

In particular, we consider the representation $V=\T^{\vee}[1/p]$ (recall that $\T=\T_pG$). \ft{$\T^{\vee}$ denotes the $\ZZ_p$-linear dual of $\T$.} We have the complex $\Cc_{\varphi, \tau}(\T^{\vee}[1/p])$ as follows:

\[\xymatrix{ 
	0\ar[rr]&& \mathcal{D}(\T^{\vee}[1/p]) \ar[r]  & \mathcal{D}(\T^{\vee}[1/p]) \bigoplus \mathcal{D}(\T^{\vee}[1/p])_{\tau, 0}  \ar[r] &  \mathcal{D}(\T^{\vee}[1/p])_{\tau, 0}  \ar[r]& 0   \\
	&& x \ar@{|->}[r]  &((\varphi-1)(x), (\tau_D-1)(x))  &&{}\\
	&&&{} (y,z) \ar@{|->}[r] &(\tau_D-1)(y)-(\varphi-1)(z).&
}  \]

 Recall that we have $\mathcal{D}(\T^{\vee}[1/p])\simeq \mathcal{D}^{\ast}(\T[1/p])\simeq \Ee\otimes_{\Ss}\Mm,$ hence we can rewrite the complex $\Cc_{\varphi, \tau}(V)$ using the Breuil-Kisin module $\Mm$ associated to $\T_pG$. 

\medskip

We have a formula for $N_{\nabla}$-action over $\Mm\otimes_{\Ss}\Oo$ by remark \ref{rem N nabla action formula for TpG}, and in particular a formula for $N_{\nabla}$-action restricted on $\Mm.$ By remark \ref{rem tau action}, we have the following formula of $\tau$-action over $\Mm$ \[ \tau:=\sum\limits_{i\geq 0}\frac{(p \mathfrak{t})^i}{i!}\cdot N_{\nabla}^{(i)}\colon \Mm \to \BB^{+}_{\cris}\otimes_{\Ss}\Mm.\]
 As we have seen, it takes values in $\W(\Oo_{C^{\flat}})\otimes_{\Ss}\Mm,$ and factors through a map 
\[ \tau_D\colon \Mm\to \Ee_{\tau}\otimes_{\Ss}\Mm,  \]
which extends into 
\[\tau_D\colon \Ee_{\tau}\otimes_{\Ss}\Mm \to \Ee_{\tau}\otimes_{\Ss}\Mm.\]

\begin{definition}
	Let $\Mm\in \BT_{\Ss}^{\varphi},$ we define the complex $\Cc_{\varphi, \tau}(\Mm)$ as follows:
	
	\[\xymatrix{ 
		0\ar[rr]&& \Ee\otimes_{\Ss}\Mm  \ar[r]  & \Ee\otimes_{\Ss}\Mm \bigoplus (\Ee_{\tau}\otimes_{\Ss}\Mm)_{0}   \ar[r] &  (\Ee_{\tau}\otimes_{\Ss}\Mm)_{0}   \ar[r]& 0   \\
		&& x \ar@{|->}[r]  &((\varphi-1)(x), (\tau_D-1)(x))  &&{}\\
		&&&{} (y,z) \ar@{|->}[r] &(\tau_D-1)(y)-(\varphi-1)(z)&
	}  \]

	where $(\Ee_{\tau}\otimes_{\Ss}\Mm)_{0}:=\big\{ x\in \Ee_{\tau}\otimes_{\Ss}\Mm;\  (\gamma\otimes 1)x=(1+\tau_D+\tau_D^{2}+\cdots + \tau_D^{(\chi(\gamma)-1)})(x) \big\}.$	
\end{definition}

\begin{corollary}\label{coro chpater 6 module BK}
	Let $G\in \BT_{\Oo_K},$ denote $\T$ its Tate module and $\Mm$ its corresponding Breuil-Kisin module. Then we have
	\[\H^{i}(\mathscr{G}_{K}, \T^{\vee}[1/p])\simeq \H^i(\Cc_{\varphi, \tau}(\Mm)) \]
	for all $i\in\NN.$
\end{corollary}
\begin{proof}
	This follows from theorem \ref{thm main result}.
\end{proof}

Let $V\in \Rep_{\QQ_p}(\mathscr{G}_K)$ and $\mathcal{D}^{\dagger}_{\rig}(V)\in \Mod_{\Rr}(\varphi, N_{\nabla})$ the corresponding $(\varphi, N_{\nabla})$-module. We have the complex $\Cc_{\varphi, N_{\nabla}}(V)$ as follows:

\[\xymatrix{ 
	0\ar[rr]&& \mathcal{D}^{\dagger}_{\rig}(V)  \ar[r]  & \mathcal{D}^{\dagger}_{\rig}(V) \oplus \mathcal{D}^{\dagger}_{\rig}(V)   \ar[r] &  \mathcal{D}^{\dagger}_{\rig}(V)   \ar[r]& 0   \\
	&& x \ar@{|->}[r]  &((\varphi-1)(x), N_{\nabla}(x))  &&{}\\
	&&&{} (y,z) \ar@{|->}[r] & N_{\nabla}(y)-(c\varphi-1)(z).&
}  \]

By proposition \ref{prop H0 H1 of complex C_R(V)} we have $\H^i(\Cc_{\varphi, N_{\nabla}}(V))\simeq \iLim{n}\H^i(\mathscr{G}_{K_n}, V)$ for $i\in \{0, 1 \}.$ By lemma \ref{lemm breuil kisin Robba ring}, for \, $G\in \BT_{\Oo_K}$ we have 
\[\Rr\otimes_{\Ee^{\dagger}}\mathcal{D}^{\ast, \dagger}(\T[1/p])\simeq \Mm\otimes_{\Ss}\Rr. \] 
Let $V=\T^{\vee}[1/p]\in \Rep_{\QQ_p}(\mathscr{G}_K),$ we can rewrite the complex with corresponding Breuil-Kisin modules. 

\begin{definition}
	For any $\Mm\in \BT_{\Ss}^{\varphi},$ we define the complex $\Cc_{\varphi, N_{\nabla}}(\Mm)$ as follows:
	\[\xymatrix{ 
		0\ar[rr]&& \Mm\otimes_{\Ss}\Rr  \ar[r]  & \Mm\otimes_{\Ss}\Rr \bigoplus \Mm\otimes_{\Ss}\Rr   \ar[r] &  \Mm\otimes_{\Ss}\Rr \ar[r]& 0   \\
		&& x \ar@{|->}[r]  &((\varphi-1)(x), N_{\nabla}(x))  &&{}\\
		&&&{} (y,z) \ar@{|->}[r] & N_{\nabla}(y)-(c\varphi-1)(z)&
	}  \]
	
	where the action of $N_{\nabla}$ is given by the formula in remark \ref{rem N nabla action formula for TpG}. 
\end{definition}

\begin{corollary}
	Let $G\in \BT_{\Oo_K},$ denote $\T$ its Tate module and $\Mm$ its corresponding Breuil-Kisin module. Then we have
	\[ \H^i(\Cc_{\varphi, N_{\nabla}}(\Mm))\simeq \iLim{n}\H^i(\mathscr{G}_{K_n}, \T^{\vee}[1/p]) \text{ for } i\in \{0, 1 \}.\] 
\end{corollary}
\begin{proof}
	This follows from proposition \ref{prop H0 H1 of complex C_R(V)}.
\end{proof}